\documentclass{article}
\usepackage{amssymb}
\usepackage{amsmath}
\usepackage{arxiv}
\usepackage[utf8]{inputenc}
\usepackage[T1]{fontenc}
\usepackage{url}
\usepackage{booktabs}
\usepackage{amsfonts}
\usepackage{nicefrac}
\usepackage{microtype}
\usepackage{color}
\usepackage{lipsum}
\usepackage{amsthm}
\usepackage{indentfirst}
\usepackage{graphicx}
\usepackage{epstopdf}
\usepackage{fourier}
\usepackage{bm}
\usepackage{mathtools}
\usepackage{latexsym,enumerate}
\usepackage{multicol}
\usepackage[skip=2pt]{caption}
\usepackage[font=small,skip=0pt]{subcaption}
\usepackage{array}
\usepackage{cancel}

\def\bm{\boldsymbol}

\newcommand{\comment}[1]{}

\newcommand{\BEA}{\begin{eqnarray}}
\newcommand{\EEA}{\end{eqnarray}}

\newtheorem{thm}{Theorem}[section]
\newtheorem{prop}[thm]{Proposition}
\newtheorem{example}[thm]{Example}
\newtheorem{lem}[thm]{Lemma}

\newtheorem{defn}[thm]{Definition}

\newtheorem{rem}[thm]{Remark}
\newtheorem{algm}[thm]{Algorithm}

\newcommand{\PreserveBackslash}[1]{\let\temp=\\#1\let\\=\temp}
\newcolumntype{C}[1]{>{\PreserveBackslash\centering}p{#1}}
\newcolumntype{R}[1]{>{\PreserveBackslash\raggedleft}p{#1}}
\newcolumntype{L}[1]{>{\PreserveBackslash\raggedright}p{#1}}

\begin{document}

\title{Ghost Point Diffusion Maps for solving elliptic PDE's on
Manifolds with Classical Boundary Conditions}
\lhead{Ghost Point Diffusion Maps}
\author{ Shixiao Willing Jiang \\
%EndAName
Institute of Mathematical Sciences \\
ShanghaiTech University, Shanghai 201210, China\\
\texttt{{jiangshx@shanghaitech.edu.cn}} \\
\And John Harlim \\
Department of Mathematics, Department of Meteorology and Atmospheric
Science, \\
Institute for Computational and Data Sciences \\
The Pennsylvania State University, University Park, PA 16802, USA\\
\texttt{jharlim@psu.edu} }
\maketitle

\begin{abstract}
In this paper, we extend the class of kernel methods, the so-called diffusion maps (DM), and its local kernel variants, to approximate second-order differential operators defined on smooth manifolds with boundaries that naturally arise in elliptic PDE models. To achieve this goal, we introduce the Ghost Point Diffusion Maps (GPDM) estimator on an extended manifold, identified by the set of point clouds on the unknown original manifold together with a set of ghost points, specified along the estimated tangential direction at the sampled points at the boundary. The resulting GPDM estimator restricts the standard DM matrix to a set of extrapolation equations that estimates the function values at the ghost points. This adjustment is analogous to the classical ghost point method in finite-difference scheme for solving PDEs on flat domain. As opposed to the classical DM which diverges near the boundary, the proposed GPDM estimator converges pointwise even near the boundary. Applying the consistent GPDM estimator to solve the well-posed elliptic PDEs with classical boundary conditions (Dirichlet, Neumann, and Robin), we establish the convergence of the approximate solution under appropriate smoothness assumptions. We numerically validate the proposed mesh-free PDE solver on various problems defined on simple sub-manifolds embedded in Euclidean spaces as well as on an unknown manifold. Numerically, we also found that the GPDM is more accurate compared to DM in solving elliptic eigenvalue problems  on bounded smooth manifolds.
 \end{abstract}


\section{Introduction}\label{intro}

% Motivate the elliptic PDE and other approaches on manifold
Elliptic Partial Differential Equations (PDEs) \cite{han2011elliptic} arise naturally in the modeling of physical phenomena, including groundwater flow \cite{mclaughlin1996reassessment}, heat conduction \cite{Feynman}, neutron diffusion \cite{Neutron}, and probability theory \cite{MoPer}. In the manifold setting, {solving the PDE formulation arises in modeling of granular flow \cite{rauter2018finite}, liquid crystal \cite{virga2018variational}, biomembranes \cite{elliott2010modeling}. In computer graphics \cite{bertalmio2001variational}, PDEs on surfaces have been used to restore damaged patterns on a surface \cite{macdonald2010implicit}, brain imaging \cite{memoli2004implicit}, among other applications.}

% FEM
Many numerical methods have been proposed to approximate the solution of PDE on the manifold setting, especially for two-dimensional surfaces. Most of these methods, however, require a parameterization of the surface, which is subsequently used to approximate the tangential derivatives along the surface. For example, the finite element method (FEM) represents the surface \cite{dziuk2013finite,camacho,bonito2016high} using triangular meshes. {\color{black}Subsequently, the PDE is solved by a Galerkin truncation on the finite-element space of functions defined on the triangular meshes. While this classical approach is popular and has been widely used in applications, it relies on the accuracy of the generated meshes. In addition to the computational task in the mesh generation, given an arbitrary set of point cloud data that lie on the manifold, constructing meshes that avoid inconsistent tangential triangulation \cite{boissonnat2010triangulating} can be challenging.

% Ambient space methods
An alternative approach is to embed the surface PDE problem to the ambient space $\mathbb{R}^n$ such that the solution of the embedded PDE problem is consistent with the original problem when restricted to $M$. One realization of such an approach is to use a level set representation \cite{bertalmio2001variational} for the surface, and subsequently, solve the embedded PDE equation in $\mathbb{R}^n$ using any standard method that works in the Euclidean domain. The level set representation, unfortunately, can lead to degenerate diffusion equations, in addition to many other limitations pointed out in \cite{ruuth2008simple}. To combat the limitations of the level set representation, the authors in \cite{ruuth2008simple} introduced the closest point representation of the surface $M$. We should point out that it is unclear how this method will perform if we are only given randomly sampled point cloud data since these points may not be the closest point. In their papers \cite{ruuth2008simple,petras2018rbf}, they tested their scheme on examples where either the analytical formula for the closest point is given or the surface has a triangular representation. Besides this minor technical issue, a bigger problem with this class of approaches is that the computational cost scales with respect to the ambient dimension-$n$. This is because the embedded PDE is solved in the ambient space $\mathbb{R}^n$, which is at least one dimension more than, for example, the two-dimensional surface $M$.

% RBF
Another class of approaches is the mesh-free radial basis function (RBF) method. While several versions of RBF solvers have been proposed \cite{piret2012orthogonal,fuselier2013high}, the key idea is to identify normal vectors at each point cloud and approximate the tangential derivative at each point cloud using the radial basis function interpolation method. In \cite{fuselier2013high}, the tangential derivatives are defined by projecting the gradient in $\mathbb{R}^n$ to the tangent space. One of the key issues with this approach is that the shape parameter of the radial basis function can be difficult to tune for high co-dimensional problems as pointed out in \cite{fuselier2013high}. Another issue that is directly related to the work in the present paper is the erratic behavior near the boundary. As far as our knowledge, the issue near boundary has only been studied on flat domains in $\mathbb{R}^n$ \cite{bayona2017role}. That work concluded that one can achieve highly accurate solutions by appropriate choice of radial basis functions with sufficiently large data. However, it is unclear how to extend their approach in the context of unknown manifolds since we cannot sample more data, let alone controlling the size of the data. In the same paper \cite{bayona2017role}, the authors also numerically demonstrated that their approach can be as effective as using the ghost points extension. While the ghost point method is computationally straightforward on flat domains, an extension to unknown non-flat geometry is a non-trivial task. The present work will introduce a numerical scheme to realize this non-trivial task and study the convergence of the approximation when it is used with the following PDE solver.
}

% Gaussian integral approximation
In this paper, we consider approximating the intrinsic second-order elliptic differential operators directly on the point clouds that lie on the manifold. Our approach rests on the fact that away from the boundary, these differential operators can be approximated by integral operators defined with appropriate Gaussian kernels, which is the theoretical underpinning of the popular nonlinear manifold learning algorithm known as the Diffusion Maps \cite{coifman2006diffusion} and its local kernel variants \cite{berry2016local}. The main advantages of this approach are that it is a consistent estimator of the intrinsic PDE problem even for sub-manifold of arbitrary co-dimension and it can naturally handle randomly distributed point cloud data. Computationally, this mesh-free algorithm does not require a parameterization of the manifold and/or an estimation of the normal vectors at each point cloud, one of which is essential in the existing approaches discussed in the previous paragraphs. As of the authors' knowledge, the idea of using such kind of integral operator for solving PDEs was first numerically realized by the Point Integral Method (PIM) for solving Poisson problems \cite{li2017point} and isotropic elliptic equations \cite{li2016convergent}. In  separate works, the same idea was realized with the Diffusion Maps (DM) algorithm \cite{berry2016local} for solving elliptic PDEs with non-symmetric advection-diffusion (Kolmogorov) operators associated with It\^o diffusion \cite{gh2019} and anisotropic diffusion \cite{harlim2019kernel}. We should point out that despite having the same vein, DM and PIM approaches are not identical as pointed out in \cite{gh2019}.

% problem with boundary
On manifolds with boundary, however, the Neumann problem is the only natural boundary condition for the Gaussian kernel integral approximation as noted in \cite{coifman2006diffusion}. Furthermore, as we shall see in this paper, even if the function satisfies the Neumann boundary condition, the diffusion maps integral approximation does not converge in the pointwise sense at interior points close to the boundary. For other types of boundary conditions, there are several approaches have been proposed. For example, the PIM approximates the Dirichlet problem with an artificial Robin boundary condition with a small first-order derivative term \cite{li2017point}. Another approach is to use a volume constraint \cite{shi2015enforce}, which is a simple version of the ghost point method that is proposed in the present paper, by setting the function values at the ghost points to be zero. In \cite{thiede2019galerkin}, they proposed an empirical approach for the Dirichlet problem by appending the discrete representation of the integral approximation at the interior points with a discrete representation of the Dirichlet boundary condition. Recent work in \cite{vaughn2019diffusion} suggests that the diffusion maps asymptotic expansion is a consistent estimator of Laplacian of bounded manifold in a weak sense and they devised a boundary integral estimator to specify the desired boundary conditions. All of these approaches, however, do not improve the integral approximation on the interior points near the boundary in the pointwise sense, and it is unclear whether they can be extended to the Robin boundary condition.

In this paper, we introduce the Ghost Point Diffusion Maps (GPDM) as a consistent estimator in the sense of pointwise, complementary to the weak sense result in \cite{vaughn2019diffusion}. The GPDM modifies the DM algorithm by a novel ghost points extension scheme, generalizing the classical ghost point method on flat domains to unknown submanifolds of $\mathbb{R}^n$. For readers convenience, let us recall the basic idea of the ghost point method in the finite difference setting for solving the Neumann boundary value problem: $u"(x)=f(x), x\in(0,1), u'(0)=u'(1)=g$. Suppose the domain is discretized as follows, $\{x_j = jh\,:\,j=0,\ldots,N, \, h=1/N\}$. Let $U_j$ denotes the finite-difference approximation to the solution, $u(x_j)$. Instead of using the one-side first-order finite-difference, consider a center-difference approximation for the boundary condition,
\BEA
u'(0) \approx \frac{U_1 -U_{-1}}{h}= g,\nonumber
\EEA
where we have introduced a new unknown, $U_{-1}\approx u(x_{-1})$ at a ghost point, $x_{-1}:=-h \notin [0,1]$. The standard ghost point method (see e.g. \cite{leveque2007finite}) specifies this function value by an additional equation that effectively imposes the PDE at the boundary point,
\BEA
\frac{1}{h^2}(U_{-1}-2U_{0}+U_1) \approx u''(x_0) = f(x_0).\label{extrapolation}
\EEA
Notice that the two key steps in this method, the specification of the ghost point $x_{-1}$ and the extrapolation of the function value $U_{-1}$, are not immediately trivial when the manifold is not a flat geometry and unknown. In the present work, we devise an algorithm to estimate the normal vectors, which in turn, allows one to carry the two key steps above along the estimated normal vectors on each point at the boundary. The proposed method uses no information of the geometry other than the available point cloud data that are possibly randomly distributed. We show that the proposed GPDM is a pointwise convergent estimator even for points close to the boundary when the function values at the ghost points are extrapolated with a set of equations that resemble matching the second-order derivatives in addition to an equation that resembles the condition in \eqref{extrapolation}. Subsequently, we apply the GPDM to solve elliptic PDEs with Dirichlet, Neumann, and Robin boundary conditions. Through theoretical analysis and numerical studies, we show that the proposed solver is a uniform convergent scheme. We also numerically show that GPDM is more accurate compared to DM in terms of solving eigenvalue problems.

% organization of the paper
The paper will be organized as follows. In Section~\ref{sec:dm}, we provide a short review of diffusion maps and its local kernel variants to approximate various types of linear second-order elliptic differential operators defined on smooth manifolds embedded in $\mathbb{R}^n$. We end the section with an example, illustrating the problem of DM near the boundary. In Section~\ref{sec:GPDM}, we present the GPDM method which overcomes the issue near the boundary. We close this section with numerical examples to support the theoretical results. In Section~\ref{sec:PDE}, we discuss the application of GPDM for solving elliptic PDEs with various boundaries. In Section~\ref{sec:EIGS}, we discuss the application of GPDM for solving eigenvalue problems corresponding to the elliptic PDEs. We close the paper with a summary and a list of open problems in Section~\ref{sec:summary}.
To improve the readability, we report the detailed proofs in several appendices.

\section{Diffusion maps and its extension with local kernels}\label{sec:dm}

In this section, we provide a short review of the diffusion maps algorithm
\cite{coifman2006diffusion} as a method to approximate the Laplacian, a class of
second-order, self-adjoint, positive-definite, differential operators that
acts on functions defined on smooth compact Riemannian manifolds. In addition,
we also review the variant of diffusion maps to approximate the
second-order elliptic diffusion operator with a given diffusion coefficient
\cite{harlim2019kernel} and the non-symmetric drifted diffusions via the local
kernels \cite{berry2016local}.

Let $M$ be a {\color{black}$C^\infty$}, $d-$dimensional compact Riemannian manifold embedded in $%
\mathbb{R}^{n}$, possibly with boundary $\partial M$. Let $u\in C^{3}(M)$
and $\epsilon >0$, for all $x\in M$ whose distance from the
boundary is larger than $\epsilon ^{r}$, where $0<r<1/2$, the integral
operator,
\begin{equation}
G_{\epsilon }u(x):=\epsilon ^{-d/2}\int_{M}\exp\Big(-\frac{|x-y|^{2}}{4\epsilon }%
\Big)u(y)dV(y)=\epsilon ^{-d/2}\int_{M_{\epsilon ,x}}\exp\Big(-\frac{|x-y|^{2}}{%
4\epsilon }\Big)u(y)dV(y)+\mathcal{O}(\epsilon ^{2})  \label{integralop1}
\end{equation}%
is effectively a local integral operator over the $\epsilon ^{r}$-ball
around $x$, $M_{\epsilon ,x}:=\{y\in M,|x-y|<\epsilon ^{r}\}$. In %
\eqref{integralop1}. The notation $|\cdot |$ denotes the standard Euclidean
norm for vectors in $\mathbb{R}^{n}$. The key idea of the diffusion maps
algorithm lies on the following asymptotic expansion. For any points $x\in
M$ whose distance from the boundary is larger than $\epsilon ^{r}$%
, where $0<r<1/2$,
\begin{equation}
G_{\epsilon }u(x)=m_{0}u(x)+\epsilon m_{2}\big(\omega (x)u(x)+\Delta _{g}u(x)%
\big)+\mathcal{O}(\epsilon ^{2}),  \label{asymp}
\end{equation}%
where $m_{0}$ and $m_{2}$ are constants that depend on {\color{black}the kernel}, $\omega $ depends also on the geometry of $M$, and $\Delta
_{g}$ denotes the negative-definite Laplace-Beltrami operator defined with
respect to the Riemannian metric $g$ inherited by $M$ from $\mathbb{R}^{n}$.
We should point out that with our choice of the exponential kernel, one can verify that $m_0=m_2$.
Based on this asymptotic expansion, one can approximate the Laplace-Beltrami
operator as,
\begin{equation}
L_{1,\epsilon }u(x):=\frac{(G_{\epsilon }1(x))^{-1}G_{\epsilon }u(x)-u(x)}{%
\epsilon }=\Delta _{g}u(x)+\mathcal{O}(\epsilon ):=\mathcal{L}_{1}u(x)+%
\mathcal{O}(\epsilon ),  \label{L1}
\end{equation}%
for all $x\in M$ whose distance from the boundary is larger than $%
\epsilon ^{r}$, where $0<r<1/2$. If one is given a strictly positive, smooth,
diffusion coefficient $\kappa :M\rightarrow (0,\infty )$, one can also
approximate the anisotropic diffusion operator,
\begin{equation}
L_{2,\epsilon }u(x):=\kappa(x)\frac{\big(G_{\epsilon }\sqrt{\kappa (x)}\big)%
^{-1}G_{\epsilon }(\sqrt{\kappa (x)}u(x))-u(x)}{\epsilon }=\mbox{div}_{g}%
\big(\kappa (x)\nabla _{g}u(x)\big)+\mathcal{O}(\epsilon ):=\mathcal{L}%
_{2}u(x)+\mathcal{O}(\epsilon ),  \label{L2}
\end{equation}%
where we have used the notations $\mbox{div}_{g}$ and $\nabla _{g}$ for the
divergence and gradient operators, respectively, defined with respect to the
Riemannian metric $g$. One can also apply the equivalent diffusion operator
using the symmetric version as reported in \cite{harlim2019kernel}.

Beyond these two self-adjoint operators, one can also approximate the
backward Kolmogorov operator,
\begin{eqnarray}
\mathcal{L}_{3}u:=b \cdot \nabla_g u+\frac{1}{2}c^{ij}\nabla _{i}\nabla
_{j}u,  \label{Eqn:L3}
\end{eqnarray}
where $\nabla _{i}$ is the covariant derivative in the $i$th direction, and $%
\nabla _{i}\nabla _{j}$ is the component of the Hessian operator. Here, the
differential operators and the dot product are defined with respect to the
Riemannian metric inherited by $M$ from $\mathbb{R}^{n}$. The vector field $%
b:M\rightarrow \mathbb{R}^{d}$ is the drift and the symmetric positive
definite diffusion tensor $c:M\rightarrow \mathbb{R}^{d}\times \mathbb{R}%
^{d} $ is a $d\times d$\ diffusion matrix, where $d$ is the dimension of
manifold $M$.

The operator in \eqref{Eqn:L3} can be accessed by employing the integral
operator in \eqref{integralop1} with the following prototypical kernel \cite%
{berry2016local}:
\begin{eqnarray}
K(\epsilon,x,y):= \exp \left(-\frac{(x+\epsilon B(x) - y)^\top
C(x)^{-1}(x+\epsilon B(x) - y)}{2\epsilon}\right)  \label{Eqn:localK}
\end{eqnarray}
where $B:$ $M\rightarrow \mathbb{R}^{n}$ and $C:M\rightarrow \mathbb{R}%
^{n}\times \mathbb{R}^{n}$ are related to $b$ and $c$, respectively, through
{\color{black}a local parameterization $\iota:U\subseteq \mathbb{R}^d\rightarrow M\subseteq\mathbb{R}^n$ of the manifold $M$} as follows:
\begin{equation}
B\left( x\right) =D\iota \left( x\right) b\left( x\right) ,\text{ \ \ }%
C(x)^{-1}=\left( D\iota \left( x\right) c\left( x\right) D\iota \left(
x\right) ^{{\top }}\right) ^{\dag }.  \label{Eqn:BC}
\end{equation}%
{\color{black}Here, the set $U\subseteq \mathbb{R}^d$ denotes a domain that contains $\iota^{-1}(x)$.}
Here, the notation $^{\dag }$\ denotes the pseudo-inverse and the
{\color{black}differential map $D\iota \left( x\right): T_{\iota^{-1}(x)}M \subseteq \mathbb{R}^{d} \to T_x\mathbb{R}^{n} \subseteq \mathbb{R}^{n}$ is an $n\times d$ matrix that is usually known as the Jacobian (or pushforward) corresponding to the map $\iota$.}
Applying the
integral operator in \eqref{integralop1} with the prototypical kernel $%
K(\epsilon,x,y)$ on manifold without boundary, we obtain,
\begin{eqnarray}
G_{K,\epsilon }u\left( x\right) :=\epsilon ^{-d/2}\int_{M}K\left(
\epsilon ,x,y\right) u\left( y\right) dV_{y}=m\left( x\right) u\left(
x\right) +\epsilon \left( \omega \left( x\right) u\left( x\right) + \mathcal{L}_{3}u\left( x\right) \right) +O\left( \epsilon
^{2}\right) ,  \label{Eqn:Gu}
\end{eqnarray}%
where $m\left( x\right) =\left( 2\pi \right) ^{d/2}\det \left( c\left(
x\right) \right) ^{1/2}$ can be approximated by $G_{K,\epsilon} 1(x) =
m(x) +\mathcal{O}(\epsilon)$. Employing the same algebraic manipulation as
in \eqref{L1}, we obtain
\begin{eqnarray}
L_{3,\epsilon} u(x) := \frac{(G_{K,\epsilon}1(x))^{-1}G_{K,\epsilon} u(x) -
u(x)}{\epsilon} = \mathcal{L}_3 u(x) + \mathcal{O}(\epsilon).  \label{L3}
\end{eqnarray}
We note that the evaluation of the prototypical kernel in \eqref{Eqn:localK}
requires the knowledge of either the intrinsic representation $b$ and $c$
together with the embedding function $\iota $ or the ambient representation $%
B$ and $C$ in \eqref{Eqn:BC}. 

Numerically, given a set of points in ambient coordinate $\{x_i\in M)\}_{i=1}^N$, which is also referred as the point cloud data, one can approximate the integrals in $L_{1,\epsilon}$ or $%
L_{2,\epsilon}$ or $L_{3,\epsilon}$ via a Monte-Carlo average, accounting for
the sampling density of the data $x_i\sim q(x)$ that are not necessarily
uniformly distributed. In particular, the function $G_{\epsilon,q}u :=G_\epsilon uq$, where $G_\epsilon$ is given in \eqref{integralop1}, can be approximated by the following Monte-Carlo average,
\BEA
G_{\epsilon,q} u(x_i) = \epsilon^{-d/2} \int_{M}\exp\Big(-\frac{|x_i-y|^{2}}{4\epsilon }%
\Big)u(y)q(y)dV(y) \approx \frac{ \epsilon^{-d/2}}{N} \sum_{j=1}^N \exp\Big(-\frac{|x_i-x_j|^{2}}{4\epsilon }
\Big)u(x_j).\nonumber
\EEA
Define also $q_\epsilon = G_{\epsilon,q}1$ as an estimator for the unknown sampling density $q$. Based on the asymptotic expansion in \eqref{asymp}, one can deduce,
\BEA
\frac{G_{\epsilon,q}(q_\epsilon^{-1}) G_{\epsilon,q}(uq_\epsilon^{-1}) - u}{\epsilon} = \Delta_g u + \mathcal{O}(\epsilon).\nonumber
\EEA
Compare to \eqref{L1}, the algebraic expression above involves a ``right-normalization'' to overcome the biases induced by non-uniform sampling density $q$ (see \cite{coifman2006diffusion,berry2016local,harlim2018} for the detailed discussion). For the non-symmetric operator, $\mathcal{L}_3$, one can repeat the same procedure as above using the non-symmetric kernel in \eqref{Eqn:localK} but estimate the sampling density $q$ using the symmetric Gaussian kernel to avoid estimating the normalization factor $m(x)$ in \eqref{Eqn:Gu} (see \cite{gh2019} for the detailed discussion).

Now we discuss the discrete estimator for $\mathcal{L}_2$, which involves an
importance sampling to debias the effect of the sampling density of the data. To
compute $G_\epsilon \sqrt{\kappa(x)}$, we first construct an $N\times N$
matrix with entry $\mathbf{K}_{ij} = \exp\big(-\frac{|x_i-x_j|^2}{4\epsilon}\big)$. Then, the estimated
unnormalized density evaluated at $x_i$ can be estimated by the $i$th
component of vector $\mathbf{q}$, that is, $q(x_i)\approx \mathbf{q}_i=\epsilon^{-d/2}N^{-1}
\sum_{j=1}^N \mathbf{K}_{ij}$. Subsequently, we define
\begin{eqnarray}
\epsilon^{d/2} G_{\epsilon} \sqrt{\kappa(x_i)} &=& \int_{M} \exp\Big(-\frac{%
|x_i-y|^2}{4\epsilon}\Big) \sqrt{\kappa}(y) dV(y) \approx \frac{1}{N}%
\sum_{j=1}^N \mathbf{K}_{ij} \frac{\sqrt{\kappa(x_j)}}{{\mathbf{q}}_j},
\label{Gerootkappa} \\
\epsilon^{d/2} G_{\epsilon} (u(x_i)\sqrt{\kappa(x_i)}) &=& \int_{M} \exp\Big(-
\frac{|x_i-y|^2}{4\epsilon}\Big) u(y)\sqrt{\kappa}(y) dV(y) \approx \frac{1}{%
N}\sum_{j=1}^N \mathbf{K}_{ij} \frac{\sqrt{\kappa(x_j)}u(x_j)}{{\mathbf{q%
}}_j}.  \notag
\end{eqnarray}
Defining $\mathbf{W}$ as an $N\times N$ matrix with entries $%
\mathbf{W}_{ij} = \mathbf{K}_{ij} \frac{\sqrt{\kappa(x_j)}}{{\mathbf{q}}%
_j}$, let $\mathbf{D}$ be a diagonal matrix with diagonal entries $%
\mathbf{D}_{ii} = \sum_{j=1}^N \mathbf{W}_{ij}$ and $\mathbf{S}$ be a diagonal matrix with diagonal entries $\mathbf{S}_{ii}=\kappa(x_i)$, then the discrete estimator
for $\mathcal{L}_{2}$ is given by,
\begin{eqnarray}
L_{2,\epsilon} \approx \mathbf{L}_2 =\frac{1}{\epsilon}\mathbf{S}\Big({\mathbf{D}}^{-1}{%
\mathbf{W}} - {\mathbf{I}}\Big).\label{discreteL2}
\end{eqnarray}
We should point out that the discrete estimator converges pointwisely,  $L_{j,\epsilon} \to \mathcal{L}_{j}$ (for each $%
j=1,2,3$) with high probability \cite{SingerEstimate,bh:16vb,gh2019}. For convenience, we state this result in Lemma~\ref{lem:old}. For the symmetric cases, $\mathcal{L}_1$ and $\mathcal{L}_2$, the spectral convergence results are also available for closed manifolds \cite{trillos2019error,bs2019consistent,calder2019improved} in $L^2$-sense and \cite{dunson2019spectral,calder2020lipschitz} in $L^{\infty}$-sense, all of which are valid in high probability.

{\color{black}
{\bf Parameter Specification:\label{paraspec}} To achieve accurate estimations, one needs to specify the appropriate bandwidth parameter, $\epsilon$. For efficient implementation, we also use $k-$nearest neighbor algorithm to avoid computing the distances of pair of points that are sufficiently large.

Our choice of $\epsilon$ follows the method that was originally proposed in \cite{coifman2008TuningEpsilon}. Basically, the idea relies on the following observation,
\BEA
S(\epsilon):=\frac{1}{Vol(M)^2}\int_M \int_{T_xM} \exp\Big(-\frac{|x-y|^2}{4\epsilon}\Big) dy\,dV(x) = \frac{1}{Vol(M)^2}\int_M (4\pi\epsilon)^{d/2} dV(x) = \frac{(4\pi\epsilon)^{d/2}}{Vol(M)}.\label{scalingS}
\EEA
Since $S$ can be approximated by a Monte-Carlo integral, for a fixed $k$, we approximate,
\BEA
S(\epsilon) \approx \frac{1}{Nk}\sum_{i,j=1}^{N,k} \exp\Big(-\frac{|x_i-x_j|^2}{4\epsilon}\Big),\nonumber
\EEA
where $\{x_j\}_{j=1}^k$ is the $k$-nearest neighbors of each  $x_i$. We choose $\epsilon$ from a domain (e.g., $[2^{-14},10]$ in our numerical implementation) such that $\frac{d\log(S)}{d\log\epsilon} \approx \frac{d}{2}$. Numerically, we found that the maximum slope of $\log(S)$ often coincides with $d/2$, which allows one to use the maximum value as an estimate for the intrinsic dimension $d$ when it is not available and choose the corresponding $\epsilon$.

For well-sampled data, we choose $k<N$ to be large enough (usually between $50-200$, depending on the size of the data) such that $\epsilon$ is smaller than the distance between $x_i$ and its $k$-neighbor, for all $i=1,\ldots, N$. With this choice, we numerically obtain $\epsilon = \mathcal{O}(N^{-\frac{2}{d}})$, for $d=1,2$, which shows accurate estimates that converge. For randomly distributed data, we set $k = \mathcal{O}(N^{1/2})$ and obtain $\epsilon =\mathcal{O} \big(\frac{k}{N}\big)^{2/d} =\mathcal{O} (N^{-1/d})$, which yields a much larger $\epsilon$ compared to the choice in the well-sampled data, that is, $N^{-1/d}> N^{-2/d}$ for $N\gg1$ and $d\geq 1$. It is worthwhile to point out that the scaling $\epsilon =\mathcal{O} \big(\frac{k}{N}\big)^{2/d}$, which we empirically found to produce convergence solutions in randomly distributed data (as we shall show later), has also been documented as a condition for the pointwise convergence estimate (see Theorem 3.6 of \cite{calder2019improved}).

}

Now, let us illustrate the problem
near the boundary of the asymptotic approximation of the weighted Laplacian in \eqref{L2} with a simple example.
\begin{example}\label{1Dexample}
In this example, we compare the {DM and GPDM estimates} of the differential
operator $\mathcal{L}_2$ on a one-dimensional ellipse {$x=(x_1,x_2)=(\cos \theta, a\sin \theta)$,}
defined with the Riemannian metric,
\begin{equation}
{g =\sin ^{2}\theta +a^{2}\cos ^{2}\theta ,\text{ \ \
for\ }0\leq \theta \leq \pi ,}  \label{Eqn:ellipse_metr}
\end{equation}%
where $a=3>1$. The diffusion coefficient in the weighted Laplacian \eqref{L2} is chosen to be
{$\kappa:=1.1+x_2/a =1.1+\sin \theta $}. In local coordinates, the
diffusion operator acting on function $u$ is given as,
\begin{eqnarray}
\mathcal{L}_{2}u&:=&\mathrm{div}\left( \kappa \nabla u\right) =\frac{1}{%
\sqrt{\vert g\vert}}\frac{\partial }{\partial \theta }\left( \sqrt{\left\vert
g\right\vert }\kappa g^{-1}\frac{\partial u}{\partial \theta }\right).
\label{explicitLs}
\end{eqnarray}%
In Fig.~\ref%
{fig1_ellip_noghost}, we plot the explicit equation in \eqref{explicitLs}
acting on a test function {$u(x)=\cos (3\theta/2 - \pi /4 )$, defined on a
semi-ellipse with $a=3$ and  $\theta \in
 [0,\pi ]$ being the intrinsic coordinate.} The discrete estimator, $\mathbf{L}_2$, of $\mathcal{L}_2$ is constructed using $N=400$
data points distributed at equal angle. Notice the agreement between the
DM estimate and the truth except near the boundaries. In the same
figure, we also show the improved estimate using the Ghost Point Diffusion
Maps (GPDM) near the boundaries that we will explain in the next section.

\end{example}

\begin{figure*}[tbp]
{\scriptsize \centering
\begin{tabular}{C{7cm}C{8cm}}
\normalsize (a) $\mathcal{L}_2 u$  & \normalsize (b) $\mathrm{FE}$ of DM=13 (upper), ${\mathrm{FE}} $ of GPDM=0.04 (bottom) \\
\includegraphics[width=.44\textwidth]{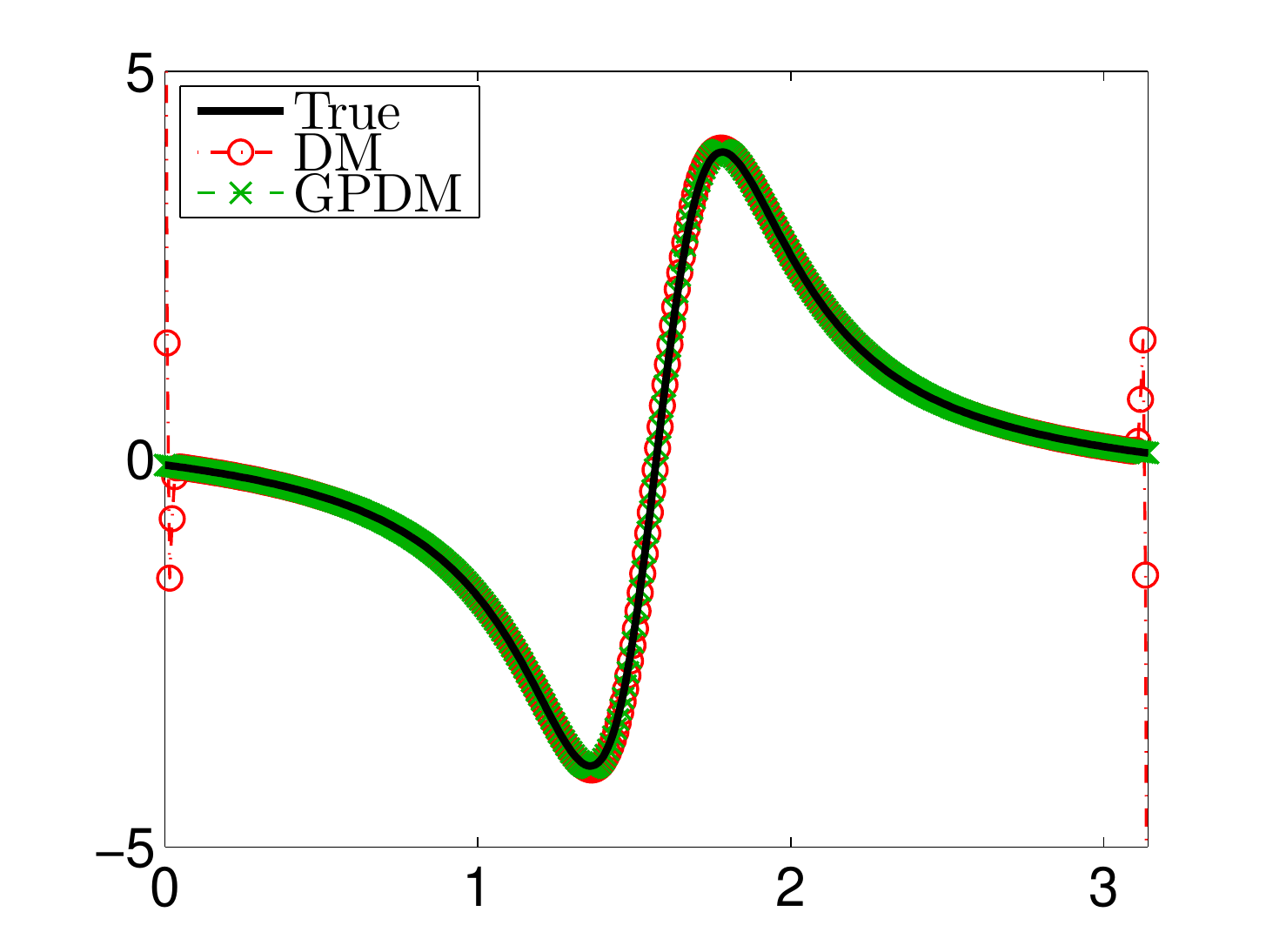}
&
\includegraphics[width=.44\textwidth, height=.35\textwidth]{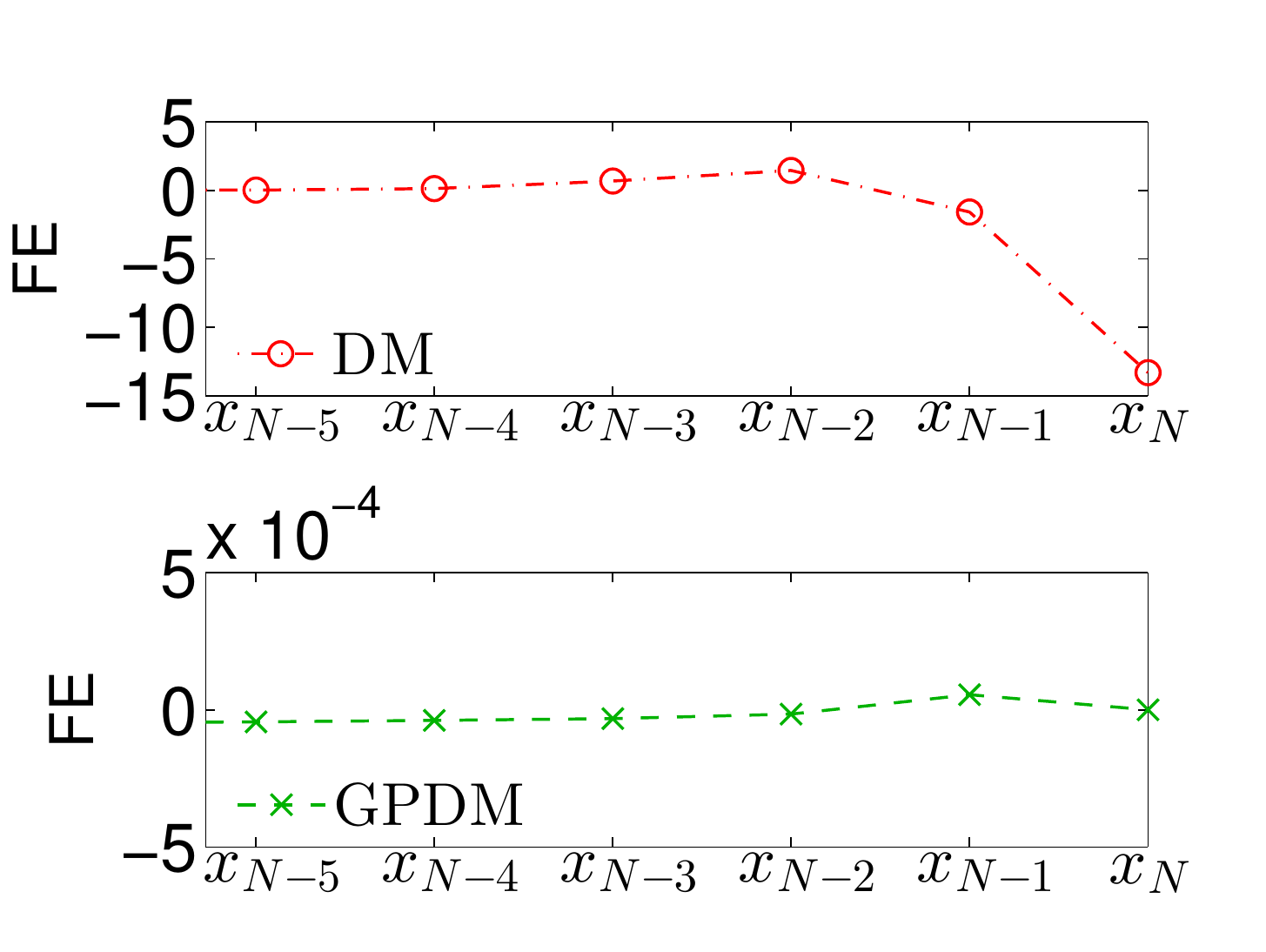}
\end{tabular}
}
\caption{(Color online) Numerical approximations of $\mathcal{L}_2 u$ (the weighted Laplacian in \eqref{L2}) on a semi-ellipse example {\color{black}with $N=400$}. (a) Comparison of {the true $%
\mathcal{L}_2 u$ and its DM and GPDM estimates.} (b) Absolute error of DM (upper panel) and GPDM (bottom panel) near the boundary. One can see that the Forward Error (FE), {defined as $\left\Vert \mathcal{L}_2 u - \mathbf{L}_2 u\right\Vert_\infty$ with the uniform norm,} using the standard DM is relatively large up to 13 near the boundary [red circles in upper panel of (b)]. However, by applying the
GPDM, the FE reduces to 0.04 [green crosses in bottom
panel of (b)]. Note that for GPDM, the FE does not reach its maximum near the boundary but in the interior of the domain instead. {\color{black}
In the bottom of (b), one can see that the FE is very small near the boundary for GPDM.}}
\label{fig1_ellip_noghost}
\end{figure*}

\section{Ghost Point Diffusion Maps for 1D and 2D manifolds}\label{sec:GPDM}

In this section, we introduce an improved method, the ghost point diffusion maps, for approximating differential operators in \eqref{L1}, \eqref{L2}, \eqref{Eqn:L3} defined on one and two-dimensional manifolds with boundaries. To facilitate the discussion, we use the conventional notations $\partial M$ and $M^o$ to denote the boundary and interior sets of
manifold $M$, respectively, that satisfy $M=M^o\cup \partial M$ and $M^o \cap \partial M=\emptyset$. We assume that $M$ is a $C^\infty$-smooth, compact domain such that the closed
subset $\partial M$ is also a compact set. For two-dimensional problems, we also assume that the boundary $\partial M$ is a smooth regular curve with additional conditions (which will be clarified in Section~\ref{ghost_ptm}) such that it is extendable along the boundary by a normal collar with radius $R=\mathcal{O}(\epsilon^r)$, for $0<r<1/2$.

The basic idea here is to follow the classical ghost point method \cite{leveque2007finite} for solving the Neumann boundary condition with the finite-difference method on flat domain, as reviewed in Section~\ref{intro}. In our configuration, we supplement ghost points near the boundary such that the diffusion maps asymptotic expansion for the estimation of the diffusion operator is valid even for points near the boundary, {where the second-order differential operator is approximated with an appropriate affine linear operator.} In this work, we assume that we have sample points at the boundary. For problems with unknown boundary points, one can use the tools developed in \cite{berry2017density} to estimate points at the boundary.

We now describe the proposed algorithm, which we refer as the Ghost Point Diffusion Maps  (GPDM). Particularly, the construction of the GPDM requires the following technical tools. In Section.~\ref{bound_estimate}, we estimate the exterior normal vector $\boldsymbol{\nu }$ to the boundary. In Section~\ref%
{est_directional_deri}, we estimate the normal derivative $\partial _{%
\boldsymbol{\nu }}u$ at ${x\in\partial M}$, which will be used for specifying the boundary conditions. In Section~\ref{ghost_ptm},\ we describe the construction of the ghost points along the normal direction $\boldsymbol{\nu }$ from
boundary points. In Section~\ref{artBC_ghost}, we discuss how to extrapolate the unknown function values at the ghost points. Here, we introduce a set of algebraic conditions on the ghost points, which ensure the consistency of the affine estimator of $\mathcal{L}_j$ in the limit of $\epsilon\to 0$ after $N\to\infty$, as reported in Section~\ref{augmentm}. Finally, we give numerical examples to validate the theory in Section~\ref{numericforward}.

\subsection{Estimation of the exterior normal direction at the boundaries}

In this section, we provide numerical methods to estimate the exterior normal direction using the point cloud data, assuming that the boundary points are given. {\color{black}We split the discussion into two subsections, concerning the well-sampled and randomly sampled data, as they require different algorithms.}

\label{bound_estimate}

\subsubsection{Well-sampled data}

{\color{black}We start our discussion on 1D manifolds. By well-sampled data, we mean that the data points are well-ordered and every consecutive points have equal (intrinsic) distance. For example, Fig. \ref{fig22_1dcurve}(a) shows the dataset $\{x_{i}\}_{i=1,\ldots
,N}$, well-ordered on a 1D semi-ellipse with $x_{1}$ and $x_{N}$ as the boundary points.} Suppose that $\gamma:\mathbb{R}\to M\subseteq\mathbb{R}^n$ is a geodesic parameterization of the one-dimensional manifold $M$ with based point $\gamma(0)=x_1\in\partial M$ and $\gamma(s)=x_2$ (see Fig.~\ref{fig22_1dcurve}(b)). The arc-length parameterization $s=\int_{0}^{s}|\gamma ^{\prime }(t)|dt,$ such that $|\gamma ^{\prime }(t)|=1$ for all $t\in \lbrack 0,s]$. Then, the inward unit normal direction to the boundary is given by the unit tangent vector $-\boldsymbol{\nu }_1 = \gamma'(0) \in\mathbb{R}^n$. When the parameterization $%
\gamma$ is unknown, we can use the secant line (see Fig.~\ref{fig22_1dcurve}(c))
to estimate this normal direction $\boldsymbol{\nu }_1$ to the boundary. {\color{black}Specifically, the secant line approximation for $\boldsymbol{\nu}_1$ is given by,
\begin{equation}
\boldsymbol{\tilde{\nu}}_{1}= \frac{x_{1}-x_{2}}{\left\vert
x_{1}-x_{2}\right\vert }.%
\text{\ }  \label{Eqn:bbnnm1}
\end{equation}
Likewise, one can approximate  $\boldsymbol{\nu }_N$ at the other boundary point, $x_N$, with $\boldsymbol{\tilde{\nu}}_{N} =\frac{x_{N}-x_{N-1}}{\left\vert x_{N}-x_{N-1}\right\vert }$.

Then, the error estimate for the normal direction $\boldsymbol{\nu}_1$ to the boundary can be formalized as follows. Here, we will focus on $\boldsymbol{\nu}_1$ but this result is also valid for the secant line approximation of the tangent vectors at any $x_i\in M$, including at $x_N$, with appropriately defined arc-length parameterization.}

\begin{figure}[tbp]
\begin{multicols}{3}
     \begin{center}
     {\normalsize (a) secant line extension for ghost points for 1D manifold} \\
    \includegraphics[width=2in]{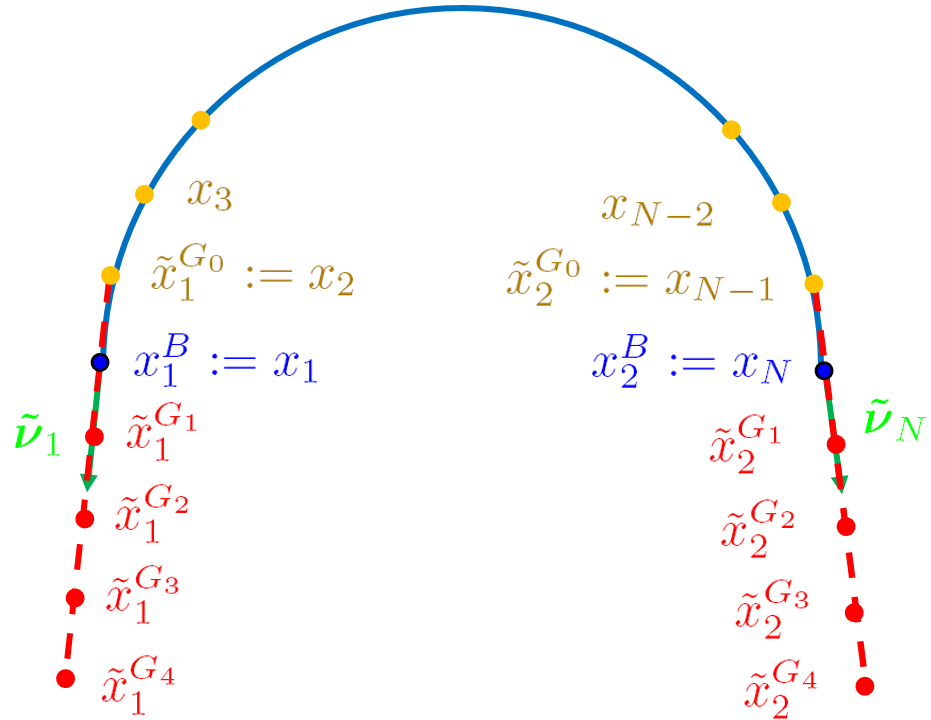}
  \end{center}
  \begin{center}
  {\normalsize (b) ideal construction: ghost point extension along true $\boldsymbol{\nu}_1$} \\
    \includegraphics[width=1.8in]{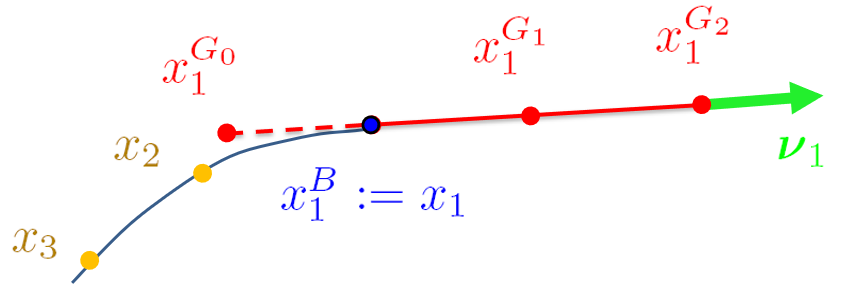}
  \end{center}
  \begin{center}
  {\normalsize (c) secant line approximation $\boldsymbol{\tilde{\nu}}_1$}\\
    \includegraphics[width=1.8in]{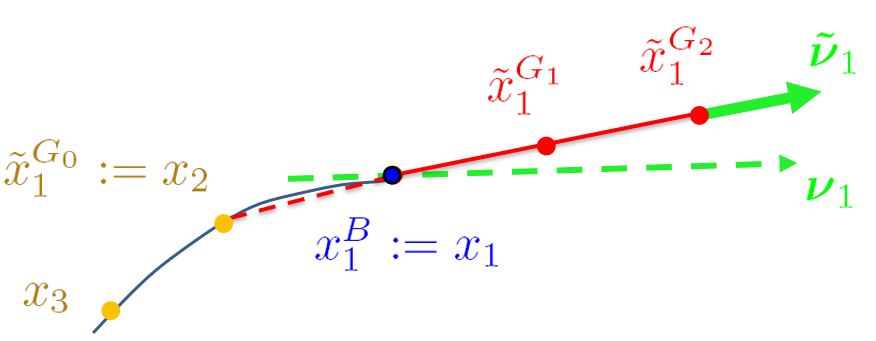}
  \end{center}
  \begin{center}
  {\normalsize (d) exterior normal direction $\boldsymbol{\tilde{\nu}}$ for well-sampled data }\\
    \includegraphics[width=2.1in]{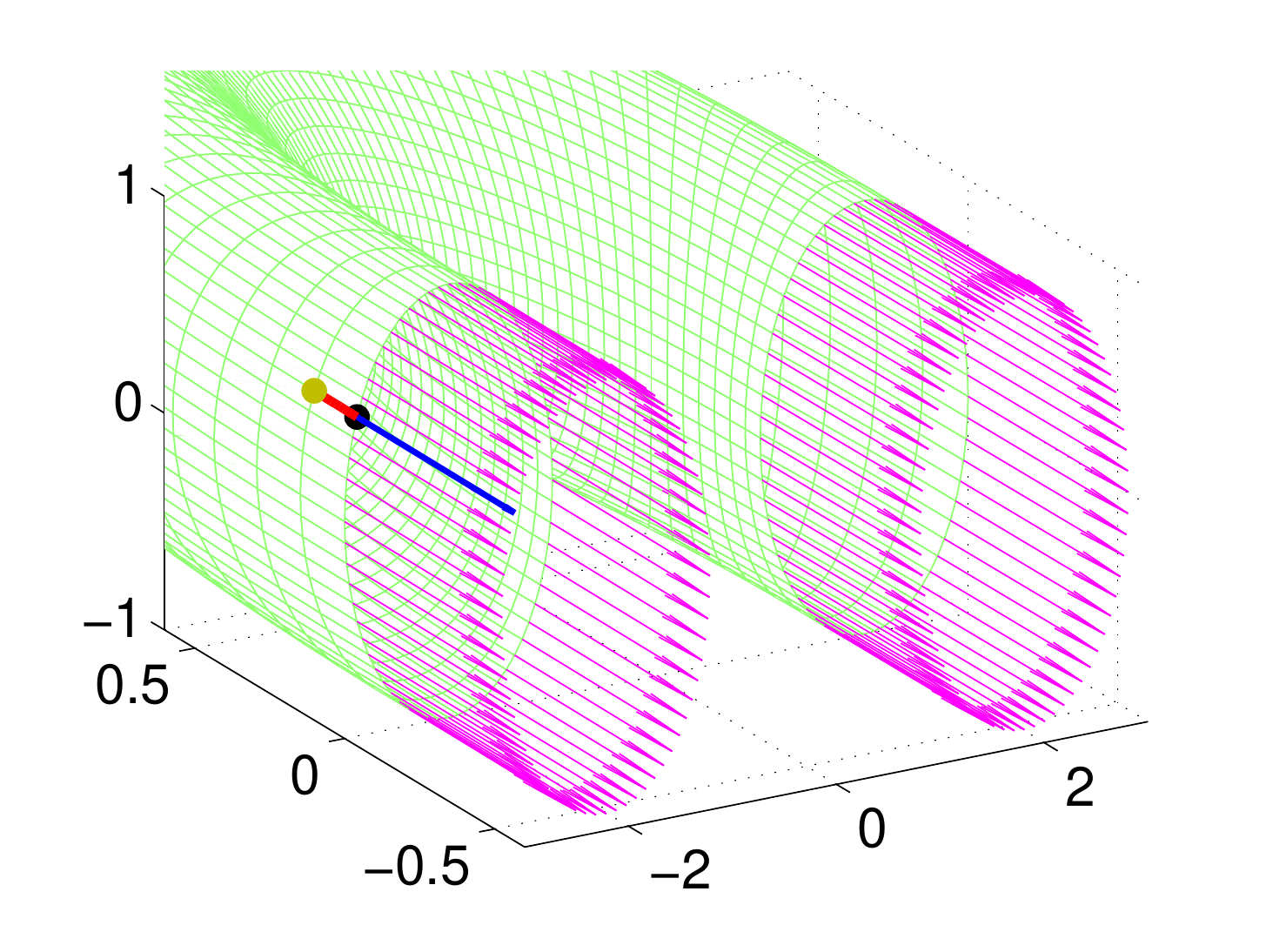}
  \end{center}
\end{multicols}
\caption{(Color online) (a) Sketch of a specification of ghost points $\{x_j^{G_k}\}$ starting from the boundary point $x_j^B$, along the secant line, on a 1D manifold.
Here, $\bm{\tilde{\nu}}_1$ and $\bm{\tilde{\nu}}_N$ are two estimated exterior normal directions which are along the secant lines connecting the boundary points and their nearest neighbors on the manifold.
(b) Ideal construction: ghost point extension for $x_1^{G_0},x_1^{G_1},x_1^{G_2}$ along true $\boldsymbol{\nu}_1$. Here, $x_2$ and $x_3$ are points on the manifold $M$, $x_1^B:=x_1$ is a point on the boundary $\partial M$.
(c) Secant line extension for ghost points $\tilde{x}_1^{G_1}, \tilde{x}_1^{G_2}$ along the estimated $\boldsymbol{\tilde{\nu}}_1$. Here, $\boldsymbol{\tilde{\nu}}_1$ is along the secant line connecting $x_2$ and $x_1$.
(d) Secant line extension for well-sampled data for the torus example. Blue line is the extension of the secant line, connecting the black boundary point and the yellow manifold point. Similarly for the other magenta lines.}
\label{fig22_1dcurve}
\end{figure}

\begin{prop}
\label{prop:diren} Let $\gamma (s)$ be a geodesic curve parameterized with the
 arc-length $s$, connecting discrete points $x_1\in \partial M$ with $x_2\in M$ (see Fig.~\ref{fig22_1dcurve}).
such that $|x_1-x_2|=\mathcal{O}(h)$, where $|\cdot |$ denotes the Euclidean $\mathbb{R}^{n}$%
-norm. Then, the
unit tangent vector ${\boldsymbol{\nu}}_1=-\gamma ^{\prime }(0)$ at point $%
x_1=\gamma (0)$\ can be estimated by ${\tilde{\boldsymbol{\nu}}}_1$ in \eqref{Eqn:bbnnm1}
with error $|{{\boldsymbol{\nu}}}_1-{\tilde{\boldsymbol{\nu}}}_1|= \mathcal{O}(h)$, where the constant in the error bound depends on the local
curvature $\omega=|\gamma''(0)|$ of the curve at $x_1=\gamma ( 0) $.
\end{prop}

\begin{proof}
For small $s$, applying Taylor's expansion on $%
\gamma $, we get
\begin{equation*}
\gamma (s)=\gamma (0)+s\gamma ^{\prime }(0)+\frac{s^{2}}{2}\gamma ^{\prime
\prime }(0)+\frac{s^{3}}{6}\gamma ^{\prime \prime \prime }(0)+\mathcal{O}%
(s^{4}).
\end{equation*}%
Since $\gamma ^{\prime \prime }(s)\perp T_{x}M$ for any $x\in M$ (by
geodesic curve), we obtain:
\begin{equation*}
|\gamma (s)-\gamma (0)|^{2}=s^{2}+s^{4}\left( \frac{1}{4}|\gamma "(0)|^{2}+%
\frac{1}{3}\langle \gamma ^{\prime }(0),\gamma ^{\prime \prime \prime
}(0)\rangle \right) +\mathcal{O}(s^{5}).
\end{equation*}%
This also means that,
\BEA
|\gamma (s)-\gamma (0)| &=& s+s^{3}\left( \frac{1}{8}|\gamma "(0)|^{2}+\frac{1}{6}\langle \gamma^{\prime }(0),\gamma ^{\prime \prime \prime }(0)\rangle \right) +\mathcal{O}(s^{4}).\nonumber%\label{Eqn:eudgeo}
\EEA
Then, we have,
\BEA
\frac{\gamma(s)-\gamma(0)}{|\gamma(s)-\gamma(0)|} = (s+\mathcal{O}(s^3))^{-1}(s\gamma'(0)+ \frac{s^2}{2}\gamma''(0)+ \mathcal{O}(s^3))  = \gamma'(0)+ \frac{s}{2}\gamma''(0)+ \mathcal{O}(s^2).\nonumber
\EEA
By the definitions of $\boldsymbol{\nu}_1$ and $\tilde{\boldsymbol{\nu}}_1$ and after some algebra, we have
\begin{eqnarray*}
\left\vert {\boldsymbol{\nu}}-{\tilde{\boldsymbol{\nu}}}\right\vert
&=&\left\vert \gamma ^{\prime }(0)-\frac{\gamma (s)-\gamma (0)}{|\gamma
(s)-\gamma (0)|}\right\vert
=\frac{s}{2}|\gamma "(0)|+\mathcal{O}(s^{2}).
\end{eqnarray*}%
Since $s =\mathcal{O}(h)$, it is clear that $\left\vert {\boldsymbol{\nu}}_1-{%
\tilde{\boldsymbol{\nu}}}_1\right\vert =\mathcal{O}(h)$ with constant that
depends on the curvature $\omega =|\gamma "(0)|$.
\end{proof}

{\color{black}In higher-dimensions, one can use the same approximation method as above for well-sampled data. In the following example, we illustrate the secant line extension on a 2D semi-torus, embedded in $\mathbb{R}^3$}

\begin{example}\label{2Dnuwell}
{Fig.~\ref{fig22_1dcurve}(d)} displays the secant line extension along $\boldsymbol{\tilde{\nu}}$ [magenta lines] for the well-sampled data on a semi-torus. In this example, the semi-torus is defined with the
standard {parameterization:}
\begin{equation}
x=\iota ( \theta ,\phi ) :=\left(
\begin{array}{c}
( a+\cos \theta ) \cos \phi \\
( a+\cos \theta ) \sin \phi \\
\sin \theta%
\end{array}%
\right) ,\text{ \ for }%
\begin{array}{c}
0\leq \theta \leq 2\pi , \\
0\leq \phi \leq \pi , \\
a=2,%
\end{array}
\label{Eqn:torus_g}
\end{equation}%
where $(\theta ,\phi) $ are the two intrinsic coordinates and $a$ is the radius of
the semi-torus. The induced Riemannian metric is given by,
\begin{equation}
g_{(\theta ,\phi ) }( u,v) =u^{{\top }}\left(
\begin{array}{cc}
1 & 0 \\
0 & \sin ^{2}\theta%
\end{array}%
\right) v,\quad\quad \forall u,v\in T_{( \theta ,\phi ) }M%
\text{.}  \label{Eqn:torus_gg}
\end{equation}%
For well-sampled data, we notice that the two bases $\frac{\partial x}{\partial \theta}$
and $\frac{\partial x}{\partial \phi}$ are perpendicular to each other. As shown in Fig. \ref{fig22_1dcurve}(d), we can extend the secant line (red), connecting the yellow and black dots to the blue line along this estimated $\boldsymbol{\tilde{\nu} }$. We apply the similar secant line extension to the other magenta lines. Then, we will add ghost points along these magenta secant lines starting from the boundary points, which will be discussed in Section \ref{ghost_ptm}.
\end{example}

{\color{black}Unfortunately, this method is not extendable for randomly distributed data on problems of dimension $d\geq 2$ since for each boundary point, we do not always sample the corresponding interior point that allows us to construct a secant line perpendicular to the boundary.}

\subsubsection{Randomly sampled data}\label{sec312}

{\color{black}For randomly sampled point clouds, $\{x_i\}$, that lie on a $d-$dimensional manifold, our basic idea here is to estimate the tangent vectors $\boldsymbol{\tilde{t}}_{1},
\boldsymbol{\tilde{t}}_{2}, \ldots, \boldsymbol{\tilde{t}}_{d}$\  that span the tangent space at each boundary point, and also estimate the
tangent vectors, $\boldsymbol{\tilde{t}}^b_{1},\ldots,\boldsymbol{\tilde{t}}^b_{d-1}$ along the $d-1$ dimensional boundary $\partial M$.
Then, we compute the normal direction $\boldsymbol{\tilde{\nu} }$\ using the
Gram--Schmidt process or QR decomposition from these directions $\{\boldsymbol{\tilde{t}}_{1}$, $%
\boldsymbol{\tilde{t}}_{2}, \ldots \boldsymbol{\tilde{t}}_{d}\}$, and $\{\boldsymbol{\tilde{t}}^b_1,\ldots,\boldsymbol{\tilde{t}}^b_{d-1}\}$. Finally, we can determine the sign of $%
\boldsymbol{\tilde{\nu} }$ from the orientation of the
manifold $M$.}

To estimate these tangent vectors, we used a
kernel-based nonlinear regression method as introduced in Corollary 3.2. of
\cite{Berry2016IDM}. Here, we give a quick review of the algorithm for
estimating the tangent vectors for an arbitrary point $x$ on a $d$-dimensional manifold $M$ embedded in $\mathbb{R}^{n}$. For a point $x$, one
defines $\mathbf{X}$ to be the $n\times K_n$ matrix with columns $\mathbf{X}%
_{j}=D(x)^{-1/2}\exp \left( -|x-x_{j}|^{2}/4\epsilon \right)
(x_{j}-x)$ where $D(x)=\sum_{j=1}^{K_n}\exp \left( -|
x-x_{j}|^{2}/2\epsilon \right) $ with $x_{j}$  ($j=1,\ldots,K_n$) being $K_n>d$
nearest neighbors of the point $x$. Then, the leading largest $d$ singular
values of matrix $\mathbf{X}$ will be of order-$\sqrt{\epsilon }$ with the
associated singular vectors parallel to the tangent space of $M$. The remaining $\min\{n,K_n\}-d
$  smaller singular values will be of order-$\epsilon $ with the
singular vectors orthogonal to the tangent space of $M$.

{\color{black}To simplify the discussion below, let us focus on 2D problems (while the same algorithm is applicable for any $d$-dimensional problems with appropriate choice of $\epsilon$ and number of boundary points, which we shall discuss in the Summary section). In the 2D case, we first estimate the two tangent vectors $\boldsymbol{\tilde{t}}_{1}$\ and $\boldsymbol{\tilde{t}}_{2}$\ for a boundary point $x \in \partial M$ using the
kernel-based nonlinear regression method. We empirically choose $K_{n_1}>d=2$
and find $K_{n_1}$ nearest neighbors of $x$ from points on the 2D manifold\ $M$%
.\ Using these $K_{n_1}$\ points, we then specify the bandwidth of the
kernel $\epsilon _{1}$ using the auto-tuned method discussed in %
Section~\ref{paraspec}. The error estimates of the two leading singular vectors $\boldsymbol{\tilde{t}}_{1}$\ and $\boldsymbol{\tilde{t}}_{2}$ for
approximating the two tangent vectors are on order-$\sqrt{\epsilon _{1}}$
(see Appendix A in \cite{Berry2016IDM} for detailed discussion).
Since there are infinitely many two linearly independent vectors that can span the 2D tangent space of $M$ at $x$, numerically we can only guarantee that $\text{Span}\{\boldsymbol{\tilde{t}}_{1}, \boldsymbol{\tilde{t}}_{2}\} = \text{Span}\{\frac{\partial x} {\partial \theta},\frac{\partial x} {\partial \phi}\}$, where  the parameterization $x =\iota (\theta ,\phi)$ with $\theta$ and $\phi$ being two intrinsic coordinates. This pair of linearly independent vectors $\boldsymbol{\tilde{t}}_{1}$\ and $\boldsymbol{\tilde{t}}_{2}$ can be different from the local bases $\frac{\partial x} {\partial \theta}$ and $\frac{\partial x} {\partial \phi}$ up to an orthonormal matrix (or a rotation).

Similarly, we apply the nonlinear regression method to estimate the tangent direction $
\boldsymbol{\tilde{t}}:=\boldsymbol{\tilde{t}}^b_1$ that is parallel to the boundary $\partial M$ for each boundary point $x\in \partial M$.
We empirically choose $K_{n_2}>d=2$ and find $K_{n_2}$ nearest neighbors of $x$
only from boundary points of the one-dimensional $\partial M$. Using these $K_{n_2}$\ points, we
auto tune the bandwidth of the kernel $\epsilon _{2}$. We can compute $\boldsymbol{\tilde{t}}$ from the first singular value of this $\mathbf{X}$ and the error estimate of
$\boldsymbol{\tilde{t}}$ is order-$\sqrt{\epsilon _{2}}$. Next, the normal
direction $\boldsymbol{\tilde{\nu} }$ can be approximated by subtracting the
orthogonal projection of $\boldsymbol{\tilde{t}}_{1}$\ (or $\boldsymbol{\tilde{t}}_{2}$)\
onto $\boldsymbol{\tilde{t}}$\ from the tangent vector $\boldsymbol{\tilde{t}}_{1}$\ (or $%
\boldsymbol{\tilde{t}}_{2}$)\ using the Gram--Schmidt process or QR decomposition,
\begin{equation}
\boldsymbol{\tilde{\nu} } = \boldsymbol{\tilde{t}}_{1} - \left\langle \boldsymbol{\tilde{t}}_{1},\boldsymbol{\tilde{t}}\right\rangle
\boldsymbol{\tilde{t}},  \notag
\end{equation}
where $\left\langle \boldsymbol{\tilde{t}}_{1},\boldsymbol{\tilde{t}}\right\rangle$ denotes the inner product of vectors $\boldsymbol{\tilde{t}}_{1}, \boldsymbol{\tilde{t}}\in\mathbb{R}^n$ and we notice that $|\boldsymbol{\tilde{t}}|=1$ for a singular vector from SVD.
Finally, the sign of $\boldsymbol{\tilde{\nu} }$ can be determined by comparing with the $k$-nearest neighbors of $x$.
The error estimate
for the normal direction $\boldsymbol{\tilde{\nu} }$ is thereafter $\mathcal{O}(%
\sqrt{\epsilon _{1}},\sqrt{\epsilon _{2}})$, that is, $\left\vert
\boldsymbol{\nu }-\boldsymbol{\tilde{\nu}}\right\vert =\mathcal{O}(\sqrt{%
\epsilon _{1}},\sqrt{\epsilon _{2}})$.

Applying the $\epsilon$ auto-tuning algorithm discussed in Section~\ref{sec:dm}, we obtain an error of order-$N^{-1/d}$, which we have verified for problems of dimensions $d=1,2$. For 2D problems, if the number of points at the boundary is $J=\mathcal{O}(N^{1/2})$, then $\epsilon_2 \sim J^{-1} \sim N^{-1/2}$ and this error rate balances with the rate $\epsilon_1 \sim N^{-1/2}$, which is also the rate of the error for the overall GPDM algorithm, as we show in the following example.

\begin{example}\label{2Dnuexample}
Fig.~\ref{fig2_normal_dire}(b) displays a comparison of the true $\boldsymbol{\nu}$ and estimated $\boldsymbol{\tilde{\nu}}$ for random data on a semi-torus. The embedding function is given by \eqref{Eqn:torus_g} and the Riemannian metric is given by \eqref{Eqn:torus_gg}.
It can be seen from Fig. \ref{fig2_normal_dire}(c) that the error rate for  $|\boldsymbol{\tilde{\nu}}-\boldsymbol{\nu }|$ is as expected to be $\mathcal{O}(\epsilon^{1/2})$, where $\epsilon=\epsilon_1$ is chosen to be the same as that in DM or GPDM method in Example \ref{torusrandom}.
\end{example}
}

\begin{figure*}[tbp]
\centering
\begin{tabular}{ccc}
(a) sketch for extension along & (b) exterior normal direction $\boldsymbol{\nu}$ & (c) error rate of $|\boldsymbol{\tilde{\nu}}-\boldsymbol{\nu}|$ \\
 estimated $\boldsymbol{\tilde{\nu}}$ for random data & and estimated $\boldsymbol{\tilde{\nu}}$ for random data &  for random data  \\
\includegraphics[width=.3\textwidth]{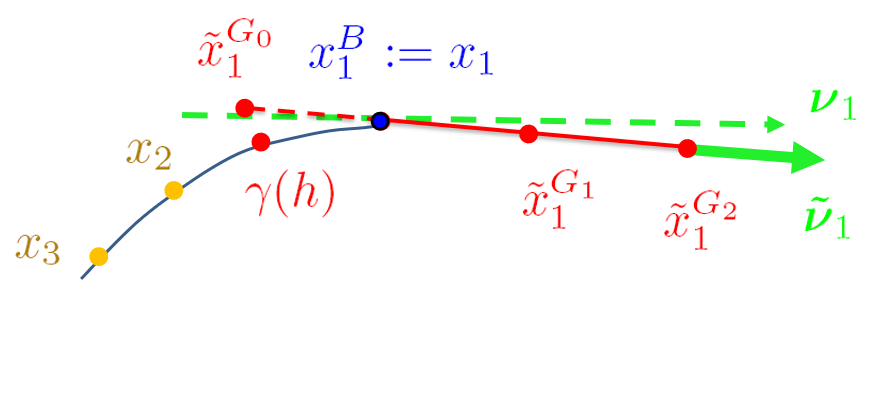}
& \includegraphics[width=.3\textwidth]{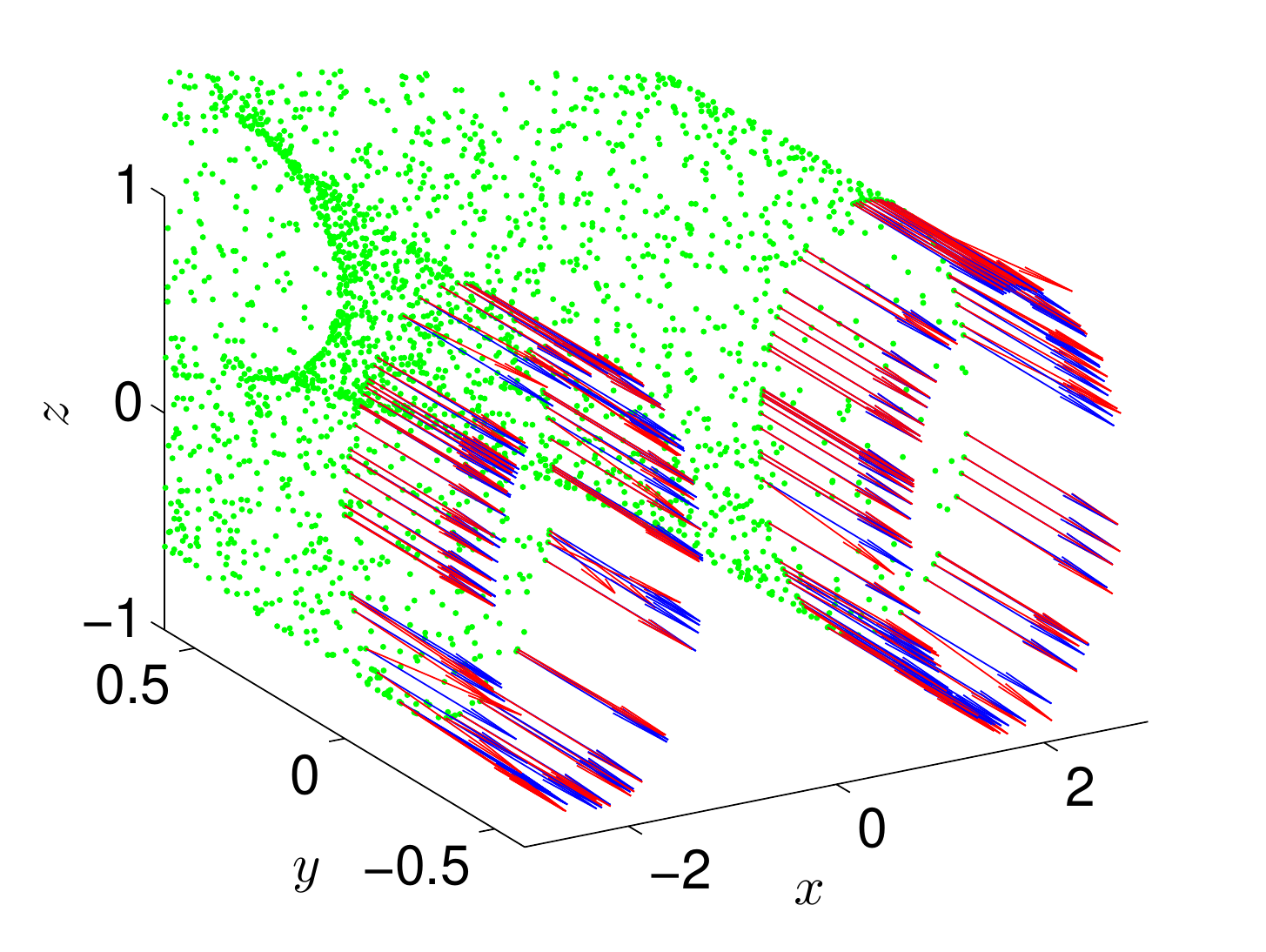}
& \includegraphics[width=.3\textwidth]{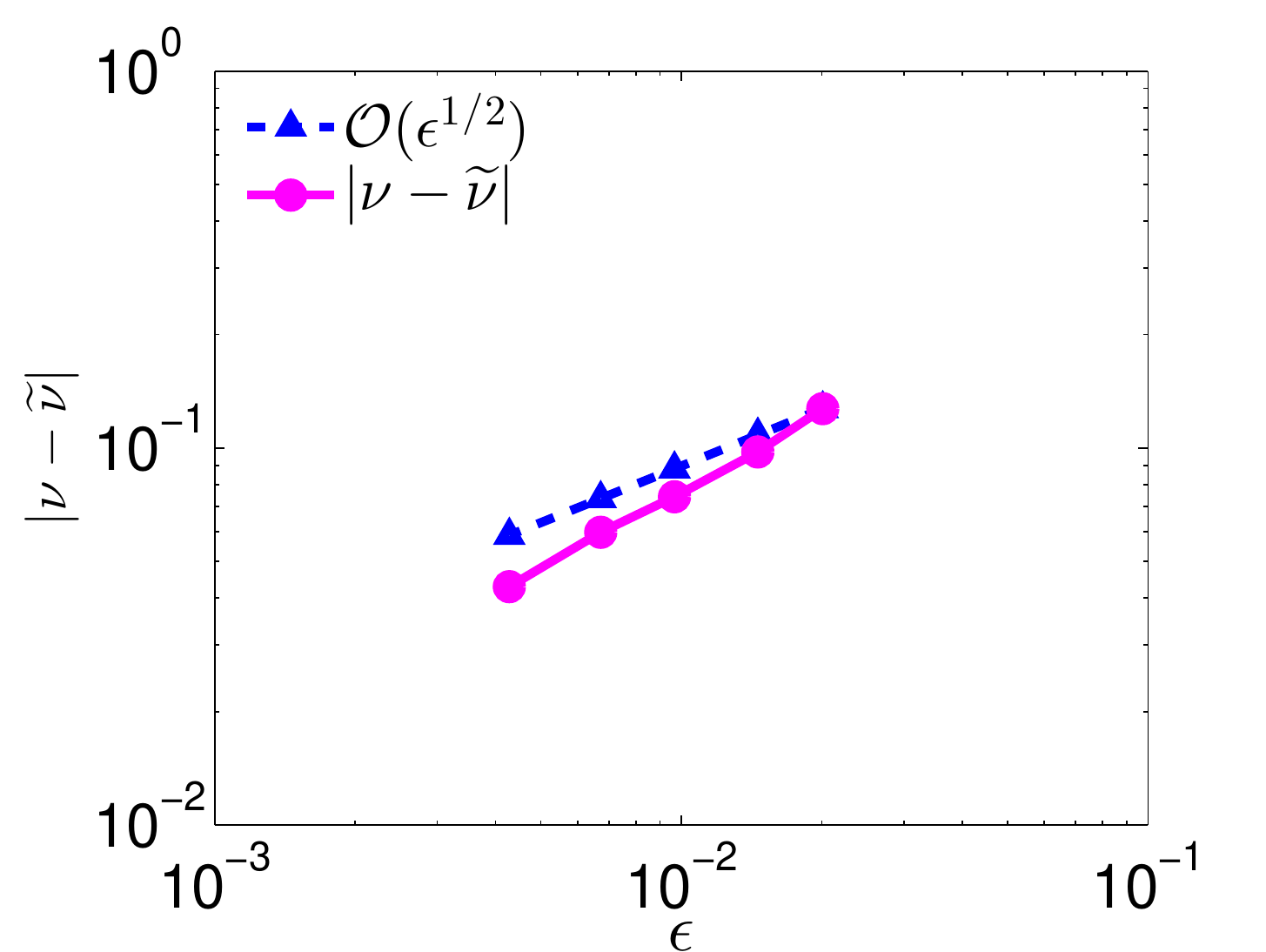}
\end{tabular}%
\caption{(Color online) (a) Sketch for a ghost point extension along the estimated $\boldsymbol{\tilde{\nu}}$ for random data. Here, $x_2$ and $x_3$ are points on the manifold $M$, $x_1^B$ is the point on the boundary $\partial M$, $\boldsymbol{\nu}_1$ is the true exterior normal direction, $\boldsymbol{\tilde{\nu}}_1$ is estimated normal direction, $\tilde{x}_1^{G_0}$ is the interior ghost point and $\gamma(h):=\exp_{x_1^B}(-h\boldsymbol{\nu}_1)$ is a projected point on the manifold $M$, and $\tilde{x}_1^{G_1}$ and $\tilde{x}_1^{G_2}$ are ghost points along $\boldsymbol{\tilde{\nu}}_1$. (b) Comparison between exact exterior normal direction $\boldsymbol{\nu}$ (blue arrows) and estimated exterior normal direction $\boldsymbol{\tilde{\nu}}$ (red arrows) for given random points on the semi-torus (\ref{Eqn:torus_g}) with unknown parameterization {for one trial when $N=64^2$}.  (c)
The expectation of the error $|\boldsymbol{\nu} - \boldsymbol{\tilde{\nu}}|$ as a function of $\epsilon$, where $\epsilon=\epsilon_1$ is chosen the same as that in DM and GPDM methods in Example \ref{torusrandom}. {The five points correspond to $N=32^2,45^2,64^2,90^2,128^2$, and larger $N$ corresponds to smaller auto-tuned $\epsilon$. For each $N$, we run 16 independent trials and then calculate the mean of $|\boldsymbol{\nu} - \boldsymbol{\tilde{\nu}}|$ versus the mean of auto-tuned $\epsilon$'s as one point in panel (c). } }
\label{fig2_normal_dire}
\end{figure*}

\subsection{Estimation of the normal derivatives on the boundaries and {\color{black}distance $h$ among neighboring ghost points}}\label{est_directional_deri}

{\color{black}
For each point $x^B$ at the boundary, we denote $\boldsymbol{\nu}:=\boldsymbol{\nu}_{x^B}\in \mathbb{R}^n$ as the corresponding normal unit vector that is pointing outward from the manifold $M$. We approximate the directional derivative of $\partial_{\boldsymbol{\nu}} u(x^B)$ with the following finite-difference method,
\BEA
\frac{\partial u}{\partial \boldsymbol{\nu }}( x^B) \approx \frac{u( x^{B}) -u( x^{G_0}) }{\left\vert
x^{B}-x^{G_0}\right\vert}, \label{finitedifference}
\EEA
where we have defined a ghost point along $-\boldsymbol{\nu}$ (see Fig.~\ref{fig22_1dcurve}(b)) as,
\BEA
x^{G_0} := x^B - h \boldsymbol{\nu},\label{xG0}
\EEA
{where $h$ characterizes the distance between neighboring ghost points as will be specified below after Definition \ref{epsball}}.
Let $\gamma:\mathbb{R}\to M\subseteq\mathbb{R}^n$ be a geodesic, parameterized with arc-length $h$, such that $\gamma(0)= x^B$ and $\gamma'(0) = -\boldsymbol{\nu}$, one can see that $\gamma(h) - x^{G_0} =  \gamma(h) - (\gamma(0) + h\gamma'(0)) = \mathcal{O}(h^2)$ (see Fig.~\ref{fig2_normal_dire}(a) for a geometric illustration of the point, $\gamma(h)$). For $u\in C^1$ on a straight line connecting $x^{G_0}$ and $\gamma(h):=\exp_{x^B}(-h\boldsymbol{\nu})\in\mathbb{R}^n$, then $u(\gamma(h)) - u(x^{G_0}) = \mathcal{O}(h^2)$. This yields the following error estimate,
\BEA
\frac{u( x^{B}) -u( x^{G_0}) }{\left\vert x^{B}- x^{G_0}\right\vert} = \frac{1}{h} \big(u(x^B) - u(\gamma(h)) \big) + \frac{1}{h}\big(u(\gamma(h)) - u(x^{G_0})\big) = -\nabla_g u(x^B) \cdot\gamma'(0) + \mathcal{O}(h) = \nabla_g u(x^B) \cdot\boldsymbol{\nu} + \mathcal{O}(h).\label{finitedifferenceerror}
\EEA

Since $\boldsymbol{\nu}$ is numerically estimated by $\boldsymbol{\tilde{\nu}}$ with {error of order-$\sqrt{\epsilon}$ and \eqref{xG0} is estimated by
\BEA
\tilde{x}^{G_0} := x^B - h \boldsymbol{\tilde{\nu}},\label{xG0tilde}
\EEA
then } it is immediately clear that, $|\tilde{x}^{G_0} - x^{G_0}| = h |\boldsymbol{\tilde{\nu}}-\boldsymbol{\nu}| = \mathcal{O}(h\sqrt{\epsilon})$. If $u\in C^1$ on a straight line connecting $\tilde{x}^{G_0}$ and $x^{G_0}$, we have $u(x^{G_0})-u( \tilde{x}^{G_0})=\mathcal{O}(h\sqrt{\epsilon})$, and
\BEA
\frac{u( x^{B}) -u( \tilde{x}^{G_0}) }{\left\vert x^{B}-\tilde{x}^{G_0}\right\vert} = \frac{1}{h}(u( x^{B}) -u( x^{G_0}))+ \frac{1}{h}\big(u(x^{G_0})-u( \tilde{x}^{G_0})\big) = {\nabla_g u(x^B) \cdot}\boldsymbol{\nu} + \mathcal{O}(h,\sqrt{\epsilon}) {= \frac{\partial u}{\partial \boldsymbol{\nu }}( x^B) + \mathcal{O}(h,\sqrt{\epsilon})},\label{finitedifferenceerror2}
\EEA
where the first term follows directly from \eqref{finitedifferenceerror}. Based on this observation, we assume that $u\in C^1$ (in fact $C^3$ for the extrapolation scheme in Section~\ref{artBC_ghost}) on the set:
\begin{defn}\label{epsball}
$B_{\epsilon^r}(\partial M) := \cup_{x\in\partial M} B_{\epsilon^r}(x)$, where $B_{\epsilon^r}(x) = \{y\in\mathbb{R}^n:|x-y|\leq \epsilon^r\}$ is an $\epsilon^r$-ball in $\mathbb{R}^n$.
\end{defn}
With this assumption, for $h\lesssim\mathcal{O}(\epsilon^r)$, it is clear that $x^{G_0},\tilde{x}^{G_0}, \gamma(h) {\in} B_{\epsilon^r}(x^B)\subset\mathbb{R}^n$, which justifies the use of the Taylor expansions along straight paths between these points.

{\bf Well-sampled data:} In this case, since $\boldsymbol{\tilde{\nu}}$ is a secant-line approximation, the estimated point $\tilde{x}^{G_0}$ coincides with the interior point adjacent to the corresponding boundary point (e.g., in Fig.~\ref{fig22_1dcurve}(c), $\tilde{x}^{G_0}$ corresponds to $x_2$ when the boundary point $x^B=x_1$). In such a case, one can immediately set $h := |\tilde{x}^{G_0} - x^B|$,
which is also used in the first equality in \eqref{finitedifferenceerror2}. This specification scales as $h=\mathcal{O}(N^{-1/d})$.

{\bf Randomly-sampled data:} In such a case, generally, $\tilde{x}^{G_0}$ does not coincide with any other randomly sampled data (see Fig.~\ref{fig2_normal_dire}(a)). To use the estimator in \eqref{xG0tilde}, one has to specify $h$. In our implementation, $h$ is estimated by the mean distance from $x^{B}$ to its $P$ (around $10$ in our numerical examples) nearest neighbors. Let $x_{p}^B \in M$ denotes the $p$th nearest neighbor of $x^B$ for $p=1,\ldots,P$. Since the distance to the nearest neighbor is a density estimator \cite{loftsgaarden1965nonparametric}, that is, $|x^B - x_{p}^B| \propto q(x^B)^{-1/d}$, where
$q$ denotes the sampling density and $d$ denotes the dimension of the manifold $M$, then the distance
 {\color{black}
\BEA
h  = \frac{1}{P} \sum_{p=1}^P \Big| x^B - x_{p}^B\Big| =\mathcal{O}\Big(q(x^B)^{-\frac{1}{2}}\Big),\notag%\label{dj}
\EEA
for two-dimensional manifolds.}
}

\subsection{Ghost points}

\label{ghost_ptm}

\begin{figure}[tbp]
\centering
\includegraphics[width=3.0
in, height=2 in]{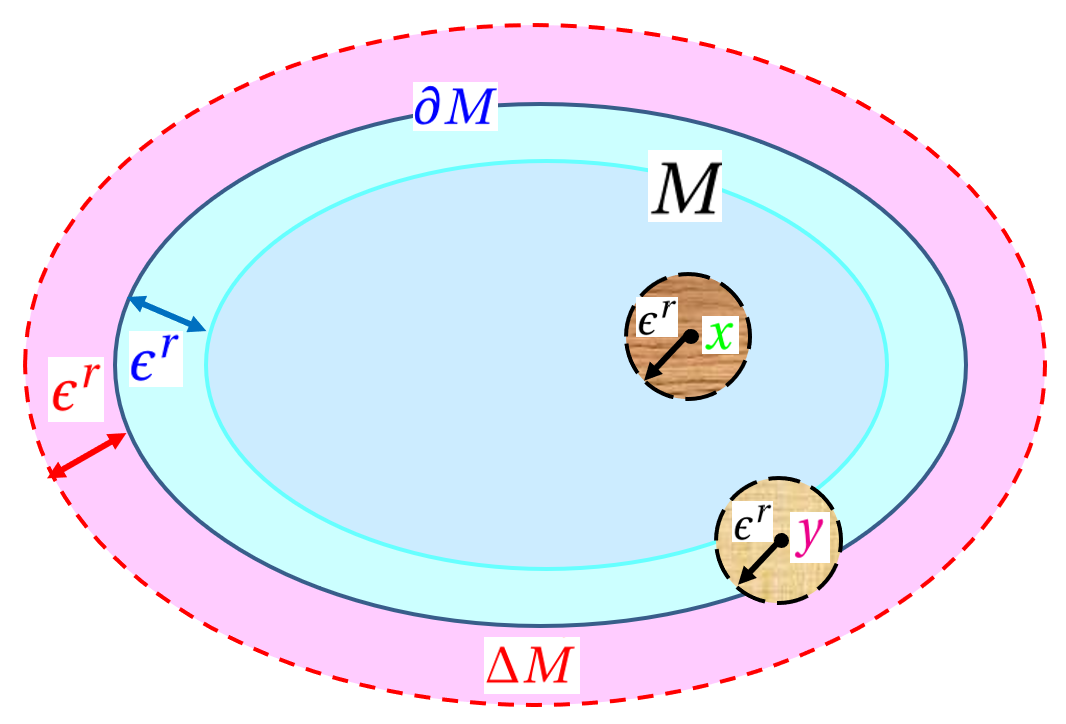}
\caption{(Color online) Sketch of the extended manifold $M \cup \Delta M$.}
\label{fig42_deltaM}
\end{figure}

It is well known that the diffusion operators defined in \eqref{L1}-\eqref{Eqn:L3} cannot be approximated accurately near the boundary of the manifold using the standard diffusion maps algorithm. The reason is that the
asymptotic expansion \eqref{asymp} is valid only for points $x\in M$ whose Euclidean distance from the boundary $\partial M$ is larger than $\epsilon ^{r}$, for $0<r<1/2$.
For points $y\in M$ whose distance from $\partial M$ is smaller than $\epsilon ^{r}$, an order-$\sqrt{\epsilon}$ term appears in the asymptotic expansion \eqref{asymp}. Geometrically, the local integral is inaccessible if there are no available data beyond $M$ (see Fig. \ref{fig42_deltaM}).

To address this issue, our idea here is to supplement the original data points on $M$
with a set of ghost points. Since the integral in the diffusion maps asymptotic expansion is effectively a local integral over a ball of radius $\epsilon ^{r}$ as discussed in \eqref{integralop1}, we
will devise a numerical scheme to specify these new points such that one can approximate the local integral over the ball of radius $\epsilon^r$ even when the integral operator is evaluated at points in $M$ whose distances are less than $\epsilon^r$ from the closest point on the boundary, $\partial M$ (e.g., $y$ in Fig.~\ref{fig42_deltaM}).
Specifically, the ghost points will be sampled from the outer normal collar that can be attached at the boundary such that the extended manifold can be isometrically embedded in $\mathbb{R}^n$ without changing the embedding function of the original $M$. In the following lemma, we provide the specific condition for such requirement to hold for two-dimensional manifolds.

{\color{black}
\begin{lem}\label{ghmanifold}
Let $M$ be a two-dimensional Riemannian manifold with nonempty, smooth boundary $\partial M$, isometrically embedded in $\mathbb{R}^n$. Suppose that $\partial M$ is a regular curve with maximum curvature of $1/R$ and any two points $x,y \in \partial M$, whose geodesic distance $d_g(x,y)> \pi R$, have Euclidean distance $|x-y|>2R$. Then there exists a submanifold, $\Delta M$ (an outer normal collar of radius $R$), such that the adjunction space, $M \cup_{id} \Delta M$, defined by attaching $M$ and $\Delta M$ along the boundary with an identity ``gluing" function $id:\partial M\to \partial (\Delta M)$, can be isometrically embedded in $\mathbb{R}^n$ with an embedding function that is consistent with the original embedding function when it is restricted to $M$.
\end{lem}

\begin{proof}
Our construction is to extend $M$ with an exterior collar of radius $R$ along the boundary. By the collar neighborhood theorem (Theorem~9.25 in \cite{lee2013smooth}), there exists a normal collar neighbor $W\subseteq M$, which is defined as the range of the following map $\phi:[0,R)\times \partial M \to W\subseteq M$,
\BEA
\phi(t,x) = \exp_x(-t\boldsymbol{\nu}_x), \quad t\in [0,R), \nonumber
\EEA
for some $R>0$. Here, $\boldsymbol{\nu}_x $ denotes the normal vector at $x\in \partial M$ that is pointing outward from the manifold $M$, so $\phi$ maps the points in the inward normal collar to the collar neighbor $W$. Define a manifold corresponding to the  pre-image of the collar neighbor that points outward as,
\BEA
\Delta M = \big\{(-t,x): t\in[0,R), x\in\partial M, \phi(t,x) \in W  \big\}, \label{DeltaM}
\EEA
which is the outer normal collar of radius $R$.

Next, we attach $M$ and $\Delta M$ along the boundary by identifying $x\in\partial M$ with the identity map $id(x)\in \{0\}\times \partial M \subset \Delta M$. Let $M \cup_{id} \Delta M:= (M\sqcup \Delta M)/\sim$ be the adjunction set defined as the quotient space of the disjoint sum induced by attaching $\Delta M$ to $M$ along the identity map, $id$. Since $\Delta M$ is the outward normal collar and $\partial M$ is a smooth boundary with bounded curvature, the adjunction set is smooth along the attached boundary $\partial M$. To finish the proof, we need to show that the adjunction set can be isometrically embedded in $\mathbb{R}^n$ with an embedding function that is consistent when restricted to the original manifold $M$.

Let $\mathcal{E}\subseteq T\mathbb{R}^n$ denotes the domain of the exponential map of $\mathbb{R}^n$ and $N(\partial M)$ denotes the normal bundle of $\partial M$ in $\mathbb{R}^n$. Then we can define $E: \mathcal{E}\cap N(\partial M) \to \mathbb{R}^n$ to be the normal exponential map of $\partial M$ in $\mathbb{R}^n$. By the tubular neighborhood theorem (see Theorem 5.25 in \cite{lee2018introduction}), $\partial M$ has a uniform tubular neighbor in $\mathbb{R}^n$. Specifically, there exists a normal neighborhood of $\partial M$, $U\subset \mathbb{R}^n$, that is diffeomorphic under $E$ to an open subset $V\subseteq \mathcal{E}\cap N(\partial M)$. Since $\boldsymbol{\nu}_x\in N_x(\partial M)$, and $\partial M$ has maximum curvature $1/R$ and any two points  with geodesic distance larger than one half of the circumference of the osculating circle of radius $R$, that is, $d_g(x,y)>  \pi R$, have Euclidean distance $|x-y|>2R$, then the open neighbor $U$ is a tubular neighborhood of radius $R$ that is homeomorphic to $\Delta M \subset V$. So, the tubular neighbor theorem ensures that $\Delta M$ as defined in \eqref{DeltaM} can be smoothly embedded in $\mathbb{R}^n$. In fact, one can define a Riemannian metric for $\Delta M$ to be the pullback of the following embedding function,
\BEA
\tilde\iota(-t,x):= x + t \boldsymbol{\nu}_x, \label{extendedembedding}
\EEA
for any $(-t,x)\in \Delta M$. Since the induced metric of $\Delta M$ is consistent with that of $M$ at the attached boundary, that is,
$\tilde\iota(0,x) = x \in \mathbb{R}^n$, then the extended manifold $M \cup_{id} \Delta M$ is isometrically embedded in $\mathbb{R}^n$.
\end{proof}
}

{\color{black}
In the remainder of this paper, we will refer to the extended manifold $M \cup_{id}\Delta M$ as the set $M \cup \Delta M$ to simplify the notation. We should also point out that for 1D manifold, since the boundary consists of only two points (e.g., as shown in Fig. \ref{fig22_1dcurve}), the assumption for $\partial M$ is in Lemma~\ref{ghmanifold} is slightly different. In this case, the extended manifold can be isometrically embedded as long as the two exterior normal lines of length $R>0$ from the boundary do not intersect. Next, will use the embedding function in \eqref{extendedembedding} to specify the ghost points with $R=\mathcal{O}(\epsilon^r)$.

Numerically, for each boundary point $x^B\in \partial M\subset\mathbb{R}^n$, let $\boldsymbol{\tilde{\nu}}\in\mathbb{R}^n$ denotes the numerical estimate of the corresponding normal vector $\boldsymbol{\nu}\in\mathbb{R}^n$ at $x^B$. Then, the ghost points,
\BEA
x^{G_k} := x^B + k h \boldsymbol{\nu},\label{ghostpontsembedding}
\EEA
are approximated by
\BEA
\tilde{x}^{G_k} := x^B + k h \boldsymbol{\tilde{\nu}},\label{approxghostpontsembedding}
\EEA
 for $k=1,\ldots,K$, where $K=\mathcal{O}(\epsilon^r h^{-1})$. Numerically, however, we specify $K$ empirically (usually $K\leq 10$). See Figs.~\ref{fig22_1dcurve}(b),(c) and \ref{fig2_normal_dire}(a) for a geometric illustration.

}

By the construction above, the Euclidean distance between any point {$x\in M^o$} and $\partial(M\cup\Delta M)$ is at least of order $\epsilon^r$. This ensures the validity of the asymptotic expansion in \eqref{asymp} for all points on $M^o$ (including points that are close to the boundary $\partial M$). {\color{black}It is worthwhile to point out that the ghost points $\tilde{x}^{G_k}$ do not exactly lie on $\Delta M$ since the true normal vectors, $\boldsymbol{\nu}$, are not available (ideal case as illustrated in Fig.~\ref{fig22_1dcurve}(b)). In the next section, we will show how this error affects the overall algorithm, especially when the data are randomly distributed.}

\subsection{Extrapolation of functions on the ghost points}

\label{artBC_ghost}

We now address the extrapolation problem on the estimated ghost points. In particular,
we need to extrapolate the solution $u$ on the estimated ghost points.
Popular extrapolation techniques include the linear and quadratic extrapolation methods, the level set
method, the ghost fluid method \cite{aslam2014static}. One idea is to extend the
function of interest with a set of artificial boundary conditions, imposed on the ghost points.
This leads us to the problem of specifying the boundary conditions on the ghost points. In particular, we will consider a discrete analog of matching the second-order derivatives of the functions evaluated at the ghost points as the extrapolation condition, which mimics the cubic spline condition as proposed in \cite{fornberg2002observations}. In addition, we also include a condition that mimics the classical finite-difference solution of Neumann (or Robin) boundary value problems with ghost points.

{\color{black}
Let $u\in C^3(M\cup B_{\epsilon^r}(\partial M))$, where the set $B_{\epsilon^r}(\partial M)\subset\mathbb{R}^n$
is stated in the Definition~\ref{epsball}. We note that the numerically estimated ghost points are components of this set, $\{\tilde{x}_{j}^{G_0}\}_{j=1}^{J}\cup \{\tilde{x}_{j}^{G_k}\}_{j,k=1}^{J,K}  \subset B_{\epsilon^r}(\partial M)$. Given the function values $u(x_i)$ at $x_i\in M$ and $u(\tilde{x}^{G_0}_j)$, our goal is to extrapolate  $u$ onto the set of ghost points, $\{\tilde{x}_{j}^{G_k}\}_{j=1,\ldots,J}^{k=1,\ldots,K}$. In the PDE applications, the function values at $\{\tilde{x}_j^{G_0}\}$ will be estimated in the same manner as the other data $\{x_i\}$ that lie on the manifold, that is, by inverting the discrete approximation of the diffusion operators. In particular, we define the matrix $\mathbf{L}^h$ as a discrete approximation to one of the diffusion operators in \eqref{L1}-\eqref{Eqn:L3} with the following important modification. We construct the matrix $\mathbf{L}^h$ by evaluating the kernel on  $\{x_i\}_{i=1}^N\cup\{\tilde{x}_{j}^{G_0}\}_{j=1}^{J} \cup \{\tilde{x}_{j}^{G_k}\}_{j,k=1}^{J,K}$. In the case of well-sampled data, the normal vector, $\boldsymbol{\nu}$, is estimated by a secant line and, therefore, some of these ghost points coincide with some interior points, that is,  $\{x_i\}_{i=1}^N \supset \{\tilde{x}_j^{G_0}\}_{j=1}^J$. In the case of randomly sampled data, {\color{black}we define $\{x_i\}_{i=1}^{N+J}:= \{x_i\}_{i=1}^N\cup\{\tilde{x}_{j}^{G_0}\}_{j=1}^{J}$ for convenience of notation, even if these interior ghost points do not lie on $M$.}

Since the argument below does not change whether the set $\{x_i\}$ has $N$ or $N+J$ points, instead of using different notations for the well-sampled and randomly sampled cases, we use the same set $\{x_i\}_{i=1}^N$ of $N$-points to denote points on the manifold as well as the interior ghost points $\{\tilde{x}_j^{G_0}\}_{j=1}^J$.}

{\color{black}
Here, the discrete extrapolation problem is to extend the function $u$ identified by the function values only on $x_i, x_j^B$ to estimate the function values  $u(\tilde{x}_j^{G_k}) $ for $k=1,\ldots,K$ by the estimated quantities $ \tilde{u}^{G_k}_{\epsilon,j}$. With these notations, we define a vector $\vec{u}_\epsilon$ by
\BEA
\vec{u}_\epsilon=(u(x_1),\ldots, u(x_N),\tilde{u}_{\epsilon,1}^{G_1},\ldots,\tilde{u}_{\epsilon,1}^{G_K},\ldots,\tilde{u}_{\epsilon,J}^{G_1},\ldots, \tilde{u}_{\epsilon,J}^{G_K})\in\mathbb{R}^{\bar{N}}, \label{longvectorue}
\EEA
where $\bar{N}=N+JK$. Since $\{\tilde{x}^{G_0}_j\}_{j=1}^J\subset \{x_i\}_{i=1}^N$, the first $N$-components include the function values $u(\tilde{x}^{G_0}_j)$.
Then, we estimate $\{\tilde{u}^{G_k}_{\epsilon,j}\}_{j,k=1}^{J,K}$, by solving the following $JK$ algebraic equations,
\BEA
\begin{aligned}\label{Eq:Uvvv}
\left(\mathbf{L}^h \vec{u}_\epsilon\right)_{B_j} & =f(x_{j}^{B}),   \\
\tilde{u}_{\epsilon ,j}^{G_{2}}-2\tilde{u}_{\epsilon ,j}^{G_{1}}+u
(x^{B}_j)& =\tilde{u}_{\epsilon ,j}^{G_{1}}- 2u(x_{j}^{B})+u(\tilde{x}_j^{G_0}), \\
\tilde{u}_{\epsilon ,j}^{G_{3}}-2\tilde{u}_{\epsilon ,j}^{G_{2}}+\tilde{u}%
_{\epsilon ,j}^{G_{1}}& =\tilde{u}_{\epsilon ,j}^{G_{2}}-2\tilde{u}%
_{\epsilon ,j}^{G_{1}}+u(x_{j}^{B}),  \\
\tilde{u}_{\epsilon ,j}^{G_{k}}-2\tilde{u}_{\epsilon ,j}^{G_{k-1}}+\tilde{u}%
_{\epsilon ,j}^{G_{k-2}}& =\tilde{u}_{\epsilon ,j}^{G_{k-1}}-2\tilde{u}%
_{\epsilon ,j}^{G_{k-2}}+\tilde{u}_{\epsilon ,j}^{G_{k-3}},\quad \quad
k=4,\ldots K,
\end{aligned}%
\EEA
for $j=1,\ldots,J$. Here, we have used the subscript-$B_j$ to denote the component corresponding to the boundary point $x^B_j$. The first equation in \eqref{Eq:Uvvv} is motivated by the classical finite-difference approach for solving the Neumann  problems in \eqref{extrapolation}, which imposes the discrete approximation of the elliptic PDE to be consistent at the boundary. The last three equations in \eqref{Eq:Uvvv} are the discrete analog of matching the second-order derivatives along $\boldsymbol{\tilde{\nu}}_j$ at the ghost points and the corresponding boundary point $x_j^B$.}

Now we report the error in approximating the function values $u(x_j^{G_k})$ with $\tilde{u}_{\epsilon,j}^{G_k}$, obtained from solving the algebraic conditions in \eqref{Eq:Uvvv}.

\begin{prop}
\label{prop:extrapolationofu} (Extrapolation error rate for $u$) Let $u\in C^{3}\left( M\cup B_{\epsilon^r}(\partial M)\right) $, where
 the extended manifold $M\cup\Delta M$ is a submanifold of $\mathbb{R}^n$, constructed by Lemma~\ref{ghmanifold} with $R=\mathcal{O}(\epsilon^r)$, where $0<r<1/2$. For each $x_j^{G_k} \in \Delta M$, let $\tilde{u}_{\epsilon,j}^{G_k}$ be the extrapolated function value at the estimated ghost point, $\tilde{x}_j^{G_k}$, obtained by solving \eqref{Eq:Uvvv}.
For any fixed $j=1,\ldots, J$,
\BEA
\left| u(x_j^{G_k}) - \tilde{u}_{\epsilon,j}^{G_k} \right| = \mathcal{O}\left({\color{black}h^3,h^2 \epsilon^{-1/2}},\bar{N}^{-1/2}\epsilon ^{-(1+d/4)},\bar{N}^{-1/2}\epsilon ^{(1/2-d/4)}\right),\label{extrapolationerror}
\EEA
in high probability, where $h$ depends on $\epsilon$ such that the first two terms vanishes as $\epsilon\to 0$ after $\bar{N}\to\infty$.
\end{prop}

\begin{proof}
See Appendix~\ref{App:C}.
\end{proof}

{\color{black}Throughout the paper, we use the notation $\mathcal{O}(f,g,w)$ as a shorthand for $\mathcal{O}(f)+\mathcal{O}(g)+\mathcal{O}(w) $ as $f,g,w \to 0$. The third and fourth error bounds in the "big-oh" notation above in \eqref{extrapolationerror} correspond to the Monte-Carlo estimates of the integral operator for fixed $\epsilon>0$ so they are defined as $\bar{N}\to\infty$. Also, since $h$ can be described as a function of $\epsilon$, this means the ``big-oh'' notation is defined as $\epsilon\to 0$ after $\bar{N}\to\infty$.}

{\color{black}\begin{rem}[Randomly sampled data] \label{rem:rs}  In this case, the leading error term in \eqref{extrapolationerror} is of order-$h^2\epsilon^{-1/2}$. This error rate is contributed by the estimated interior ghost points that do not lie on $M$ and the exterior ghost points that do not lie on $\Delta M$. In Appendix~\ref{App:C}, we shall see how the distances between the estimated ghost points and the points on the extended manifold,
\BEA
|\gamma_j(h)- \tilde{x}_j^{G_0}| = \mathcal{O}(h\sqrt{\epsilon}),\quad\quad
|x_j^{G_k}- \tilde{x}_j^{G_k}| = \mathcal{O}(h\sqrt{\epsilon}), \nonumber
\EEA
where $\gamma_j(h):=\exp_{x_j^B}(-h\boldsymbol{\nu}_{x_j^B})\in M$ and $\{x^{G_k}_j\}_{j,k=1}^{J,K} \subset \Delta M$,
contribute to this error rate.
\end{rem}

\begin{rem}[Well-sampled data]\label{rem:ws}
In this case, since the secant line approximation is used to approximate $\boldsymbol{\nu}$, the estimated ghost points $\{\tilde{x}_j^{G_0}\}$ coincide with some components of $\{x_i\in M\}$. Since these interior ghost points lie on the manifold, they do not contribute to the error rate-$h^2 \epsilon^{-1/2}$. While the estimated exterior ghost points, $\{\tilde{x}_j^{G_k}\}_{j,k=1}^{J,K}$, do not exactly lie on $\Delta M$, we numerically also found that they do not contribute to the error of order-$h^2 \epsilon^{-1/2}$. We suspect that this is because the diffusion maps algorithm, applied on the extended data, $\{x_i\in M\}_{i=1}^N\cup\{\tilde{x}^{G_k}_j\}_{j,k=1}^{J,K}$, is approximating the differential operator on a different smooth extended domain $\widetilde{M\cup\Delta M}$ that contain these points and the error rate in Lemma~\ref{lem:old} is still valid for the matrix $\boldsymbol{L}^h$ with the assumption that $u\in C^{3}\left( M\cup B_{\epsilon^r}(\partial M)\right)$. In light of this, for well-sampled data, the leading error is the first error term of order-$h^3$ in \eqref{extrapolationerror}.
\end{rem}}

\subsection{The ghost point diffusion maps estimator} \label{augmentm}

{\color{black}Here, we continue using the notation $\vec{u}_{\epsilon }$ as defined in \eqref{longvectorue}, where the first $N$-components contain the function value at the estimated interior ghost points $\tilde{x}^{G_0}_j$ that may or may not lie exactly on $M$, depending on the distribution of the data. For the discussion below, we also define the column vectors,
\BEA
\vec{u}_{\epsilon }^{M}&=&(u(x_{1}),\ldots ,u(\tilde{x}^{G_0}_1),\ldots,u(\tilde{x}^{G_0}_J),\ldots, u(x_N)), \notag\\
\vec{u}_{\epsilon }^{G}&=&(\tilde{u}_{\epsilon,1}^{G_{1}},\ldots ,\tilde{u}_{\epsilon,J}^{G_{K}}) \label{vectoru}\\
\vec{u}_\epsilon &=& (\vec{u}_{\epsilon }^{M},\vec{u}_{\epsilon }^{G}),\notag
\EEA
where we emphasized that some of the components of $\vec{u}_{\epsilon }^{M}\in\mathbb{R}^N$ are $u(\tilde{x}^{G_0}_j)$ in the definition above.
Similarly, we also define
\BEA
\vec{u}^{M}&=&(u(x_{1}),\ldots ,u(\gamma_1(h)),\ldots,u(\gamma_J(h)), \ldots, u(x_N)), \notag\\
\vec{u}^{G}&=&(u(x^{G_{1}}_1),\ldots,u(x^{G_{K}}_J)), \label{truevectoru}\\
\vec{u} &=& (\vec{u}^M,\vec{u}^G), \notag
\EEA
where $\vec{u}^M \in\mathbb{R}^N$ contains $u(\gamma_j(h))$, replacing each component $u(\tilde{x}^{G_0}_j)$ of $\vec{u}_{\epsilon }^{M}$. These definitions imply that,
\BEA
\vec{u}_{\epsilon }^{M} =\begin{cases} \vec{u}^{M} +\mathcal{O}(h\sqrt{\epsilon}), & \mbox{if $x_i$ are randomly sampled,}\\
 \vec{u}^{M}, & \mbox{if $x_i$ are well-sampled.}\end{cases}\label{vectoru2}
\EEA

}

Here, Eq.~\eqref{Eq:Uvvv} consists of a system of $JK$ equations and it has a unique solution that can be written in compact form as,
\BEA
\vec{u}_{\epsilon}^{G} = \mathbf{A} \vec{u}^{M}_\epsilon + \vec{b},\label{vecUGK}
\EEA
where one can see the detailed expression of $\mathbf{A} \in \mathbb{R}^{JK\times N}$ and $\vec{b}\in\mathbb{R}^{JK}$ for the 1D case in Appendix~\ref{app:D}.  Here, components of $\vec{b}$ depend on $f(x^B_j)$.
To this end, we denote the discrete approximation with a non-square matrix $\mathbf{L}^h=(\mathbf{L}^{(1)}, \mathbf{L}^{(2)}) \in\mathbb{R}^{N\times \bar{N}}$ that maps vectors $\vec{u}_\epsilon\in\mathbb{R}^{\bar{N}}$ into $\mathbf{L}^h\vec{u}_\epsilon\in\mathbb{R}^N$, where the matrix $\mathbf{L}^h$ is constructed as discussed in Section~\ref{artBC_ghost}. For the discussion below, we define the matrix $\mathbf{L}\in \mathbb{R}^{N\times\bar{N}}$, as a discrete estimator of $\mathcal{L}$ that is constructed in analogous to $\mathbf{L}^h$, except that the kernel is evaluated on $\gamma_j(h)\in M$ (and $\{x^{G_k}_j\in \Delta M\}_{j,k=1}^{J,K}$) in placed of $\tilde{x}^{G_0}_j$ (and $\{\tilde{x}^{G_k}_j\}_{jk,=1}^{J,K}$), in addition to the evaluation at all sampled points of $M$ in the construction of $\mathbf{L}^h$.
 We should point out that each row of the non-square matrices $\mathbf{L}$ and $\mathbf{L}^h$ corresponds to the kernel evaluation at the components of $\{x_i\}_{i=1}^N$, where the former includes $\{\gamma_j(h)\}$ and the latter includes $\{\tilde{x}^{G_0}_j\}$.

{\color{black}
Since we are interested in approximating $\vec{u}^M\in\mathbb{R}^N$ with the constraint that $\vec{u}^G$ is not available, we define the GPDM estimator $\mathbf{L}^g:\mathbb{R}^N\to\mathbb{R}^N$ as the following affine operator,
\BEA
\mathbf{L}^g(\vec{u}^M) := \big(\mathbf{L}^{(1)} + \mathbf{L}^{(2)}\mathbf{A}\big)\vec{u}^{M} +\mathbf{L}^{(2)} \vec{b}.\label{GPDM}
\EEA
With this definition, we should point out that $\mathbf{L}^g(\vec{u}^M_\epsilon) =  \mathbf{L}^{(1)} \vec{u}^{M}_\epsilon  + \mathbf{L}^{(2)}\big(\mathbf{A}\vec{u}^{M}_\epsilon + \vec{b}\big) =  \mathbf{L}^{(1)} \vec{u}_\epsilon^{M}  + \mathbf{L}^{(2)} \vec{u}^{G}_\epsilon = \mathbf{L}^h\vec{u}_\epsilon$, where we have used \eqref{vecUGK}. We should also point out that clearly $\mathbf{L}^g(\vec{u}^M)\neq \mathbf{L}^h\vec{u}$ since \eqref{vecUGK} is not valid for the pair of $\vec{u}^G$ and $\vec{u}^M$, that is, $\vec{u}^G\neq \mathbf{A}\vec{u}^M + \vec{b}$. With all these definitions, we now state the consistency of the GPDM estimator in \eqref{GPDM}.

\begin{thm}\label{theorem1} (Consistency of the GPDM)
Let $u\in C^{3}\left( M\cup B_{\epsilon^r}(\partial M)\right) $, where
 the extended manifold $M\cup\Delta M$ is a submanifold of $\mathbb{R}^n$, constructed by Lemma~\ref{ghmanifold} with $R=\mathcal{O}(\epsilon^r)$, where $0<r<1/2$, such that $\Delta M\subset B_{\epsilon^r}(\partial M)$. For each $x_i\in M$, where $\{x_i\}_{i=1}^N \supset \{\gamma_j(h)\}_{j=1}^J$,
\BEA
\left| \big(\mathbf{L}^g(\vec{u}^M)\big)_{i}-\mathcal{L}u(x_{i})\right| = O\left(h^3\epsilon ^{-1},h^2\epsilon^{-3/2},\bar{N}^{-1/2}\epsilon ^{-(2+d/4)},\bar{N}^{-1/2}\epsilon ^{-(1/2+d/4)} \right),\nonumber
\EEA
in high probability, where $h$ depends on $\epsilon$ such that the first two terms vanishes as $\epsilon\to 0$ after $\bar{N}\to\infty$.
\end{thm}

\begin{proof}
For each $i=1,\ldots,N$, using the definitions in \eqref{vectoru}-\eqref{vectoru2}, 
\BEA
\left| (\mathbf{L}^g(\vec{u}^M))_{i}-\mathcal{L}u(x_{i})\right|  &=&  \left| \big(\mathbf{L}^g(\vec{u}^M_{\epsilon })\big)_{i}-\mathcal{L}u(x_{i}) + {\big(\mathbf{L}^g(\vec{u}^M)-\mathbf{L}^g(\vec{u}^M_\epsilon)\big)_i}\right| \nonumber
\\
&=& \left| \big(\mathbf{L}^{(1)} \vec{u}^M_\epsilon + \mathbf{L}^{(2)} \vec{u}^G_\epsilon )\big)_i -\mathcal{L}u(x_{i}) + \big((\mathbf{L}^{(1)}+\mathbf{L}^{(2)}\mathbf{A})(\vec{u}^M-\vec{u}^M_\epsilon)\big)_i \right| \nonumber
\\
&=& \left|(\mathbf{L}^{(1)}\vec{u}^{M})_{i}+(\mathbf{L}^{(2)}\vec{u}^{G})_{i}-\mathcal{L}u(x_{i})+\left(\mathbf{L}^{(2)}\mathbf{A}( \vec{u}^{M}-\vec{u}^{M}_\epsilon)\right)_{i} + \left(\mathbf{L}^{(2)}( \vec{u}_{\epsilon }^{G}-\vec{u}^{G}) \right)_{i}\right|\nonumber
\\
&=& \left|(\mathbf{L}^h\vec{u})_{i}-\mathcal{L}u(x_{i})+\left(\mathbf{L}^{(2)}\mathbf{A}( \vec{u}^{M}-\vec{u}^{M}_\epsilon)\right)_{i} + \left(\mathbf{L}^{(2)}( \vec{u}_{\epsilon }^{G}-\vec{u}^{G}) \right)_{i}\right|\nonumber
\\
&\leq& \left| (\mathbf{L}^h\vec{u})_{i}-(\mathbf{L}\vec{u})_{i}\right|  +\left| (\mathbf{L}\vec{u})_{i}-\mathcal{L}u(x_{i})\right|+\left| \left(\mathbf{L}^{(2)}\mathbf{A}( \vec{u}_{\epsilon }^{M}-\vec{u}^{M}) \right)_{i}\right|+\left|\left(\mathbf{L}^{(2)}( \vec{u}_{\epsilon }^{G}-\vec{u}^{G}) \right)_{i}\right|, \nonumber\\
&\leq& \mathcal{O}(h^2\epsilon^{-3/2}) + \mathcal{O}(\epsilon,\bar{N}^{-1/2}\epsilon ^{-(2+d/4)},\bar{N}^{-1/2}\epsilon ^{-(1/2+d/4)}) + \mathcal{O}(h\epsilon^{-1/2}) \nonumber \\ &&+ \mathcal{O}(h^{3}\epsilon^{-1},h^2\epsilon^{-3/2},\bar{N}^{-1/2}\epsilon ^{-(2+d/4)},\bar{N}^{-1/2}\epsilon ^{-(1/2+d/4)}).\nonumber
\EEA

The first error term is a consequence of the diffusion matrices from two different sets of points as seen in Lemma~\ref{lem:noisymatrix}, which is Eq.~\eqref{errorofLh}. The second error term is the pointwise error bound of the standard diffusion maps in Lemma~\ref{lem:old}. The third error term results from the estimated interior ghost points and is bounded by \eqref{vectoru2} multiplied by $\epsilon^{-1}$ from the components of $\mathbf{L}^{(2)}\mathbf{A}$. The last error term results from the estimated exterior ghost points and is given by Prop.~\ref{prop:extrapolationofu}, multiplied by  $\epsilon^{-1}$ from the components of $\mathbf{L}^{(2)}$.

For the well-sampled data, the third error term also vanishes based on \eqref{vectoru2}. The first error bound and the second term $h^2\epsilon^{-3/2}$ in the fourth error bound are, both, not applicable as discussed in Remark~\ref{rem:ws}. Thus, the leading error term is $h^3\epsilon^{-1}$ from the fourth error bound for well-sampled data. For the randomly sampled data, the leading error term is $h^2\epsilon^{-3/2}$. Anticipating $h=\mathcal{O}(\epsilon)$, we neglect the third error bound of order-$h\epsilon^{-1/2}$. {\color{black}Collecting these information, we only report the order-$h^3\epsilon^{-1}$ term that is applicable for well-sampled data and the order-$h^2\epsilon^{-3/2}$ that is applicable for the randomly sampled data.}
\end{proof}

}

{\color{black}We should point out that for the approximation of the operators $\mathcal{L}_2$ and $\mathcal{L}_3$ in \eqref{L2} and \eqref{Eqn:L3}, respectively, we assume that $\kappa, B,$ and $C$ are well-defined functions of the domain $M\cup B_{\epsilon^r}(\partial M)$. This is due to the fact that the associated asymptotic expansion in Eq.~\eqref{L2} or the kernel in Eq.~\eqref{Eqn:localK} require evaluations of the associated kernel at the estimated ghost points $\{\tilde{x}_j^{G_k}\}_{j=1,k=0}^{J,K}\subset B_{\epsilon^r}(\partial M)$. While one can devise an extrapolation method to determine the function values at these ghost points had the functions were only defined on $M$, we neglect it in the present work to avoid the extra complication in the analysis above. In our numerics below, we assume that we are given $\kappa, B, C$ that can be evaluated on any point cloud $x\in M\cup B_{\epsilon^r}(\partial M)\subset \mathbb{R}^n$.}

\subsection{Numerical verification}\label{numericforward}

In this section, we provide supporting numerical results of the GPDM method on the semi-ellipse Example~\ref{1Dexample} and assess the error of the affine operator in \eqref{GPDM} in estimating $\mathcal{L}_2u$ for functions $u$ that satisfy various boundary conditions. For the Robin boundary condition, $\beta_1 \partial_{\boldsymbol{\nu}}u + \beta_2 u = g$, we set $\beta _{1}\left( x\right) =1$, $\beta _{2}\left( x\right)
=3/\left( 2a\right) $ with the homogeneous $g=0$ at both boundary points, $x_{1}$
and $x_{N}$. For this numerical example, we set $\kappa =1.1+\sin \theta $. Choosing the true function
to be $u=\cos (3\theta /2-\pi /4)$, one can check that this function
satisfies the above Robin boundary condition. The analytic $f=\mathcal{L}%
_{2}u$ can be calculated from \eqref{explicitLs}. For the Dirichlet and Neumann boundary conditions, we choose the appropriate $u$ that satisfies the boundary conditions and proceed in the similar fashion.

The components of $u$ are evaluated at equally angle distributed points $\{\theta _{i}=\frac{\left( i-1\right) \pi}{N-1}\}_{i=1,...,N}$. In the following numerical experiment, we set $N=400$ and $k=50$ nearest neighbors (this is the same configuration that produces Fig.~\ref{fig1_ellip_noghost}). Fig.~\ref{Fig6_eps_semiellipse2} shows the Forward Error (FE) defined as $\left\Vert \mathbf{L}^g(\vec{u}^M)-%
\mathcal{L}_{2}u\right\Vert _{\infty }$ as a function of
the bandwidth parameter, $\epsilon$, for various boundary conditions. One can see from Fig.~\ref
{Fig6_eps_semiellipse2} that with the GPDM, the uniform FE reduces
substantially on a wide range of  $\epsilon=10^{-5}-10^{-2}$. This indicates that the solution of the GPDM
becomes much more accurate for the $\epsilon$ tuning compared to the standard
DM, even in the Neumann case. 

In Fig.~\ref{Fig6_eps_semiellipse2}, we also show the results obtained from the auto-tuning algorithm discussed in Section~\ref{sec:dm}. While this automated tuning strategy may not necessarily give the best estimates on the resulting operator estimation (for example, notice that the yellow and blue points in Fig.~\ref{Fig6_eps_semiellipse2} do not correspond to the minimum Forward Error), it often gives a starting point for further tuning and is numerically cheap. For a theoretically justifiable method, yet computationally more elaborate, one can also use the local singular value decomposition technique in \cite{Berry2016IDM}.

\begin{figure}
{\scriptsize \centering
\begin{tabular}{ccc}
{\normalsize Robin} & {\normalsize Dirichlet} & {\normalsize Neumann}  \\
\includegraphics[scale=0.35]{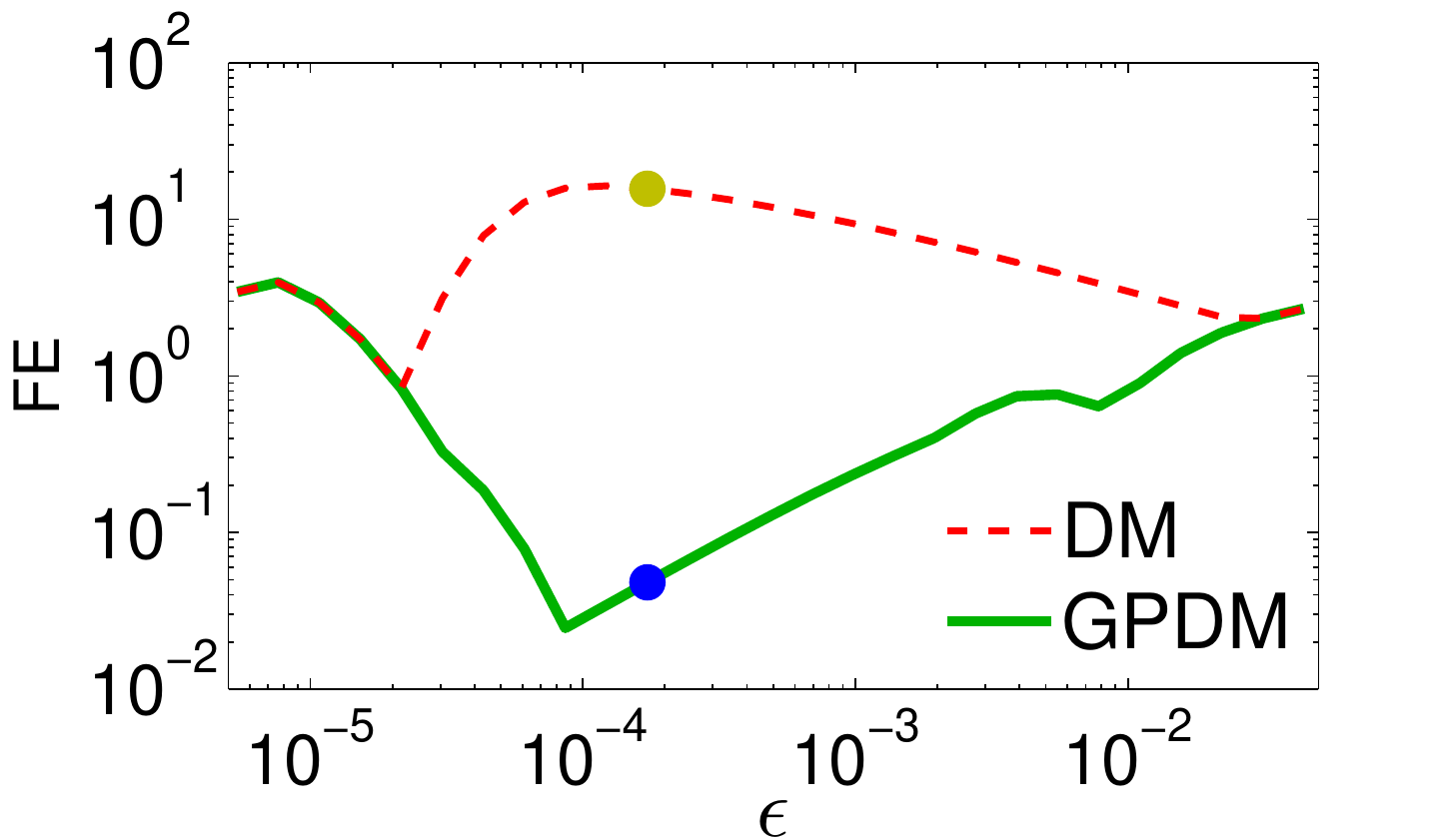}
&
\includegraphics[scale=0.35]{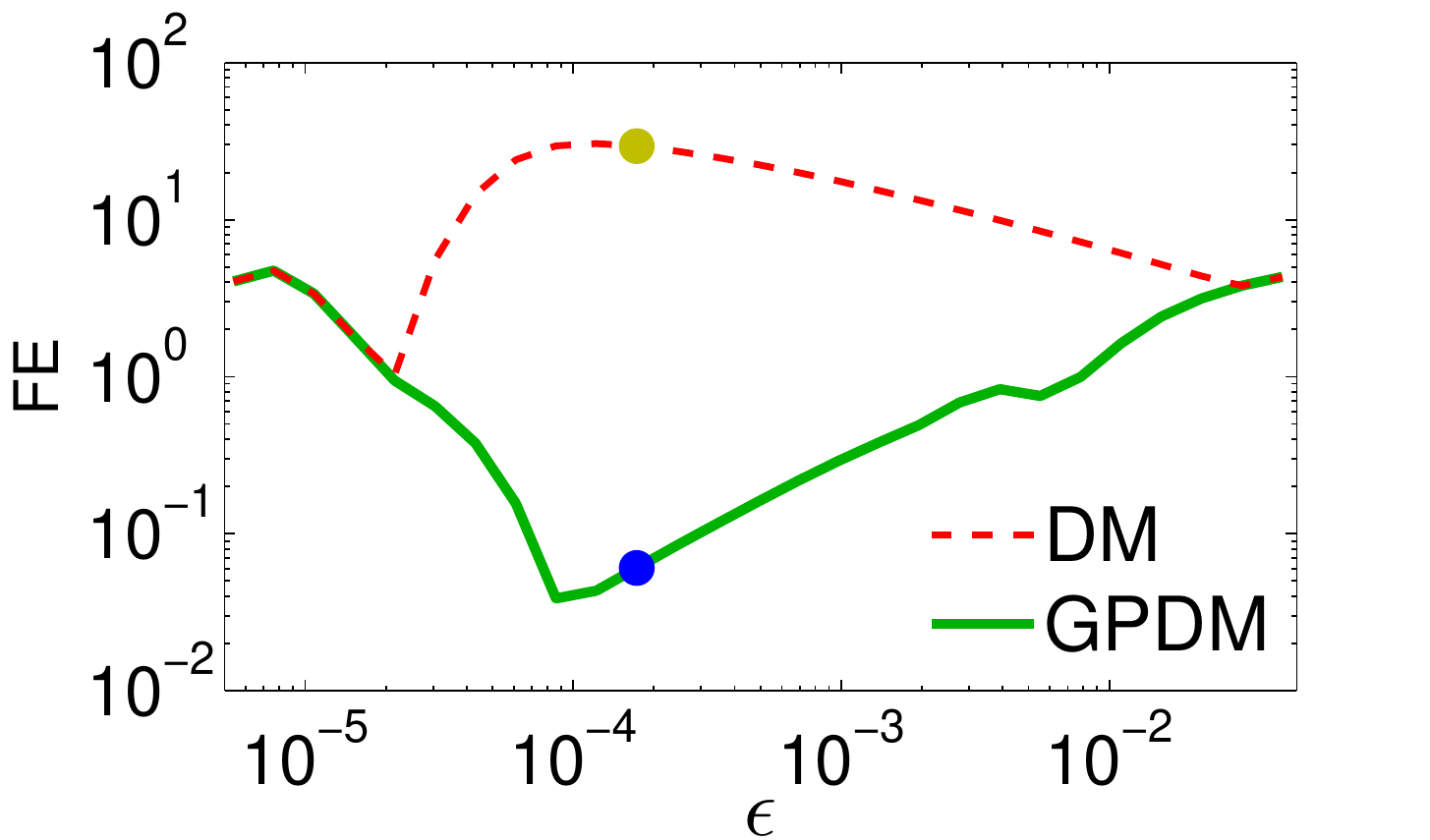}
&
\includegraphics[scale=0.35]{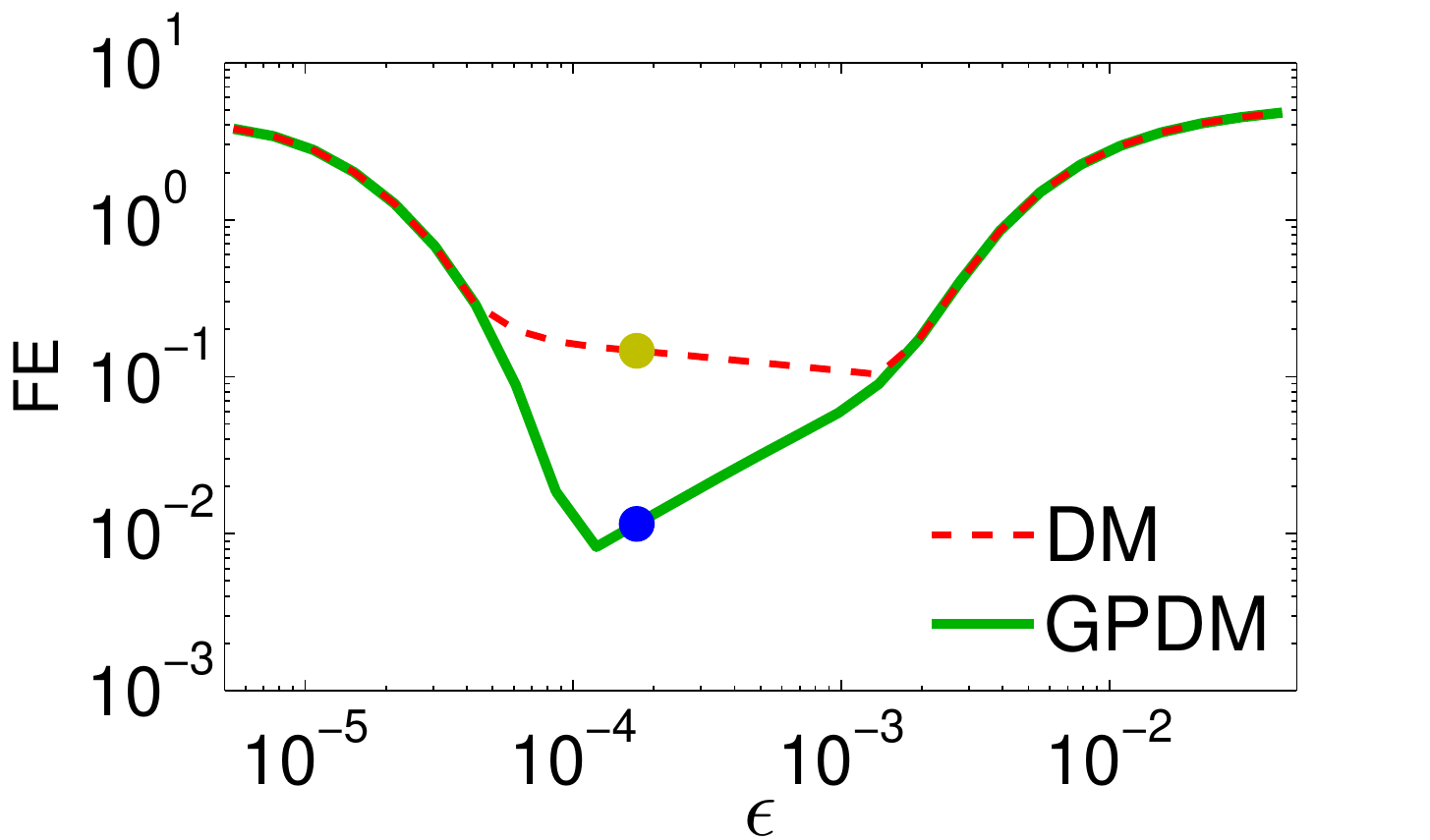}
\end{tabular}
}
\caption{(Color online) Forward Error (FE) of the operator estimation as a function of the bandwidth $\protect\epsilon$ for the semi-ellipse example {with fixed $N=400$ well-sampled data}. The operator acts on a test function satisfying homogeneous: (a) Robin, (b) Dirichlet, and (c) Neumann boundary conditions. The
yellow point and blue point correspond to the auto-tuned $\protect\epsilon$
for DM and GPDM, respectively. }
\label{Fig6_eps_semiellipse2}
\end{figure}

Fig.~\ref{Fig7_N_semiellipse2} shows the FE as a function of the number of points $N$. For comparison, we also show numerical results obtained from the standard DM without adding ghost points. The GPDM FE $\Vert
\mathbf{L}^g(\vec{u}^M)-\mathcal{L}_{2}u\Vert _{\infty }$ (green curve)\ is
a uniform error computed at all $N$ points on manifold $M$. The DM FE $\Vert \mathbf{L}\vec{u}^M-\mathcal{L}_{2}u\Vert _{\infty }$
depicted by the black dashed curve is computed at all $N$ points, whereas
the DM FE depicted by the red dashed curve is computed only at points $%
x_{10}-x_{N-9}$ away from boundary. One can see from Fig. \ref%
{Fig7_N_semiellipse2} that the DM FE on the
interior of $M$ and the GPDM error on all points of $M$ decay on $\mathcal{O}(N^{-2})$ whereas the DM FE on $M$ increases on $%
\mathcal{O}(N^{1})$ for both Robin and Dirichlet BC's and of $\mathcal{O}(1)$ for Neumann BC's.
This indicates that for DM, the increasing FE comes from the boundary when $N$ increases.
Incidentally, we notice that for the case of no boundary for manifold $M$, FE decays as $\mathcal{O}(N^{-2})$\
(see \cite{gh2019}). However, in the presence of boundary conditions, only the GPDM FE decays as $\mathcal{O}(N^{-2})$.

\begin{figure*}[tbp]
{\scriptsize \centering
\begin{tabular}{ccc}
{\normalsize Robin} & {\normalsize Dirichlet} & {\normalsize Neumann}  \\
\includegraphics[scale=0.35]{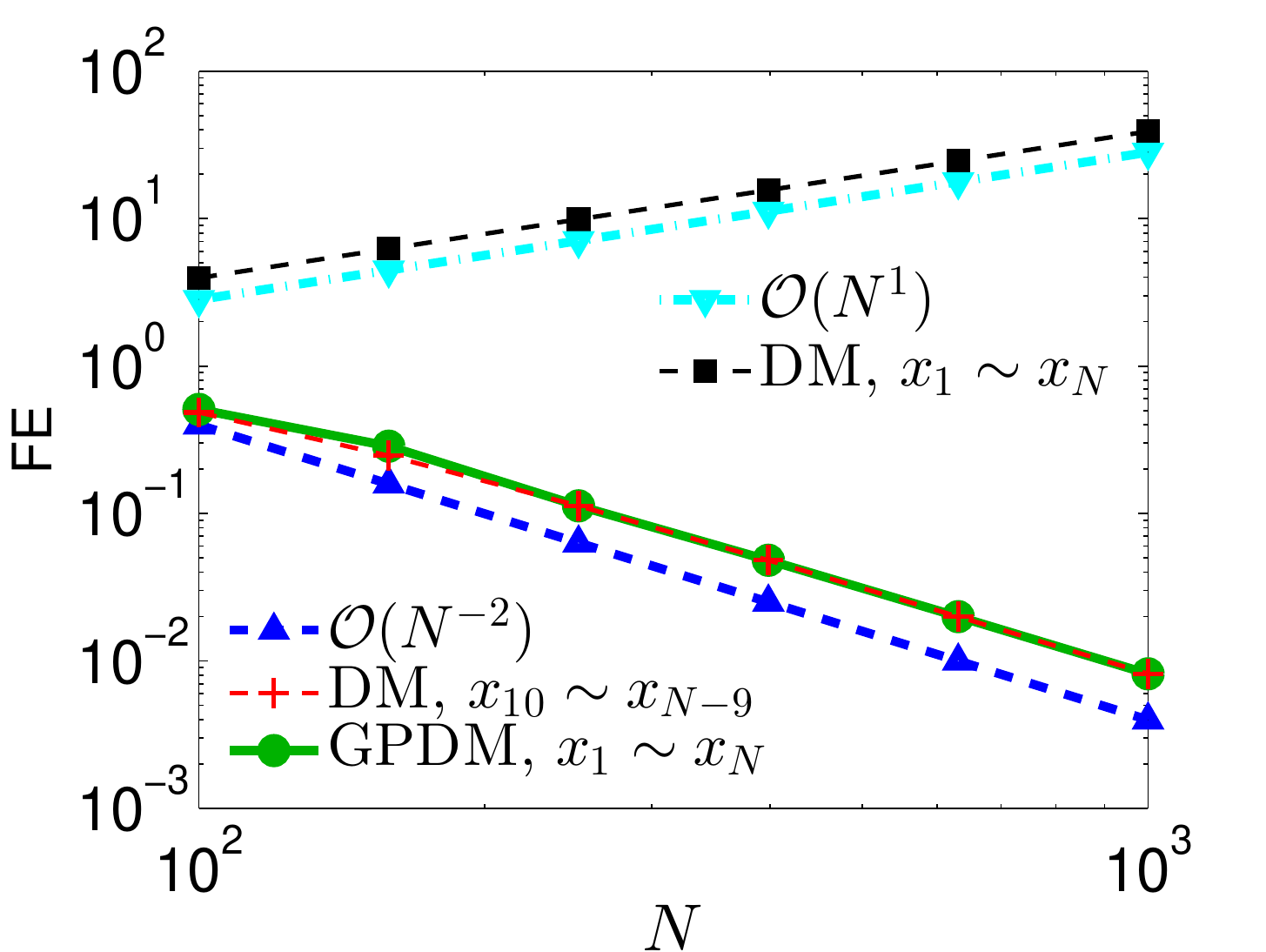}
& %
\includegraphics[scale=0.35]{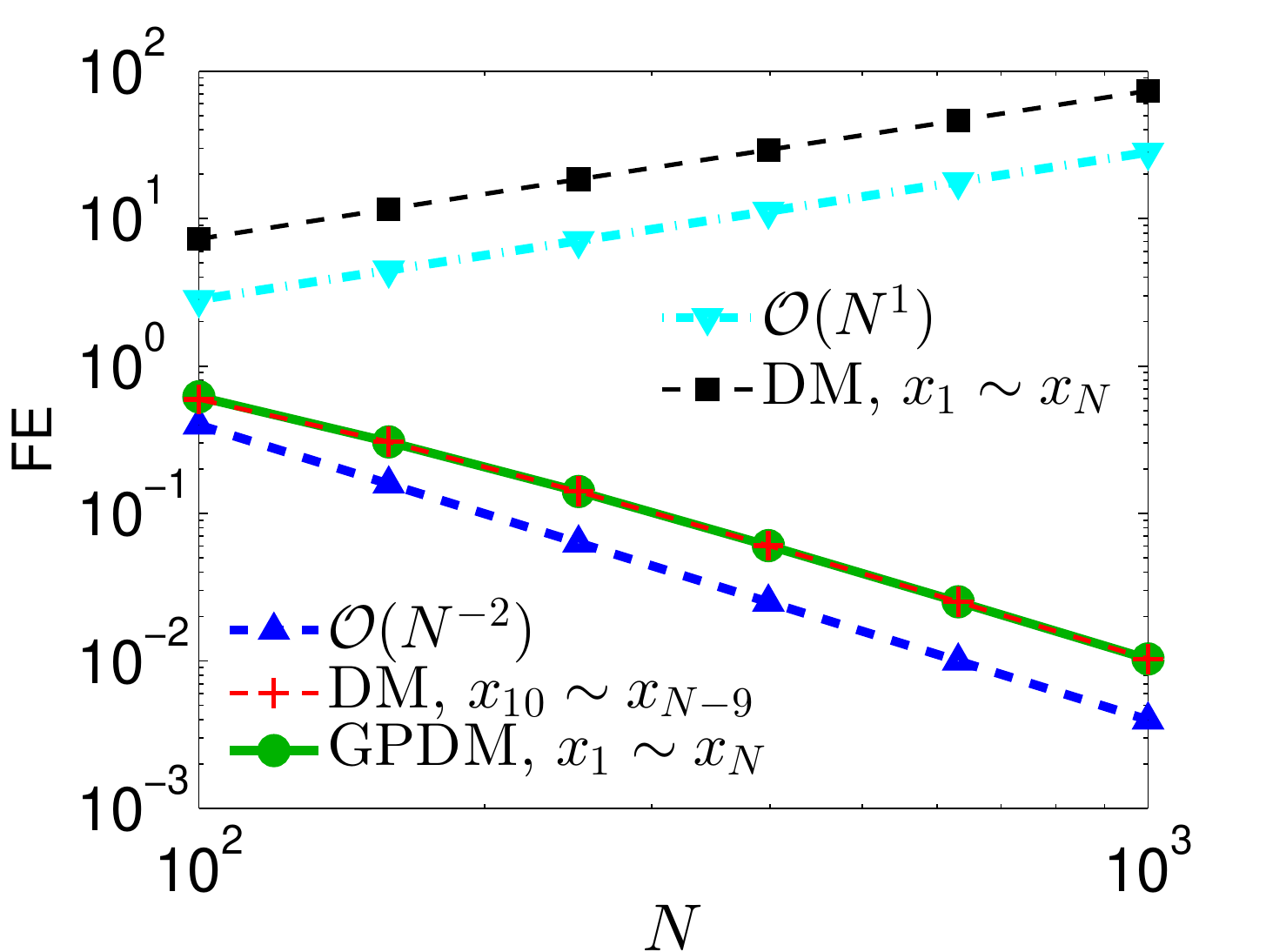}
&
\includegraphics[scale=0.35]{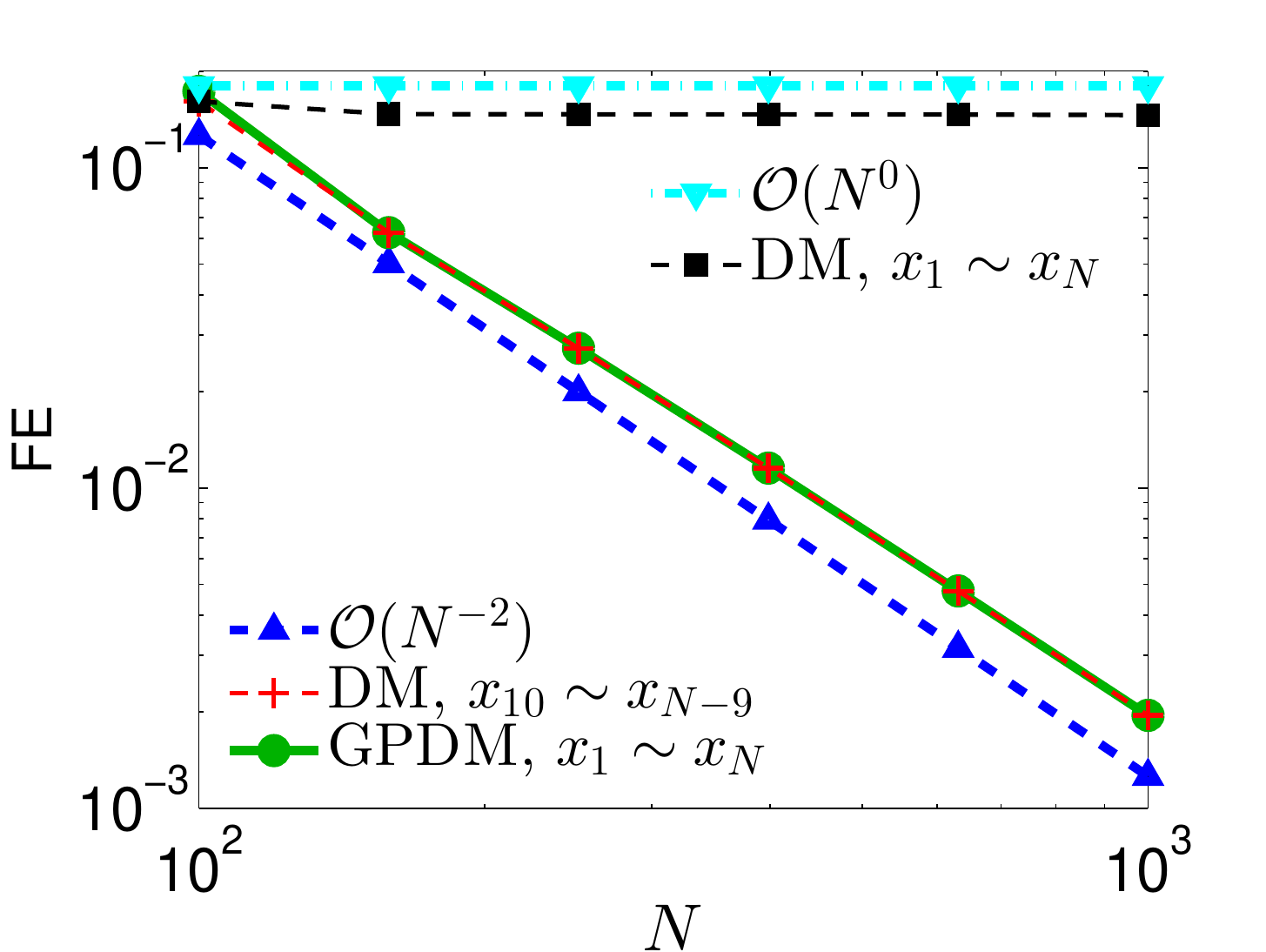}
\end{tabular}
}
\caption{(Color online) Comparisons of Forward Errors (FEs) of the estimated operators as functions of the number of points $N$ for the semi-ellipse example. The operator acts on a test function satisfying homogeneous: (a) Robin, (b) Dirichlet, and (c) Neumann boundary conditions. For GPDM, the
FE $\left\Vert \mathbf{L}^g(\vec{u}^M)-\mathcal{L}_{2}u\right\Vert _{\infty }$ is
computed on all points on the manifold, ${M}$ (green solid line). For DM, The FE $\Vert \mathbf{L}\vec{u}^M-\mathcal{L}_{2}u\Vert _{\infty }$ is computed on all points on $M$ (black dashed
line) and only on interior points $\{x_i\}_{i=10,\ldots,N-9}$ away from the
boundary (red dashed line); i.e., neglecting nine closest points from each boundary point. The bandwidth $\protect\epsilon$ is auto-tuned for each $N$ {number of well-sampled data}. }
\label{Fig7_N_semiellipse2}
\end{figure*}

\section{Applications: Solving linear elliptic PDE's}\label{sec:PDE}

In this section, we consider solving the elliptic PDE's problem on smooth manifold $M$,
\BEA
\begin{aligned}\label{PDE}
\mathcal{L} u = f, \quad\quad x\in M^o, \\
\mathcal{B}u:= (\beta_1 \partial_{\boldsymbol{\nu}} + \beta_2) u = g, \quad\quad x\in \partial M,
\end{aligned}
\EEA
where $\beta_1,\beta_2$ are smooth real-valued functions such that $\beta_1\beta_2>0$ on $\partial M$. Here, the differential operator $\mathcal{L}$ is one of \eqref{L1}-\eqref{Eqn:L3} and is assumed to be uniformly elliptic with smooth coefficients (if any). Here, the smoothness will determine the regularity of the solution. When $\beta_1=0$, we have the Dirichlet boundary condition, when $\beta_2=0$, we have the Neumann boundary condition, and when both are nonzero, we have the Robin boundary condition. For the Neumann boundary condition, we will consider the PDE $(\mathcal{L}-a)u=f$ with $a(x) \geq a_{min}>0$, $\forall x\in M$ for a well-posed problem. For the Robin boundary condition, we also add $-a$ for convenience of the convergence study. {\color{black}For $f \in C^{1,\alpha}(M)$, where $0<\alpha<1$, the PDE problem with appropriate smoothness of the coefficients (if any) admits a unique classical solution $u\in C^{3,\alpha}(M)$, when $g\in C^{3,\alpha}(M)$ for Dirichlet boundary or $g\in C^{2,\alpha}(M)$ for both the Robin and Neumann boundary conditions. We should point out that we impose one-order derivative higher than the usual Schauder estimates ($u\in C^{2,\alpha}$) since the diffusion maps asymptotic expansion (Lemma~\ref{lem:old}) requires a $C^3$-function. For the detailed statement of the Schauder estimates, see Theorem 6.11 of \cite{han2011elliptic} or Theorem 6.25 of \cite{gilbarg2015elliptic} for the Dirichlet problem, Theorem 6.31 of \cite{gilbarg2015elliptic} for the Robin Problem, and \cite{nardi2014} for the Neumann problem.

We should also point out that since the convergence analysis will rely on the consistency of the GPDM estimator, then we require that $u$ is $C^3$ not only on $M$ but on the extended domain $M\cup B_{\epsilon^r}(\partial M)$, which is assumed in Theorem~\ref{theorem1}. The appropriate regularity for $\kappa, B, C$ on the extended domain is also implicitly assumed for the consistent GPDM estimators to, both, the differential operators, $\mathcal{L}_2$ and $\mathcal{L}_3$. Likewise, the function values of $f$ on the estimated interior ghost points $\tilde{x}_j^{G_0}$ are also assumed to be well defined and $f$ is also $C^1$ on the $B_{\epsilon^r}(\partial M)$.
}

In Section~\ref{discretemethods}, we present and report the convergence of the proposed solver, constructed using the GPDM discretization. In Section~\ref{numericsinverse}, we provide supporting numerical examples on simple manifolds. In Section~\ref{sec:face}, we test the PDE solver on problems defined on an ``unknown'' manifold and compare the estimates with the Finite Element Method (FEM) solution.

\subsection{The GPDM discretization method}\label{discretemethods}

Numerically, we will approximate the PDE in \eqref{PDE} with the affine operator in \eqref{GPDM} {for our GPDM method}. To be concise, we define $\hat{u}^{M}=(\hat{u}_1,\ldots ,\hat{u}_N)$, whose components are the numerical solution of the elliptic problem at $\{x_i\}_{i=1}^N$ that also include solutions at the estimated ghost points, $\{\tilde{x}^{G_0}_j\}_{j=1}^J$. Then, the PDE is discretized as,
\BEA
\mathbf{L}^g(\hat{u}^M) = \big(\mathbf{L}^{(1)} + \mathbf{L}^{(2)}\mathbf{A}\big)\hat{u}^M +\mathbf{L}^{(2)} \vec{b} = \vec{f},\label{PDEdiscrete}
\EEA
where $\vec{f}\in\mathbb{R}^{N}$, with components $f_i = f(x_i),\,x_i\in M$ and $x_i=\tilde{x}^{G_0}_j$ for some $i,j$. In the analysis below, we will establish the convergence of the solution $\hat{u}^M$ of the linear problem in \eqref{PDEdiscrete} to the true solution, $\vec{u}^M$, as defined in \eqref{truevectoru}, subjected to boundary conditions.

As for the boundary condition, we discretize the boundary operator for each $x_i\in \partial M$ as,
\BEA\beta_1(x_i) \partial_{\nu} + \beta_2(x_i)  \approx \beta_1(x_i)\left(\frac{\delta( x_i) -\delta(\tilde{x}^{G_0}_i) }{h} \right) + \beta_2(x_i) \delta(x_i): =\mathbf{B}_i,\notag% \label{LBC}
\EEA
following Eq.~\eqref{finitedifferenceerror2}. The Kronecker delta notation $\delta(x)$, which is equal to $1$ on $x$ and $0$ otherwise, is used to clarify that the row vector $\mathbf{B}_i$ (of size $1\times N$) has nonzero components on entries associated to $\tilde{x}^{G_0}_i$ and $x_i\in\partial M$. With this notation, the estimated boundary condition can be written in a compact form as,
\BEA
\mathbf{B}\hat{u}^M = \vec{g},\label{BCdiscrete}
\EEA
where $\vec{g}\in\mathbb{R}^J$, with components $g_i = g(x_i)$ for all $x_i\in \partial M$. Then we have,
{\color{black}
\BEA
\left(\mathbf{B}\hat{u}^M -\mathbf{B}\vec{u}^M\right)_j  &=& g(x_j^B) -(\mathbf{B}\vec{u}^M )_j =  \mathcal{B}u(x_j^B) -(\mathbf{B}\vec{u}^M )_j =
\mathcal{B}u(x_j^B) -\big(\mathbf{B}(\vec{u}^M_\epsilon + \mathcal{O}(h\epsilon^{1/2}))\big)_j \notag\\ &=& \beta_1(x_j^B) \left(\partial_{\boldsymbol{\nu}}u(x_j^B) - \frac{u(x_j^B)-u(\tilde{x}_j^{G_0})}{h} + \mathcal{O}(\epsilon^{1/2}) \right) = \mathcal{O}(h,\epsilon^{1/2}),\label{BCerror}
\EEA
for $j=1,\ldots,J$. For the third equality in the first line, we have used \eqref{vectoru2} such that the error of order $h\epsilon^{1/2}$ only occurs for the randomly sampled data. As for the equality in the second line, for well-sampled data, $\tilde{x}_j^{G_0}$ coincides with one of the interior points (due to the secant line approximation), and the error bound in the approximation of the directional derivative is of order-$h$. For the randomly sampled data, the error bound in the approximation of the directional derivative is given by \eqref{finitedifferenceerror2}.}

\noindent {\bf Dirichlet Problem:}  Numerically, we consider solving an $(N-J) \times (N-J)$ linear problem that is obtained by asserting \eqref{BCdiscrete} to the first $N-J$ row of \eqref{PDEdiscrete}.
To clarify, let us define the submatrices $\mathbf{L}^I\in\mathbb{R}^{(N-J)\times(N-J)}, \mathbf{L}^B\in\mathbb{R}^{(N-J) \times J}$ that satisfy,
\BEA
\begin{pmatrix} \mathbf{L}^I & | & \mathbf{L}^B \end{pmatrix} = \begin{pmatrix}(\mathbf{L}^{(1)} + \mathbf{L}^{(2)}\mathbf{A})_{1} \\ \ldots \\ (\mathbf{L}^{(1)} + \mathbf{L}^{(2)}\mathbf{A})_{N-J}
\end{pmatrix} \in\mathbb{R}^{(N-J) \times N}.\label{defnL}
\EEA
and decompose {\color{black} the estimated solution $\hat{u}^M = (\hat{u}^I,\hat{u}^B)$ to
\BEA
\notag
\begin{aligned}
\hat{u}^I &= (\hat{u}_1,\ldots ,\hat{u}_{N-J}), \quad \quad \mbox{for the interior components},\\
\hat{u}^B &= (\hat{u}_{N-J+1},\ldots ,\hat{u}_{N}), \quad \mbox{for the boundary components.}
\end{aligned}
\EEA
Similarly, we will decompose the true solution as $\vec{u}^M=(\vec{u}^I,\vec{u}^B)$ with,
\BEA
\notag
\begin{aligned}
\vec{u}^I&=(u(x_1),\ldots,u(\gamma_j(h)), \ldots,u(x_{N-J})), \mbox{ for the the interior components,} \\
\vec{u}^B &= ({u}_{N-J+1},\ldots ,{u}_{N}) = (u(x_1^B),\ldots, u(x_J^B)), \quad \mbox{for the boundary components.}
\end{aligned}
\EEA
}

For the Dirichlet boundary condition,
$u(x_j^B)= g(x_j^B)$ for $j=1,\ldots, J$, then one can directly replace $\hat{u}_{N-J+j} = u(x_j^B)= g(x_j^B)$, applying the decomposition in \eqref{defnL} on the first $N-J$ rows of \eqref{PDEdiscrete}, we arrive at the following reduced system,
\BEA
\mathbf{L}^I \hat{u}^I  = \vec{f}^I - \mathbf{L}^B \vec{g}\label{Dirichletsystem}
\EEA
where we have also defined $\vec{f}^I=(f(x_1)- (\mathbf{L}^{(2)}\vec{b})_1,f(x_2)-(\mathbf{L}^{(2)}\vec{b})_2,\ldots,f(x_{N-J})-(\mathbf{L}^{(2)}\vec{b})_{N-J})$. We can show that the solution of \eqref{Dirichletsystem} converges to the solution of the PDE in \eqref{PDE} with Dirichlet boundary condition.

\begin{thm} (Convergence of the Dirichlet Problem)\label{theorem2}
Let $u$ be the solution of PDE in \eqref{PDE} with Dirichlet boundary condition, $u(x_j^B)= g(x_j^B)$ for $j=1,\ldots, J$. Let $\hat{u}_i$ be the solution of the linear system in \eqref{Dirichletsystem}, where the diffusion operator $\mathcal{L}$ is approximated using the GPDM affine estimator in \eqref{GPDM}, constructed with $N$ grid points on the manifold and ghost points of distance, $h\lesssim\mathcal{O}(\epsilon^{r})$ and $0<r<1/2$. Assume that the differential operator $\mathcal{L}$ satisfies the maximum principle, then for any $x_i\in M^o$,
$\hat{u}_i$ converges to $u(x_i)$ with an error bound given as,
\BEA
 | \hat{u}_i - u(x_i)|  = \mathcal{O}\left( h^3\epsilon^{-1},h^2\epsilon^{-3/2},\bar{N}^{-1/2}\epsilon ^{-(2+d/4)},\bar{N}^{-1/2}\epsilon ^{-(1/2+d/4)} \right),\nonumber
\EEA
\end{thm}
in high probability, where $h$ depends on $\epsilon$ such that the first two terms vanishes as $\epsilon\to 0$ after $\bar{N}\to\infty$.
\begin{proof}
See Appendix~\ref{app:D}.
\end{proof}

Recall that some components of $\{\hat{u}_i\}_{i=1}^{N-J}$ correspond to the numerical solutions at the ghost points $\{\tilde{x}_j^{G_0}\}_{j=1}^J$. For these components, $\hat{u}_i$ converges to the true solution $u$, evaluated at the corresponding point, $\gamma_j(h)\in M$. We will elaborate this case in the proof of the next Theorem~\ref{theorem3} (see the discussion after \eqref{robinconsistency}).

\noindent {\bf Robin and Neumann problems:} Here, we consider
\BEA
\begin{aligned}\label{PDENeumann}
(-a+ \mathcal{L}) u = f, \quad\quad x\in M^o, \\
\mathcal{B}u:= (\beta_1 \partial_{\boldsymbol{\nu}} + \beta_2) u = g, \quad\quad x\in \partial M,
\end{aligned}
\EEA
with $a(x) \geq a_{min}>0$, $\forall x\in M$ such that $-a+ \mathcal{L}$ is strictly negative definite. Here, the additional $-a$ term is to ensure the well-posedness of the Neumann problem and for convenience of the convergence study of the Robin problem.

For the discussion below, we write the discrete approximation of the boundary operator as $\mathbf{B} = (\mathbf{B}^I;\mathbf{B}^B)$, where $\mathbf{B}^I\in\mathbb{R}^{J\times(N-J)}$ and $\mathbf{B}^B\in\mathbb{R}^{J\times J}$. Then, the discrete approximation to the PDE problem in \eqref{PDENeumann} is to solve the following $N\times N$ system,
\BEA
\mathbf{N} \hat{u}^M := \begin{pmatrix} -\mathbf{a}+ \mathbf{L}^I & \mathbf{L}^B \\  \mathbf{B}^I & \mathbf{B}^B\end{pmatrix}\begin{pmatrix} \hat{u}^I \\  \hat{u}^B \\  \end{pmatrix} = \begin{pmatrix} \vec{f}^I  \\ \vec{g} \end{pmatrix},\label{discreteRobin}
\EEA
where $\mathbf{a}$ denotes a diagonal matrix with diagonal components $\{a(x_i)\}$. Numerically, one can also solve the last $J$ rows corresponding to the boundary conditions,
\BEA
\hat{u}^B = (\mathbf{B}^B)^{-1}( \vec{g} - \mathbf{B}^I\hat{u}^I ),\label{Robinub}
\EEA
and insert this solution to the first $(N-J)$-rows in problem \eqref{discreteRobin} to obtain a reduced $(N-J)\times(N-J)$ system.

For the Robin problem, we have the following convergence result.
\begin{thm} (Convergence of the Robin Problem)\label{theorem3}
Let $u$ be the solution of PDE in \eqref{PDENeumann} with Robin boundary condition and $\beta_1, \beta_2>0$.
Let the corresponding GPDM estimator be constructed as in Theorem~\ref{theorem2} and we assume that $a\in C^1(M\cup B_{\epsilon}(\partial M))$.  Then, for any $x_i\in M^o$, the solution $\hat{u}_i$ of the linear system in \eqref{discreteRobin} converges to $u(x_i)$ with error bound given as,
\BEA
 | \hat{u}_i - u(x_i)|  = \mathcal{O}\left(h^{3}\epsilon^{-1},h^2\epsilon^{-3/2},\epsilon^{1/2},\bar{N}^{-1/2}\epsilon ^{-(2+d/4)},\bar{N}^{-1/2}\epsilon ^{-(1/2+d/4)} \right),\nonumber
\EEA
in high probability, where $h$ depends on $\epsilon$ such that the first two terms vanishes as $\epsilon\to 0$ after $\bar{N}\to\infty$. 
\end{thm}

\begin{proof}
Using the definition of $\vec{f}^I$ and the decomposition in \eqref{defnL}, one can immediately see the consistency. Multiplying the matrix $\mathbf{N}$ in \eqref{discreteRobin} with a vector consists of the difference between the estimated and the true solutions, we obtain
\BEA
\left(  (-\mathbf{a}+ \mathbf{L}^I )(\hat{u}^I-\vec{u}^I)+  \mathbf{L}^B  (\hat{u}^B-\vec{u}^B)\right)_i
&=& \left(\vec{f}^I - (-\mathbf{a}+ \mathbf{L}^I)\vec{u}_I - \mathbf{L}^B \vec{u}_B\right)_i \nonumber\\
&=& f(x_i) + a(x_i)u(x_i) - \left(\mathbf{L}^{(2)}\vec{b} + \mathbf{L}^I\vec{u}_I  + \mathbf{L}^B \vec{u}_B\right)_i \nonumber\\
&=& \mathcal{L}u(x_i) - \left(\mathbf{L}^g(\vec{u}^M)\right)_i.\label{robinconsistency}
\EEA
for $i=1,\ldots, N-J$. For the randomly sampled case, some of the elements of $\{x_i\}$ are $\tilde{x}_j^{G_0}$ that do not lie on $M$.  For such components, we have,
\BEA
\left(  (-\mathbf{a}+ \mathbf{L}^I )(\hat{u}^I-\vec{u}^I)+  \mathbf{L}^B  (\hat{u}^B-\vec{u}^B)\right)_i
&=& f(\tilde{x}^{G_0}_j)+a(\tilde{x}^{G_0}_j)u(\tilde{x}^{G_0}_j) - \left(\mathbf{L}^g(\vec{u}^M)\right)_i \nonumber \\ &=&  \mathcal{L}u(\gamma_j(h)) - \left(\mathbf{L}^g(\vec{u}^M)\right)_i  \nonumber \\ && + \big( f(\tilde{x}^{G_0}_j)-f(\gamma_j(h)) +a(\tilde{x}^{G_0}_j)u(\tilde{x}^{G_0}_j) -a(\gamma_j(h))u(\gamma_j(h))\big)\nonumber \\
&=&  \mathcal{L}u(\gamma_j(h)) - \left(\mathbf{L}^g(\vec{u}^M)\right)_i  + \mathcal{O}(h\epsilon^{1/2}),\nonumber
\EEA
where the last term is valid under the assumption that $a,f,u \in C^1(B_{\epsilon^r}(\partial M))$.
This additional error of order-$h\epsilon^{1/2}$ will be dominated by the error of order-$h^2\epsilon^{-3/2}$ of Theorem~\ref{theorem1} since $h=\mathcal{O}(\epsilon)$. Consequently, we will just ignore it and refer to the consistency equation in \eqref{robinconsistency}.
The last $J$ rows corresponding to the boundary points are nothing but \eqref{BCerror}.

From Eq.~\eqref{sumzero} in Appendix~\ref{app:D}, the column sum of each row of the matrix $\mathbf{M}= \epsilon(\mathbf{L}^{(1)} + \mathbf{L}^{(2)}\mathbf{A})$ is zero and that $\mathbf{M}_{i,i}<0$ and $\mathbf{M}_{i,j}>0$ for all $j\neq i$. Since the first $N-J$ rows of $\mathbf{N}$ is nothing but $-a(x_i)+(\mathbf{L}^{(1)} + \mathbf{L}^{(2)}\mathbf{A})_{i}$, we have
\BEA
|\mathbf{N}_{i,i}| - \sum_{\stackrel{j=1}{j\neq i}}^N |\mathbf{N}_{i,j}| = |-a(x_i) + \epsilon^{-1}\mathbf{M}_{i,i}| - \epsilon^{-1} \sum_{\stackrel{j=1}{j\neq i}}^N |\mathbf{M}_{i,j}|  = a(x_i) -\epsilon^{-1}\sum_{j=1}^N \mathbf{M}_{i,j}=a(x_i) \geq a_{min}>0,\notag %\label{SDD}
\EEA
for $i=1,\ldots,N-J$. Also, the last $J$ rows of the matrix $\mathbf{N}$ are strictly diagonal dominant as long as $\beta_2>0$. For example, in 1D case where $J=2$, the last two rows of \eqref{discreteRobin} is given as,
\BEA
\mathbf{B}^B\hat{u}_B + \mathbf{B}^I\hat{u}_I := \begin{pmatrix}\frac{\beta_1(x_1)}{h}+\beta_2(x_1) & 0 \\ 0 & \frac{\beta_1(x_N)}{h}+\beta_2(x_N) \end{pmatrix} \begin{pmatrix} \hat{u}_1 \\ \hat{u}_N\end{pmatrix} + \begin{pmatrix}-\frac{\beta_1(x_1)}{h} & 0 & \ldots & 0 \\ 0 &\ldots & 0 & -\frac{\beta_1(x_N)}{h} \end{pmatrix} \begin{pmatrix} \hat{u}_2 \\ \hat{u}_3 \\ \vdots \\ \hat{u}_{N-1}\end{pmatrix}=\begin{pmatrix} g(x_1) \\ g(x_N)\end{pmatrix}:=\vec{g}.\nonumber
\EEA
In this case, $|\mathbf{N}_{i,i}| - \sum_{\stackrel{j=1}{j\neq i}}^N |\mathbf{N}_{i,j}| = \beta_2(x_i)>0$ for $i>N-J$. Therefore, the matrix $\mathbf{N}$ is strictly diagonal dominant and nonsingular. By the Ahlberg-Nilson-Varah bound \cite{ahlberg1963,varah1975}, the inverse matrix is uniformly bounded,
\BEA
\|\mathbf{N}^{-1}\|_{\infty} \leq \frac{1}{\min_i (|\mathbf{N}_{ii}|-\sum_{\stackrel{ j=1}{j\neq i}}^{N} |\mathbf{N}_{ij}|)} = \frac{1}{\min\{a_{min},\beta_2\}}.\nonumber %\label{ANVbound}
\EEA
Thus, multiplying $\mathbf{N}^{-1}$ to a vector where the first $N-J$ components consist of \eqref{robinconsistency} and the last $J$ components consist of \eqref{BCerror}, we have
\BEA
|\hat{u}_i-u(x_i)| &\leq& \|\mathbf{N}^{-1}\|_{\infty} \left( \max_{i=1,\ldots,N-J} \left|\left(  (-\mathbf{a}+ \mathbf{L}^I )(\hat{u}^I-\vec{u}^I)+  \mathbf{L}^B  (\hat{u}^B-\vec{u}^B)\right)_i \right|, \max_{j=1,\ldots,J}  \left|\left(\mathbf{B}(\hat{u}^M -\vec{u}^M\right)_j\right| \right)\nonumber\\
&=& \|\mathbf{N}^{-1}\|_{\infty}\left( \max_{i=1,\ldots,N-J}  \left| \mathcal{L}u(x_i) - (\mathbf{L}^g(\vec{u}^M))_i \right|, \max_{j=1,\ldots,J}  \left|\left(\mathbf{B}(\hat{u}^M -\vec{u}^M\right)_j\right| \right),\nonumber
\EEA
for all $x_i\in M$.
Since the GPDM is consistent, $ \left |\mathcal{L}u(x_i) - ((\mathbf{L}^g(\vec{u}^M))_i\right| \to 0$ as $\epsilon\to 0$ after $N\to\infty$ with error rate given in Theorem~\ref{theorem1}, together with the error bound in \eqref{BCerror}, the proof is completed.
\end{proof}

For the Neumann problem, the last $J$ components of $\mathbf{N}$ is not strictly diagonal dominant, since $\beta_2=0$. To achieve the convergence, one can consider (without loss of generality) the homogeneous Neumann problem $g=0$ such that
\eqref{Robinub} simplifies to $\hat{u}_B = -(\mathbf{B}^B)^{-1}\mathbf{B}^I\hat{u}^I $. For example in well-sampled case, the discrete approximation in \eqref{finitedifference} yields $\hat{u}_N= \hat{u}_{N-1}$ and $\hat{u}_1= \hat{u}_{2}$. Substituting these solutions ($J$ equations in general) to the first $N-J$ rows of \eqref{discreteRobin}, one can verify that the reduced $N-J$ problem is nonsingular and it has an inverse that is uniformly bounded by $1/a_{min}$. Thus, the convergence can be achieved using the similar argument as in the proof above.

\subsection{Numerical examples on simple manifolds}\label{numericsinverse}

In this section, we discuss three examples of problems defined on simple manifolds. First, we verify the convergence rate with the 1D example in Example~\ref{1Dexample}. In the second example, we test the solver on a semi-torus embedded in $\mathbb{R}^3$ with a mixed-type Dirichlet-Neumann, boundary condition. In the third example, we verify the effectiveness of the proposed method on the randomly-sampled data for the semi-torus PDE problem.

Numerically, we will compare GPDM with the standard DM. To account for other than homogeneous Neumann boundary conditions, we modified the standard DM as follows. We consider the $N-J$ rows corresponding to the interior points $x_i \in M^{o}$ of the equation, $\mathbf{L}_{\rm DM} \hat{u}^M=\vec{f}$, where $\mathbf{L}_{\rm DM}$ is the standard diffusion maps operator. To approximate boundary conditions that involve normal derivatives, we use the Algorithm in Appendix~\ref{App:A} that requires no interior ghost points. Then the inverse problem consists of solving the reduced linear system (arised from imposing the appropriate boundary conditions), analogous to the reduced linear problem with GPDM.

\subsubsection{Anisotropic diffusion on a semi-ellipse with well-sampled data} First, let us present the results of the 1D problem in Example~\ref{1Dexample} in solving
\BEA
\mathcal{L}_2 u = f,\label{PDEproblem2}
\EEA
with the three boundary conditions. In this numerical experiment, the configuration is the same as in Section~\ref{numericforward}. Particularly, Fig.~\ref{Fig9_eps_semiellipse2} demonstrates the error of the solutions, $\|\hat{u}^M -\vec{u}^M\|_\infty$, which we refer to as the Inverse Error (IE), as a function of $\epsilon$ for fixed $N=400, k=50$. Compared to the standard diffusion maps, notice that GPDM is more robust for the case of Robin and Dirichlet boundary conditions, as expected. The advantage of GPDM over DM on Robin and Dirichlet boundary conditions is more apparent in Fig. \ref{Fig10_N_semiellipse2}. Particularly, for the Robin BC, one can see that the GPDM IE decays on $\mathcal{O}(N^{-1})$ whereas the DM IE does not decay and is nearly constant.  For the Dirichlet BC, GPDM IE decays faster compared to the DM IE. For the Neumann BC, we see comparable IEs as functions of $N$, as expected.

\begin{figure}
{\scriptsize \centering
\begin{tabular}{ccc}
{\normalsize Robin BC} & {\normalsize Dirichlet BC} & {\normalsize Neumann BC}  \\
\includegraphics[scale=0.35]{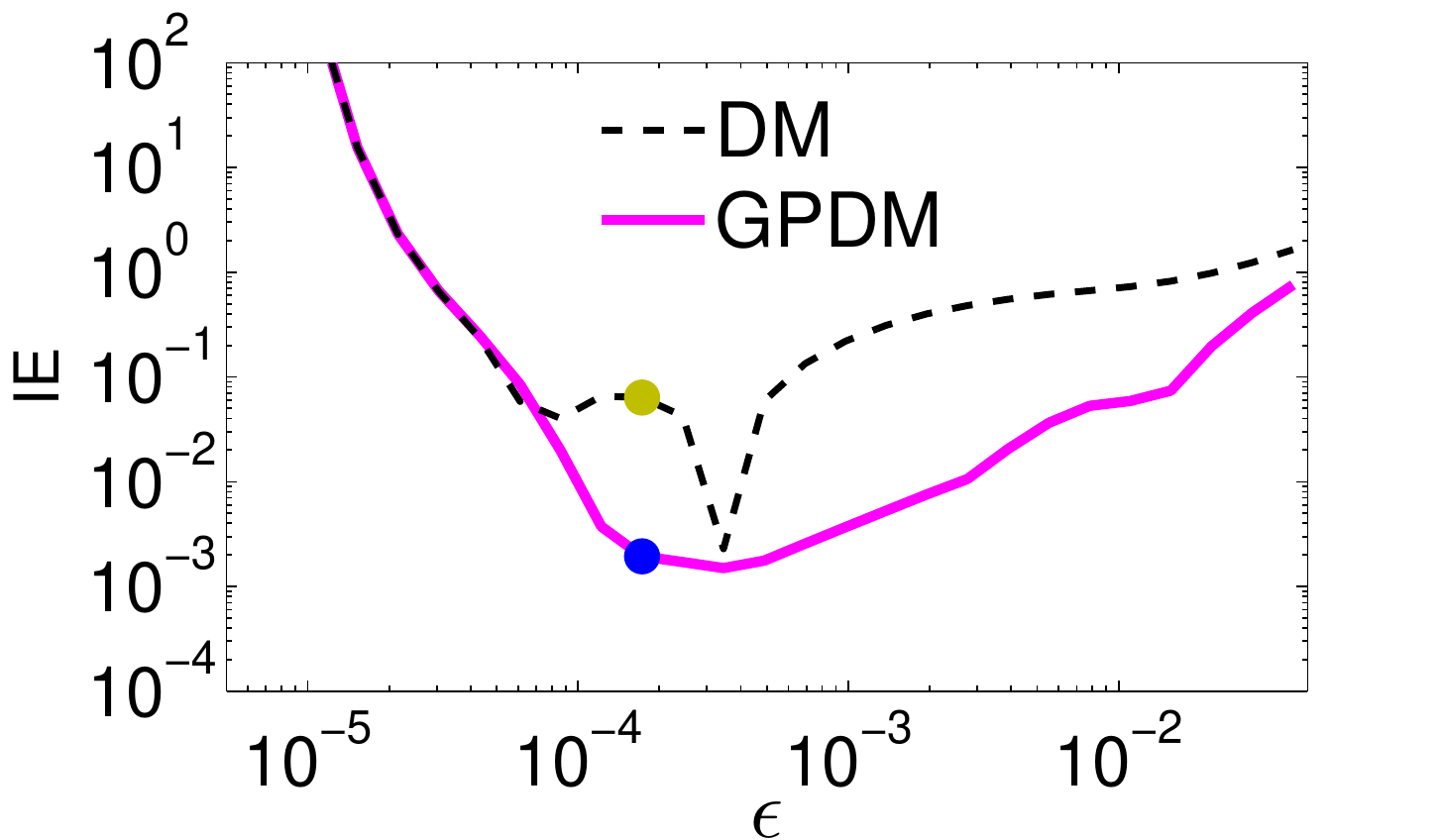}
&
\includegraphics[scale=0.35]{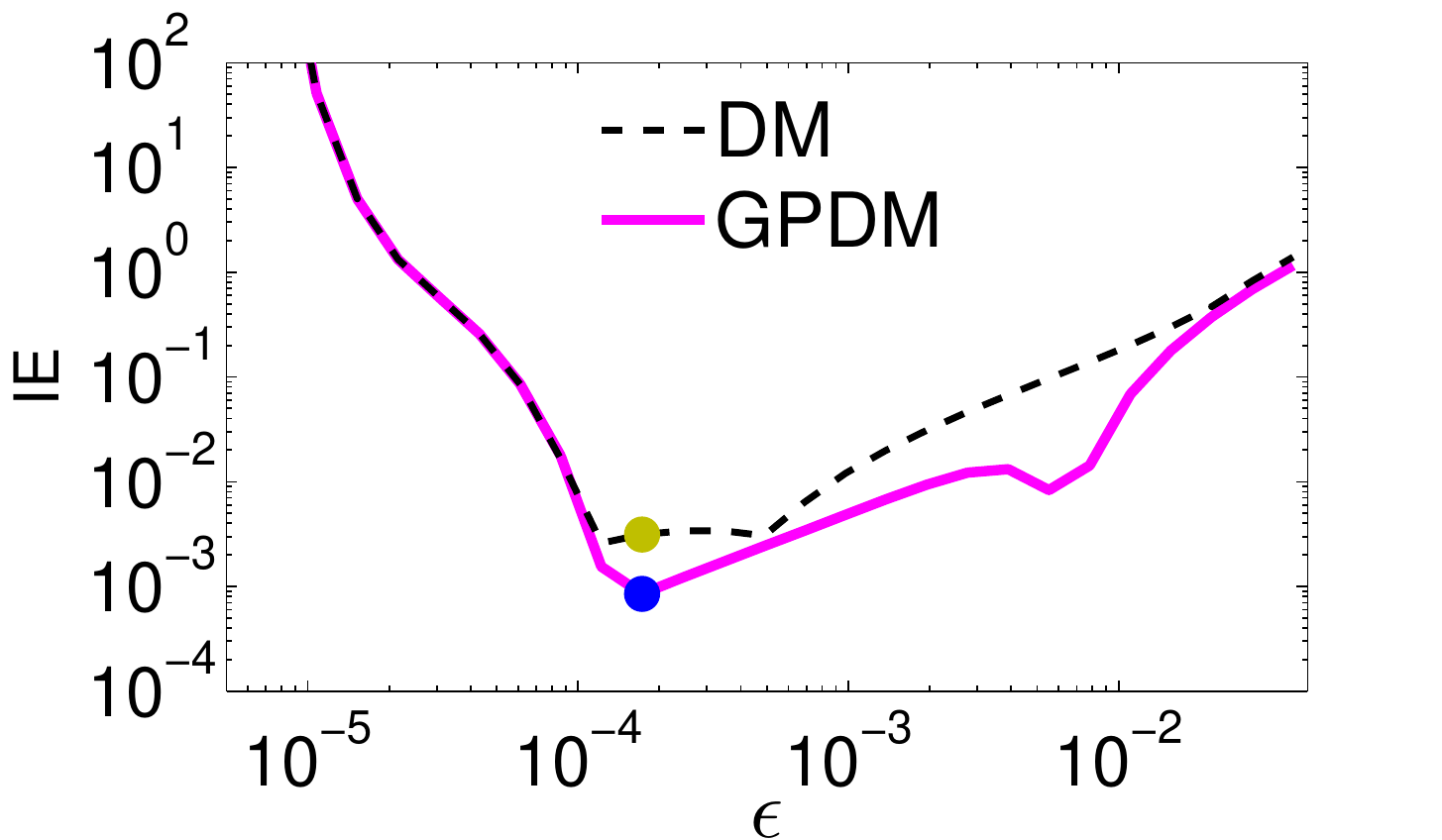}
&
\includegraphics[scale=0.35]{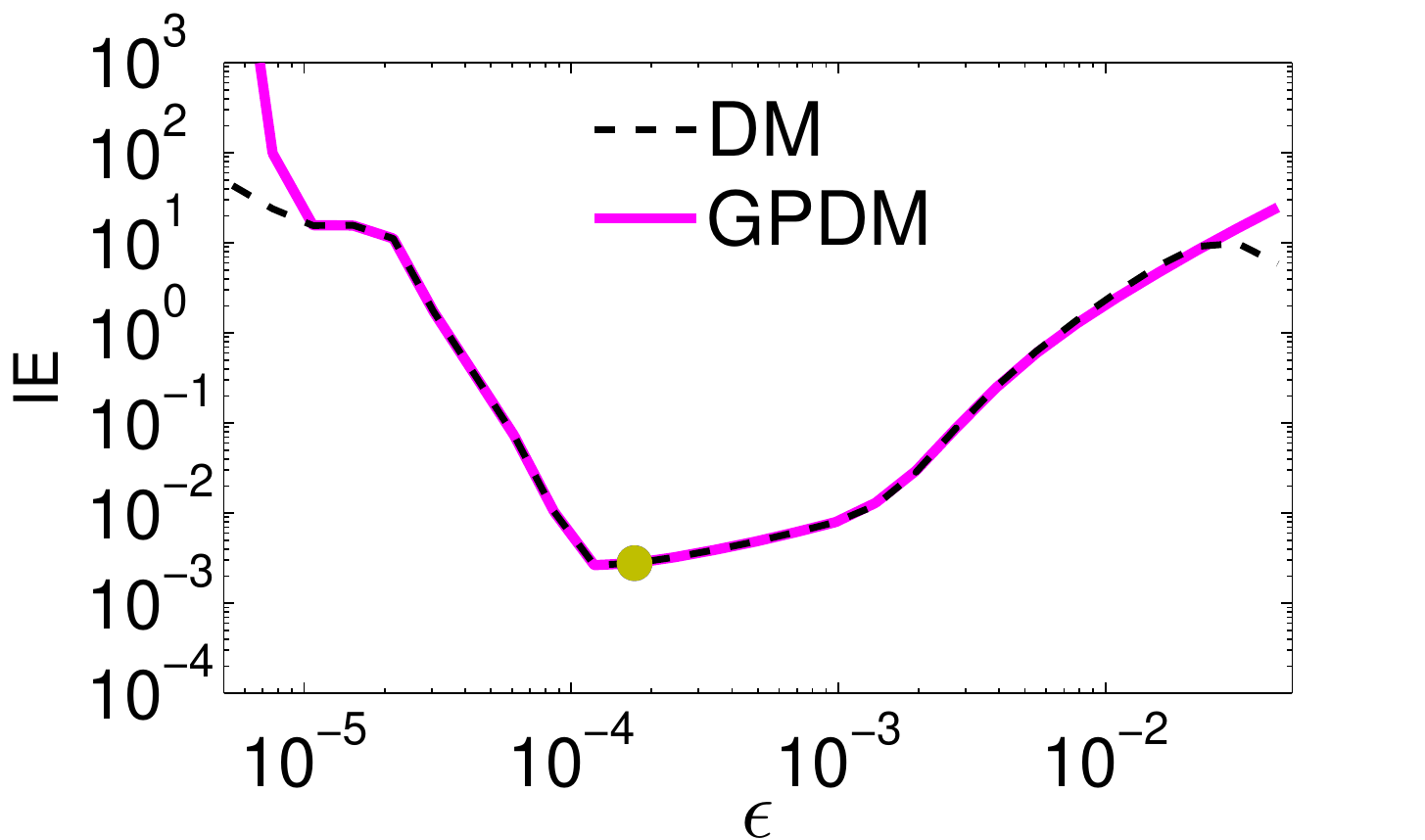}
\end{tabular}
}
\caption{(Color online) Pointwise Inverse Error (IE) of the solution of \eqref{PDEproblem2} as a function of the bandwidth $\protect\epsilon$ {for  the semi-ellipse example with fixed $N=400$ well-sampled data}. The
yellow point and blue point correspond to the auto-tuned $\protect\epsilon$
for DM and GPDM, respectively.}
\label{Fig9_eps_semiellipse2}
\end{figure}

\begin{figure*}[tbp]
{\scriptsize \centering
\begin{tabular}{ccc}
{\normalsize Robin BC} & {\normalsize Dirichlet BC} & {\normalsize Neumann BC}  \\
\includegraphics[scale=0.35]{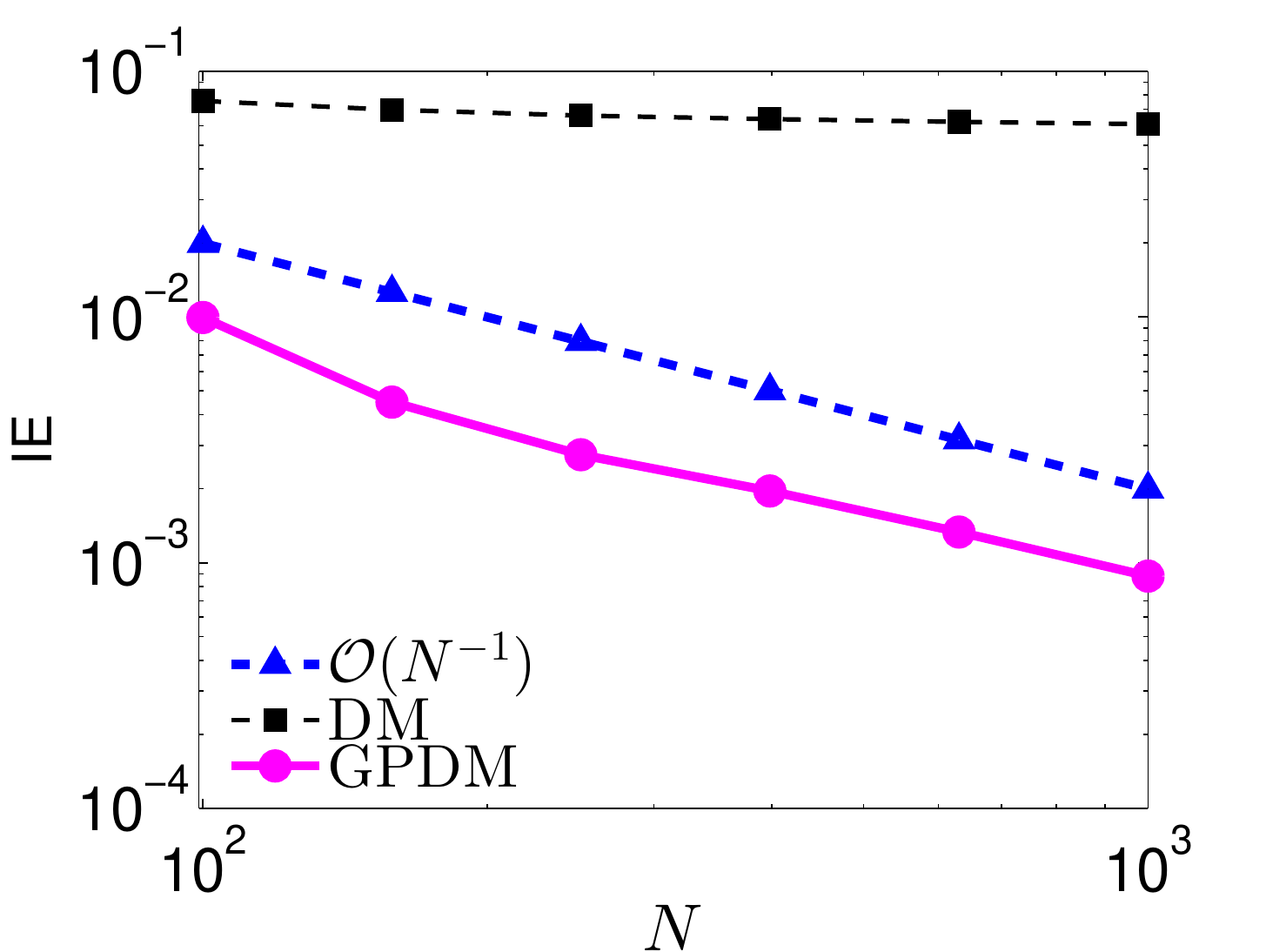}
& %
\includegraphics[scale=0.35]{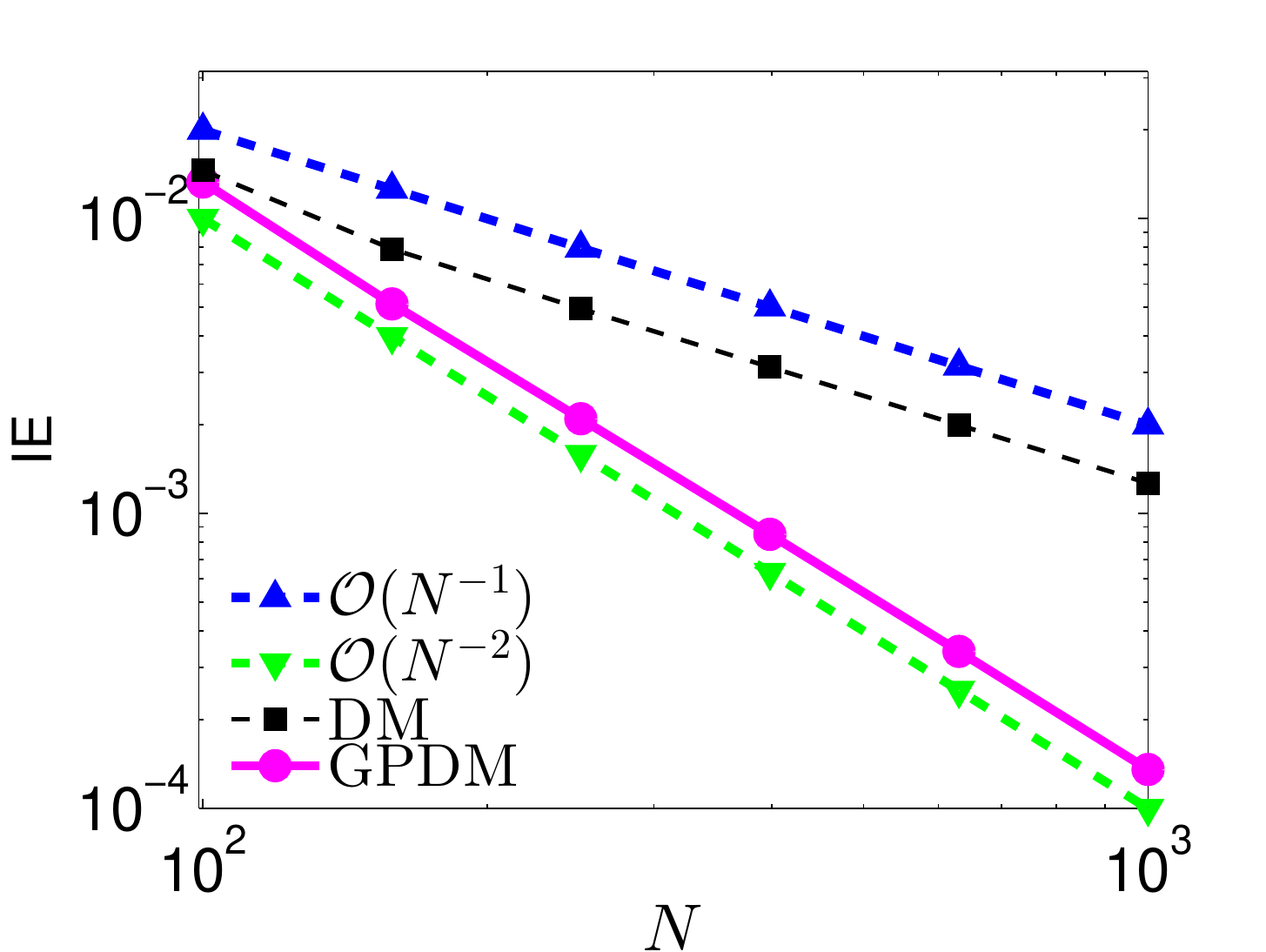}
&
\includegraphics[scale=0.35]{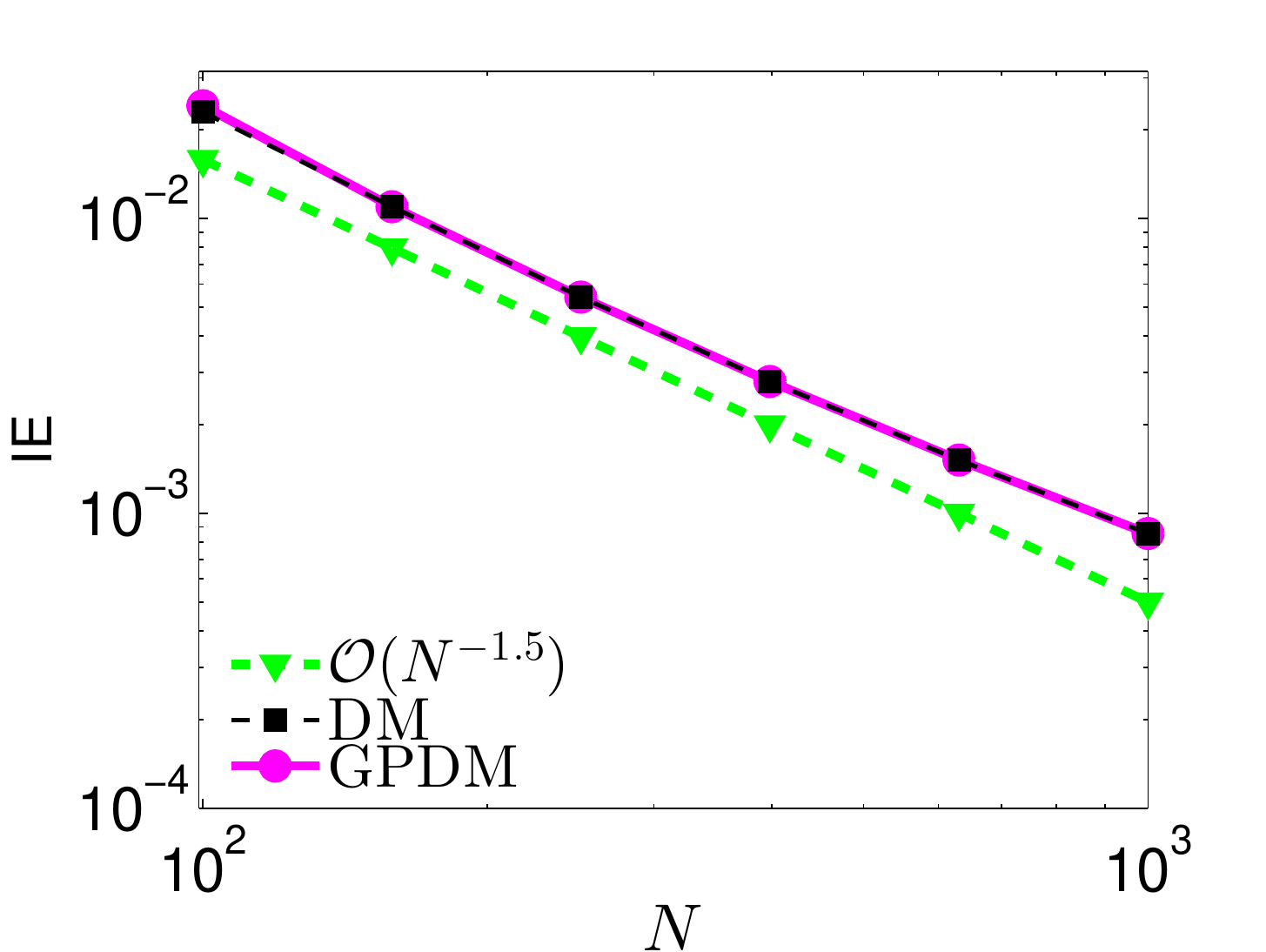}
\end{tabular}
}
\caption{(Color online) Comparisons of the Inverse Errors (IEs) of the solutions of \eqref{PDEproblem2} as functions of $N$ for the semi-ellipse example.  {The bandwidth $\protect\epsilon$ is auto-tuned for each $N$ number of well-sampled data}.}
\label{Fig10_N_semiellipse2}
\end{figure*}

\subsubsection{Non-symmetric backward Kolmogorov elliptic PDE on a semi-torus with well-sampled data \label{L3toruswell}}

\begin{figure}[tbp]
\flushleft
\begin{tabular}{cccc}
{\normalsize DM, $64 \times 64$ } & {\normalsize GPDM, $64 \times 64$} &
{\normalsize DM, $128 \times 128$} & {\normalsize GPDM, $128 \times 128$} \\
\includegraphics[scale=0.24]{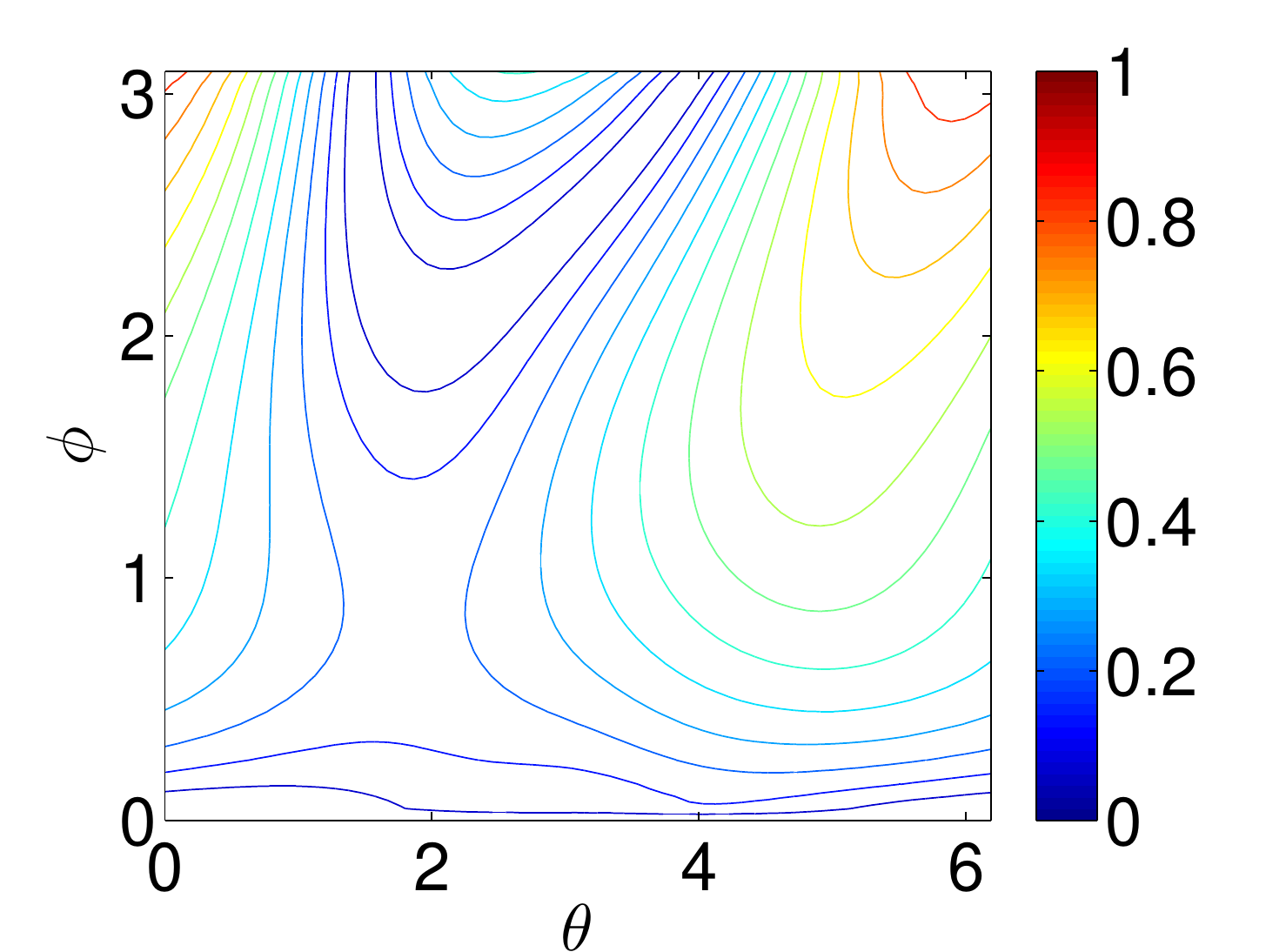}
& %
\includegraphics[scale=0.24]{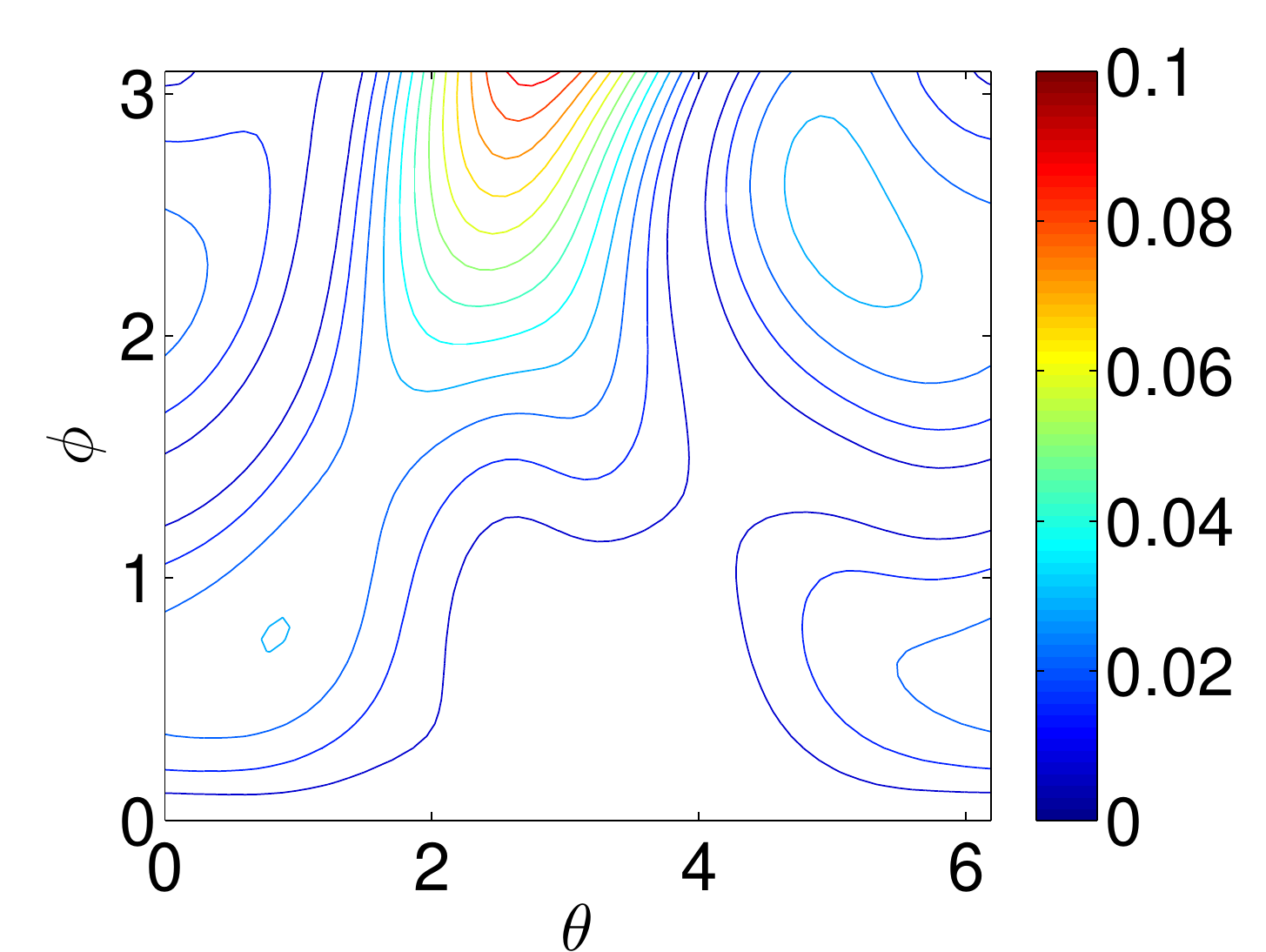}
& %
\includegraphics[scale=0.24]{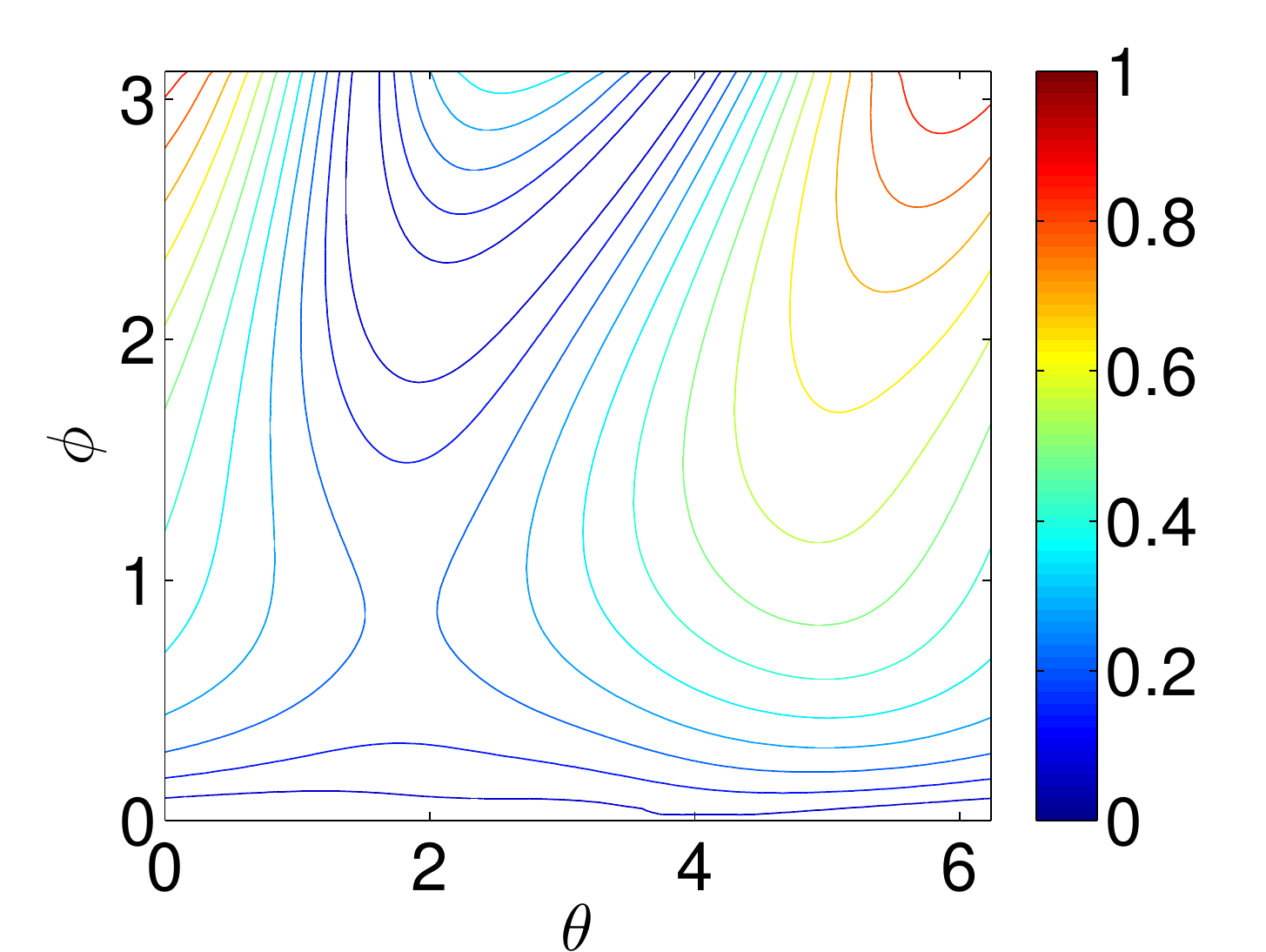}
& %
\includegraphics[scale=0.24]{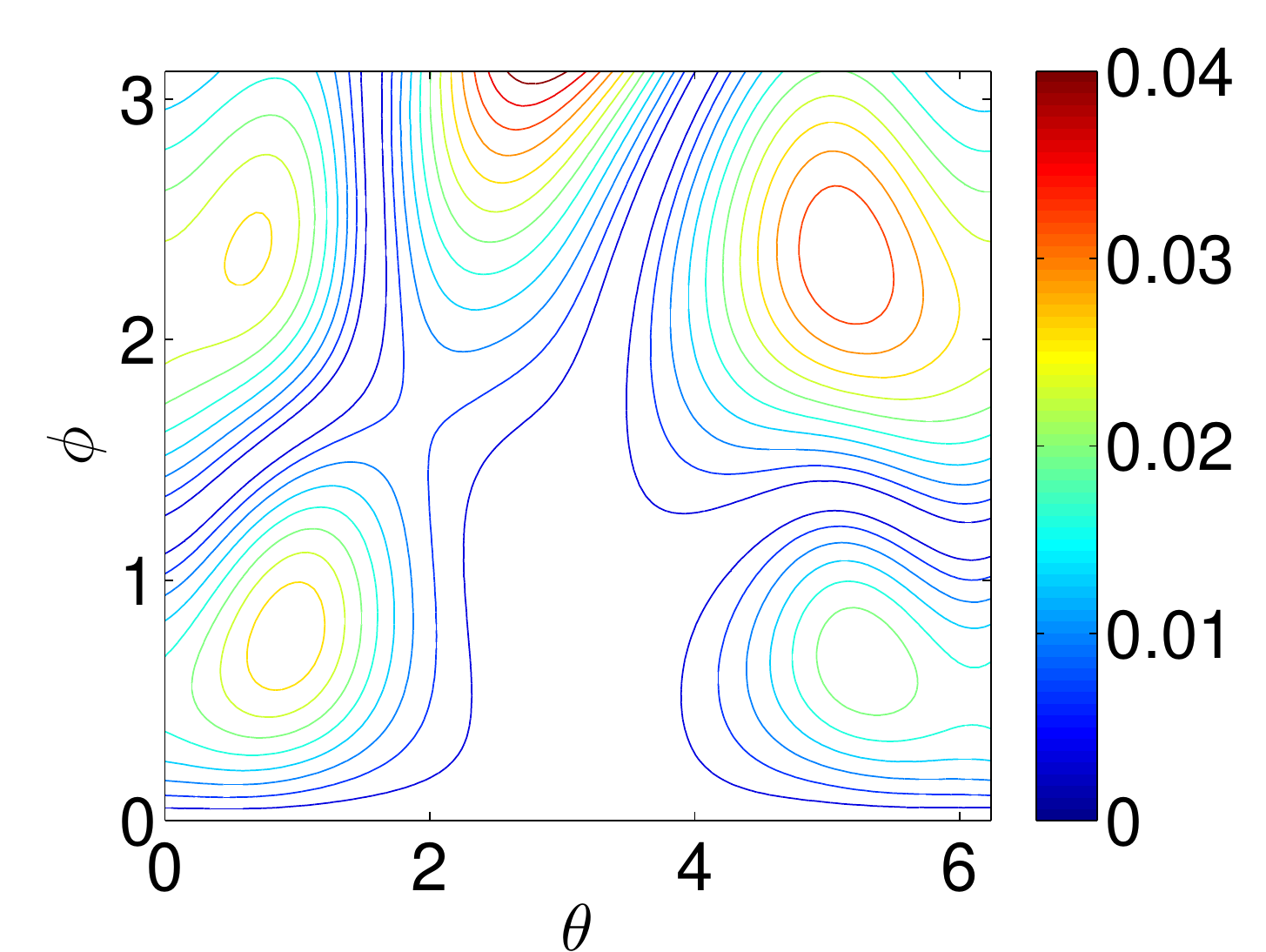}%
\end{tabular}%
\caption{(Color online) Absolute errors in the estimated solutions for the semi-torus example with well-sampled data: (a) DM for $%
N=64 \times 64$ with $\Vert\hat{u}^M -\vec{u}^M \Vert _\infty = 0.88$%
, (b) GPDM for $N=64 \times 64$ with $\Vert\hat{u}^M -\vec{u}^M \Vert _\infty
 = {0.095}$, (c) DM for $N=128 \times 128$ with $\Vert\hat{u}^M -\vec{u}^M \Vert _\infty
  = 0.91$, (d) GPDM for $N=128 \times 128$
with $\Vert\hat{u}^M -\vec{u}^M \Vert _\infty = {0.042}$. }
\label{Fig16_semitorus_IE3}
\end{figure}

\begin{figure*}[tbp]
{\scriptsize \centering

\begin{tabular}{ccc}
 {\normalsize(a) IE vs. $N$ } & {\normalsize(b) DM, IE vs. $\epsilon$} & {\normalsize(c) GPDM, IE vs. $\epsilon$} \\
\includegraphics[width=.32\textwidth]{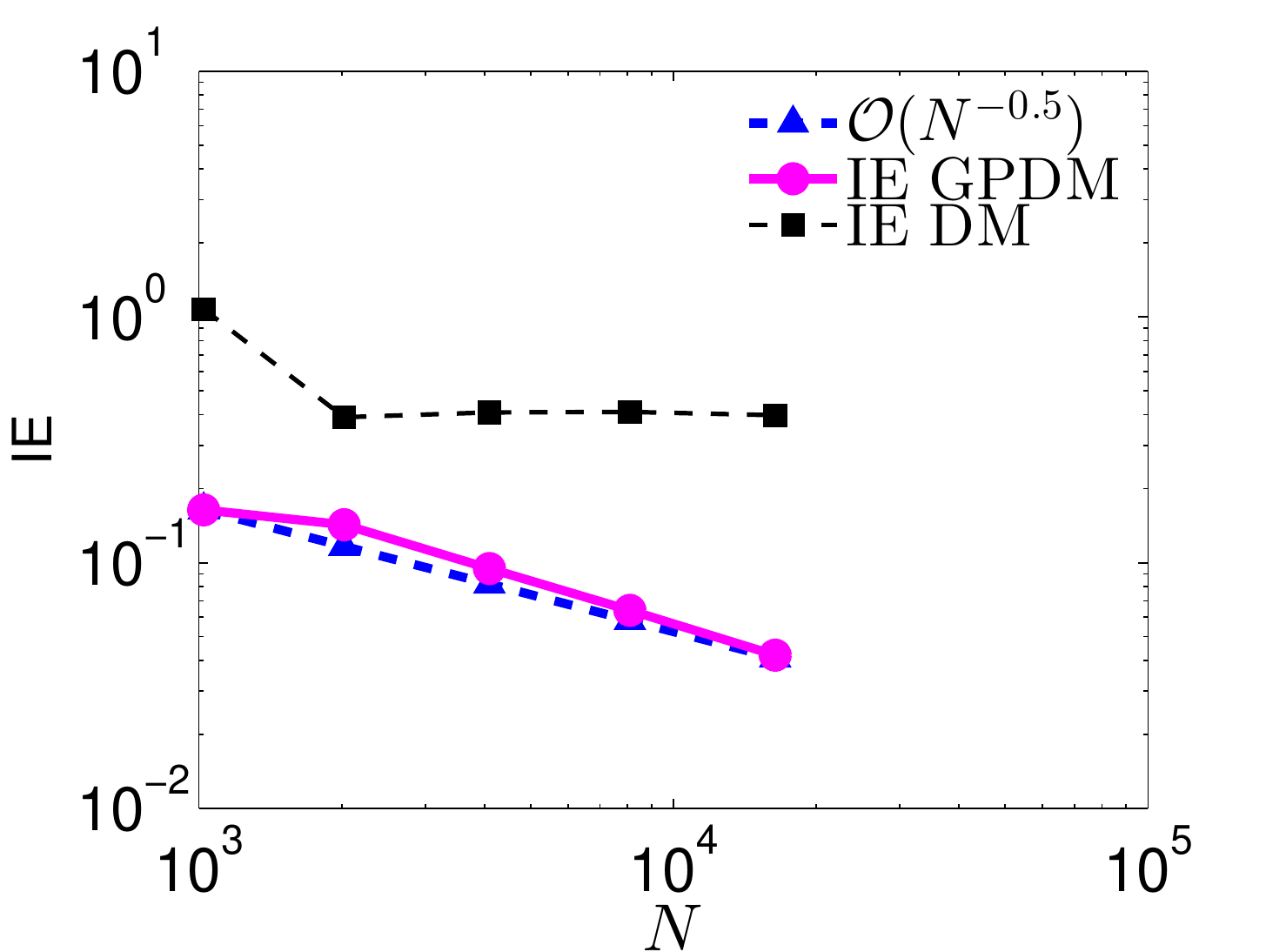} &
\includegraphics[width=.308\textwidth]{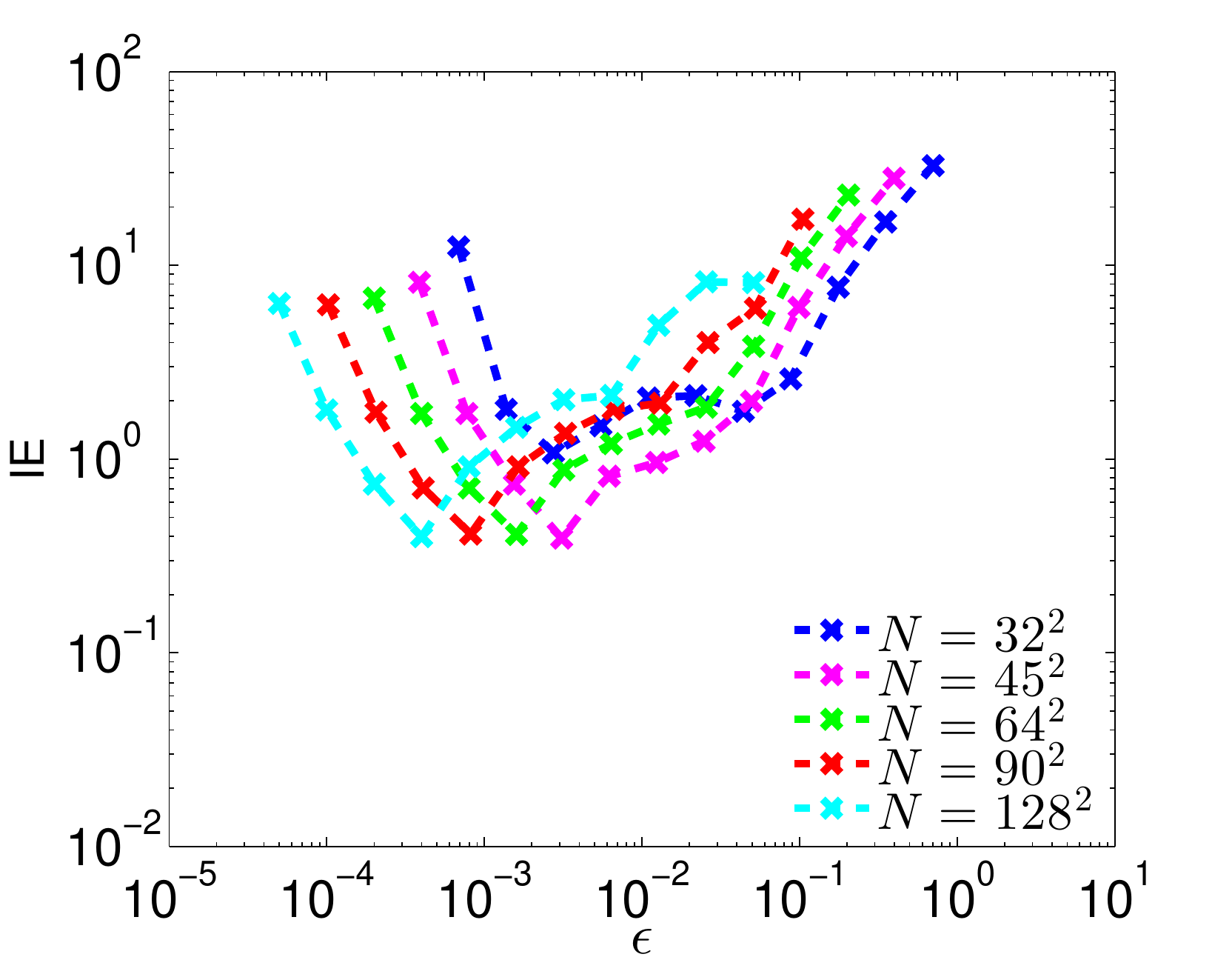} &
\includegraphics[width=.30\textwidth]{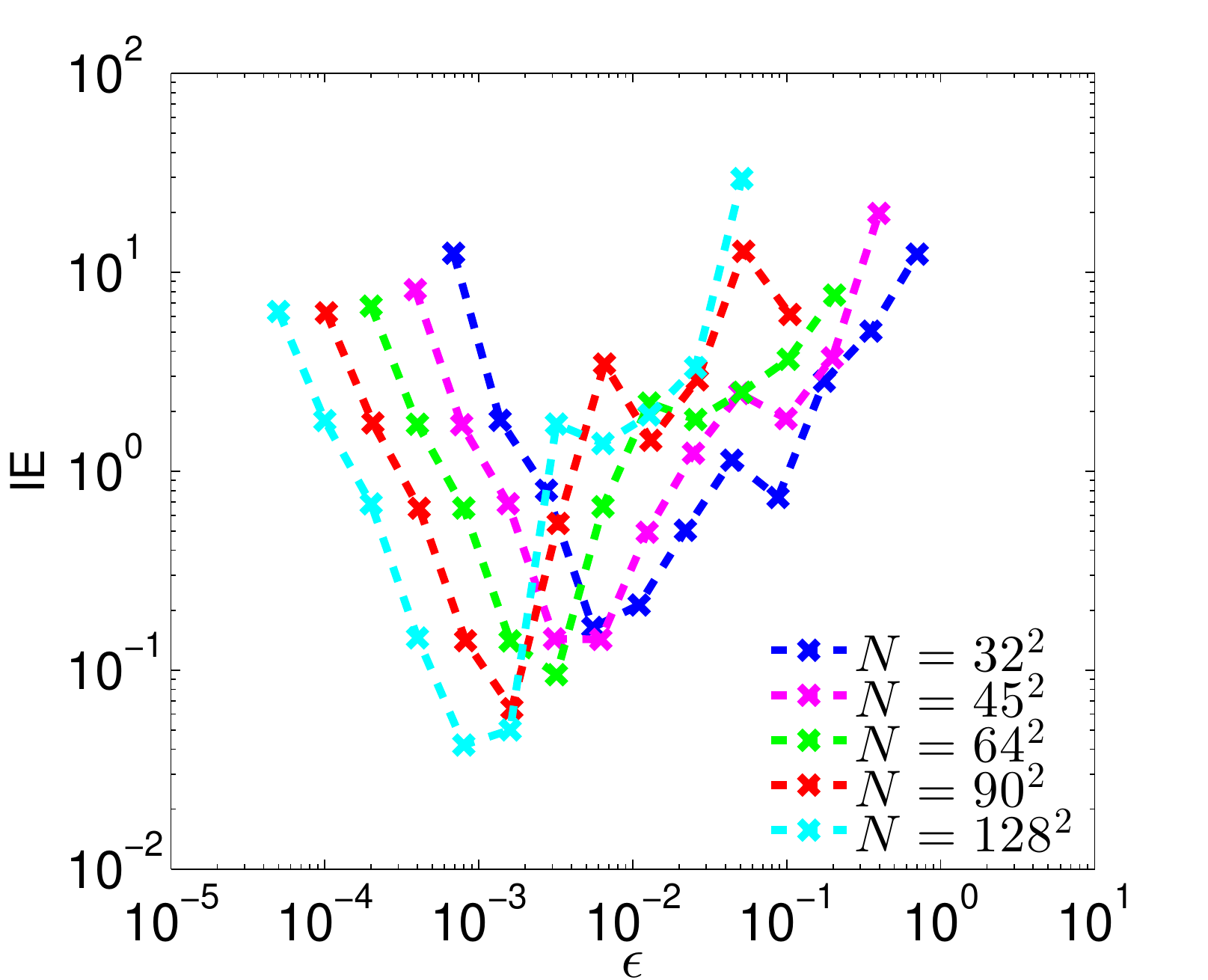}
\end{tabular}
}
\caption{(Color online) The semi-torus example with the well-sampled data in Section \ref{L3toruswell}. 
(a) IEs of DM and GPDM methods as functions of $N$. {For each $N$, the IE is obtained from the minimal inverse error for different $\epsilon$.} IEs of (b) DM and (c) GPDM methods as functions of bandwidth $\epsilon$ for different $N$. }
\label{Fig10_welltorus_IE}
\end{figure*}

In the next example, we consider solving $\mathcal{L}_{3}u=f$, with a mixed
Dirichlet-Neumann boundary conditions on a semi-torus $M\subset \mathbb{R}%
^{3}$. The parametrization of the torus is given in \eqref{Eqn:torus_g} and the
corresponding Riemannian metric is defined in \eqref{Eqn:torus_gg} with $(\theta,\phi)$ being the two intrinsic coordinates.
{\color{black}
The differential operator $\mathcal{L}_{3}$\ is defined as in \eqref{%
Eqn:L3} with%
\begin{eqnarray}
\left(
\begin{array}{c}
b^{1}(x) \\
b^{2}(x)%
\end{array}%
\right)
&:=&\left(
\begin{array}{c}
2+x_3 \\
({x_1^2+x_2^2})^{1/2}%
\end{array}%
\right)
=\left(
\begin{array}{c}
2+\sin \theta  \\
2+\cos \theta
\end{array}%
\right) ,  \notag\label{Eqn:bc_torus2} \\
\left(
\begin{array}{cc}
c^{11}(x) & c^{12}(x) \\
c^{21}(x) & c^{22}(x)%
\end{array}%
\right)
&:=&\left(
\begin{array}{cc}
3+x_1/(x_1^2+x_2^2)^{1/2}  & 1/10 \\
1/10 & 2%
\end{array}%
\right)
=\left(
\begin{array}{cc}
3+\cos \phi  & 1/10 \\
1/10 & 2%
\end{array}%
\right) .  \nonumber
\end{eqnarray}
The semi-torus is defined with the standard parameterization function as in \eqref{Eqn:torus_g} so that the induced Riemannian metric is given in \eqref{Eqn:torus_gg}. We set an analytic solution of this problem to be
\begin{equation}
u\left( x \right) =\left( \sin 2\phi -\frac{2\cos 2\phi }{2+\cos
\theta }\right) \cos \theta ,  \label{Eqn:usol2}
\end{equation}%
where
\begin{equation}
\cos \theta   =(x_1^2+x_2^2)^{1/2}-2, \ \ \
\sin \theta=x_3, \ \ \
\cos \phi = \frac{x_1}{(x_1^2+x_2^2)^{1/2}}, \ \ \
\sin \phi = \frac{x_2}{(x_1^2+x_2^2)^{1/2}},  \label{Eqn:thetox}
\end{equation}%
with $\sin 2\phi=2\sin\phi\cos\phi$ and $\cos 2\phi=2\cos^2\phi-1$.} Next, we calculate $f:=\mathcal{L}_{3}u$ and $g:=\beta _{1}\partial _{\boldsymbol{%
\nu }}u+\beta _{2}u$ at $\phi =0$ and $\phi =\pi $.  In this semi-torus example, the explicit
expression for $f$ is given by%
\[
f:=\mathcal{L}_{3}u=b\cdot \nabla u+\frac{1}{2}c^{ij}\nabla _{i}\nabla
_{j}u=b^{1}\frac{\partial u}{\partial \theta }+b^{2}\frac{\partial u}{%
\partial \phi }+\frac{1}{2}c^{11}\frac{\partial ^{2}u}{\partial \theta ^{2}}%
+c^{12}\left( \frac{\partial ^{2}u}{\partial \theta \partial \phi }-\Gamma
_{12}^{2}\frac{\partial u}{\partial \phi }\right) +\frac{1}{2}c^{22}\left(
\frac{\partial ^{2}u}{\partial \phi ^{2}}-\Gamma _{22}^{1}\frac{\partial u}{%
\partial \theta }\right) ,
\]%
where $\Gamma _{12}^{2}$ and $\Gamma _{22}^{1}$ are the only nontrivial
Christoffel symbols of the second kind%
\[
\Gamma _{12}^{2}=-\frac{\sin \theta }{2+\cos \theta },\text{ \ }\Gamma
_{22}^{1}=\sin \theta \left( 2+\cos \theta \right) ,
\]%
with the trigonometric functions defined in \eqref{Eqn:thetox}.
At one boundary $\phi =0$, the parameters are $\beta _{1}=0$ and $\beta
_{2}=1\ $(Dirichlet boundary condition) so that $g:=u(\phi =0)$, where $u$
is the analytic solution in \eqref{Eqn:usol2}. At the other boundary $\phi
=\pi $, the parameters are $\beta _{1}=1$ and $\beta _{2}=1\ $(Robin
boundary condition) so that the expression for $g$ at $\phi =\pi $ is
\[
g:=\beta _{1}\partial _{\boldsymbol{\nu }}u+\beta _{2}u=\left( \frac{1}{%
2+\cos \theta }\frac{\partial u}{\partial \phi }+u\right) \left( \phi =\pi
\right) =0,
\]%
where the analytic $u$ in \eqref{Eqn:usol2} and $\phi =\pi $ have been used.
Then, we approximate the
solution in \eqref{Eqn:usol2} for the PDE problem in \eqref{PDE}, subjected
to the manufactured $f$ and $g$.
Numerically, the grid points $\left\{ \theta _{i},\phi _{j}\right\} $ are
uniformly distributed on $\left[ 0,2\pi \right] \times \left[ 0,\pi \right] ,
$ with $i,j=1,\ldots ,64$ or $i,j=1,\ldots ,128$\ points in each direction,
resulting in a total of $N=4096$ or $N=16384$ grid points. To apply the
local kernel in \eqref{Eqn:localK}, we use $k=200$ nearest neighbors for all
$N$ and manually tune the kernel bandwidth as $\epsilon =0.0032$ for $N=4096$
and $\epsilon =8\times 10^{-4}$ for $N=16384$. {\color{black}We found that the auto-tuned
method discussed in Section~\ref{sec:dm} is not so robust for the estimation of $\mathcal{L}_3$, and we suspect that this is because the covariance in the Gaussian kernel is not constant such that the scaling used in \eqref{scalingS} may not be appropriate.

In Fig.~\ref{Fig16_semitorus_IE3}, we\ show the
absolute errors between the true and the estimated solutions obtained using
DM and GPDM for $N=64\times 64$ and $N=128\times 128$. For DM, the IE $%
\left\Vert \vec{u}^M-\hat{u}^M\right\Vert _{\infty }=0.9$ is relatively
large and IE does not decrease even as $N$ increases. On the other hand, the
inverse error (IE) of GPDM is one magnitude order smaller than the IE of DM
and decreases from $0.095$ to $0.042$ as $N$ is increased from $64\times 64$
to $128\times 128$.

{Fig.~\ref{Fig10_welltorus_IE}(a)} shows
the IEs as functions of $N$ for DM and GPDM methods. One can see that GPDM
solutions converge whereas DM solutions do not converge. {Fig.~\ref%
{Fig10_welltorus_IE}(b) and (c)} show IEs of DM and GPDM methods,
respectively, as functions of bandwidth $\epsilon $ for different $N$. One
can see that as $N$ increases, IE of GPDM decreases (at the rate of $%
\mathcal{O}(N^{-1/2})$) whereas IE of DM does not decrease.
}

\subsubsection{Anisotropic diffusion on a semi-torus with random data \label{torusrandom}}

\begin{figure}[tbp]
\flushleft
\begin{tabular}{cccc}
{\normalsize DM, $64 \times 64$ } & {\normalsize GPDM, $64 \times 64$} &
{\normalsize DM, $128 \times 128$} & {\normalsize GPDM, $128 \times 128$} \\
\includegraphics[scale=0.24]{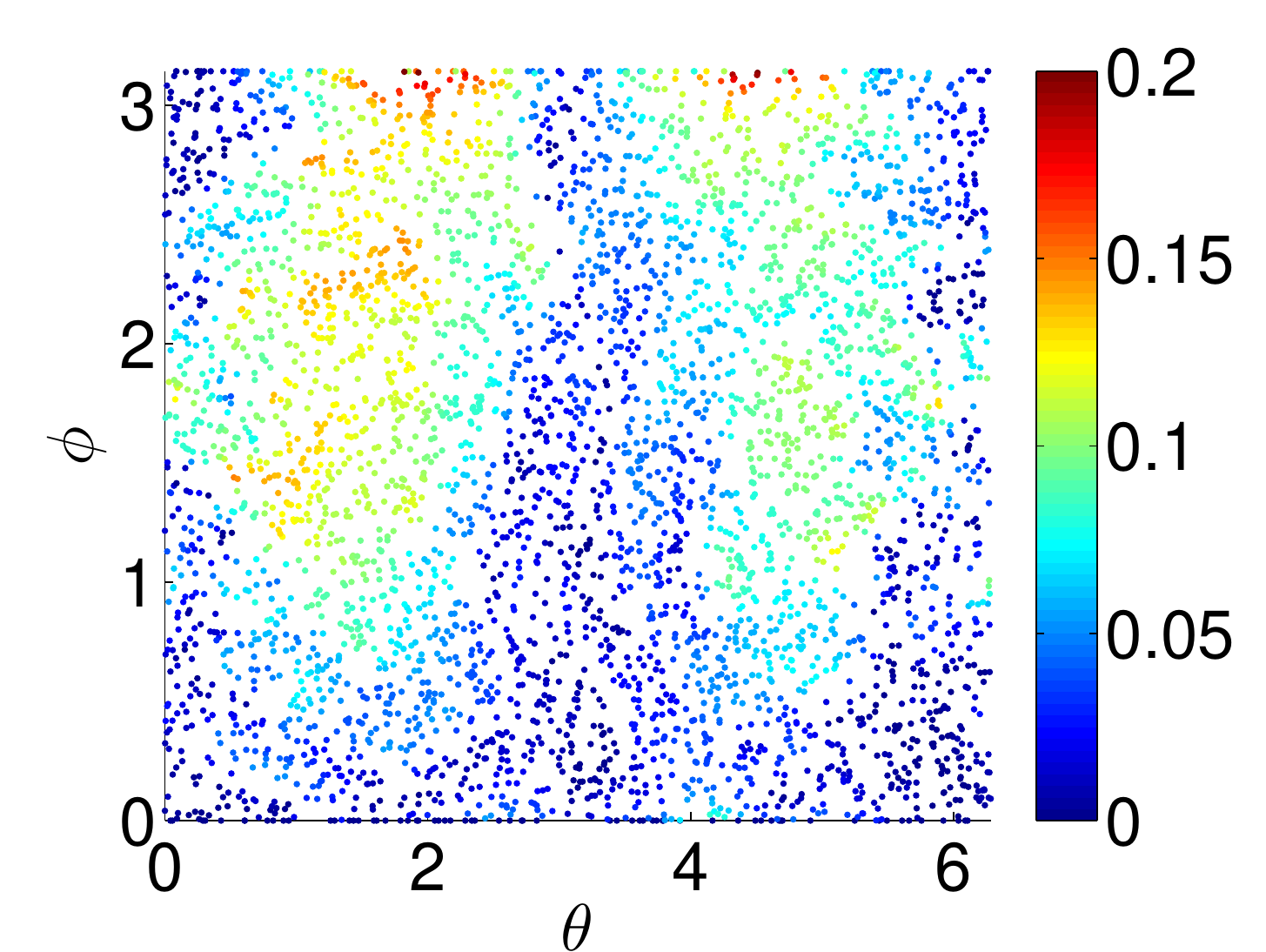}
& %
\includegraphics[scale=0.24]{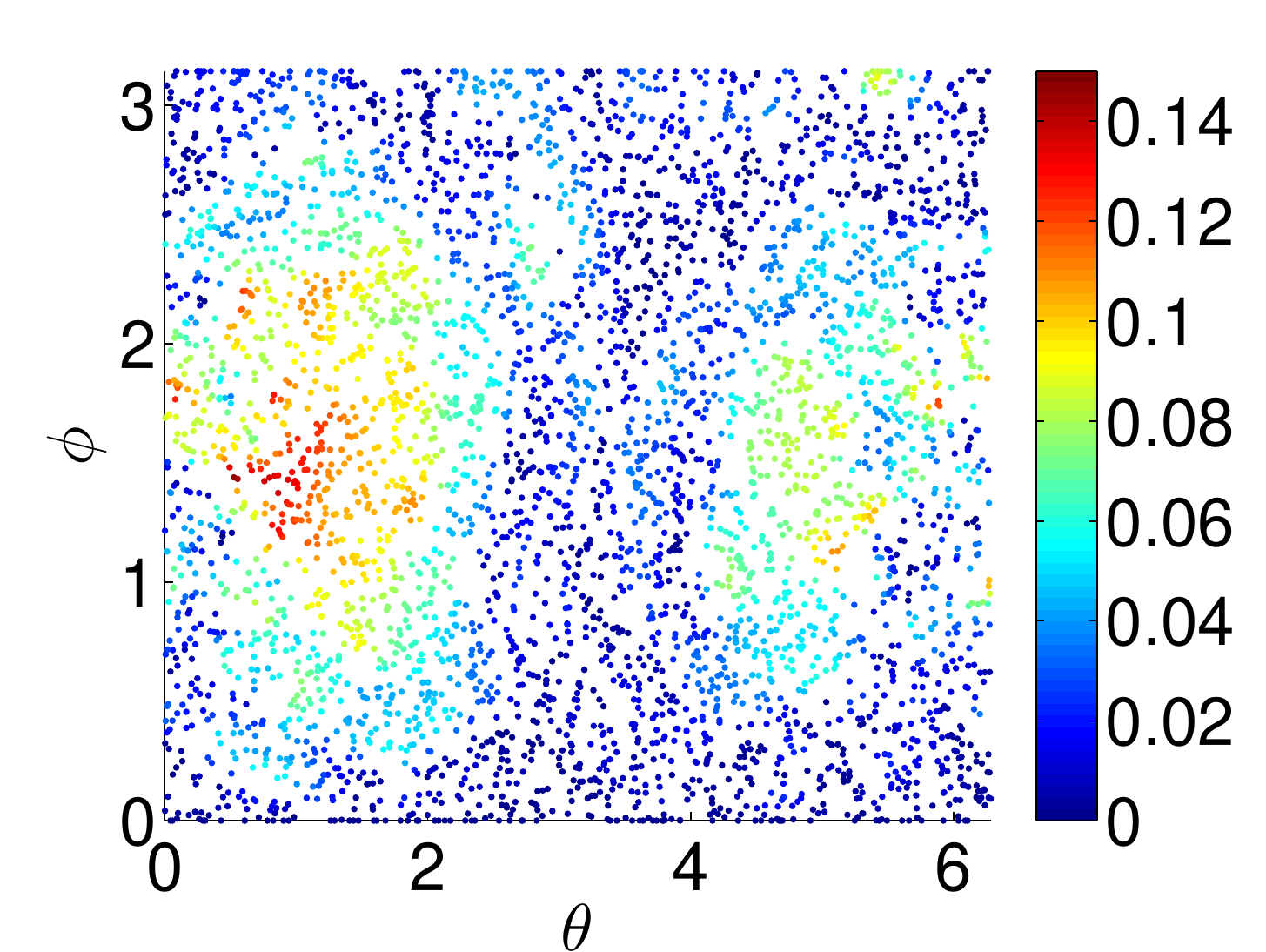}
& %
\includegraphics[scale=0.24]{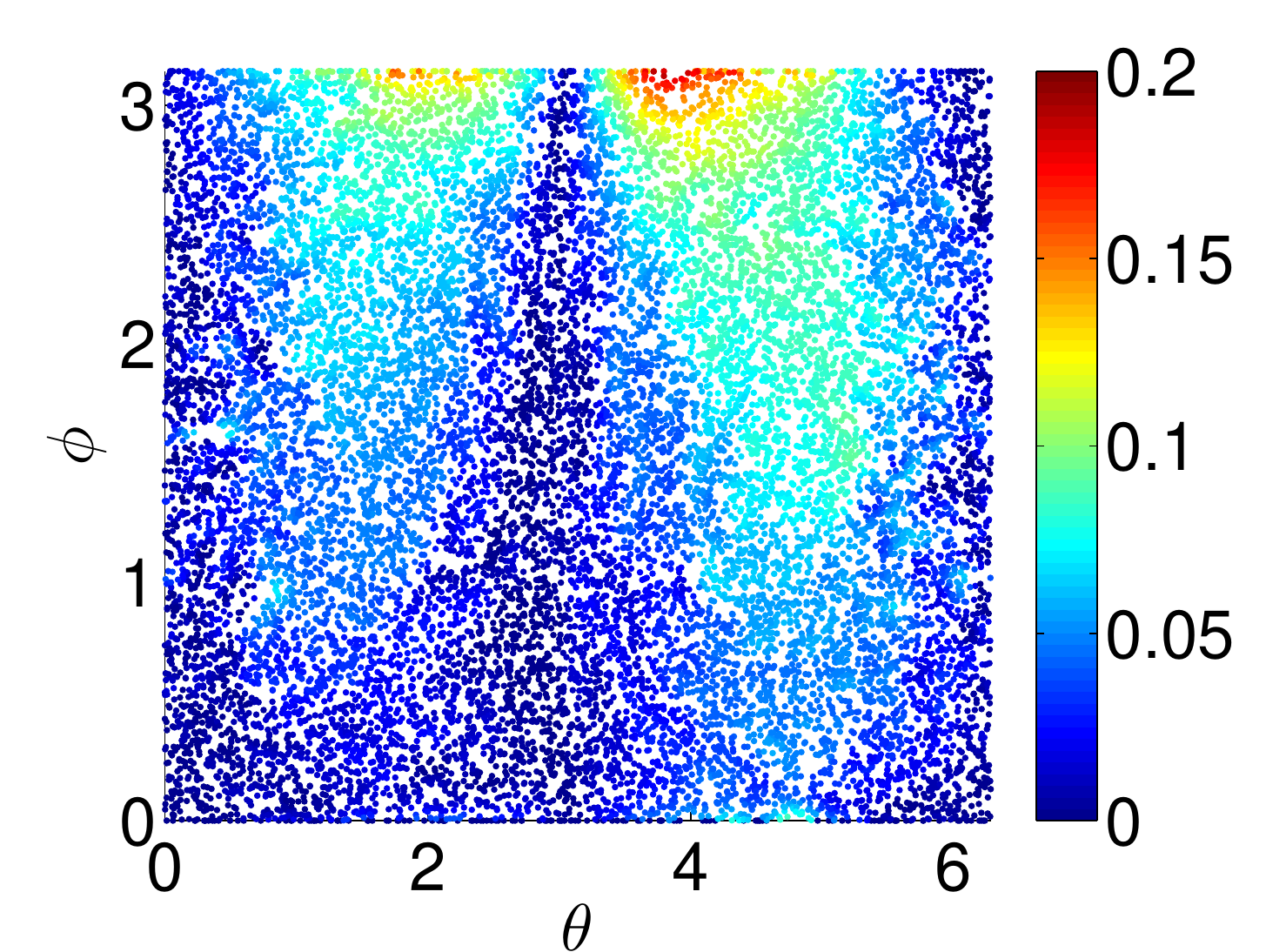}
& %
\includegraphics[scale=0.24]{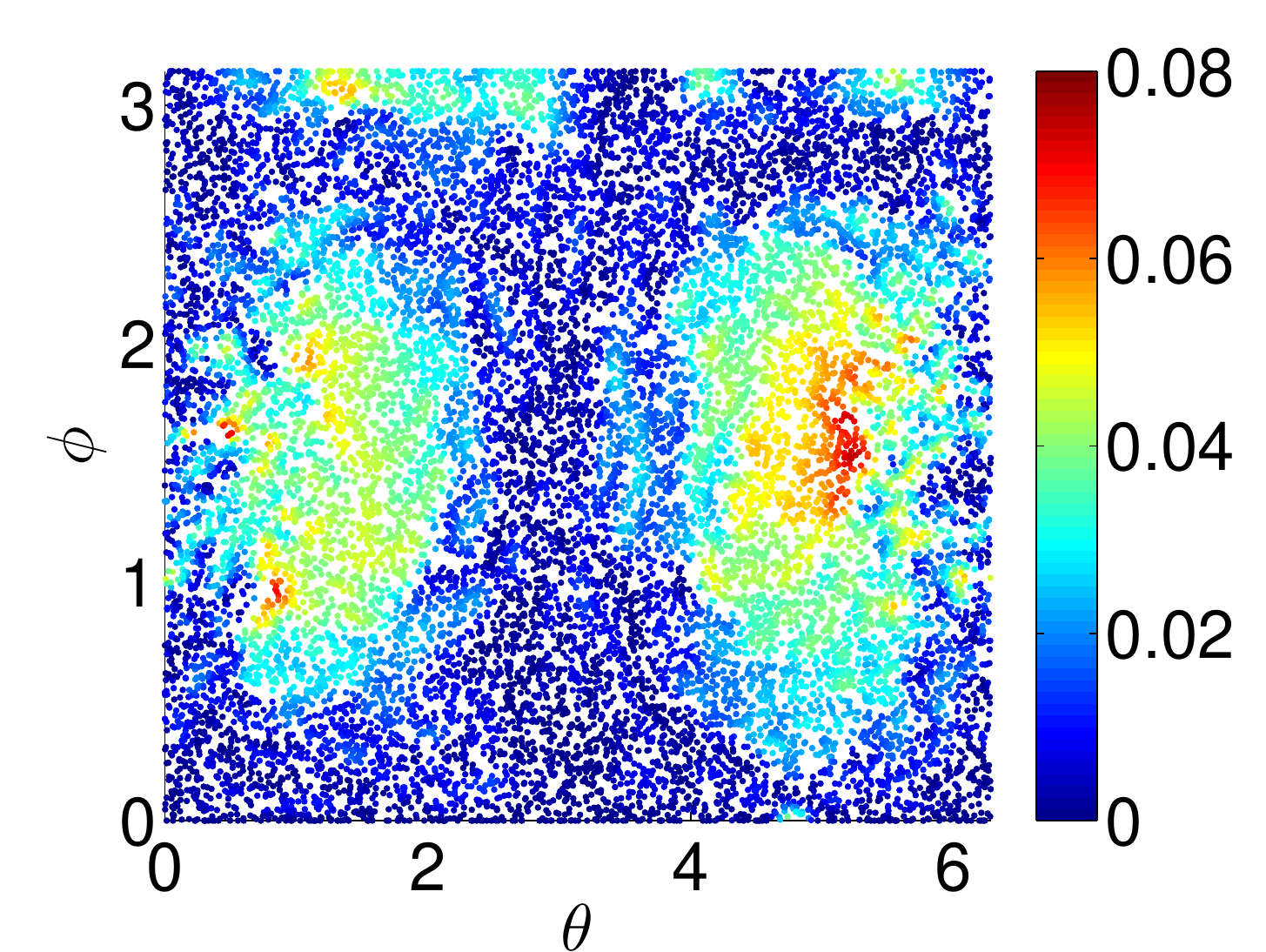}%
\end{tabular}%
\caption{(Color online) Absolute errors in the estimated solutions for the semi-torus example with random data: (a) DM for $%
N=64 \times 64$ with $\Vert \vec{u}^M-\hat{u}^M \Vert _\infty = 0.203$%
, (b) GPDM for $N=64 \times 64$ with $\Vert \vec{u}^M- \hat{u}^M
\Vert _\infty = {0.146}$, (c) DM for $N=128 \times 128$ with $\Vert\vec{u}^M-\hat{u}^M \Vert _\infty = 0.186$, (d) GPDM for $N=128 \times 128$
with $\Vert \vec{u}^M-\hat{u}^M \Vert _\infty = {0.074}$. }
\label{Fig11_semitorus_pcolor}
\end{figure}

\begin{figure*}[tbp]
{\scriptsize \centering
\begin{tabular}{cc}
{\normalsize(a) $\epsilon$ vs. $N$ } & {\normalsize(b) IE vs. $N$ }\\
\includegraphics[width=.45\textwidth]{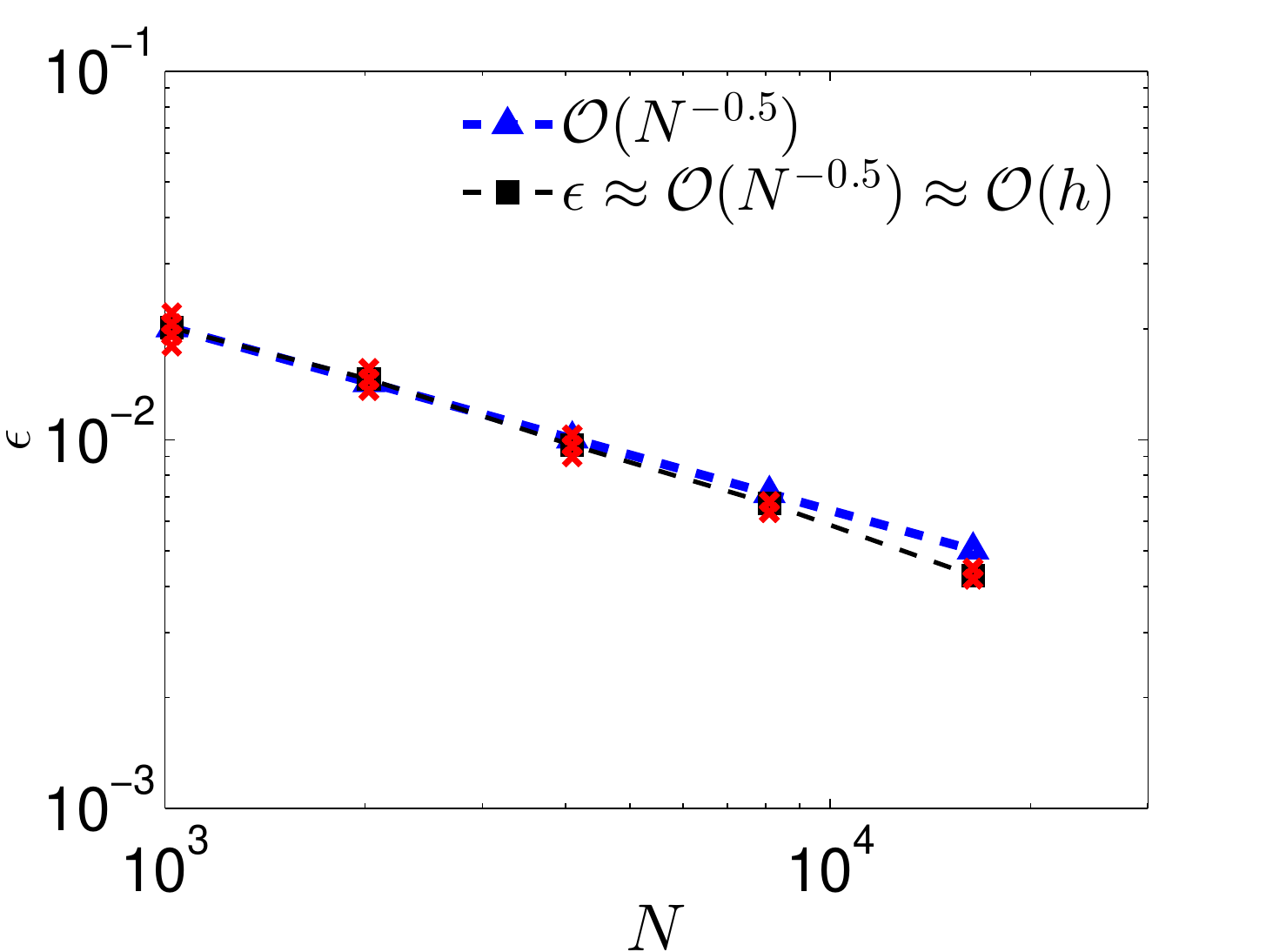} &
\includegraphics[width=.45\textwidth]{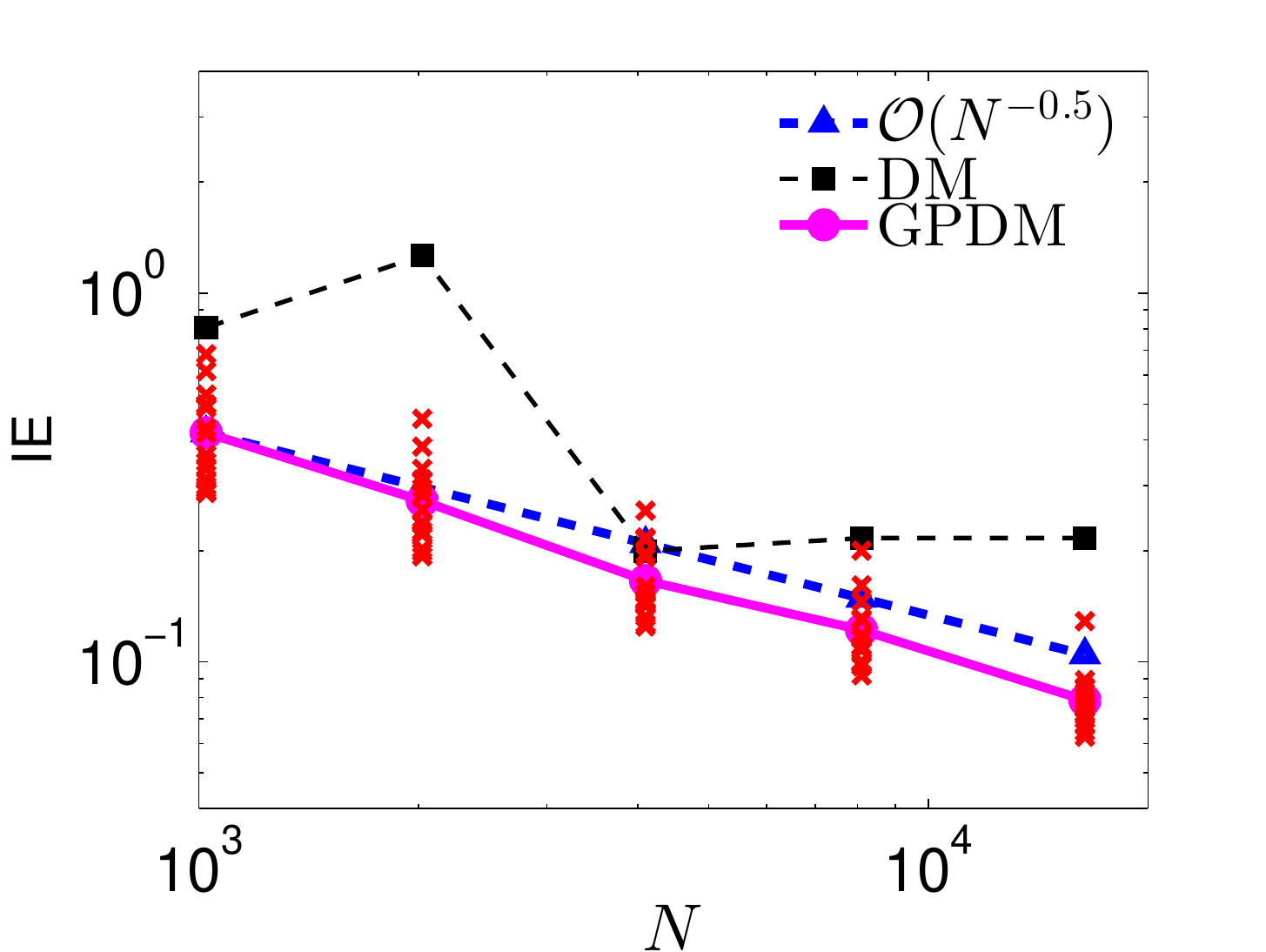}\\%
{\normalsize(c) DM, IE vs. $\epsilon$} & {\normalsize(d) GPDM, IE vs. $\epsilon$} \\
\includegraphics[width=.45\textwidth]{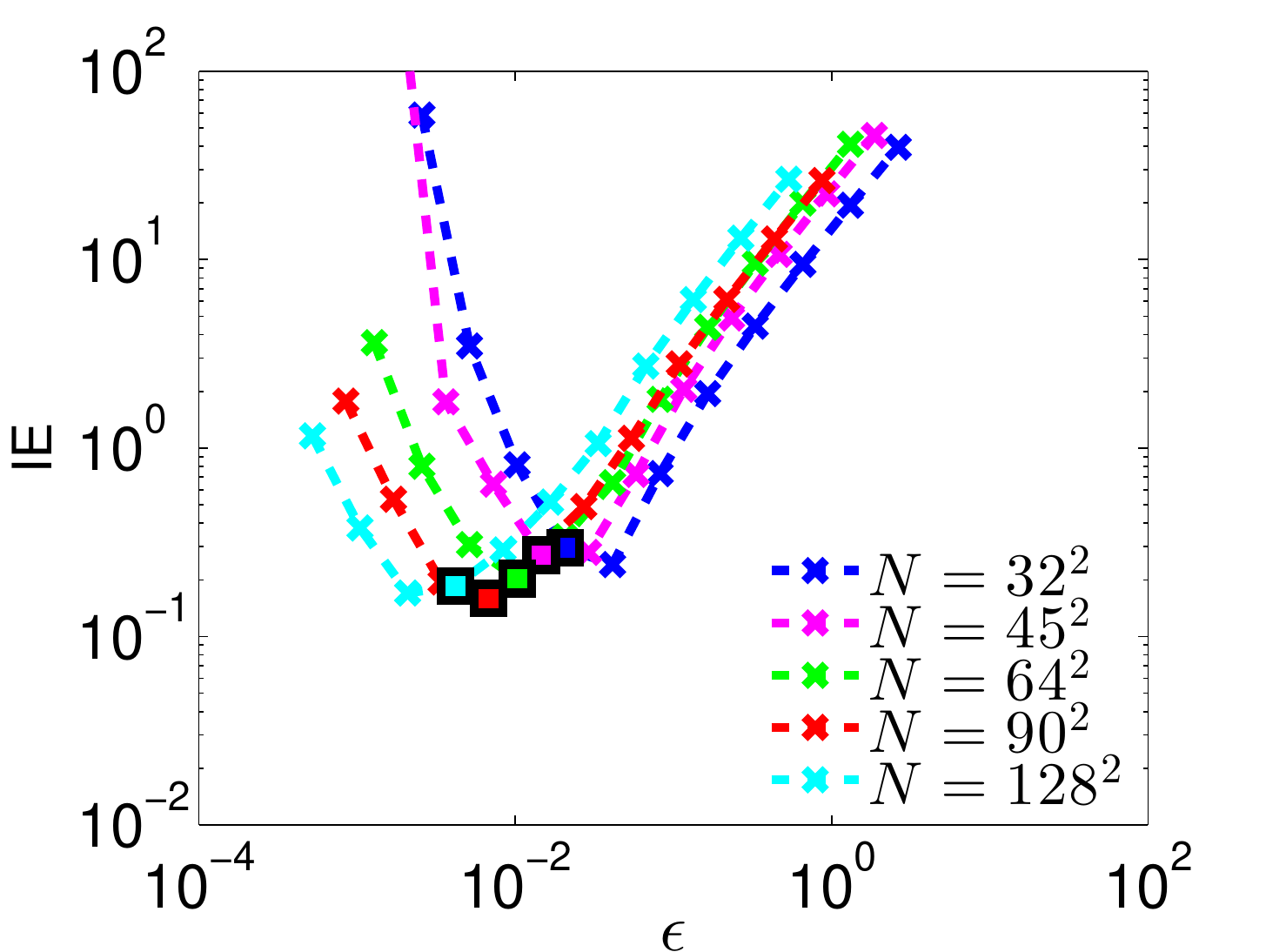} &
\includegraphics[width=.45\textwidth]{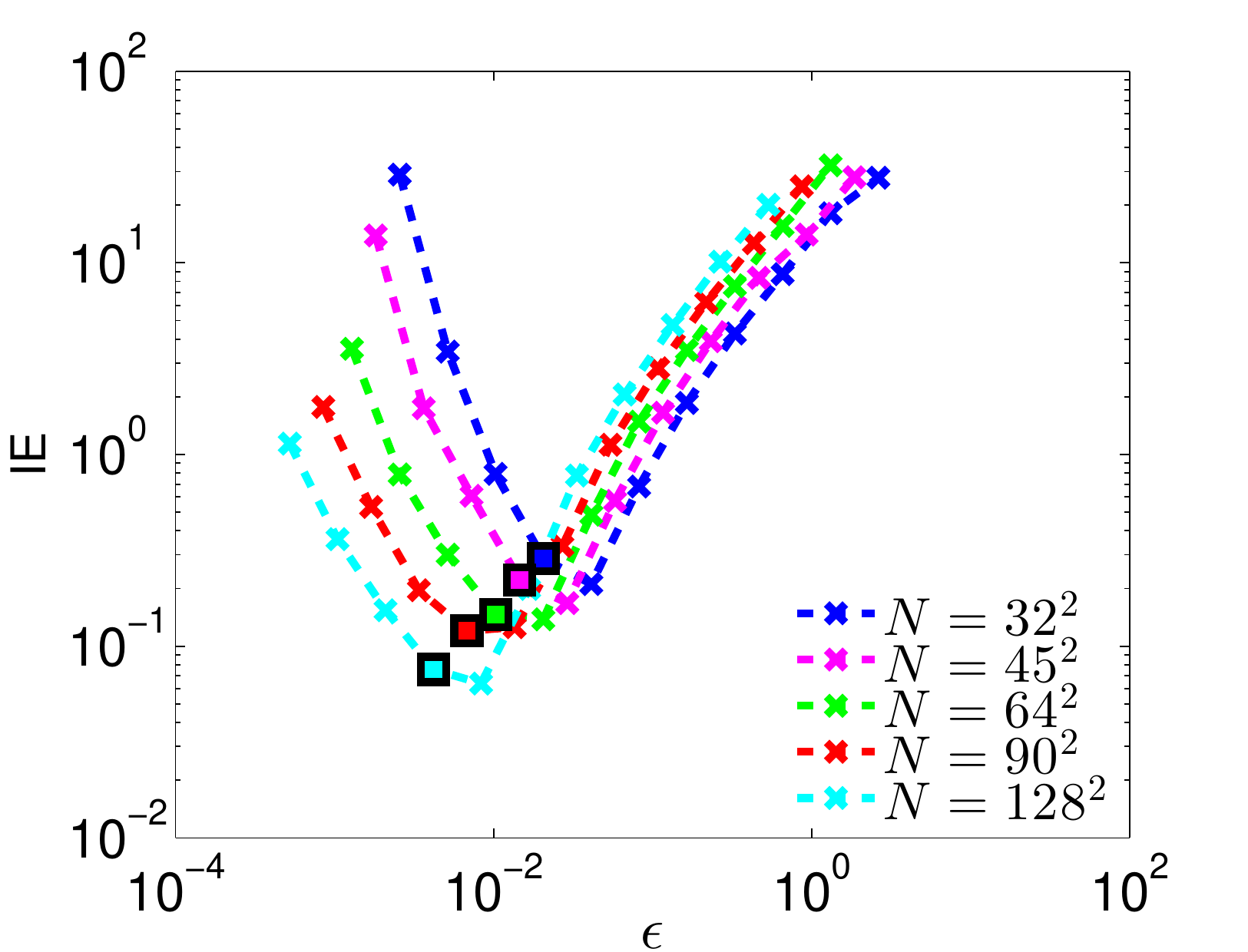}
\end{tabular}
}
\caption{(Color online) The semi-torus example with random data in Section~\ref{torusrandom}. Totally 16 independent trials are run. (a) The auto-tuned bandwidth $\epsilon$ as a function of number of points $N$. Each red cross corresponds to an auto tuned $\epsilon$ of one trial and black square corresponds to the mean of these auto-tuned $\epsilon$.
(b) IEs of DM and GPDM methods as functions of $N$. Each red cross is the IE for one trial. For one independent trial, plotted are (c) IEs of  DM and (d) IEs of  GPDM  as functions of  bandwidth $\epsilon$ for different $N$. Squares correspond to the auto-tuned $\epsilon$.}
\label{Fig12_randtorus_IE}
\end{figure*}

{\color{black}

In this example, we consider solving $\mathcal{L}_{2}u=f$, with a mixed
Dirichlet-Neumann boundary conditions on a semi-torus $M\subset \mathbb{R}%
^{3}$. The differential operator $\mathcal{L}_{2}$ is defined as in \eqref{Eqn:L3} with%
\begin{equation}
\kappa(x) =1.1+\sin ^{2}\theta \cos ^{2}\phi , \notag% \label{Eqn:ktori2}
\end{equation}%
where  the trigonometric functions for $(\theta,\phi)$ as  functions $x$ are still given in \eqref{Eqn:thetox}.
The semi-torus is still defined with the embedding function as in (\ref%
{Eqn:torus_g}). We set an analytic solution of this problem to be
\begin{equation}
u\left( x \right) =\sin \phi \sin \theta , \nonumber \label{Eqn:utori2}
\end{equation}%
and calculate
\begin{equation}
f:=\mathcal{L}_{2}u=\frac{1}{\sqrt{\left\vert g\right\vert }}\partial
_{i}\left( \kappa \sqrt{\left\vert g\right\vert }g^{ij}\partial _{j}u\right).\notag
\end{equation}%
The boundary conditions $g:=\beta _{1}\partial _{\boldsymbol{\nu }}u+\beta
_{2}u$ at $\phi =0$ and $\phi =\pi $ are given the same as those in Section %
\ref{L3toruswell}. Then, we approximate the solution for the PDE problem,
subjected to the manufactured $f$ and $g$.

\textbf{Randomly sampled data:} Numerically, the grid points $\left\{ \theta _{i},\phi
_{j}\right\} $ are randomly uniformly distributed on $\left[ 0,2\pi \right]
\times \left[ 0,\pi \right] $. For $N=32^{2},45^{2},64^{2},90^{2},128^{2}$
grid points, we set $k\sim\sqrt{N}$
and apply the $\epsilon$-auto tuning method discussed in Section~\ref{paraspec}%
. For each $N$, we show results for 16 independent trials. In Fig. \ref%
{Fig11_semitorus_pcolor}, we\ show the absolute errors in $\theta $ and $%
\phi $\ between the true and the estimated solutions obtained for DM and
GPDM methods for $N=64^{2}$ and $N=128^{2}$. For DM, the IEs are relatively
large and do not follow a clear decreasing pattern as $N$ increases ($\left\Vert \vec{u}%
^M-\hat{u}^M\right\Vert _{\infty }=0.203$ for $N=64^{2}$ and $\left\Vert
\vec{u}^M-\hat{u}^M\right\Vert _{\infty }=0.186$ for $N=128^{2}$). On the
other hand, the inverse error (IE) of GPDM is smaller than that of DM and
decreases from $0.146$ to $0.074$ as $N$ is increased from $64^{2}$ to $%
128^{2}$.

Fig.~\ref{Fig12_randtorus_IE}(a) shows the auto-tuned bandwidth $\epsilon $
as a function of $N$. Fig~\ref{Fig12_randtorus_IE}(b) shows IEs as
functions of $N$ for both DM and GPDM. One can see that GPDM solutions
converge whereas DM solutions do not converge. Fig.~\ref%
{Fig12_randtorus_IE}(c) and (d) show IEs of DM and GPDM methods,
respectively, as functions of bandwidth $\epsilon $ for different $N$ for
one independent trial. One can see that as $N$ increases, IE of GPDM
decreases whereas IE of DM does not decreases. For completeness,
we show the auto-tuned $\epsilon$ in square. Note that for the GPDM, the auto-tuned
$\epsilon$ seems to correspond to the lowest IE.

\begin{figure*}[tbp]
{\scriptsize \centering
\begin{tabular}{cc}
{\normalsize(a) $\epsilon$ vs. $N$ } & {\normalsize(b) IE vs. $N$ }\\
\includegraphics[width=.45\textwidth]{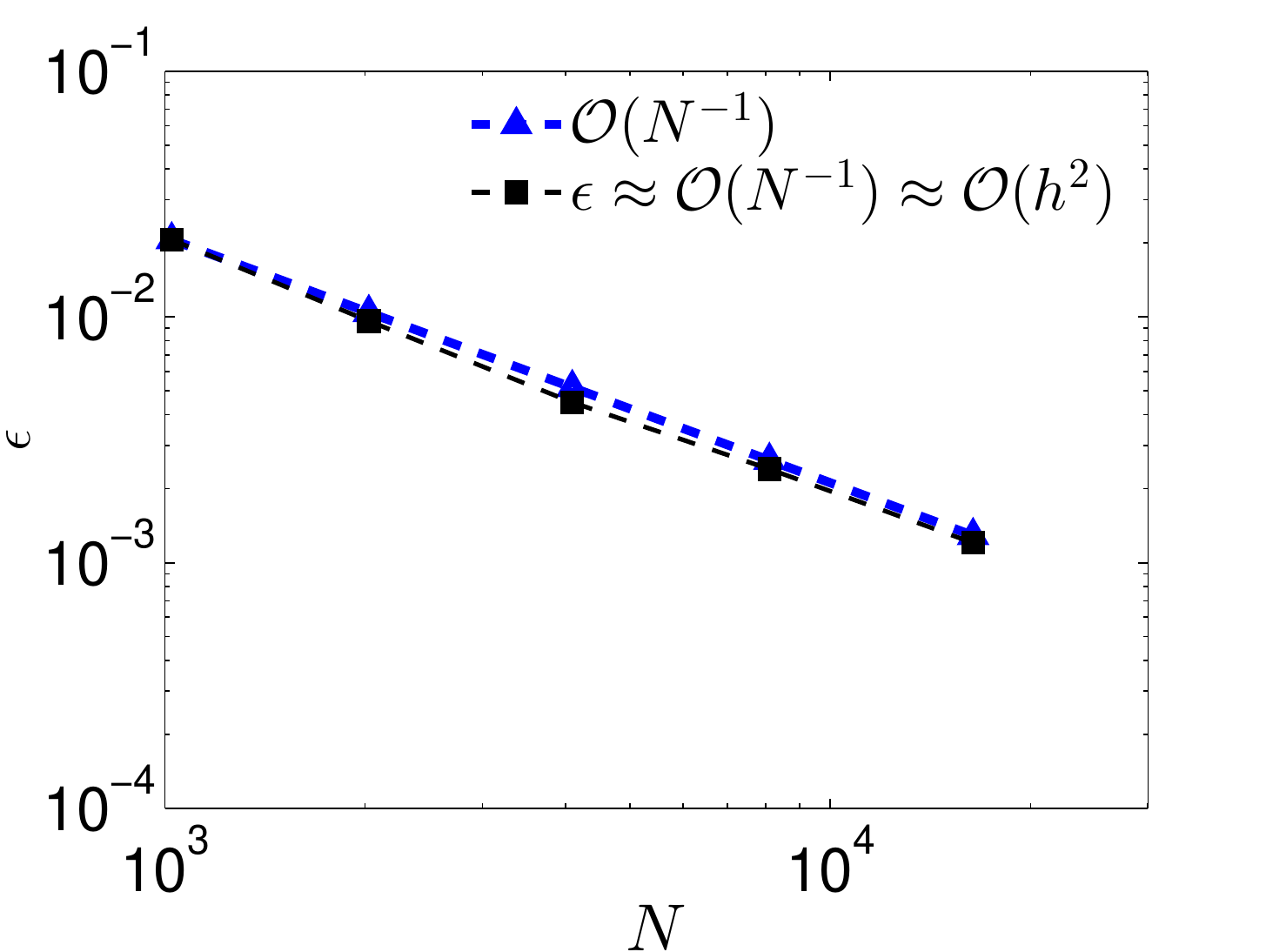} &
\includegraphics[width=.45\textwidth]{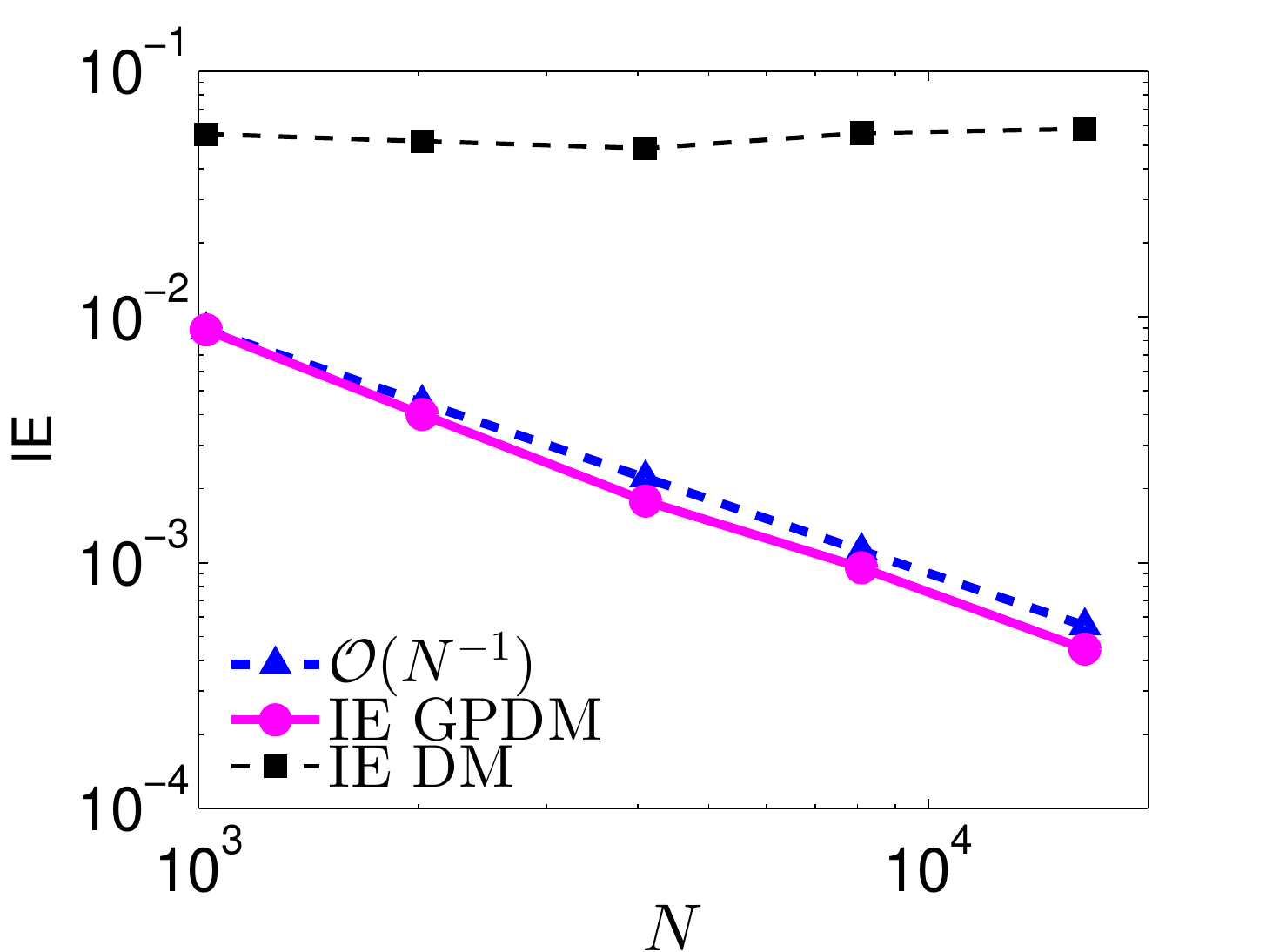}
\end{tabular}
}
\caption{(Color online) The semi-torus example with the well-sampled data in Section~\ref{torusrandom}. (a) The auto-tuned bandwidth $\epsilon$ as a function of number of points $N$.
(b) IEs of DM and GPDM methods as functions of $N$. }
\label{Fig12b_welltorus_IE}
\end{figure*}

For comparison, we also show numerical results with well-sampled data.

\textbf{Well-sampled data:} The grid points $\left\{ \theta _{i},\phi
_{j}\right\} $ are well uniformly distributed on $\left[ 0,2\pi \right]
\times \left[ 0,\pi \right] ,$ with $i,j$, both, equal to $32,45,64,90,128$\ points in
each direction, which is the same as those in Section~\ref{L3toruswell}. For
different $N$ grid points, we fixed $k=121$ nearest neighbors and then apply the $\epsilon$-auto
tuning method discussed in Section~\ref{paraspec}%
. One can see from Fig. \ref{Fig12b_welltorus_IE}(a) that the auto tuned
bandwidth $\epsilon $ is on order of $N^{-1}$. This rate for the well-sampled
data is faster than that for random data as shown in Fig. \ref%
{Fig12_randtorus_IE}(a). Fig.~\ref{Fig12b_welltorus_IE}(b) shows that IE
of GPDM decays on the order of $N^{-1}$ whereas IE of DM does not decay for
different $N$.

\textbf{Other choices of }$k$\textbf{\ nearest neighbors and auto-tuned }$%
\epsilon $:\ For well-sampled data, we also examined the auto-tuned $\epsilon $\
under variable $k$ nearest neighbors, that is, we choose {$k\sim \sqrt{N}$}.
We found that for well-sampled data, the rates preserve as in Fig. \ref%
{Fig12b_welltorus_IE}, that is,\  the bandwidth $\epsilon =\mathcal{O}(N^{-1})$ and IE
is $\mathcal{O}(N^{-1})$\ as well [not shown here].

For random data, we also examined the auto-tuned $\epsilon $\ under fixed $k=200$
nearest neighbors. However, we found that the results are different between
using fixed $k$ and variable $k$. For variable $k$, the bandwidth $%
\epsilon =\mathcal{O}(N^{-1/2})$ and IE is $\mathcal{O}(N^{-1/2})$ as shown in Fig. \ref%
{Fig12_randtorus_IE}. For fixed $k$, the bandwidth $\epsilon =\mathcal{O}(N^{-1})$ and
IE is $\mathcal{O}(N^{-1/4})$ [not shown here].

}

\subsection{\label{sec:face}Anisotropic diffusion on an unknown ``face" manifold}

In this section, we consider solving the boundary value problem in \eqref{PDEproblem2}
 with $\kappa =1.1+\sin ^{2}(10x_{1})$ and $f=\cos (10x_{2})$\ on an
unknown manifold example of a two-dimensional ``face" $x=(x_{1},x_{2},x_{3})\in
M\subset \mathbb{R}^{3}$. We consider the Robin boundary
condition on the one dimensional-closed boundary curve of the face. The
surface used in this section is from Keenan Crane's 3D repository \cite{crane}. Notice that we have no access to the analytic solution since we do not know the embedding of the face surface. For comparison, we numerically solve the
problem with finite element method (FEM) using the FELICITY FEM Matlab toolbox \cite{walker2018felicity}.

\begin{figure*}[tbp]
{\scriptsize \centering
\begin{tabular}{ccc}
{\normalsize(a) FEM Solution} & {\normalsize(b) Difference between FEM \& DM} & {\normalsize(c) Difference between
FEM \& GPDM }\\
\includegraphics[width=1.99
in, height=1.4 in]{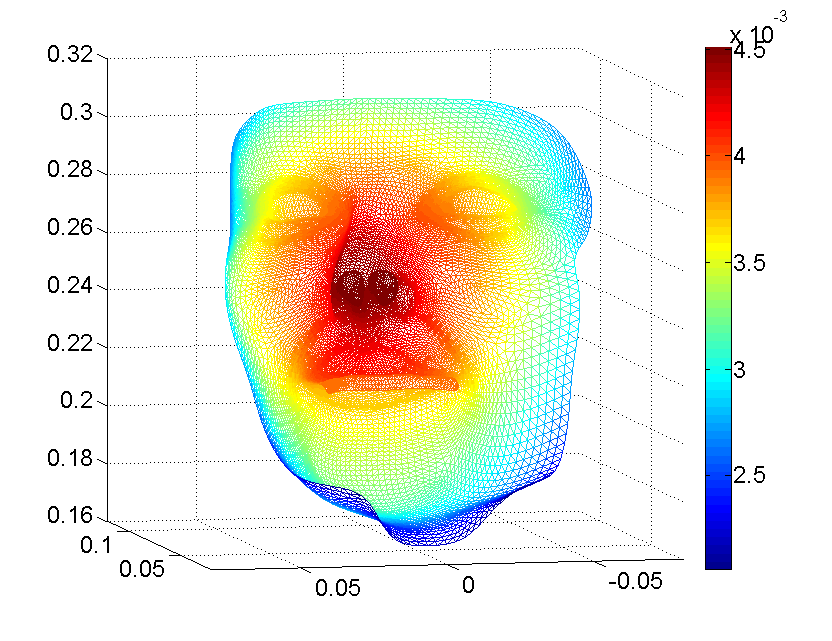} &
\includegraphics[width=1.99
in, height=1.4 in]{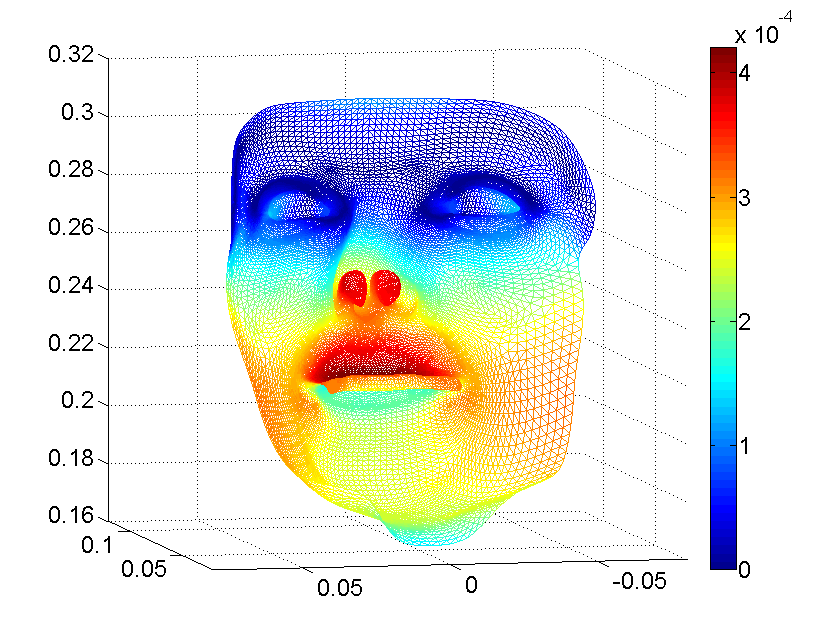} &
\includegraphics[width=1.99
in, height=1.4 in]{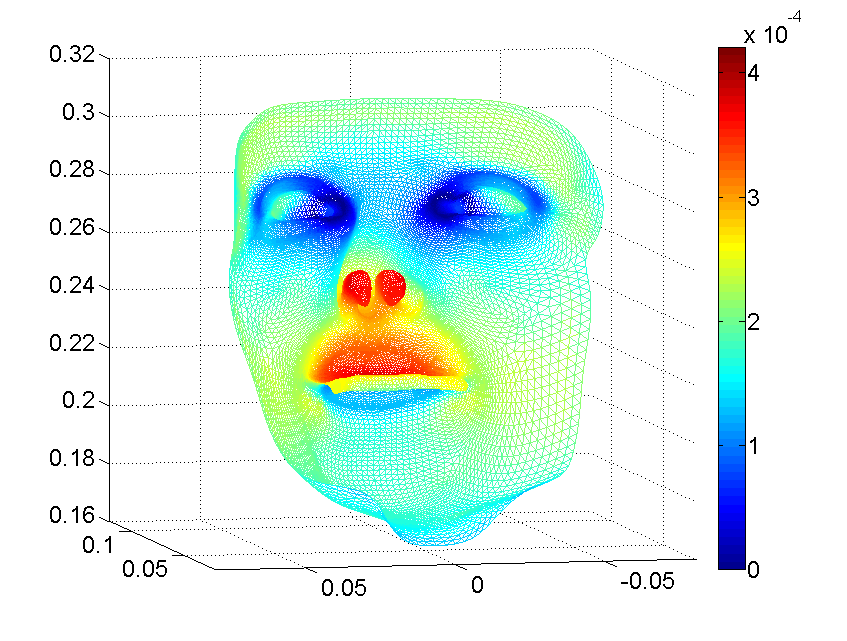}%
\end{tabular}
}
\caption{(Color online) Comparison of the PDE solutions among FEM, DM, and GPDM
on the ``face'' example with the Robin boundary condition. (a) FEM solution. (b)
{Absolute difference} between FEM and DM solutions. (c) {Absolute difference} between FEM and GPDM
solutions. }
\label{FigX3_face}
\end{figure*}

Fig.~\ref{FigX3_face} shows the comparison of the
solutions among FEM, DM, and GPDM methods corresponding the Robin boundary condition ($\partial _{\boldsymbol{\nu }%
}u+10u=0$ on $\partial M$). To compute the FEM solution as a
benchmark, we applied FELICITY toolbox in Matlab using the triangulated mesh of the
surface, which consisted of $17157$ points and a connectivity matrix for the
triangle elements. We use a linear finite element space in the FEM
algorithm. We used $k=512$ nearest neighbors and tuned the kernel bandwidth parameter as $\epsilon
=3\times 10^{-6}$. For GPDM, we used $K=6$ layers of ghost points for $168$
boundary points so that we used $168\times 6$ ghost points in total. In Fig. %
\ref{FigX3_face}, we found that the inverse error (IE) between GPDM and FEM
solutions (about $3.2\times 10^{-4}$) is smaller than that between DM and FEM
solutions (about $4.3\times 10^{-4}$); here the scaling of the true solution is on the order of $10^{-3}$. In this case, one can see that larger errors of GPDM are locally concentrated near the nose and the mouth whereas the larger errors for DM are evenly
distributed on the lower face. Thus, for the Robin boundary condition, one can
see that GPDM exhibits a better performance than the standard DM.

\section{Applications: Solving elliptic eigenvalue problems}\label{sec:EIGS}

In this section, we apply the GPDM algorithm for
solving the eigenvalue problem $\mathcal{L}\psi _{k}=\lambda _{k}\psi _{k}$\ on
manifold with boundary, where $\mathcal{L}$ is either the Laplace-Beltrami in \eqref{L1} or the weighted
Laplacian operator in \eqref{L2}.  Since there are no $f$ in this problem, we cannot use the quadratic extrapolation formula in \eqref{Eq:Uvvv}. Instead, we extrapolate $u$ using the linear extrapolation formula defined as follow:
\BEA
\begin{aligned}\label{Eqn:uvv_g2}
\tilde{u}_{\epsilon,j}^{G_{1}} - 2u(x_{j}^{B}) + u(\tilde{x}_{j}^{G_0}) &= 0,\\
\tilde{u}_{\epsilon ,j}^{G_{2}}-2\tilde{u}_{\epsilon ,j}^{G_{1}}+ u(x_j^{B})&=0, \\
\tilde{u}_{\epsilon ,j}^{G_{k}}-2\tilde{u}_{\epsilon ,j}^{G_{k-1}}+\tilde{u}_{\epsilon ,j}^{G_{k-2}}&=0,  \quad\quad k = 3,\ldots K, \\
\end{aligned}
\EEA
where $\tilde{u}_{\epsilon ,j}^{G_{k}}$ are the function values to be specified. It is worth noting that if we replace the quadratic extrapolation in \eqref{Eq:Uvvv} with \eqref{Eqn:uvv_g2}, one can deduce the error rate analogous to Proposition~\ref{prop:extrapolationofu}, except that the first error bound $h^3$ is replaced with $h^2$. With this linear extrapolation formula, we consider the following algorithm.

\begin{algm}\label{algm5_1}
GPDM algorithm for eigenvalue problems:
\begin{enumerate}

\item Supplement the ghost points as in Section~\ref{ghost_ptm} and
construct the augmented $\bar{N}\times \bar{N}$ matrix using DM based on all
points on manifold and ghost points.

\item Construct the GPDM estimator, an $(N-J)\times (N-J)$\ matrix, based on
the homogeneous extrapolation formula \eqref{Eqn:uvv_g2} for u at the ghost
points and the homogeneous boundary condition \eqref{BCdiscrete}. Here, $J$
is the number of boundary points. The homogeneous equations \eqref{Eqn:uvv_g2} and %
\eqref{BCdiscrete} have a unique solution that can be written in a compact
form as a the column vector,
\begin{equation*}
\left(\hat{u}^{B},\vec{u}_{\epsilon }^{G}\right)=%
\mathbf{C}\hat{u}^{I},
\end{equation*}%
where $\vec{u}_{\epsilon }^{G}$, $\hat{u}^{B}$, and $\hat{u}^I$  are column vectors with components consist of the estimated function values of $u$ at the estimated ghost points, boundary points, and interior points, {\color{black}respectively, as defined in \eqref{vectoru}-\eqref{truevectoru}. Here, $\mathbf{C}$
is a $(JK+J)\times (N-J)$ matrix. Denoting the column vector $\hat{u}:=(\hat{u}^I,\hat{u}^B,\vec{u}_{\epsilon }^{G})\in \mathbb{R}^{\bar{N}}$, the diffusion operator $\mathcal{L}$
is approximated with the following matrix,
\begin{equation*}
\mathbf{L}^h \hat{u}:=\mathbf{L}^{(1)}\hat{u}^{I}+%
\mathbf{L}^{(2)}\left( \hat{u}^{B},\vec{u}_{\epsilon }^{G}\right)=\mathbf{L}^{(1)}\hat{u}^{I}+\mathbf{L}^{(2)}%
\mathbf{C}\hat{u}^{I} = \big(\mathbf{L}^{(1)}+\mathbf{L}^{(2)}%
\mathbf{C}\big)\hat{u}^{I}.
\end{equation*}}
Here, we have defined the submatrices $\mathbf{L}^{(1)}\in\mathbb{R}^{(N-J)\times(N-J)}$ and $\mathbf{L}^{(2)}\in\mathbb{R}^{(N-J)\times(JK+J)}$ of the augmented ${(N-J)}\times\bar{N}$ matrix $\mathbf{L}^h\equiv (\mathbf{L}^{(1)},\mathbf{L}^{(2)})$, and we should point out that these submatrices are different than those defined in \eqref{GPDM}.

\item Solve the eigenvalue problem of the diffusion matrix $\mathbf{L}%
^{(1)}+\mathbf{L}^{(2)}\mathbf{C}$.
\end{enumerate}
\end{algm}

{For comparison, we also apply the standard DM algorithm for solving the eigenvalue problem $\mathcal{L}\psi _{k}=\lambda _{k}\psi _{k}$\ with the following modification to incorporate boundary conditions other than homogeneous Neumann.}

{
\begin{algm}
DM algorithm for eigenvalue problems with non-Neumann boundary conditions:

\begin{enumerate}
\item Construct the DM estimator, an $(N-J)\times (N-J)$\ matrix, based on
the homogeneous boundary condition \eqref{BCdiscrete}. Here, $J$ is the
number of boundary points. The homogeneous boundary condition %
\eqref{BCdiscrete} has a unique solution that can be written in a compact
form as,
\begin{equation*}
\hat{u}^{B}=\mathbf{C}_{\rm DM}\hat{u}^{I},
\end{equation*}%
where $\hat{u}^{B}$ and $\hat{u}^{I}$ are vectors with components consist of
the estimated function values of $u$ evaluated at the boundary points and interior points,
respectively. For boundary conditions that involve normal derivatives, we used the Algorithm in Appendix~\ref{App:A} to approximate the
normal derivatives without adding ghost points. Here, $\mathbf{C}_{\rm DM}$ is a $(J)\times (N-J)$ matrix. For the formula below, we define
column vector $\hat{u}=(\hat{u}^I,\hat{u}^B)$. Then, the diffusion operator $\mathcal{L}$ can be approximated with the following
matrix,
\begin{equation*}
\mathbf{L}_{\rm DM}\hat{u}=\mathbf{L}_{\rm DM}^{(1)}\hat{u}^{I}+\mathbf{L}_{\rm DM}%
^{(2)}\hat{u}^{B}=\mathbf{L}_{\rm DM}^{(1)}\hat{u}^{I}+\mathbf{L}_{\rm DM}^{(2)}\mathbf{C}_{\rm DM}%
\hat{u}^{I}\equiv \big(\mathbf{L}_{\rm DM}^{(1)}+\mathbf{L}_{\rm DM}^{(2)}\mathbf{C}_{\rm DM}\big)%
\hat{u}^{I}.
\end{equation*}%
Here, we have defined the submatrices $\mathbf{L}_{\rm DM}^{(1)}\in \mathbb{R}%
^{(N-J)\times (N-J)}$ and $\mathbf{L}_{\rm DM}^{(2)}\in \mathbb{R}^{(N-J)\times (J)}$
of the ${(N-J)}\times N$ DM matrix $\mathbf{L}_{\rm DM}\equiv (\mathbf{L}_{\rm DM}^{(1)},%
\mathbf{L}_{\rm DM}^{(2)})$.

\item Solve the eigenvalue problem of the diffusion matrix $\mathbf{L}_{\rm DM}^{(1)}+%
\mathbf{L}_{\rm DM}^{(2)}\mathbf{C}_{\rm DM}$.
\end{enumerate}
\end{algm}
}

Next, we compare the numerical performance of the DM and GPDM
in solving the eigenvalue problems $\mathcal{L}\psi _{k}=\lambda _{k}\psi
_{k}$\ on manifolds with boundary for various test examples. We begin with the singular Sturm-Liouville eigenvalue problem of Legendre polynomials on a flat domain $[-1,1]$. Next, we show numerical results of the Laplace-Beltrami operator on various embedded smooth manifolds, such as a 1D semi-circle in $\mathbb{R}^{2}$\ with Dirichlet and Robin boundary conditions, and a 2D semi-torus in $\mathbb{R}^{2}$ with mixed boundary conditions.

\subsection{A singular Sturm-Liouville problem}

First, we consider solving the Legendre differential equation on the flat
domain $[-1,1],$
\begin{equation}
\mathcal{L}\psi _{k}:=\frac{d}{dx}\left[ \left( 1-x^{2}\right) \frac{d\psi
_{k}}{dx}\right] =-k\left( k+1\right) \psi _{k},  \label{Eqn:Legend}
\end{equation}%
where the eigenvalues are $\lambda _{k}=-k\left( k+1\right) $ with $%
k=0,1,2,\ldots $, and the eigenfunctions $\psi _{k}$ are Legendre
polynomials. The Legendre polynomials are orthogonal with respect to a
uniformly distributed weight over the domain $[-1,1]$. The completeness of
the set of eigenfunctions follows from the framework of Sturm-Liouville
theory. It is well-known that the differential equation \eqref{Eqn:Legend}
has singular points at the boundary $x=\pm 1$, so that the eigenfunctions $%
\psi _{k}$ are required to be regular at $x=\pm 1$.

Numerically, the operator $\mathcal{L}$\ in \eqref{Eqn:Legend} is estimated
by choosing $\kappa =$ $1-x^{2}$ in the weighted Laplacian operator $%
\mathcal{L}_{2}$ in \eqref{L2} using the GPDM method. At the boundaries $x=\pm
1 $, $\mathcal{L}$ reduces to a first-order differential operator $\mathcal{L%
}\psi _{k}=-2x\frac{d\psi _{k}}{dx}$, so that it can be treated as a boundary
condition which is estimated using a finite-difference method. In particular,
we construct an $N\times N$ diffusion matrix on $N$ equally spaced discrete
grids $\left\{ x_{i}=2(i-1)/(N-1)-1\right\} _{i=1,\ldots ,N}$ on $[-1,1]$. For
efficient computation, the sparse diffusion matrix is represented using the
kernel generated from $k=50$ nearest neighbors based on the Euclidean
distance of $x_{i}$ \cite{harlim2018}. The bandwidth $\epsilon =1.5\times
10^{-5}$ is chosen for $N=400$ by the auto-tuning algorithm discussed in Section~\ref{sec:dm}.

\begin{figure*}[tbp]
\centering
\includegraphics[width=.9\textwidth]{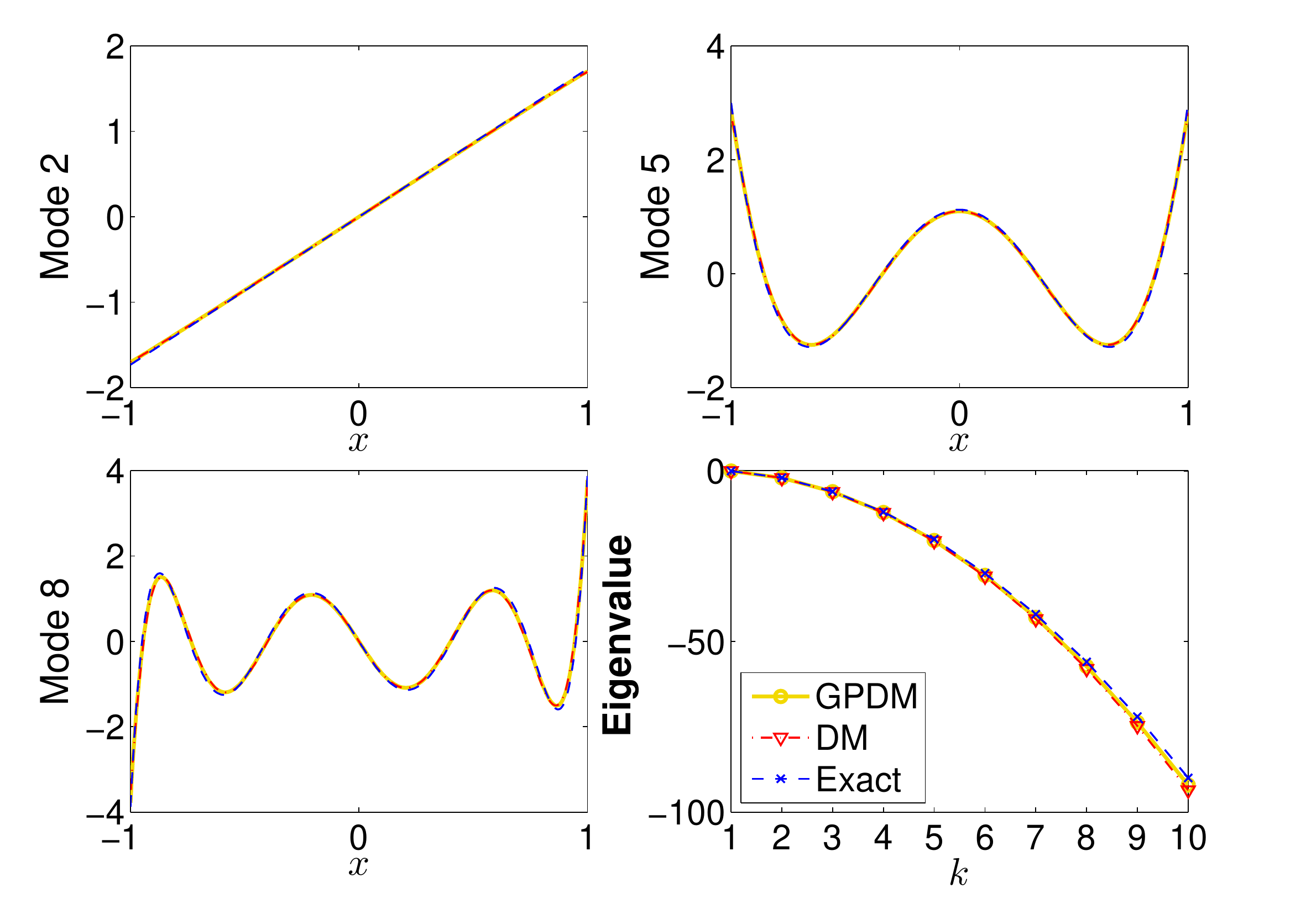}
\caption{(Color online) DM and GPDM estimation of eigenvalues and
eigenfunctions for the Legendre polynomials on flat domain $[-1,1]$.}
\label{Fig13_SL_eigs}
\end{figure*}

\begin{figure*}[tbp]
{\scriptsize \centering
\begin{tabular}{cc}
{\normalsize(a) Error of Eigenvalues} & {\normalsize(b) Error of Eigenfunctions} \\
\includegraphics[width=.45\textwidth]{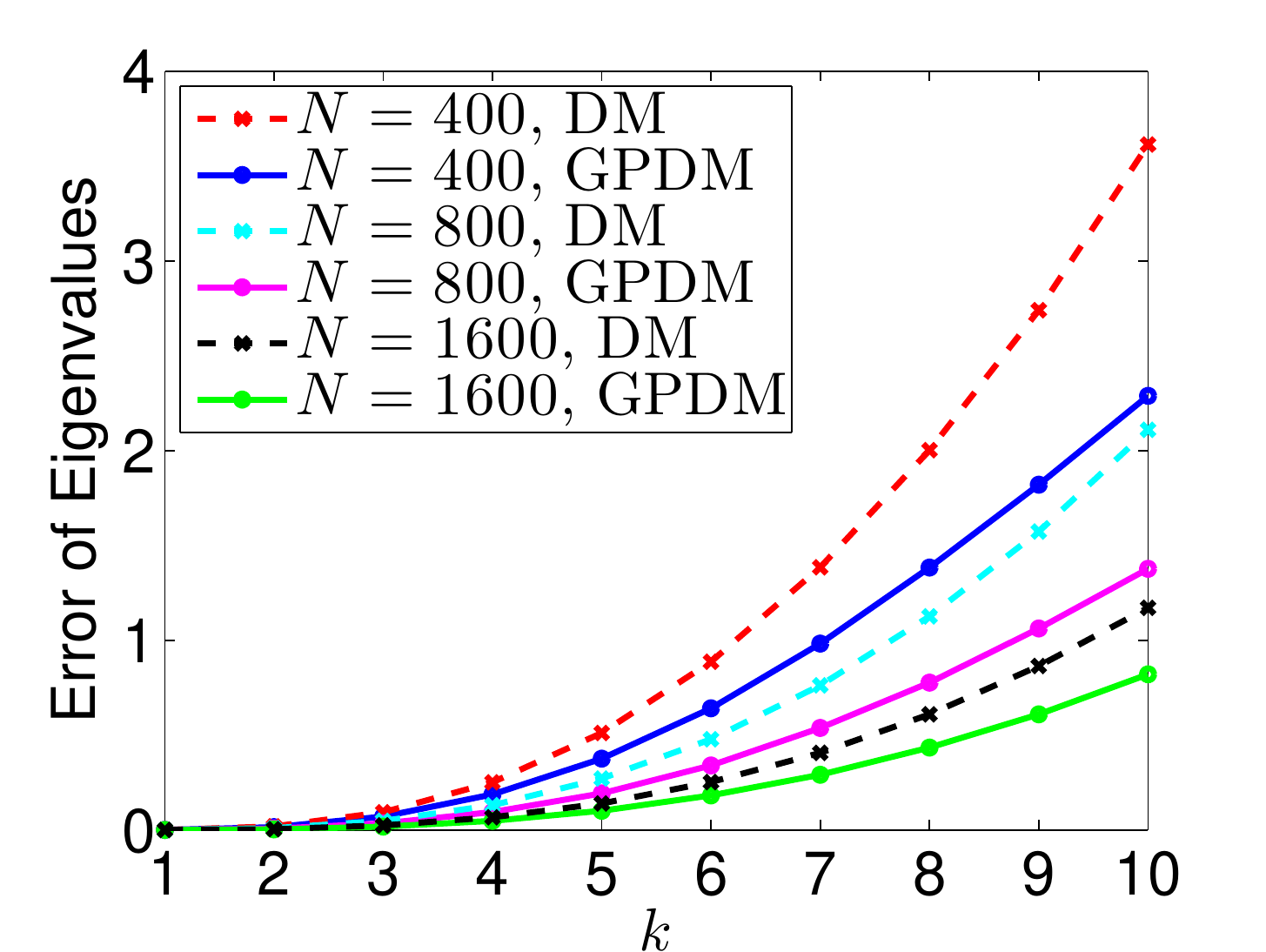}
&
\includegraphics[width=.45\textwidth]{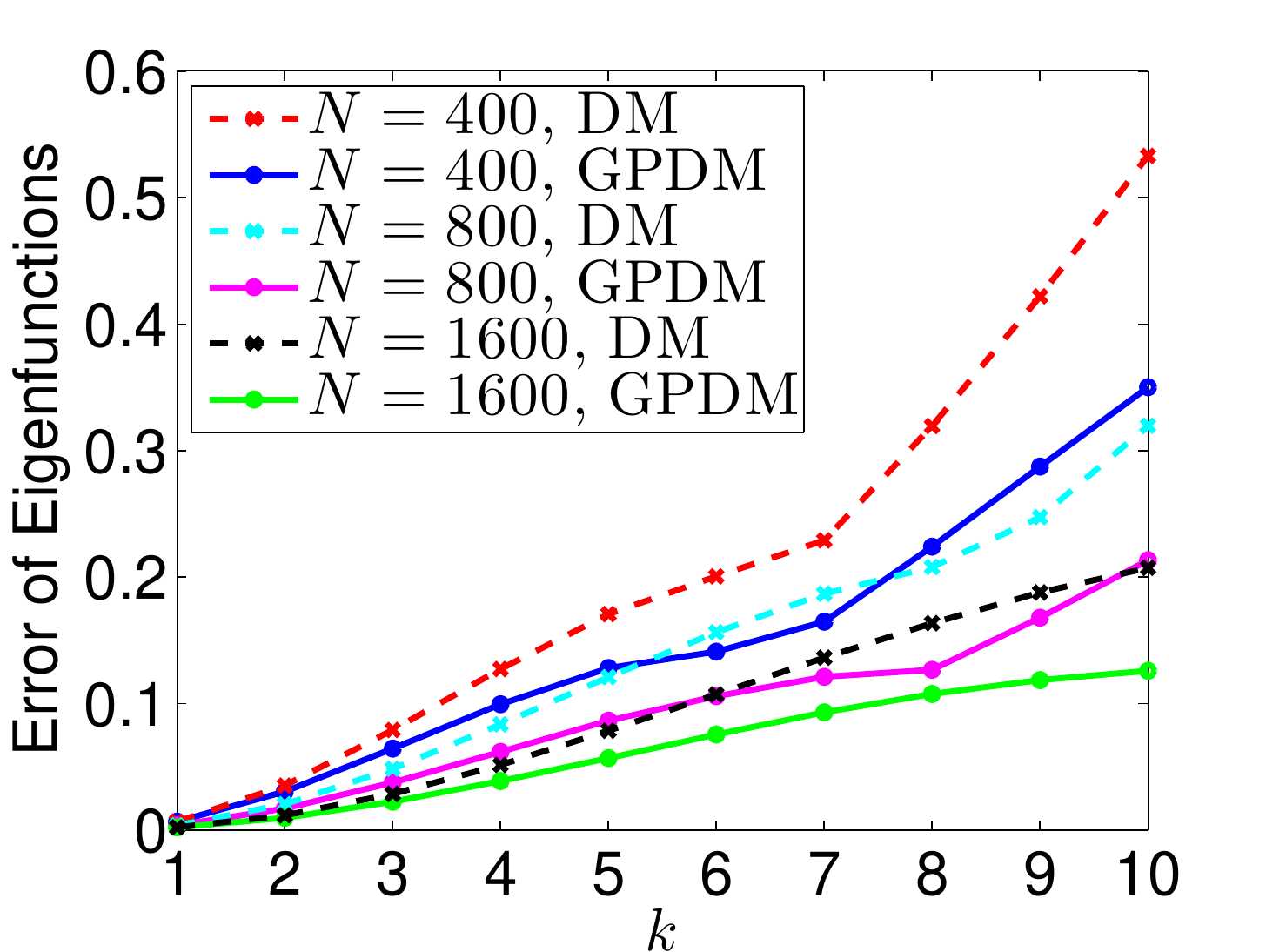}%
\end{tabular}
}
\caption{(Color online) Sturm-Liouville problem: Error of (a) eigenvalues and (b) eigenfunctions {as functions of the mode-$k$} for different number of points $N$%
. }
\label{Fig14_SL_eigs}
\end{figure*}

Fig.~\ref{Fig13_SL_eigs} shows the comparison of the eigenvalues and
eigenfunctions between the analytic Legendre polynomials and the numerical
results from DM and GPDM. It can be seen from Fig.~\ref{Fig13_SL_eigs}
that both eigenvalues and eigenfunctions can be well approximated within
numerical accuracy. For a detailed inspection, we show the errors of the
eigenvalues and the eigenfunctions as functions of the mode-$k$, respectively, for the different number of points $N$ in Fig.~\ref{Fig14_SL_eigs}. It can be seen that both DM and GPDM
provide convergent eigenvalues and eigenfunctions as $N$ increases. The errors of GPDM are slightly smaller than those of DM.

\subsection{Laplace-Beltrami operator on a semi-circle}

In this example, we consider solving the eigenvalue problem $\Delta \psi
_{k}=\lambda _{k}\psi _{k}$ on a 1D semi-circle with Dirichlet and
Robin boundary conditions. We neglect to show results with the Neumann boundary condition since
the performances of GPDM and DM are identical. The Riemannian metric of the semi-circle is given
by \eqref{Eqn:ellipse_metr} with $a=1$. For the Dirichlet boundary condition
$\psi _{k}=0$ at both ends $\theta =0$ and $\pi $, one can check that the
eigenvalues and eigenfunctions are
\begin{equation}
\lambda _{k}=-k^{2},\text{ \ \ }\psi _{k}=\sin \left( kx\right) ,\text{ \
for }k=1,2,3,\ldots . \notag %\label{Eqn:diri_eig}
\end{equation}%
For the Robin boundary condition $-\partial _{\nu }\psi _{k}+\psi _{k}=0$ at
$\theta =0$ and $\partial _{\nu }\psi _{k}+\psi _{k}=0$ at $\theta =\pi ,$
we can find the explicit expression of both the eigenvalues and eigenfunctions,
\begin{equation}
\lambda _{k}=\left\{
\begin{array}{c}
1 \\
-\left( k-1\right) ^{2}%
\end{array}%
\right. ,\text{ \ \ }\psi _{k}=\left\{
\begin{array}{l}
\exp \left( -x\right) ,\text{ \ \ \ \ \ \ \ \ \quad \quad \quad\ \ \ \ \ \ \
\ \ \ \ \ \ \ \ \ \ \ \ \ \ \ \ for\ }k=1 \\
\sin \left( \left( k-1\right) x\right) -\left( k-1\right) \cos \left( \left(
k-1\right) x\right) ,\text{ \ \ for\ }k=2,3,\ldots%
\end{array}%
\right.\notag%\label{Eqn:rob_eig}
\end{equation}%
We should point out that the Robin boundary condition at $\theta =0$ corresponds to unphysical problems.

\begin{figure*}[tbp]
{\scriptsize \centering
\begin{tabular}{cc}
{\normalsize(a) Dirichlet} & {\normalsize(b) Robin} \\
\includegraphics[width=.48\textwidth]{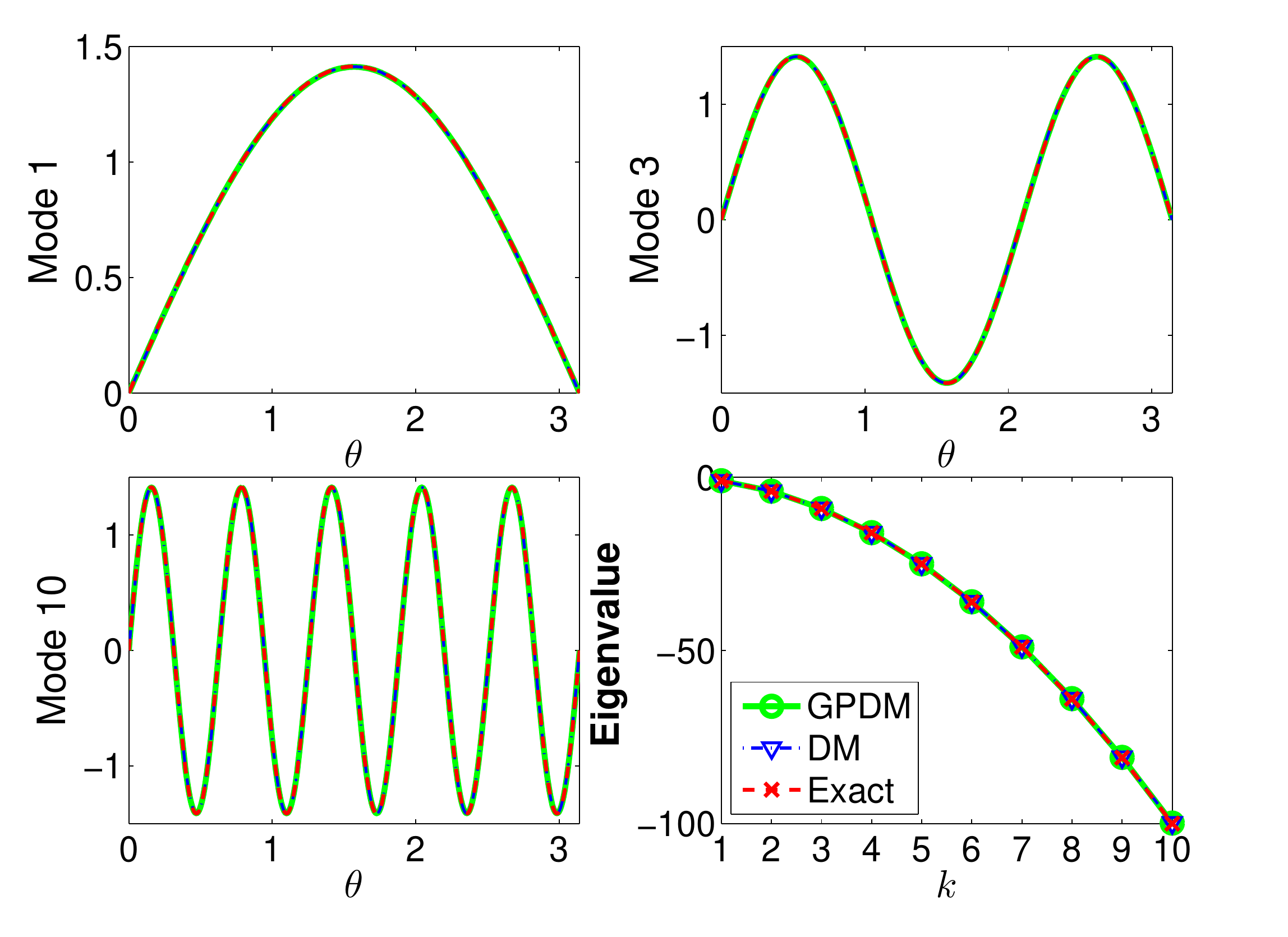}
&
\includegraphics[width=.48\textwidth]{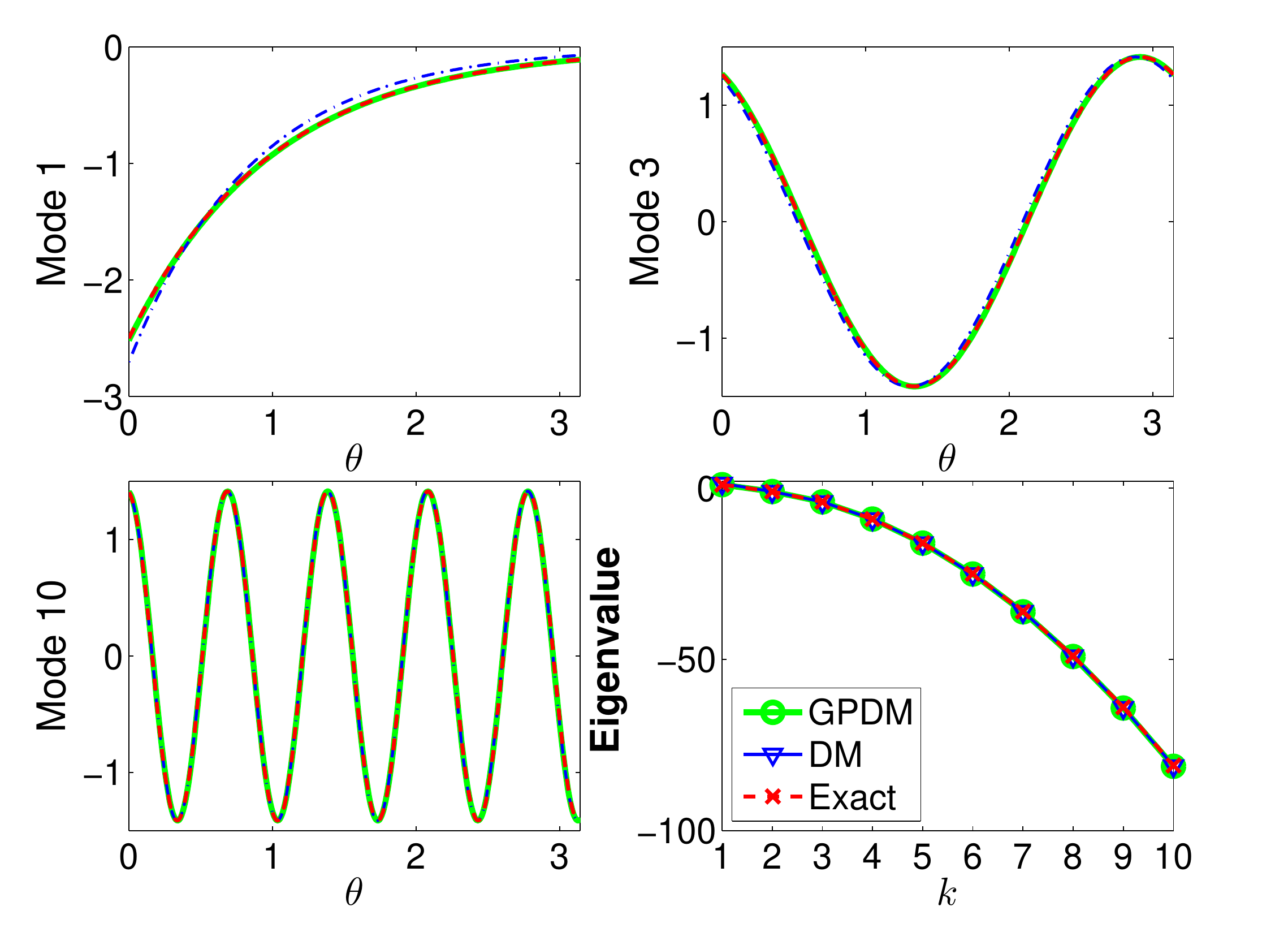}%
\end{tabular}
}
\caption{(Color online) Comparisons of eigenvalues and eigenfunctions
between DM and GPDM for semi-circle example with (a) Dirichlet and (b) Robin boundary conditions. The Riemannian metric is given by (\protect\ref%
{Eqn:ellipse_metr}) with $a=1$.}
\label{Fig5_eigs_semiellipse}
\end{figure*}

The Laplace-Beltrami operator $\mathcal{L}_{1}$\ is numerically estimated
using DM and GPDM from formula \eqref{L1}. We construct an $N\times N$
matrix on $N$ equally spaced discrete grids $\left\{ x_{i}=\left( \cos
\left( (i-1)\pi /(N-1)\right) ,\sin \left( (i-1)\pi /(N-1)\right) \right) \right\}
_{i=1,\ldots ,N}$. The kernel uses $k=50$ nearest neighbors and the
bandwidth $\epsilon =2.1\times 10^{-5}$, that is auto tuned using a fixed (for $N=400$) grid points for
all types of boundary conditions. The numerical results are shown in Fig. %
\ref{Fig5_eigs_semiellipse}. In these two problems, the eigenvalues and eigenfunctions can be well approximated by both
DM and GPDM, although DM is less accurate for the Robin boundary condition (as seen in the estimation of mode-1).

\begin{figure*}[tbp]
{\scriptsize \centering
\begin{tabular}{cc}
\normalsize (a) Error of eigenvalues vs. $k$ & \normalsize (b) Error of eigenfunctions vs. $k$
\\
\includegraphics[width=3.0
in, height=2.1 in]{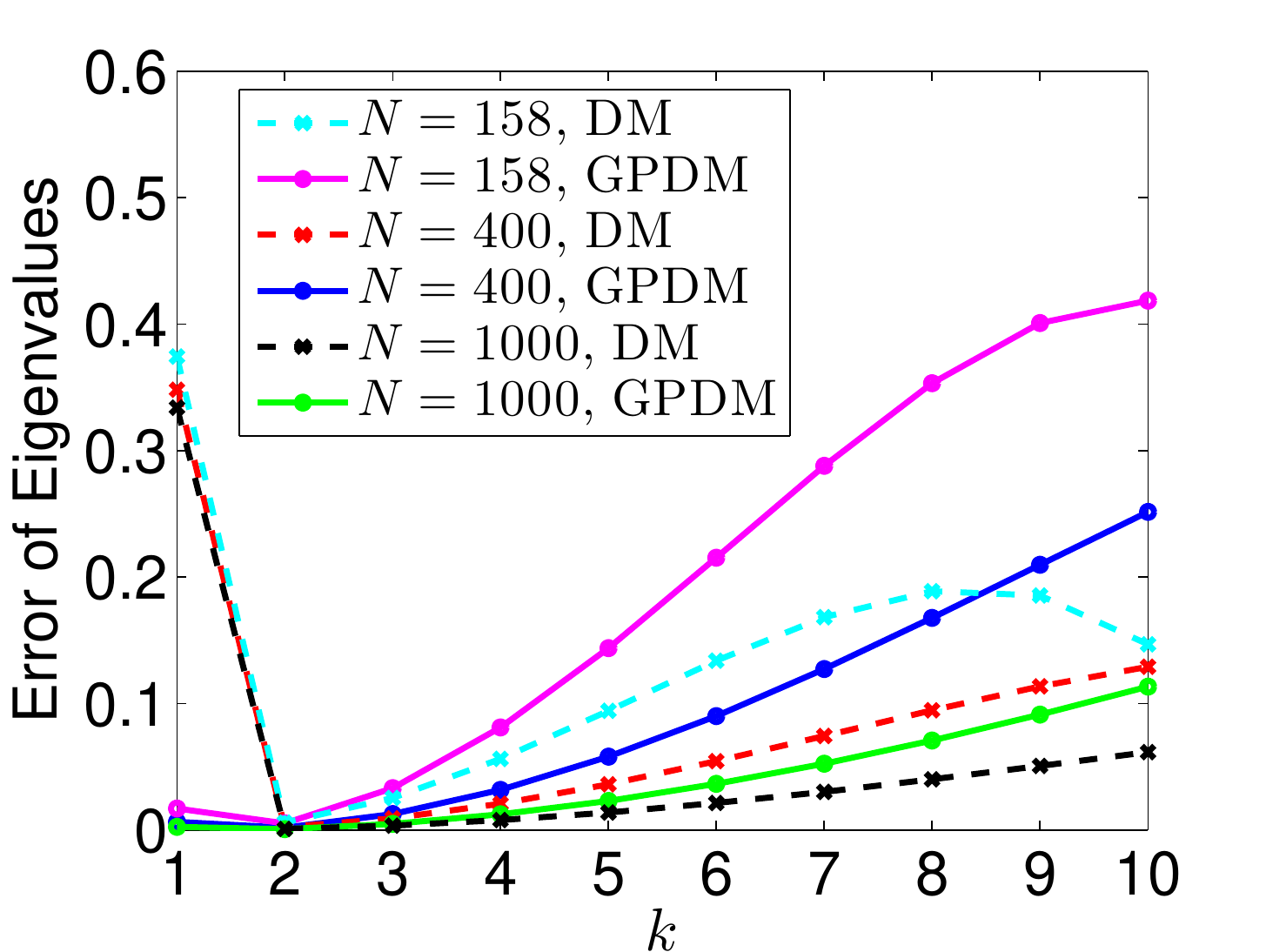}
&
\includegraphics[width=3.0
in, height=2.1 in]{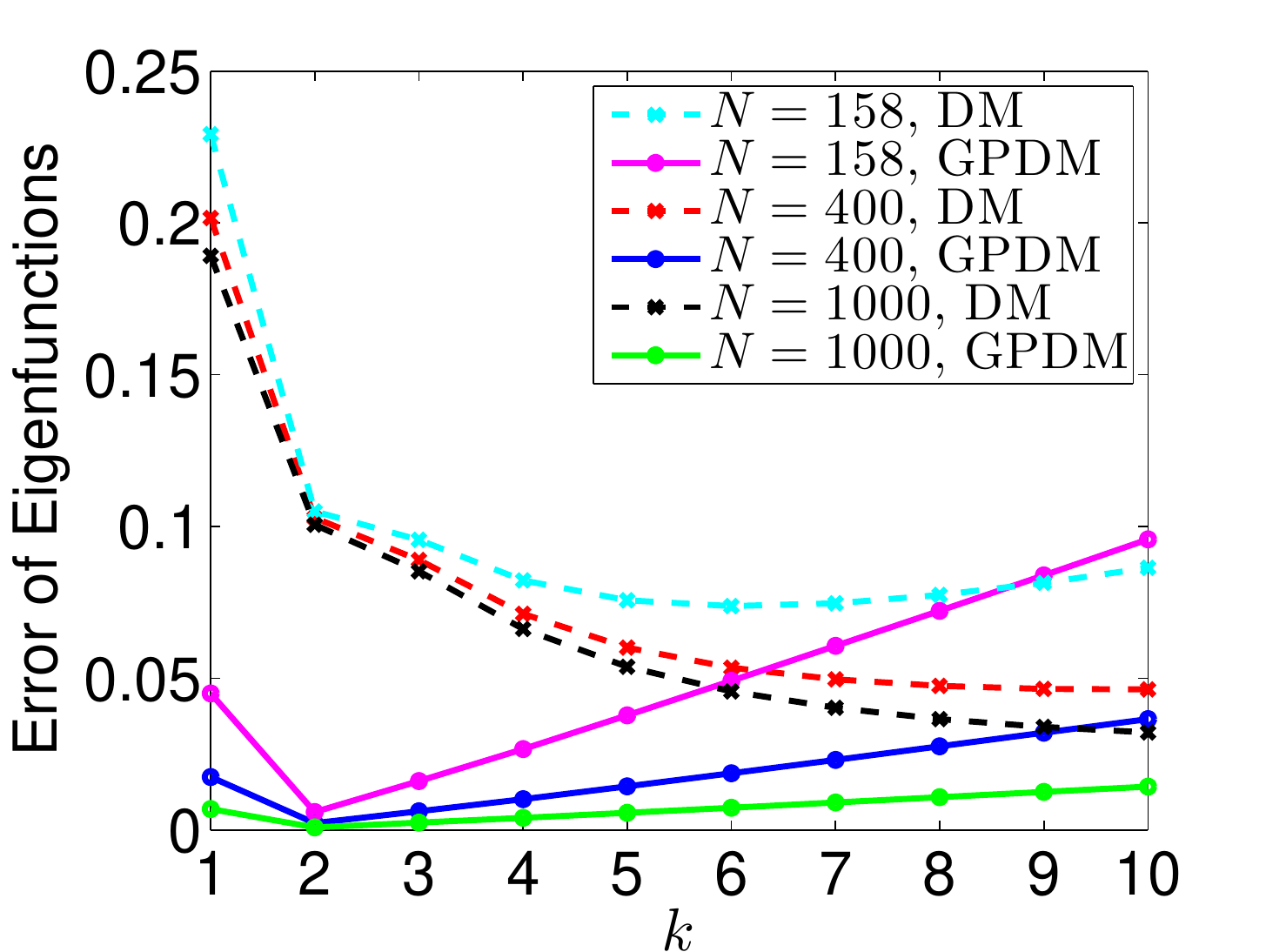}\\
\normalsize (c) Error of eigenvalues vs. $N$ & \normalsize (d) Error of eigenfunctions vs. $N$\\
\includegraphics[width=3
in, height=2.1 in]{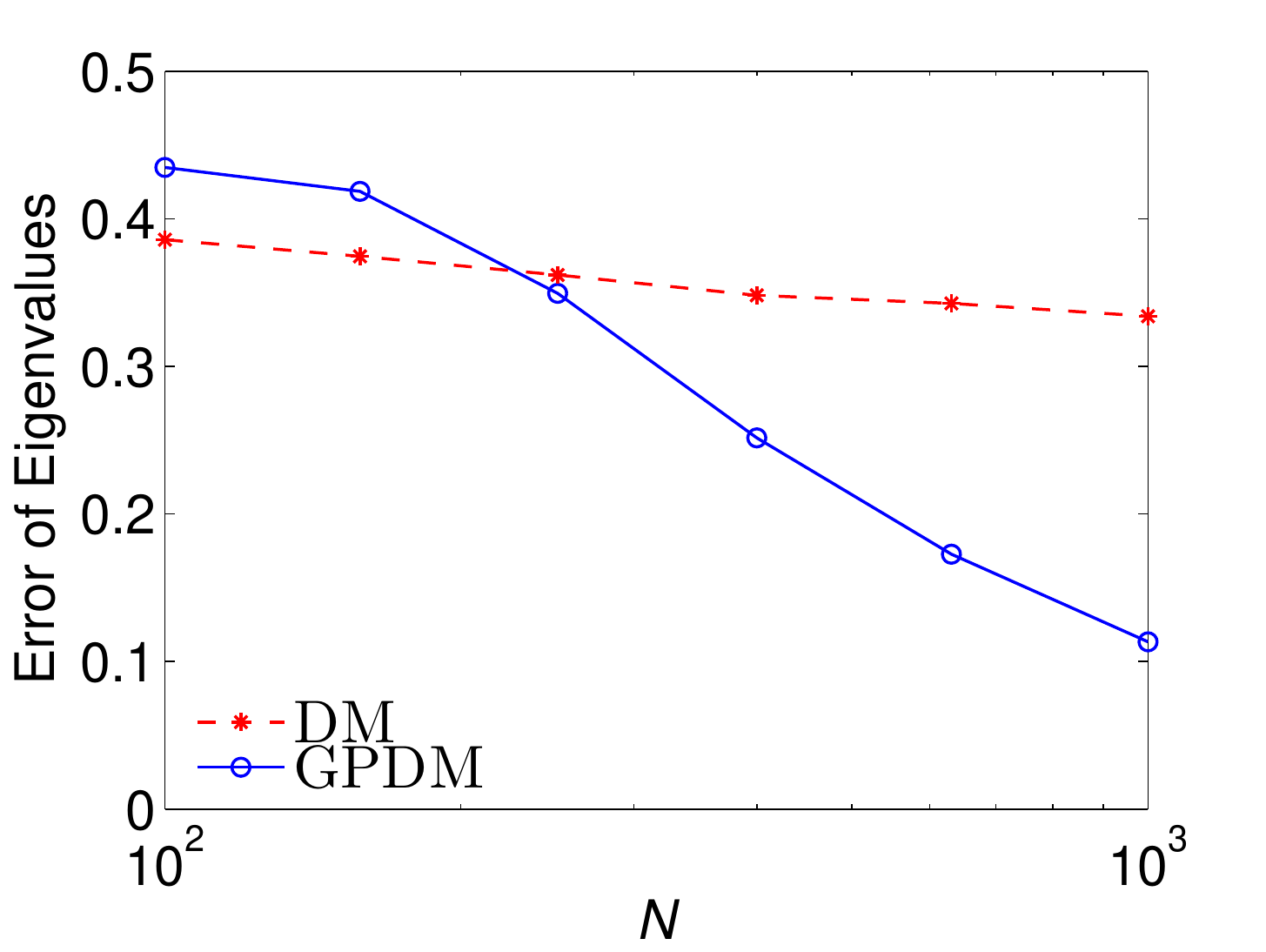}
&
\includegraphics[width=3
in, height=2.1 in]{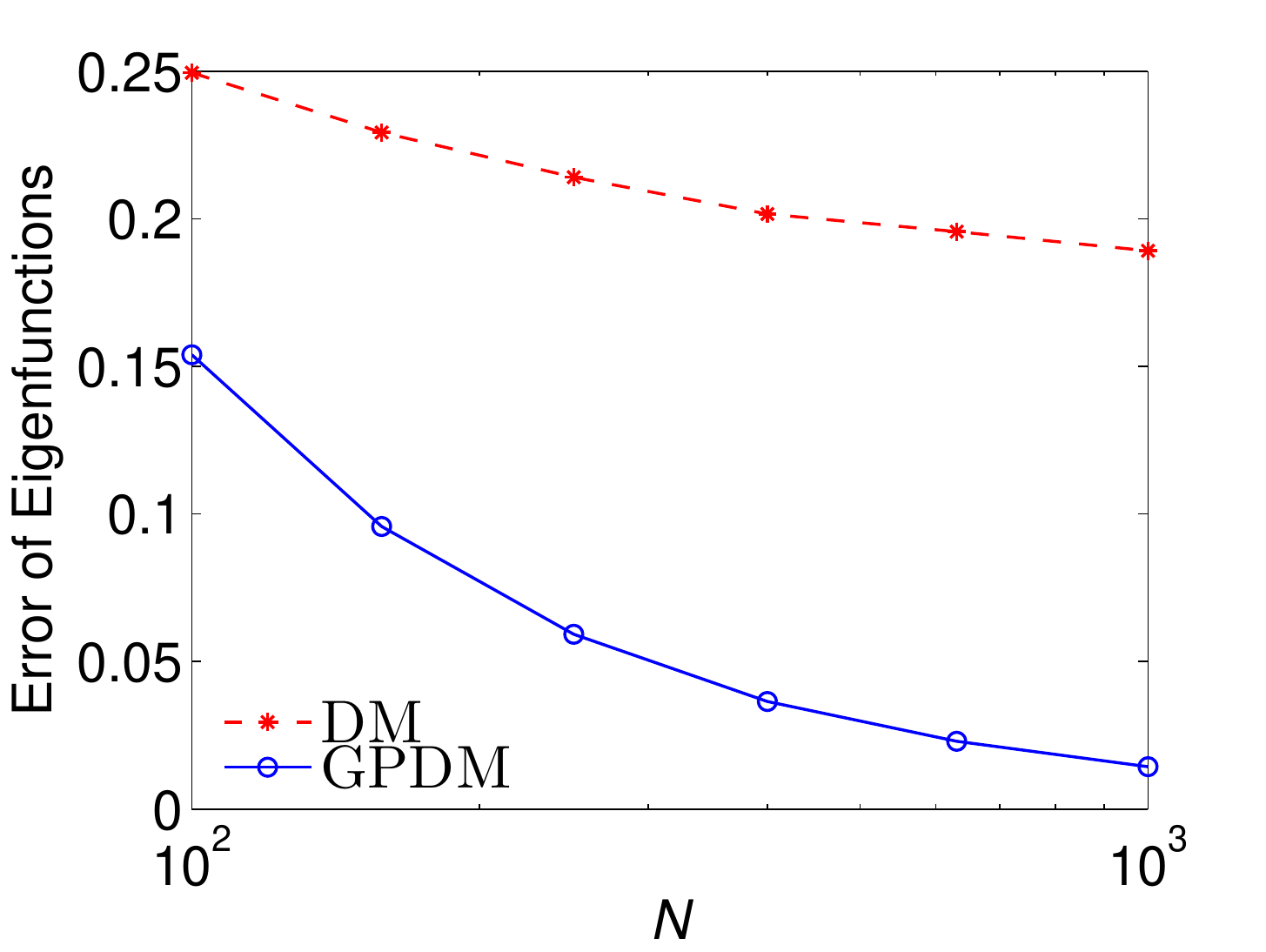}%
\end{tabular}
}
\caption{(Color online) Error of (a) eigenvalues and (b) eigenfunctions as functions of $k$ for
different number of points $N$ for the semi-circle example with the Robin boundary
condition. (c) Error of eigenvalues  and (d) error of eigenfunctions
 vs. $N$. Note that for DM, there is no convergence for the
leading eigenvalues and eigenfunctions. }
\label{Fig52_eigenlapvsk}
\end{figure*}

Figs.~\ref{Fig52_eigenlapvsk} (a) and (b) show errors of the eigenvalues and
eigenfunctions, respectively, as functions of mode-$k$ for different number of points $N$ on a
semi-circle example with Robin boundary condition. Figs.~\ref%
{Fig52_eigenlapvsk}(c) and (d) show the errors of the eigenvalues and eigenfunctions as functions of $N$, respectively. It can be seen that for DM, there is no convergence in the estimation of the leading eigenvalues and eigenfunctions as $N$
increases. In comparison, for GPDM, there is convergence in the estimation of the leading eigenvalues and
eigenfunctions.

\subsection{Laplace-Beltrami operator on a semi-torus}

In this example, we consider solving the eigenvalue problems $\Delta \psi
_{k}=\lambda _{k}\psi _{k}$ on a 2D semi-torus embedded in\ $\mathbb{R}^{3}$%
\ with Dirichlet and Dirichlet-Neumann mixed boundary conditions.
Here, the torus is defined with the standard embedding function \eqref%
{Eqn:torus_g} with the Riemannian metric \eqref{Eqn:torus_gg}, the
parameter $a=2$, and the intrinsic coordinates $\left( \theta ,\phi \right) $ on $%
\left[ 0,2\pi \right] \times \left[ 0,\pi \right] $. Then, we can check that
the Laplace-Beltrami operator in the intrinsic coordinates $\left( \theta ,\phi
\right) $ can be written as:%
\begin{equation}
\Delta \psi _{k}=\frac{1}{\left( a+\cos \theta \right) ^{2}}\frac{\partial
^{2}\psi _{k}}{\partial \phi ^{2}}+\frac{\partial ^{2}\psi _{k}}{\partial
\theta ^{2}}-\frac{\sin \theta }{a+\cos \theta }\frac{\partial \psi _{k}}{%
\partial \theta }=\lambda _{k}\psi _{k}.  \label{Eqn:lp_tor}
\end{equation}%
We can use the method of separation of variables to solve this eigenvalue problem \eqref%
{Eqn:lp_tor}, satisfying Dirichlet and the mixed boundary conditions.
That is, we set $\psi _{k}=\Phi _{k}\left( \phi \right) \Theta _{k}\left(
\theta \right) $ and substitute $\psi _{k}$ back into \eqref{Eqn:lp_tor} to deduce the eigenvalue problems for $\Phi _{k}$ and $\Theta _{k}$:
\begin{eqnarray}
\Phi _{k}^{\prime \prime }+m_{k}^{2}\Phi _{k} &=&0,  \label{Eqn:THE} \\
\Theta _{k}^{\prime \prime }-\frac{\sin \theta }{a+\cos \theta }\Theta
_{k}^{\prime }-\frac{m_{k}^{2}}{\left( a+\cos \theta \right) ^{2}}\Theta
_{k} &=&\lambda _{k}\Theta _{k},  \label{Eqn:PHI}
\end{eqnarray}%
where the derivatives in \eqref{Eqn:THE} and \eqref{Eqn:PHI} are taken with
respect to $\phi $ and $\theta $, respectively. The discrete values of $%
m_{k} $ are chosen such that $\Phi _{k}$ satisfies \eqref{Eqn:THE} with
two types of boundary conditions. In particular, type (a) is the Dirichlet
boundary condition at both sides ($\Phi _{k}\left( 0\right) =\Phi _{k}\left(
\pi \right) =0$) and type (b) is the Dirichlet-Neumann mixed boundary condition ($\Phi
_{k}\left( 0\right) =0$ and $\Phi _{k}^{\prime }\left( \pi \right) =0$).
Then, the eigenvalue problem \eqref{Eqn:PHI} can be numerically solved for $%
\lambda _{k}$\ with high-order accuracy. The eigenvalue $\lambda _{k}$
associated with the eigenfunction $\psi _{k}$\ obtained by the approach above are
treated as the exact solutions of the eigenvalue problem \eqref{Eqn:lp_tor}.

In our numerical implementation, the grid points $\left\{ \theta _{i},\phi
_{j}\right\} $ are uniformly distributed on $\left[ 0,2\pi \right] \times %
\left[ 0,\pi \right] ,$ with $i,j=1,\ldots ,64$ points in each direction
resulting in a total of $N=4096$ grid points. We assume that we do not know
the embedding function \eqref{Eqn:torus_g} when solving the eigenvalue
problem using DM and GPDM. For the GPDM method, we estimate normal direction
$\boldsymbol{\nu }$ to the boundary, add ghost points along $\boldsymbol{\nu
}$, construct an augmented matrix using standard DM, and finally construct
the $N\times N$ diffusion matrix based on the extrapolation formula and
boundary conditions. We {\color{black}use $k=200$ nearest neighbors} to construct a sparse
matrix $\mathbf{L}^h$\ for computational efficiency. The kernel bandwidth $%
\epsilon =0.004$ is auto tuned for all types of boundary conditions.

\begin{figure*}[tbp]
{\scriptsize \centering
\begin{tabular}{cc}
{\normalsize(a) Dirichlet} & {\normalsize(b) Dirichlet-Neumann mixed} \\
\includegraphics[width=.45\textwidth]{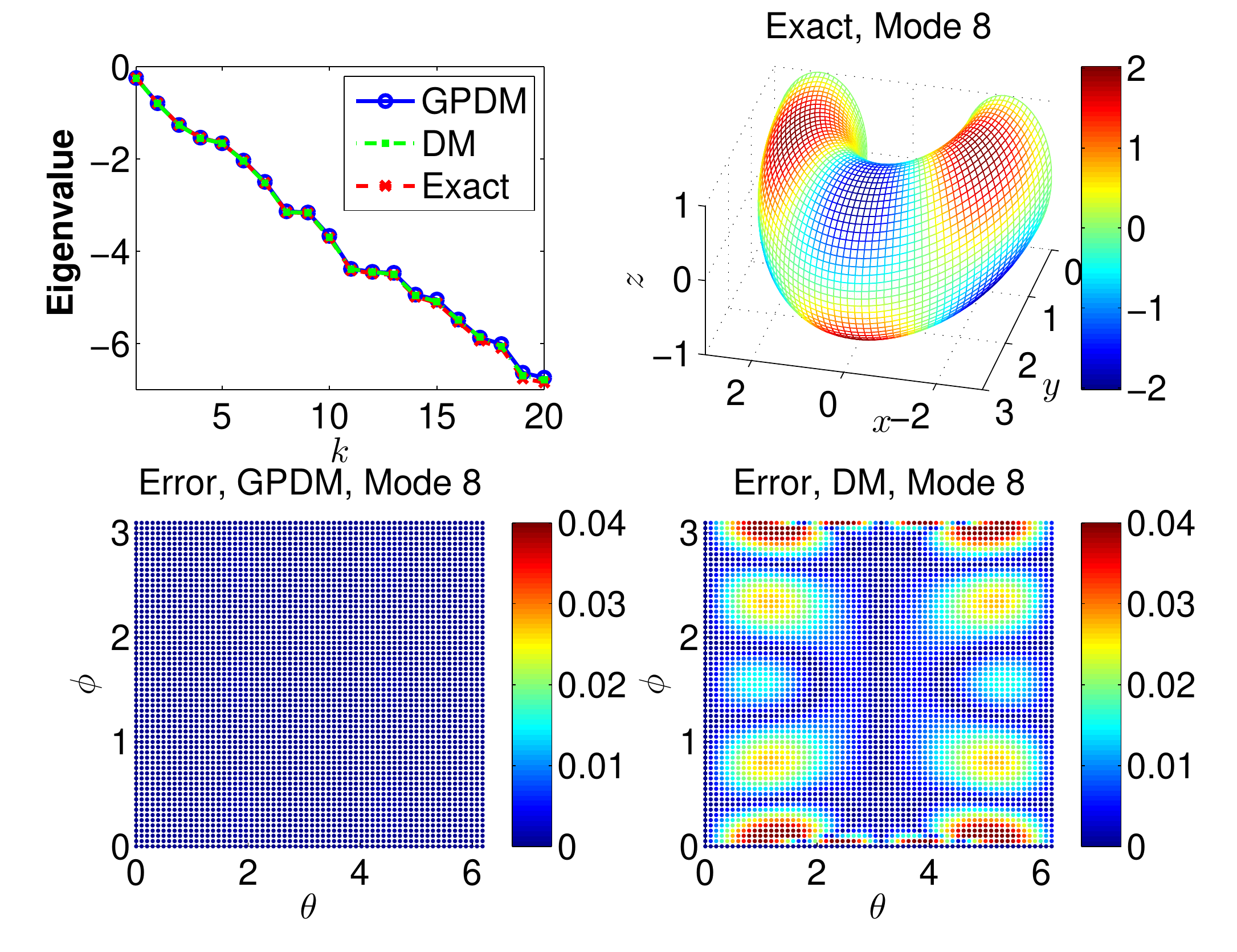}
&
\includegraphics[width=.45\textwidth]{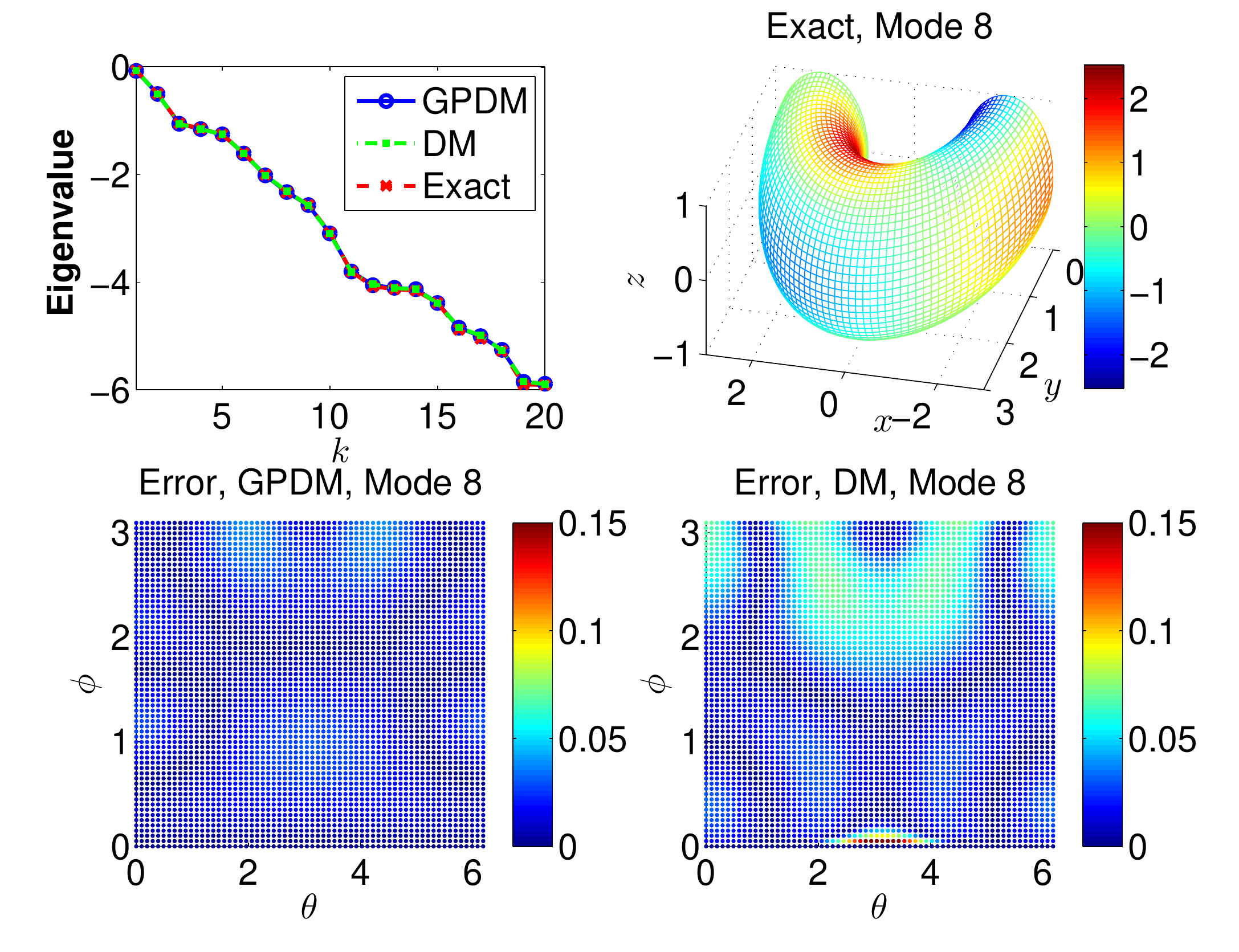}%
\end{tabular}
}
\caption{(Color online) Comparisons of eigenvalues and eigenfunctions with
(a) Dirichlet, (b) Dirichlet-Neumann mixed boundary conditions
on the semi-torus example with Riemannian metric (\protect\ref{Eqn:torus_gg})
with $a=2$.}
\label{Fig7_eigs_semitorus}
\end{figure*}

Fig.~\ref{Fig7_eigs_semitorus} shows the numerical estimates of the first
20 eigenvalues and the 8th eigenfunction for (a) Dirichlet and
(b) the mixed boundary conditions. One can see from Fig. \ref%
{Fig7_eigs_semitorus} that the eigenvalues and the 8th eigenfunction can be
approximated well by both DM and GPDM. For Dirichlet boundary condition, the largest errors of the first twenty eigenvalues are comparable as $%
0.08$ and $0.12$ using the standard DM and GPDM, respectively. The largest $\ell
^{\infty }$-norm error of the first twenty eigenfunctions using GPDM ($=0.01$%
) is much smaller than that using DM ($=0.31$).
For the mixed
boundary condition, the largest errors of the first twenty eigenvalues are comparable as $%
0.06$ and $0.04$ using the standard DM and GPDM, respectively. The largest $\ell ^{\infty }$-norm error of the first
twenty eigenfunctions using GPDM ($=0.99$) is  comparable to that using DM ($=1.03$). However, a close inspection, e.g. the 8th eigenfunctions, suggests that the GPDM errors occur on smaller regions of the domain compared to those of DM.

\section{Summary}\label{sec:summary}

In this paper, we introduced the Ghost Points Diffusion Maps (GPDM) to estimate second-order elliptic differential operators defined on smooth manifolds with boundaries. The proposed method overcomes the inconsistency of the diffusion maps (DM) algorithm in estimating these differential operators near the boundaries. We provided theoretical convergence study as well as numerical verification on test problems with tractable solutions and on the unknown ``face" manifold to validate our claim. The key idea of GPDM is motivated by the standard ghost points approach that is used to obtain a higher-order finite-difference approximation of Neumann/Robin type boundary conditions on the flat domain. Our key contribution is to realize this idea with a concrete numerical algorithm on unknown manifolds, identified only by the point clouds, that is guaranteed to be consistent.

We considered solving elliptic PDEs \eqref{PDE} with the GPDM operator estimation method. We showed that the PDE solver, which is a mesh-free technique, is a convergent method under the standard assumption of the well-posedness of the PDE problem. Numerically, we validated the solver on a series of 1D and 2D test examples with and without explicit solutions. On problems with unknown manifold where the explicit solution is unknown, we compare the result against the FEM solution. Overall, GPDM is much more accurate and robust relative to DM except on the Neumann boundary for which DM is expected to work well as shown in \cite{gh2019}. Numerically, we also found that GPDM is more accurate compared to DM in terms of solving eigenvalue problems associated with the operators \eqref{L1}-\eqref{L2}.

While the proposed approach is encouraging, it also poses many open questions, namely:
\begin{itemize}
\item The ghost points are constructed by extending points along the exterior normal direction from the boundary. Since the errors in estimating the directional derivatives and normal vectors depend on Hessian (second-order derivatives), the error can be large if the curvature is very large at the boundary. This suggests that the method can be improved by specifying ghost points that account for the curvature at the boundary. In our context, where the manifold is unknown, this requires an estimation of the boundary curvature from the point clouds, which is a problem that we are not currently familiar to.
\item In this work, we have verified the proposed method on 1D and 2D manifolds. For higher-dimensional manifolds, while the numerical method can be used, the conditions to achieve the conclusion in the Lemma~\ref{ghmanifold} requires further studies.

\item The proposed technique assumed that the boundary points are given. In the case of well-sampled data, the number of points at the boundary is specified explicitly, $J=N^{1/d}$. For the randomly sampled data, when we employ the local kernel, the auto-tuned $\epsilon$ yields error rates of $\epsilon_1 \sim N^{-1/2}$ and $\epsilon_2 \sim J^{-1}$. To have a balanced error, $\epsilon_1\sim \epsilon_2$, we require $J = N^{1/2}$, as we numerically verified in 2D examples. If this scaling is valid for arbitrary dimension, that is, $\epsilon_1\sim N^{-1/d}$ and $\epsilon_2\sim J^{-\frac{1}{d-1}}$, then the number of points at the boundary required to achieve balanced error rates of order $N^{-1/d}$ is $J = N^{\frac{d-1}{d}}$. While this estimate seems to indicate a severe limitation of this method, intuitively, this is consistent with the well-known fact that the distribution of high-dimensional random variables on a bounded domain tends to lie near the boundary. Further investigation is required to understand this thoroughly.

\item While the numerical demonstration showed convincing results in solving eigenvalue problems, spectral convergence, as well as the error estimate of the eigenfunctions, are not known. One possible avenue is to extend the result in \cite{trillos2019error,bs2019consistent,calder2019improved,dunson2019spectral,calder2020lipschitz} to manifolds with general boundary conditions. In this direction, the result for the Neumann boundary condition was recently reported in \cite{lu2020}.
\end{itemize}

\section*{Acknowledgment}

The research of JH was partially supported under the NSF grant DMS-1854299. This research was supported in part by a Seed Grant award from the Institute for Computational and Data Sciences at the Pennsylvania State University. The authors also thank Faheem Gilani for providing an initial sample code of FELICITY FEM, Ryan Vaughn and Tyrus Berry for the helpful discussion on various aspects of differential geometry.

\appendix

\section{Proof of Proposition~\ref{prop:extrapolationofu}}\label{App:C}

Before proving the main result (Proposition~\ref{prop:extrapolationofu}), we state the pointwise error estimates that are known from literature. Subsequently, we deduce several lemmas before proving the main result.

From previous results \cite{coifman2006diffusion,SingerEstimate,berry2016local,harlim2018,harlim2019kernel}, we have the
pointwise error estimation under these three situations: (1) on manifolds without
boundary, or (2) for the test function $u$ with Neumann boundary condition
on manifold with boundary, or (3) for any $u$ on manifold with boundary but
only for points away from the boundary with distance at least $\mathcal{O}%
\left( \epsilon ^{r}\right)$, $0< r< 1/2$. For reader's convenience, we quote the following error estimation based on the third situation.

\begin{lem}
\label{lem:old} (Pointwise forward error estimate) Let  $M\cup \Delta M$ be a smooth $d$-dimensional manifold embedded in $\mathbb{R}^n.$ Let the assumptions in Proposition~\ref{prop:extrapolationofu} for the extended manifold $M\cup\Delta M,$ and $x\in M$ hold. Let  $x_{i}\in M$ for $i=1,\ldots ,N$ and $x_j^{G_k}\in \Delta M$\ for $j=1,\ldots ,J, k=1,\ldots,K$ be i.i.d samples with
sampling density $q\in C^{3}\left( M\cup \Delta M\right) $ defined with
respect to the volume form inherited by the $d$-dimensional smooth augmented
manifold $M\cup \Delta M$ from the ambient space $\mathbb{R}^{n}$. For any $%
u\in C^{3}\left( M\cup \Delta M\right) $, define a vector $\vec{u} = (u(x_1),\ldots, u(x_N), u(x_{1}^{G_1}),\ldots, u(x_{J}^{G_K}))^\top \in \mathbb{R}^{\bar{N}}$. Then for $i=1,\ldots, N$,
\BEA
\left\vert  (\mathbf{L}_{j'}\vec{u})_i -\mathcal{L}_{j'} u(x_i)\right\vert  &=&O\left( \epsilon ,\frac{q\left( x_{i}\right) ^{1/2}}{%
\sqrt{\bar{N}}\epsilon ^{2+d/4}},\frac{\left\vert \nabla u\left( x_{i}\right)
\right\vert q\left( x_{i}\right) ^{-1/2}}{\sqrt{\bar{N}}\epsilon ^{1/2+d/4}}%
\right),\quad\quad j'=1,3,  \label{Eqn:err_nobound}\\
\left\vert  (\mathbf{L}_2\vec{u})_i -\mathcal{L}_2 u(x_i)\right\vert  &=&O\left( \epsilon ,\frac{q\left( x_{i}\right) ^{1/2}}{%
\sqrt{\bar{N}}\epsilon ^{2+d/4}},\frac{\left\vert \nabla (\sqrt{\kappa(x_i)}u( x_{i}))
\right\vert q\left( x_{i}\right) ^{-1/2}}{\sqrt{\bar{N}}\epsilon ^{1/2+d/4}}%
\right)\label{Eqn:err_L2}
\EEA%
in high probability as $\epsilon\to 0$ after $\bar{N}\to\infty$. For $\mathcal{L}_{1}$ and $\mathcal{L}_{2}$, the gradient operator is defined with respect to the Riemannian metric $g\left( u,v\right) $ for all $u,v\in T_{x}\left( M\cup \Delta
M\right) $, inherited by $M$\ from
the ambient space. For $\mathcal{L}_{3}$, the gradient operator is defined
with respect to a new metric, $\tilde{g}\left( u,v\right) :=g\left(
c^{-1/2}u,c^{-1/2}v\right) $ for all $u,v\in T_{x}\left( M\cup \Delta
M\right) $, where $c$ denotes the symmetric positive definite diffusion
tensor.
\end{lem}

In \eqref{Eqn:err_nobound}-\eqref{Eqn:err_L2}, the first error-term, which is valid as $\epsilon\to 0$, is due to the continuous asymptotic expansion in \eqref{L1}, \eqref{L2}, and \eqref{L3}. The second error term is due to the estimation of the sampling density through \eqref{qehat}, and the final error term is the bias induced by the discrete estimator, both of these are valid as $\bar{N}\to\infty$ and fixed $\epsilon>0$.

\begin{proof}
The proofs for the cases $j'=1$ and $3$ are readily available in  \cite{coifman2006diffusion,SingerEstimate,berry2016local,harlim2018,harlim2019kernel}.
For $j'=2$, the proof follows directly the steps in Appendix A of \cite{harlim2019kernel} with the following modification. Define a matrix  $\mathbf{K}_{ij} = K(\epsilon,x_i,x_j)=\exp\big(-\frac{|x_i-x_j|^2}{4\epsilon}\big)$. Let
\BEA
\hat{q}_\epsilon(x_j) := \frac{\epsilon^{d/2}}{\bar{N}}\sum_{i=1}^{\bar{N}} \mathbf{K}_{ji},\label{qehat}
\EEA
as an estimator to the sampling density of the data $q(x_j)$. With this definition, we define $F_i(x_j): = \frac{K(\epsilon,x_i,x_j)\sqrt{\kappa(x_j)}u(x_j)}{\hat{q}_\epsilon(x_j)}$ and $G_i(x_j): = \frac{K(\epsilon,x_i,x_j)\sqrt{\kappa(x_j)}}{\hat{q}_\epsilon(x_j)}$. Following exactly the steps in the proof in \cite{harlim2019kernel} with the asymptotic expansion in \eqref{asymp}, one obtains the error estimate in \eqref{Eqn:err_L2} for a discrete estimator that converges to $\kappa^{-1}\mathcal{L}_2$. Thus, the error for estimating $\mathcal{L}_2$ is no difference since the discrete estimator involves only a left multiplication by a diagonal matrix with diagonal components $\kappa(x_j)$ (which we denoted by $\mathbf{S}$ in \eqref{discreteL2}).
\end{proof}

Now, we will deduce several intermediate results that will simplify the proof of Proposition~\ref{prop:extrapolationofu}. {\color{black}For the discussion below, we define the matrix,
\BEA
\mathbf{L}^h=\frac{1}{\epsilon}(\mathbf{\tilde{D}}-\mathbf{I}) :=\frac{1}{\epsilon}((\mathbf{D}^h)^{-1}\mathbf{K}^h-\mathbf{I}), \label{Ltildeinlemma}
\EEA
obtained from the standard diffusion maps as a discrete approximation to one of the diffusion operators in \eqref{explicitLs} with the entries of $\mathbf{%
\tilde{D}}$ denoted by $\tilde{D}_{i,j}$. As we discussed in Section~\ref{artBC_ghost},
the matrix $\mathbf{L}^h$ as a discrete approximation to one of the diffusion operators in \eqref{L1}-\eqref{Eqn:L3} with the following important modification. We construct the matrix $\mathbf{L}^h$ by evaluating the kernel on  $\{x_i\}_{i=1}^N\cup \{\tilde{x}_j^{G_k}\}_{j,k=1}^{J,K}$,
where the interior ghost points are denoted as components of $\{x_i\}$, that is, $\{\tilde{x}_j^{G_0}\} \subset\{x_i\}$. To be consistent with the notation in Section~\ref{artBC_ghost}, we emphasize that
$\mathbf{L}^h \in \mathbb{R}^{N\times\bar{N}}$ is a non-square matrix with $\bar{N}=N+JK$, where the $N$-rows correspond to the kernel evaluation at $\{x_i\}_{i=1}^N$. {\color{black}Based on the discussion in the Remark~\ref{rem:ws}, we will only proof the next Lemma for randomly sampled data, which error is mainly due to:
\BEA
\begin{aligned}\label{distance}
|\gamma_j(h)- \tilde{x}_j^{G_0}| &\leq |\gamma_{j}(h)- {x}_j^{G_0}| + |x_j^{G_0}- \tilde{x}_j^{G_0}|=\mathcal{O}(h^2) + h|\boldsymbol{\nu}-\boldsymbol{\tilde{\nu}}| = \mathcal{O}(h\sqrt{\epsilon}). \\
 |x_j^{G_k}- \tilde{x}_j^{G_k}| &= kh |\boldsymbol{\nu}-\boldsymbol{\tilde{\nu}}| = \mathcal{O}(h\sqrt{\epsilon})
 \end{aligned}
\EEA
for $j=1,\ldots,J$ and $k=1,\ldots, K$, as pointed out in Remark~\ref{rem:rs}. In \eqref{distance}, we have used the fact that $|\boldsymbol{\nu}-\boldsymbol{\tilde{\nu}}|=\mathcal{O}(\sqrt{\epsilon})$ for randomly sampled data.} For convenience, we recall that $\gamma_j(h) := \exp_{\boldsymbol{x}_j^B}(-h\boldsymbol{\nu}_{x_j^B})\in M$ (see Fig.~\ref{fig2_normal_dire}(a) for a geometric illustration).

Let $\mathbf{L}$ be defined as in \eqref{Ltildeinlemma} such that the matrix is constructed by evaluating the kernel on $\{x_i\}\cup \{x^{G_k}_j\}$, where we replace the interior ghost points $\{\tilde{x}^{G_0}_j\}$ in the construction of $\mathbf{L}^h$ with the corresponding $\{\gamma_j(h)\}\in M$.
 Basically $\mathbf{L}$ is constructed based on points that lie on the extended manifold, $M\cup\Delta M$. 
 With this construction, we have:
\begin{lem}\label{lem:noisymatrix} Let $\mathbf{L}^h$ and $\mathbf{L}$ be constructed as in the discussion above. Suppose that $x_j:=\gamma_j(h)$ such that $|x_i -x_j| = \mathcal{O}(h)$ for all $i\neq j$, then for $u\in C(B_{\epsilon^r}(\partial M))$,
\BEA
\sum_{j=1}^{\tilde N} \mathbf{L}^h_{ij} \tilde{u}_j =\sum_{j=1}^{\tilde N}  \mathbf{L}_{ij}u_j + \mathcal{O}(h^2\epsilon^{-3/2},h\epsilon^{-1/2}),\label{noisymatrix}
\EEA
where $h$ depends on $\epsilon$ such that the first two terms vanishes as $\epsilon\to 0$.
\end{lem}
where $\tilde{u}_j$ and $u_j$ are different only on the ghost points, particularly, if $\tilde{u}_j= u(\tilde{x}_j^{G_0})$, then $u_j= u(\gamma_j(h))=u(x_j)$. Also, if $\tilde{u}_j= u(\tilde{x}_j^{G_k})$, then $u_j=u(x_j^{G_k})$ for all $k=1,\ldots, K$.

\begin{proof}
Suppose we consider $x_j = \gamma_j(h)$ that satisfies \eqref{distance}. For $|x_i-x_j|=\mathcal{O}(h)$, one can show that,
\BEA
\mathbf{K}^h_{ij}&:=&\exp\Big(-\frac{|x_i- \tilde{x}_j^{G_0}|^{2}}{4\epsilon }\Big) = \exp\Big(-\frac{|x_i-x_j|^{2}}{4\epsilon }\Big) \exp\Big( -\frac{c|x_i- x_j|h\sqrt{\epsilon}+ \mathcal{O}(h^2\epsilon)}{4\epsilon}\Big)\nonumber\\
&=& \exp\Big(-\frac{|x_i- x_j|^{2}}{4\epsilon }\Big)  (1+\mathcal{O}(h^2\epsilon^{-1/2})).\nonumber
\EEA
Therefore,
\BEA
\mathbf{D}^h_i:=\sum_{j=1}^{\bar{N}} \mathbf{K}^h_{ij} = \sum_{j=1}^{\bar{N}} \mathbf{K}_{ij} + \mathcal{O}(h^2\epsilon^{-\frac{1}{2}}) := \mathbf{D}_i + \mathcal{O}(h^2\epsilon^{-\frac{1}{2}}),\nonumber
\EEA
where the constant in the big-oh notation absorbs the number perturbed points, which is much smaller than $k$ when $k$-nearest neighbor summand is used. This means,
\BEA
(\mathbf{D}^h_i)^{-1}  := \mathbf{D}_i^{-1} (1- \mathbf{D}_i^{-1}\mathcal{O}(h^2\epsilon^{-1/2})).\notag
\EEA
and
\BEA
\sum_{j=1}^{\bar{N}} \mathbf{L}_{ij}^h \tilde{u}_j  &:=& \epsilon^{-1}\Big( (\mathbf{D}^h_i)^{-1} \sum_{j=1}^{\bar{N}} \mathbf{K}^h_{ij} \tilde{u}_j - \tilde{u}_i\Big) \nonumber \\
&=& \epsilon^{-1}  \Big(  \mathbf{D}_i^{-1}(1 - \mathbf{D}_i^{-1}\mathcal{O}(h^2\epsilon^{-1/2})) \sum_{j=1}^{\bar{N}} \mathbf{K}^h_{ij} \tilde{u}_j   - \tilde{u}_i\Big) \nonumber \\
&=& \epsilon^{-1}  \Big(  \mathbf{D}_i^{-1} (1- \mathbf{D}_i^{-1}\mathcal{O}(h^2\epsilon^{-1/2})) \big(\sum_{j=1}^{\bar{N}} \mathbf{K}_{ij} \tilde{u}_j + \sum_{j=1}^{\bar{N}} \tilde{u}_j\mathcal{O}(h^2\epsilon^{-1/2})\big)  - \tilde{u}_i\Big) \nonumber \\
&=& \epsilon^{-1}  \Big(  \mathbf{D}_i^{-1} \sum_{j=1}^{\bar{N}} \mathbf{K}_{ij} \tilde{u}_j - \tilde{u}_i+  \mathbf{D}_i^{-1} \sum_{j=1}^{\bar{N}} \tilde{u}_j\mathcal{O}(h^2\epsilon^{-1/2}) -\mathbf{D}_i^{-1}\sum_{j=1}^{\bar{N}} \mathbf{K}_{ij} \tilde{u}_j \mathbf{D}_i^{-1}\mathcal{O}(h^2\epsilon^{-1/2}) \Big) \nonumber \\
&=& \sum_{j=1}^{\bar{N}} \mathbf{L}_{ij} \tilde{u}_j + \mathbf{D}_i^{-1}\big(\sum_{j=1}^{\bar{N}} (1-\mathbf{D}_i^{-1}\mathbf{K}_{ij})\tilde{u}_j \big) \mathcal{O}(h^2\epsilon^{-3/2})\big) \nonumber\\
&=& \sum_{j=1}^{\bar{N}} \mathbf{L}_{ij} u_j + \mathcal{O}(h^2\epsilon^{-3/2},h\epsilon^{-1/2}),\nonumber
\EEA
where in the last equality, we have used the fact $u_j-\tilde{u}_j = \mathcal{O}(h\epsilon^{1/2})$ on the estimated ghost points due to \eqref{distance}, $u\in C^{1}(B_{\epsilon^r}\partial M)$, and $1-\mathbf{D}_i^{-1}\mathbf{K}_{ij} \leq 1$ and $\mathbf{D}_i \geq 1$ such that
$\mathbf{D}_i^{-1}\big(\sum_{j=1}^{\bar{N}} (1-\mathbf{D}_i^{-1}\mathbf{K}_{ij})\tilde{u}_j\big) \leq | {\tilde{u}}|$, where $\tilde{u}\in\mathbb{R}^{\bar{N}}$ is defined below in \eqref{utildevec}.

\end{proof}

The assumption that $|x_i -x_j| = \mathcal{O}(h)$, where $x_j:=\gamma_j(h)$ is rather natural in the numerical implementation with $k$-nearest neighbor, even if there are many other points $x_i$ that are further away with distance of order-$\sqrt{\epsilon}$. In particular, our construction is such that the perturbed points $\{\tilde{x}^{G_0}_j\}$ are defined to be of order-$h$ away from each boundary point and the corresponding ghost points $\{\tilde{x}^{G_k}_j\}$ as defined in \eqref{approxghostpontsembedding}. Therefore, when $k$-nearest neighbor is used in constructing the matrix $\mathbf{L}^h$, then these estimated ghost points are either belong to the $k$-nearest sets of other points whose distance are of order-$h$ or they have $k$-nearest neighbor of mostly points of order-$h$ away.

Since $h=\mathcal{O}(\epsilon)$ for the randomly sample data case, the two error bounds are equivalent. Next, we define the column vector
\BEA
\tilde{u} := (\tilde{u}_1,\ldots,\tilde{u}_{\bar{N}}) = (u(x_1), \ldots, u(\tilde{x}_j^{G_0}), \ldots, u(x_N), u(\tilde{x}_1^{G_1}),\ldots,u(\tilde{x}_J^{G_K}) ) \in \mathbb{R}^{\bar{N}}.\label{utildevec}
\EEA
For $\vec{u}$ as defined in \eqref{truevectoru}, the error rate in \eqref{noisymatrix} can be written in a compact form as,
\BEA
\mathbf{L}^h \tilde{u}= \mathbf{L}\vec{u} + \mathcal{O}(h^2\epsilon^{-3/2}).\label{errorofLh1compact}
\EEA
This error also implies,
\BEA
\mathbf{L}^h \vec{u}= \mathbf{L}\vec{u} + \mathcal{O}(h^2\epsilon^{-3/2}),\label{errorofLh}
\EEA
since $\tilde{u}-\vec{u}=\mathcal{O}(h\epsilon^{1/2})$.

Next, we will deduce the consistency of the extrapolation formula  \eqref{Eq:Uvvv}. For this purpose, we define $\vec{u}_{\epsilon }=(u(x_{1}),\ldots, u(\tilde{x}_j^{G_0}) ,\ldots,u(x_{N}),\tilde{u}_{\epsilon,1 }^{G_{1}},\ldots ,
\tilde{u}_{\epsilon,J }^{G_{K}})^{\top }$ as in \eqref{longvectorue}. Here, $u(\tilde{x}_j^{G_0})$ replaces $u(\gamma_j(h))$ in the first $N-$terms of $\vec{u}$ as in \eqref{truevectoru}.

}

\begin{lem} (Consistency of the extrapolation formula)
\label{prop:exterpo_dis4} Under the assumptions of the Proposition~\ref{prop:extrapolationofu}, then for each boundary point $x_j^B\in\partial M$, the truncation error for the first equation in the extrapolation formula \eqref{Eq:Uvvv} is given by,
\begin{equation}
\left|\sum_{j',k=1}^{J,K}\tilde{D}_{B_j,\left( N+(j'-1)K+k\right)}\left(u(\tilde{x}^{G_{k}}_{j'}) -\tilde{u}_{\epsilon,j' }^{G_{k}}\right)\right|
{=\epsilon\left\vert \left(\mathbf{L}^h  \tilde{u}\right)_{B_j}-\left( \mathbf{L}^h\vec{u}_{\epsilon
}\right) _{B_j}\right\vert
=}\mathcal{O}%
\left( \epsilon \bigg(h^2\epsilon^{-3/2},\epsilon,
\bar{N}^{-1/2}\epsilon ^{-(2+d/4)},\bar{N}^{-1/2}\epsilon ^{-(1/2+d/4)}%
\bigg)\right) ,  \label{Eqn:DG4}
\end{equation}%
in high probability, where $h$ depends on $\epsilon$ such that the first two terms vanishes as $\epsilon\to 0$ after $\bar{N}\to\infty$. {\color{black}Here the sub-index $B_j$ corresponds to the boundary point $x^{B}_j$. For the last three equations in \eqref{Eq:Uvvv}, we have,
\BEA
\left|\left( u(\tilde{x}^{G_{2}}_j)-3u(\tilde{x}^{G_{1}}_j)\right) -\left(\tilde{u}%
_{\epsilon,j }^{G_{2}}-3\tilde{u}_{\epsilon,j }^{G_{1}}\right)\right| &=&%
\mathcal{O}\left( h^{3}\right) ,  \notag \\
\left|\left( u(\tilde{x}^{G_{3}}_j)-3u(\tilde{x}^{G_{2}}_j)+3u(\tilde{x}^{G_{1}}_j)\right) -\left(\tilde{u}%
_{\epsilon,j }^{G_{3}}-3\tilde{u}_{\epsilon,j }^{G_{2}}+3\tilde{u}_{\epsilon,j
}^{G_{1}})\right)\right| &=& \mathcal{O}
\left( h^{3}\right) ,  \label{Eqn:ext_f} \\
\left|\left( u(\tilde{x}^{G_{k}}_j)-3u(\tilde{x}^{G_{k-1}}_j)+3u(\tilde{x}^{G_{k-2}}_j)-u(\tilde{x}^{G_{k-3}}_j)\right) -\left(%
\tilde{u}_{\epsilon,j }^{G_{k}}-3\tilde{u}_{\epsilon,j }^{G_{k-1}}+3\tilde{u}%
_{\epsilon,j }^{G_{k-2}}-\tilde{u}_{\epsilon,j }^{G_{k-3}}\right)\right| &=&\mathcal{O}\left( h^{3}\right),  \notag
\EEA%
for $k=4,\ldots, K$.}
\end{lem}

\begin{proof} First, let us proof \eqref{Eqn:DG4}. For this case, we only consider the $B_j$th row corresponding to the boundary point $x^{B}_j$. {\color{black}In the case of randomly sampled data, we have
\BEA
\left\vert (\mathbf{L}^h\tilde{u})_{B_j}-( \mathbf{L}^h\vec{u}_{\epsilon
}) _{B_j}\right\vert  &\leq& \left\vert (\mathbf{L}^h\tilde{u})_{B_j}-( \mathbf{L}\vec{u}) _{B_j}\right\vert  + \left\vert (\mathbf{L}\vec{u})_{B_j}-( \mathbf{L}^h\vec{u}_{\epsilon
}) _{B_j}\right\vert \nonumber \\
&=&  \left\vert (\mathbf{L}^h\tilde{u})_{B_j}-( \mathbf{L}\vec{u}) _{B_j}\right\vert + \left\vert (\mathbf{L}\vec{u})_{B_j}- \mathcal{L}u(x_{j}^B)\right\vert  \nonumber \\
&=&
\mathcal{O}(h^2\epsilon^{-3/2}) + \mathcal{O}\left( \epsilon,
\bar{N}^{-1/2}\epsilon ^{-(2+d/4)},\bar{N}^{-1/2}\epsilon ^{-(1/2+d/4)}\right),\label{temp1}
\EEA
where we have used the fact that, $\mathcal{L}u(x_{j}^B)=f(x_{j}^B) $ and $\left( \mathbf{L}^h\vec{u}%
_{\epsilon }\right) _{B_j}=f(x_{j}^B)$, which is the first equation of the extrapolation formula  in \eqref{Eq:Uvvv}, in deducing the second line above. To obtain the third line, we directly used \eqref{errorofLh1compact} for the first bound and Lemma~\ref{lem:old} for the second error bound,
where we have suppressed the dependence on $q(x_j^B), \nabla u(x_j^B), \nabla (\kappa^{1/2}(x_j^B)u(x_j^B))$ in \eqref{Eqn:err_nobound} and \eqref{Eqn:err_L2} to simplify the discussion.

For the well-sampled data, based on the discussion in Remark~\ref{rem:ws}, the first error term in \eqref{temp1} is not applicable and we treat $\mathbf{L}^h$ as $\mathbf{L}$.} Since the first $N$ components of $\tilde{u}_i$ (see \eqref{utildevec}) are equal to the components of $\vec{u}_\epsilon^M$ defined in \eqref{vectoru} and the identity $%
\mathbf{I}$ only contributes to the coefficient of $u(x_{j}^B)=\tilde{u}_{\epsilon,j
}^{B_j} $ for the boundary point $x^{B}_j$, we can simplify the left-hand-side of \eqref{temp1} as,
\begin{eqnarray}
\left\vert \left(\mathbf{L}^h  {\tilde{u}}\right)_{B_j}-\left( \mathbf{L}^h\vec{u}_{\epsilon
}\right) _{B_j}\right\vert =\frac{1}{\epsilon }\left|\left( \mathbf{\tilde{D}}%
{\tilde{u}} \right)_{B_j}-\left( \mathbf{\tilde{D}}\vec{u}_{\epsilon }\right)
_{B_j}\right| =\frac{1}{\epsilon }\left|\sum_{j',k=1}^{J,K}\tilde{D}_{B_j,\left( N+(j'-1)K+k\right)}\left(u(\tilde{x}^{G_{k}}_{j'}) -\tilde{u}_{\epsilon,j' }^{G_{k}}\right)\right|.  \label{Eqn:Lui}
\end{eqnarray}%
Thus, from \eqref{temp1} and \eqref{Eqn:Lui}, we obtain the result in %
\eqref{Eqn:DG4}.

{\color{black}The proof for \eqref{Eqn:ext_f} is straightforward. In particular, for each $j=1,\ldots, J$, note that $\{\tilde{x}^{G_0}_j,x^B,\tilde{x}^{G_1}_j,\ldots,\tilde{x}^{G_K}_j \}$ are points that lie on a straight line in the direction of $\boldsymbol{\tilde{\nu}}$ where the distances between the consecutive points are identical, namely, $h$. For $u\in C^3(M\cup B_{\epsilon^r}(\partial M))$, where $B_{\epsilon^r}(\partial M)\supset \Delta M$ is as in Definition~\ref{epsball}, then one can can deduce  \eqref{Eqn:ext_f} by employing the standard Taylor's expansion on the interval $[\tilde{x}_j^{G_0},\tilde{x}_j^{G_K}]$.}

\end{proof}

{\bf Proof of Proposition~\ref{prop:extrapolationofu}.} Notice that we can write \eqref{Eqn:DG4} and \eqref{Eqn:ext_f} in Proposition \ref{prop:exterpo_dis4} in a matrix form,%
\begin{equation}
\mathbf{E}\delta \vec{u}_{\epsilon }^{G}=\mathcal{O}\left( h
^{3},h^2\epsilon^{-1/2},\epsilon^{2},\bar{N}^{-1/2}\epsilon ^{-(1+d/4)},\bar{N}^{-1/2}\epsilon ^{(1/2-d/4)}\right),\label{Eqn:uueps}
\end{equation}%
where $\delta \vec{u}_{\epsilon }^{G}=(|\tilde{u}_{\epsilon,1
}^{G_{1}}-u(\tilde{x}^{G_{1}}_1)|,\ldots ,|\tilde{u}_{\epsilon,J
}^{G_{K}}-u(\tilde{x}^{G_{K}}_J)|)^{\top }$ and the matrix $\mathbf{E}$ is of size $JK\times JK$. We first show the stability of $\mathbf{E}$, namely $\mathbf{E}$ is invertible with uniformly bounded inverse. For simplicity of the discussion, we set $J=1$ corresponds to a boundary point. One can use the same idea for the case of $J>1$.

In this case, the matrix $\mathbf{E}$ in \eqref{Eqn:uueps} is given by
\begin{equation*}
\mathbf{E}=\begin{pmatrix}
\tilde{D}_{B_1,\left( N+1\right) } & \tilde{D}_{B_1,\left( N+2\right) } & \tilde{D}%
_{B_1,\left( N+3\right) } & \tilde{D}_{B_1,\left( N+4\right) } & \cdots  & \tilde{D%
}_{B_1,\left( N+K\right) }  \\
-3 & 1 &  &  &  & 0 \\
3 & -3 & 1 &  &  & \vdots  \\
-1 & 3 & -3 & 1 &  & \vdots  \\
\vdots  &  &  &  & \ddots  & 0 \\
0 & \cdots  & -1 & 3 & -3 & 1%
\end{pmatrix}.
\end{equation*}%

We can obtain the uniform error between $\vec{u}$ and $\vec{u}_{\epsilon }$\
in \eqref{Eqn:uueps} once showing that $\left\Vert \mathbf{E}%
^{-1}\right\Vert _{\infty }<C$. We have the following decomposition for
matrix $\mathbf{E}$, $\mathbf{E=E}_{0}+\mathbf{v}_{1}\mathbf{v}_{2}^{\top }$,%
\begin{equation*}
\mathbf{E}_{0}=\begin{pmatrix}
1 & 0 & 0 & \cdots  & \cdots  & 0 \\
-3 & 1 &  &  &  & 0 \\
3 & -3 & 1 &  &  & \vdots  \\
-1 & 3 & -3 & 1 &  & \vdots  \\
\vdots  &  &  &  & \ddots  & 0 \\
0 & \cdots  & -1 & 3 & -3 & 1%
\end{pmatrix} ,\quad \mathbf{v}_{1}=\begin{pmatrix}
1 \\
0 \\
\vdots  \\
\vdots  \\
0%
\end{pmatrix},\quad \mathbf{v}_{2}=\begin{pmatrix}
\tilde{D}_{B_1,\left( N+1\right) }-1 \\
\tilde{D}_{B_1,\left( N+2\right) } \\
\tilde{D}_{B_1,\left( N+3\right) } \\
\vdots  \\
\tilde{D}_{B_1,\left( N+K\right) }%
\end{pmatrix}.
\end{equation*}%
By induction, one can show that
\begin{equation*}
\mathbf{E}_{0}^{-1}=\begin{pmatrix}
1 & 0 & 0 & \cdots  & \cdots  & 0 \\
3 & 1 &  &  &  & 0 \\
6 & 3 & 1 &  &  & \vdots  \\
10 & 6 & 3 & 1 &  & \vdots  \\
\vdots  & 10 & \ddots  & \ddots  & \ddots  & 0 \\
K(K+1)/2 & \cdots  & 10 & 6 & 3 & 1%
\end{pmatrix},
\end{equation*}%
so that $\left\Vert \mathbf{E}_{0}^{-1}\right\Vert _{\infty }=K(K+1)(K+2)/6<C
$ by noticing that $K$ is always fixed to be less than 10 even when $%
N\rightarrow \infty $. One can calculate that $\mathbf{E}_{0}^{-1}\mathbf{v}%
_{1}=$ $(1,3,6,\ldots ,K(K+1)/2)^{\top }$\ and $1+\mathbf{v}_{2}^{\top }%
\mathbf{E}_{0}^{-1}\mathbf{v}_{1}=\sum_{k=1}^{K}k(k+1)\tilde{D}_{B_1,\left(
N+k\right) }/2$ which is nonzero. Thus, according to Sherman-Morrison
formula, we have%
\begin{equation*}
\left\Vert \mathbf{E}^{-1}\right\Vert _{\infty }=\left\Vert \left( \mathbf{E}%
_{0}+\mathbf{v}_{1}\mathbf{v}_{2}^{\top }\right) ^{-1}\right\Vert _{\infty
}=\left\Vert \left( \mathbf{I}-\frac{\mathbf{E}_{0}^{-1}\mathbf{v}_{1}%
\mathbf{v}_{2}^{\top }}{1+\mathbf{v}_{2}^{\top }\mathbf{E}_{0}^{-1}\mathbf{v}%
_{1}}\right) \mathbf{E}_{0}^{-1}\right\Vert _{\infty }<C.
\end{equation*}

{\color{black}Inverting $\mathbf{E}$ in \eqref{Eqn:uueps}, for each $j=1,\ldots, J, k=1,\ldots, K$, we have,
\BEA
\Big|\tilde{u}_{\epsilon,1}^{G_{k}}-u(\tilde{x}^{G_{k}}_j)\Big| = \Big(\delta \vec{u}_{\epsilon }^{G}\Big)_{j,k} = \mathcal{O}\left( h
^{3},h^2\epsilon^{-1/2},\epsilon^{2},\bar{N}^{-1/2}\epsilon ^{-(1+d/4)},\bar{N}^{-1/2}\epsilon ^{(1/2-d/4)}\right),
\EEA
and the proof is complete by comparing these error rates with $|u(x^{G_k}_j) - u(\tilde{x}^{G_k}_j)|= \mathcal{O}(h\sqrt{\epsilon})$. We will neglect this error term since it is equivalent to the error rate of order-$h^2\epsilon^{-1/2}$, anticipating that $h=\mathcal{O}(\epsilon)$. In Proposition~\ref{prop:extrapolationofu}, we also neglect the order-$\epsilon^2$ error term since it is dominated by the first two error bounds in \eqref{Eqn:uueps}.}

\section{Proof of Theorem~\ref{theorem2}} \label{app:D}

The proof here follows the standard approach for proving the convergence of the finite-difference method presented in many numerical PDE texts (e.g., see \cite{larsson2008partial}). That is, we will show that $\mathbf{L}^I$ in \eqref{Dirichletsystem} satisfies a discrete maximum principle. Subsequently, a comparison function is chosen using the maximum principle of the Dirichlet PDE problem to establish the stability condition. The convergence is achieved with the consistency of the GPDM estimator in Theorem~\ref{theorem1}. Before we proceed with these steps, let us first analyze the resulting GPDM estimator, $\mathbf{L}^g(\vec{u}^M):=(\mathbf{L}^{(1)} + \mathbf{L}^{(2)}\mathbf{A})\vec{u}^M + \mathbf{L}^{(2)}\vec{b}$, as defined in \eqref{GPDM}.

To simplify the discussion, we present the 1D case with $J=2$ boundary points, denoted by $x_{1}^B = x_1$ and $x_{2}^B = x_N$ (see Fig.~\ref{fig22_1dcurve}(a)). The corresponding ghost points are $\tilde{x}^{G_0}_1=x_2$ and $\tilde{x}^{G_0}_2=x_{N-1}$, using the secant line approximation. Otherwise, the same analysis can be carried by relabelling $\tilde{x}^{G_0}_j$ by other arbitrary $x_i$.
%Replacing $-\tilde{d}_{1}\frac{\Delta u}{\Delta \boldsymbol{\tilde{\nu}}_{1}}(x_{1}^B) = u(x_2) - u(x_1)$ and $-\tilde{d}_{2}\frac{\Delta u}{\Delta \boldsymbol{\tilde{\nu}}_{N}}(x_{2}^B) = u(x_{N-1}) - u(x_N)$ in the second equation of \eqref{Eq:Uvvv} for $j=1,2$, we can write
The last three equations in \eqref{Eq:Uvvv} can be written as,
\BEA
\begin{aligned}\label{UGK}
\tilde{u}^{G_k}_{\epsilon,1} &= \frac{k(k+1)}{2} \tilde{u}^{G_1}_{\epsilon,1} - (k^2-1) u_1 + \frac{k(k-1)}{2}u_2,  \\
\tilde{u}^{G_k}_{\epsilon,2} &= \frac{k(k+1)}{2} \tilde{u}^{G_1}_{\epsilon,2} - (k^2-1) u_N + \frac{k(k-1)}{2}u_{N-1},
\end{aligned}\quad k=2,\ldots, K.
\EEA
Using the same notation as in \eqref{Ltildeinlemma}, we let $\mathbf{L}^h=(\mathbf{\tilde{D}}-%
\mathbf{I})/\epsilon :=((\mathbf{D}^h)^{-1}\mathbf{K}^h-\mathbf{I})/\epsilon $ be
the $N\times\bar{N}$ matrix, where $\bar{N}=N+2K$, obtained from the standard diffusion maps as an discrete approximation to one of the diffusion operators in \eqref{L1}-\eqref{Eqn:L3} with the entries of $\mathbf{%
\tilde{D}}$ denoted by $\tilde{D}_{i,j}$. To be consistent with the notation in Section~\ref{artBC_ghost}, we emphasize that
$\mathbf{L}^h \in \mathbb{R}^{N\times\bar{N}}$ is a non-square matrix with $\bar{N}=N+2K$, where the $N$-rows correspond to the kernel evaluation at the points $\{x_i\in M\}_{i=1}^N$.

Then, the $i$th component of $(\mathbf{\tilde{D}}-\mathbf{I})\vec{u}_\epsilon$ is given by,
\BEA
\sum_{j=1}^{N+2K} \tilde{D}_{i,j} u_j - u_i &=& \sum_{j=3}^{N-2} \tilde{D}_{i,j} u_j - u_i + \tilde{D}_{i,1} u_1 + \tilde{D}_{i,2} u_2 + \tilde{D}_{i,N-1} u_{N-1} + \tilde{D}_{i,N} u_N + \sum_{j,k=1}^{2,K} \tilde{D}_{i,N+(j-1)K+k}\tilde{u}^{G_k}_{\epsilon,j}\nonumber \\
&=& \sum_{j=3}^{N-2} \tilde{D}_{i,j} u_j - u_i + \left( \tilde{D}_{i,1} - \sum_{k=2}^{K}(k^2-1)\tilde{D}_{i,N+k} \right) u_1 + \left(\tilde{D}_{i,2}+\sum_{k=2}^{K}\frac{k(k-1)}{2}\tilde{D}_{i,N+k} \right) u_2 \nonumber\\
&& + \left(\tilde{D}_{i,N-1}+\sum_{k=2}^{K}\frac{k(k-1)}{2}\tilde{D}_{i,N+K+k} \right) u_{N-1} + \left( \tilde{D}_{i,N} - \sum_{k=2}^{K}(k^2-1)\tilde{D}_{i,N+K+k} \right) u_N \nonumber \\ && +\left(\sum_{k=1}^K \frac{k(k+1)}{2}\tilde{D}_{i,N+k}\right) \tilde{u}^{G_1}_{\epsilon,1} + \left(\sum_{k=1}^K \frac{k(k+1)}{2}\tilde{D}_{i,N+K+k}\right) \tilde{u}^{G_1}_{\epsilon,2}\nonumber \\
&=& \sum_{j=3}^{N-2} \tilde{D}_{i,j} u_j - u_i +c_{i,1} u_1 + c_{i,2} u_2 + c_{i,N-1} u_{N-1} +c_{i,N} u_N + c_{i,0} \tilde{u}^{G_1}_{\epsilon,1} + c_{i,N+1} \tilde{u}^{G_1}_{\epsilon,2},
\label{rowsum}
\EEA
where we have defined,
\BEA\notag%\label{cis}
\begin{aligned}
c_{i,1} &= \tilde{D}_{i,1} - \sum_{k=2}^{K}(k^2-1)\tilde{D}_{i,N+k},\quad\quad \quad\quad\quad
c_{i,2} = \tilde{D}_{i,2}+\sum_{k=2}^{K}\frac{k(k-1)}{2}\tilde{D}_{i,N+k},\\
c_{i,N-1}&= \tilde{D}_{i,N-1}+\sum_{k=2}^{K}\frac{k(k-1)}{2}\tilde{D}_{i,N+K+k}, \quad\quad
c_{i,N} =  \tilde{D}_{i,N} - \sum_{k=2}^{K}(k^2-1)\tilde{D}_{i,N+K+k},  \\
c_{i,0} &= \sum_{k=1}^K \frac{k(k+1)}{2}\tilde{D}_{i,N+k},\quad\quad \quad\quad\quad\quad\,\,c_{i,N+1} = \sum_{k=1}^K \frac{k(k+1)}{2}\tilde{D}_{i,N+K+k},
\end{aligned}
\EEA
for convenience. From the first equation in \eqref{Eq:Uvvv}, we have,
\BEA
\begin{aligned}\label{FS}
\sum_{j=3}^{N-2} \tilde{D}_{1,j} u_j - u_1 +c_{1,1} u_1 + c_{1,2} u_2 + c_{1,N-1} u_{N-1} +c_{1,N} u_N + c_{1,0} \tilde{u}^{G_1}_{\epsilon,1} + c_{1,N+1} \tilde{u}^{G_1}_{\epsilon,2}&= \epsilon f(x_1),\\
\sum_{j=3}^{N-2} \tilde{D}_{N,j} u_j - u_N +c_{N,1} u_1 + c_{N,2} u_2 + c_{N,N-1} u_{N-1} +c_{N,N} u_N + c_{N,0} \tilde{u}^{G_1}_{\epsilon,1} + c_{N,N+1} \tilde{u}^{G_1}_{\epsilon,2}&= \epsilon f(x_N).
\end{aligned}
\EEA

Since $c_{1,N+1}=c_{N,0 }\approx 0$, we obtain,
\BEA
\begin{aligned}\label{UG1}
\tilde{u}^{G_1}_{\epsilon,1} &= \frac{1}{c_{1,0}} (\epsilon f(x_1) - \sum_{j=3}^{N-2} \tilde{D}_{1,j} u_j +(1-c_{1,1}) u_1 - c_{1,2} u_2 - c_{1,N-1} u_{N-1} - c_{1,N} u_N), \\
\tilde{u}^{G_1}_{\epsilon,2} &=  \frac{1}{c_{N,N+1}} (\epsilon f(x_N) - \sum_{j=3}^{N-2} \tilde{D}_{N,j} u_j - c_{N,1} u_1 - c_{N,2} u_2 - c_{N,N-1} u_{N-1} +(1-c_{N,N}) u_N).
\end{aligned}
\EEA
We should point out that Eqs.~\eqref{UG1} and \eqref{UGK} are components of \eqref{vecUGK}. Therefore, the $i$th row in \eqref{rowsum} becomes,
\BEA
\sum_{j=1}^{N+2K} \tilde{D}_{i,j} u_j - u_i &=& \sum_{j=3}^{N-2} \tilde{D}_{i,j} u_j - u_i +c_{i,1} u_1 + c_{i,2} u_2 + c_{i,N-1} u_{N-1} +c_{i,N} u_N \nonumber \\&& + \frac{c_{i,0}}{c_{1,0}} (\epsilon f(x_1) - \sum_{j=3}^{N-2} \tilde{D}_{1,j} u_j +(1-c_{1,1}) u_1 - c_{1,2} u_2 - c_{1,N-1} u_{N-1} - c_{1,N} u_N) \nonumber \\ && + \frac{c_{i,N+1}}{c_{N,N+1}} (\epsilon f(x_N) - \sum_{j=3}^{N-2} \tilde{D}_{N,j} u_j - c_{N,1} u_1 - c_{N,2} u_2 - c_{N,N-1} u_{N-1} +(1-c_{N,N}) u_N) ) \nonumber\\
&=&  \sum_{j=3}^{N-2} \left(\tilde{D}_{i,j}- \frac{c_{i,0}}{c_{1,0}} \tilde{D}_{1,j} -\frac{c_{i,N+1}}{c_{N,N+1}} \tilde{D}_{N,j}   \right) u_j - u_i + \left(c_{i,1} + \frac{c_{i,0}}{c_{1,0}}(1-c_{1,1}) - \frac{c_{i,N+1}}{c_{N,N+1}}c_{N,1} \right)u_1 \nonumber \\ &+&\left( c_{i,2}- \frac{c_{i,0}}{c_{1,0}}c_{1,2}- \frac{c_{i,N+1}}{c_{N,N+1}}c_{N,2} \right) u_2 + \left( c_{i,N-1}- \frac{c_{i,0}}{c_{1,0}}c_{1,N-1}- \frac{c_{i,N+1}}{c_{N,N+1}}c_{N,N-1} \right) u_{N-1} \nonumber \\
&& + \left(c_{i,N} - \frac{c_{i,0}}{c_{1,0}}c_{1,N} + \frac{c_{i,N+1}}{c_{N,N+1}}(1-c_{N,N}) \right)u_N
+ \epsilon\left( \frac{c_{i,0}}{c_{1,0}} f(x_1) + \frac{c_{i,N+1}}{c_{N,N+1}} f(x_N)\right). \label{Ltilde}
\EEA

It is clear that $0<\frac{c_{i,0}}{c_{1,0}},\frac{c_{i,N+1}}{c_{N,N+1}}<1$, for all, $i=2,\ldots, N-1$. Also, $\tilde{D}_{i,j}>\tilde{D}_{1,j}$ and $\tilde{D}_{i,j}>\tilde{D}_{N,j}$ for $i=2,\ldots,N-1$ and $j=3,\dots,N-2$. This implies,
\BEA
\tilde{D}_{i,j}- \frac{c_{i,0}}{c_{1,0}} \tilde{D}_{1,j} -\frac{c_{i,N+1}}{c_{N,N+1}} \tilde{D}_{N,j}  > \tilde{D}_{i,j} (1-  \frac{c_{i,0}}{c_{1,0}} -\frac{c_{i,N+1}}{c_{N,N+1}})>0.\nonumber
\EEA
In fact, since $c_{i,2}>c_{1,2}$ and $c_{i,2}>c_{N,2}$ for $i=2,\ldots, N-1$, it is clear that
\BEA
c_{i,2}- \frac{c_{i,0}}{c_{1,0}}c_{1,2} -\frac{c_{i,N+1}}{c_{N,N+1}} c_{N,2}  > c_{i,2} (1-  \frac{c_{i,0}}{c_{1,0}} -\frac{c_{i,N+1}}{c_{N,N+1}})>0.\nonumber
\EEA
Likewise, we have
\BEA
c_{i,N-1}- \frac{c_{i,0}}{c_{1,0}}c_{1,N-1} -\frac{c_{i,N+1}}{c_{N,N+1}} c_{N,N-1}  > c_{i,N-1} (1-  \frac{c_{i,0}}{c_{1,0}} -\frac{c_{i,N+1}}{c_{N,N+1}})>0.\nonumber
\EEA
The coefficients on the boundary points,
\BEA
c_{i,1} + \frac{c_{i,0}}{c_{1,0}}(1-c_{1,1})  - \frac{c_{i,N+1}}{c_{N,N+1}}c_{N,1}  > c_{i,1}  (1- \frac{c_{i,N+1}}{c_{N,N+1}}) >0.\nonumber\\
c_{i,N} - \frac{c_{i,0}}{c_{1,0}}c_{1,N} + \frac{c_{i,N+1}}{c_{N,N+1}}(1-c_{N,N})   > c_{i,N}  (1- \frac{c_{i,0}}{c_{1,0}}) >0\nonumber
\EEA
are also strictly positive. Thus, all of the nondiagonal coefficients of \eqref{Ltilde} are strictly positive.

We should point out that the expression on the right-hand-side of \eqref{Ltilde} is nothing but the $i$th row of the affine operator in \eqref{GPDM}, that is,
\BEA
\sum_{j=1}^{N+2K} \tilde{D}_{i,j} u_j - u_i = \epsilon\left((\mathbf{L}^{(1)} + \mathbf{L}^{(2)}\mathbf{A})\vec{u}^M + \mathbf{L}^{(2)}\vec{b}\right)_i. \notag%\label{dww}
\EEA
Let us denote $\mathbf{M}= \epsilon(\mathbf{L}^{(1)} + \mathbf{L}^{(2)}\mathbf{A})$.  Notice that if $u_i=1$ for all $i=1,\ldots,N+2K$, then from \eqref{FS} and the fact that $\sum_{j=1}^{N+2K}\tilde{D}_{i,j} = 1$, one can verify that $f(x_1)=f(x_N)=0$, which means $(\mathbf{L}^{(2)}\vec{b})_i=0$. Evaluating \eqref{Ltilde} at $u_i=1$, one can see that,
\BEA
0 &=& \sum_{j=1}^{N+2K} \tilde{D}_{i,j} - 1 = \sum_{\stackrel{j=3}{j\neq i}}^{N-2} \left(\tilde{D}_{i,j}- \frac{c_{i,0}}{c_{1,0}} \tilde{D}_{1,j} -\frac{c_{i,N+1}}{c_{N,N+1}} \tilde{D}_{N,j}   \right) + \left(\left(\tilde{D}_{i,i}- \frac{c_{i,0}}{c_{1,0}} \tilde{D}_{1,i} -\frac{c_{i,N+1}}{c_{N,N+1}} \tilde{D}_{N,i}   \right)- 1\right) \nonumber\\ && +  \left(c_{i,1}+\frac{c_{i,0}}{c_{1,0}}(1-c_{1,1}) - \frac{c_{i,N+1}}{c_{N,N+1}}c_{N,1}\right)    +\left( c_{i,2}- \frac{c_{i,0}}{c_{1,0}}c_{1,2}- \frac{c_{i,N+1}}{c_{N,N+1}}c_{N,2} \right)  \nonumber\\ &&+ \left( c_{i,N-1}- \frac{c_{i,0}}{c_{1,0}}c_{1,N-1}- \frac{c_{i,N+1}}{c_{N,N+1}}c_{N,N-1} \right) + \left(c_{i,N} - \frac{c_{i,0}}{c_{1,0}}c_{1,N} + \frac{c_{i,N+1}}{c_{N,N+1}}(1-c_{N,N}) \right) \nonumber\\
&=&
\sum_{\stackrel{j=3}{j\neq i}}^{N-2} \mathbf{M}_{i,j} + \mathbf{M}_{i,i} + \mathbf{M}_{i,1}+\mathbf{M}_{i,2}+ \mathbf{M}_{i,N-1}+\mathbf{M}_{i,N},\label{sumzero}
\EEA
where $\mathbf{M}_{i,i}<0$ and $\mathbf{M}_{i,j}>0$ for all $j\neq i$ are defined as in the brackets in the previous equality, respectively.

\noindent {\bf Discrete Maximum Principle:} Suppose $\vec{v}=(v(x_2),\ldots,v(x_{N-1})$ is such that $\mathbf{L}^I\vec{v}>0$. Suppose the maximum occurs at the interior point $x_i$, that is $v(x_i)\geq v(x_j)$ for all $j\neq i$. Then,
\BEA
-\mathbf{M}_{i,i} v(x_i) = \sum_{\stackrel{j=2}{j\neq i}}^{N-1} \mathbf{M}_{i,j}v(x_j) - \epsilon(\mathbf{L}^I\vec{v})_i \leq \sum_{\stackrel{j=2}{j\neq i}}^{N-1} \mathbf{M}_{i,j}v(x_j) \leq \left(\sum_{\stackrel{j=2}{j\neq i}}^{N-1} \mathbf{M}_{i,j}\right)v(x_i).\label{contra}
\EEA
Here, we use the fact that the matrix $\epsilon\mathbf{L}^I$ (as defined in \eqref{defnL}) is nothing but the submatrix of $\mathbf{M}$, ignoring the first and $N$th columns.
From \eqref{sumzero}, $-\mathbf{M}_{i,i} =  \sum_{\stackrel{j=1}{j\neq i}}^{N} \mathbf{M}_{i,j} > \sum_{\stackrel{j=2}{j\neq i}}^{N-1} \mathbf{M}_{i,j}$, which contradicts \eqref{contra} so $v$ cannot attain the maximum at $x_i$. Repeating the same argument on all interior points, it is clear that the maximum has to occur at the boundary. That is,
\BEA
\max_{1 \leq j\leq N} v(x_j)  = \{v(x_1),v(x_N)\}.\label{discretemax}
\EEA
Using the same argument, one can also show that the minimum occurs at the boundaries.

\noindent{\bf Stability:}  By assumption, the PDE satisfies a maximum principle. Consider $v\in C^{2}(M)$ that solves $\mathcal{L}v(x)=C$ for all $x\in M^o$, $v(x)\vert_{x\in\partial M}=0$, and a constant $C>0$ to be determined.  Here, the existence of the unique solution $v$ follows from the well-posedness assumption of the Dirichlet problem. By the maximum principle, it is clear that $v(x)\leq 0$. Also, since $M$ is compact, it attains the global minima on $M$. Define $v_s(x) := v(x)-v_{min}$, where $v_{min}=\min_{x\in M}{v(x)}\leq 0$. Thus it is clear that $0 \leq v_s(x)\leq C_2=|v_{min}|$ solves $\mathcal{L}v_s=C$ and $v_s(x)\vert_{x\in\partial M}=C_2$. In this case, since GPDM is consistent (see Theorem~\ref{theorem1}), it is clear that for the column vector, $\vec{v}_s^M:=(v_s(x_1),\vec{v}_s^I,v_s(x_N))\in \mathbb{R}^N$,
where $\vec{v}_s^I:=(v_s(x_2),\ldots, v_s(x_{N-1}))$, we have $\big|\big(\mathbf{L}^g(\vec{v}_{s}^M)\big)_i -\mathcal{L}v_s(x_i)\big| \leq c_1 \delta $, where $\delta := \max\{h^3\epsilon^{-1},h^2\epsilon^{-3/2}\}$.
Notice that,
\BEA
\left| \mathcal{L}v_s(x_i) - \big(\mathbf{L}^g (\vec{v}_{s}^M)\big)_i\right| &=&\left| \mathcal{L}v_s(x_i) -  \left((\mathbf{L}^{(1)} + \mathbf{L}^{(2)}\mathbf{A})\vec{v}^M_s + \mathbf{L}^{(2)}\vec{b}\right)_i \right| = \left| \mathcal{L}v_s(x_i) -(\mathbf{L}^{(2)}\vec{b})_i- (\mathbf{L}^B \vec{g} + \mathbf{L}^I \vec{v}^I_s)_i  \right|,   \label{equivalentrep}
 \EEA
 where we have used the decomposition in \eqref{defnL} and the affine estimator \eqref{GPDM}. This means,
 \BEA
( \mathbf{L}^I\vec{v}^I_s)_i  \geq  C - c_1\delta - (\mathbf{L}^{(2)}\vec{b})_i- (\mathbf{L}^B \vec{g})_i.\nonumber
 \EEA
Choose  $C=2 + \|\mathbf{L}^{(2)}\vec{b}\|_\infty + \|\mathbf{L}^B\vec{g}\|_\infty$, we obtain
 \BEA
(\mathbf{L}^I\vec{v}^I_s)_i  &\geq& 2 - c_1\delta+\left(\|\mathbf{L}^{(2)}\vec{b}\|_\infty -(\mathbf{L}^{(2)}\vec{b})_i\right) + \left(\|\mathbf{L}^B\vec{g}\|_\infty -(\mathbf{L}^B \vec{g})_i \right) \geq 2- c_1 \delta \geq 0.\nonumber
\EEA

Basically $0 \leq v_s(x_i) \leq C_2$ is a comparison function that we have identified for proving the stability of the solution. Let $M = \|\vec{f}^I - \mathbf{L}^B \vec{g} \|_{\infty}$ be the maximum of the right-hand-side in \eqref{Dirichletsystem}, then for $\hat{u}^I$ that solves  \eqref{Dirichletsystem}, we have,
\BEA
\mathbf{L}^I (\hat{u}^I+ M \vec{v}^I_s) \geq \vec{f}^I- \mathbf{L}^B \vec{g} + (2-c_1 \delta) M \geq 0,\nonumber
\EEA
for small enough $\delta$, which depends on $h$ and fixed $0<\epsilon\ll 1$.
By the discrete maximum principle in \eqref{discretemax}, it is clear that,
\BEA
\max_{x_i\in M} \hat{u}^I \leq \max_{x_i\in M} (\hat{u}^I+ M \vec{v}^I_s)\leq \max_{x_i\in\partial M}\hat{u}^B + \max_{x_i\in \partial M} M \vec{v}^I_s\leq \|\hat{u}^B\|_\infty + C_2 \|\vec{f}^I - \mathbf{L}^B \vec{g}\|_\infty. \nonumber
\EEA
Using similar argument on $-\hat{u}^I$, we obtain the stability of the approximate solution
\BEA
\|\hat{u}^I\|_\infty \leq \|\hat{u}^B\|_\infty + C_2 \|\vec{f}^I - \mathbf{L}^B \vec{g}\|_\infty.\label{stability}
\EEA

\noindent {\bf Convergence:} Applying \eqref{stability} on $\hat{u}^I - \vec{u}^I$, where components of $\vec{u}^I$ are the true solution of the PDE in \eqref{PDE} with Dirichlet boundary condition, we obtain,
 \BEA
\|\hat{u}^I-\vec{u}^I\|_\infty \leq \|\hat{u}^B-\vec{u}^B \|_\infty + C_2 \|\vec{f}^I - \mathbf{L}^B \vec{g}- \mathbf{L}^I \vec{u}^B\|_\infty. \label{convergence1}
\EEA
Using the same argument as in \eqref{equivalentrep} and the error bound in Theorem~\ref{theorem1}, we immediately see the consistency of the estimator, that is,
\BEA
\left |f(x_i) -(\mathbf{L}^{(2)}\vec{b})_i- (\mathbf{L}^B \vec{g} + \mathbf{L}^I \vec{u}^I)_i  \right|
=\left| \mathcal{L}u(x_i) - \big(\mathbf{L}^g (\vec{u}^M)\big)_i\right| =  \mathcal{O}\left( h^3\epsilon ^{-1},h^2\epsilon^{-3/2},\bar{N}^{-1/2}\epsilon ^{-(2+d/4)},\bar{N}^{-1/2}\epsilon ^{-(1/2+d/4)} \right),\label{consistency1}
 \EEA
in high probability, where $h$ depends on $\epsilon$ such that the first two terms vanishes as $\epsilon\to 0$ after $\bar{N}\to\infty$.
Since $\vec{u}^B=\hat{u}^B=\vec{g}$, combining \eqref{convergence1} and \eqref{consistency1}, the proof is completed.

\section{An alternative method for estimating the normal derivatives}\label{App:A}

In this Appendix, we discuss a method for estimating normal derivatives at the boundary of 2D manifold which requires no specification of ghost points. This scheme is used for estimating the directional derivatives of Neumann or Robin boundary conditions used in the classical diffusion maps algorithm. Specifically, the normal derivatives are estimated as follows.

\begin{algm}
\label{algm:deri} Assume that $\boldsymbol{\nu }$\ is the exterior normal
direction to the boundary $\partial M$ and $\boldsymbol{\tilde{\nu}}$ is its
numerical estimate as defined in Section~\ref{sec312} at a boundary point $x^{B}\in\partial M$. Then, the
normal derivative $\partial _{\boldsymbol{\nu }}u$ at $x^{B}$\ is estimated as follows:
\begin{enumerate}
\item Find the "left" nearest neighbor $x^{L}$\ and "right" nearest neighbor $%
x^{R}$ for the boundary point $x^{B}\in\partial M$. Then, one can compute
the normalized vectors,%
\begin{equation*}
\boldsymbol{\tilde{\nu}}^{L}:=\frac{x^{L}-x^{B}}{\left\vert
x^{L}-x^{B}\right\vert }\text{ and }\boldsymbol{\tilde{\nu}}^{R}:=\frac{%
x^{R}-x^{B}}{\left\vert x^{R}-x^{B}\right\vert }.
\end{equation*}%
Here, $x^{L}$\ is the nearest point to $x^{B}$\ in the region such that the
angle between $\boldsymbol{\tilde{\nu}}^{L}\ $and $-\boldsymbol{\tilde{\nu}}%
\ $satisfying $\Theta \left( \boldsymbol{\tilde{\nu}}^{L},-\boldsymbol{%
\tilde{\nu}}\right) <\Theta _{0}$ (in our implementation, $\Theta _{0}=\pi
/4 $). Similarly, idea applies for $x^{R}$. Moreover, the "left" and "right" can be numerically distinguished by the
negative inner product $\left\langle \boldsymbol{\tilde{w}}^{L},\boldsymbol{%
\tilde{w}}^{R}\right\rangle <0$ where $\boldsymbol{\tilde{w}}^{L}$ and $%
\boldsymbol{\tilde{w}}^{R}$ are components orthogonal to $-\boldsymbol{%
\tilde{\nu}}$, that is, $\boldsymbol{\tilde{w}}^{L}=\boldsymbol{\tilde{\nu}}%
^{L}-(\boldsymbol{\tilde{\nu}}^{L}\cdot \boldsymbol{\tilde{\nu}})\left(
\boldsymbol{\tilde{\nu}}\right) $ and $\boldsymbol{\tilde{w}}^{R}=%
\boldsymbol{\tilde{\nu}}^{R}-(\boldsymbol{\tilde{\nu}}^{R}\cdot \boldsymbol{%
\tilde{\nu}}))(\boldsymbol{\tilde{\nu}})$.

\item Write $-\boldsymbol{\tilde{\nu}}$ as a linear combination of $%
\boldsymbol{\tilde{\nu}}^{L}$ and $\boldsymbol{\tilde{\nu}}^{R}$\ using the
linear regression,
\begin{equation}
-\boldsymbol{\tilde{\nu}}=\tilde{a}^{L}\boldsymbol{\tilde{\nu}}^{L}+\tilde{a}%
^{R}\boldsymbol{\tilde{\nu}}^{R},  \label{Eqn:lc}
\end{equation}%
where $\tilde{a}^{L}$ and $\tilde{a}^{R}$\ are the regression coefficients.

\item Estimate the normal derivative $-\partial _{\boldsymbol{%
\nu }}u$\ numerically using the difference method,%
\begin{equation}
\frac{\partial u}{\partial \left( -\boldsymbol{\nu }\right) }\left(
x^{B}\right) \approx \frac{\Delta u}{\Delta \left( -\boldsymbol{\tilde{\nu}}%
\right) }\left( x^{B}\right) :=\left[ \tilde{a}^{L}\frac{\Delta u}{\Delta
\boldsymbol{\tilde{\nu}}^{L}}+\tilde{a}^{R}\frac{\Delta u}{\Delta
\boldsymbol{\tilde{\nu}}^{R}}\right] \left( x^{B}\right) := \tilde{a}^{L}%
\frac{u\left( x^{L}\right) -u\left( x^{B}\right) }{\left\vert
x^{L}-x^{B}\right\vert }+\tilde{a}^{R}\frac{u\left( x^{R}\right) -u\left(
x^{B}\right) }{\left\vert x^{R}-x^{B}\right\vert },  \label{Eqn:udv}
\end{equation}%
where we have used Eq.~\eqref{Eqn:lc} and the fact that $\boldsymbol{\tilde{\nu}}$,
$\boldsymbol{\tilde{\nu}}^{L}$, and $\boldsymbol{\tilde{\nu}}^{R}$\ are all
unit vectors. Then, the normal derivative $\partial _{\boldsymbol{\nu }}u$
term in the boundary condition \eqref{PDE} in the following section can be numerically estimated
using Eq.~\eqref{Eqn:udv} for all points on the boundary.
\end{enumerate}
\end{algm}

Next, we provide the error rate for estimating the directional
derivative $\partial _{\boldsymbol{\nu }}u$ with Eq.~\eqref{Eqn:udv} at the
boundary points.

\begin{prop}
\label{prop:interpo} Let $u\in C^{3}\left(
M\right) $ be a smooth function on a 2D manifold $M$ with 1D boundary $%
\partial M$. Let $\{x_{1},\ldots, x_N\}\subset M$ be a set of data points, among which
some labeled points lie on the boundary $\partial M$. Let $x^{B}$ be a boundary point on the 1D smooth $\partial M\ $and $\boldsymbol{\nu }%
\ $be the unit exterior normal direction to the boundary $\partial M\ $at $%
x^{B}$. Let $x^{L}$\ and $x^{R}\in \{x_{1},\ldots,x_N\}$\ be the
"left"\ and "right" nearest neighbors, respectively, for the boundary point $%
x^{B}$.\ Then, the normal derivative $\partial _{\boldsymbol{\nu }}u$ at $%
x_{B}$ estimated by Eq.~\eqref{Eqn:udv} in Algorithm \ref{algm:deri}\
has an error rate of
\begin{equation*}
\left\vert \frac{\partial u}{\partial \boldsymbol{\nu }}\left( x^{B}\right) -%
\frac{\Delta u}{\Delta \boldsymbol{\tilde{\nu}}}\left( x^{B}\right)
\right\vert =\mathcal{O}(h ),
\end{equation*}%
where $h$ characterizes the distance of the neighboring points and $\epsilon $
characterizes the bandwidth of the kernel. The constant depends on the local curvature and the norm of the second-order derivative of $u$ (that is, $%
\left\vert \nabla _{i}\nabla _{j}u\left( x^{B}\right) \right\vert $ with $%
\nabla _{i}\nabla _{j}$ being the Hessian operator).
\end{prop}

\begin{proof}
The error has two parts, one from the
regression coefficients $\tilde{a}^{L}$\ and $\tilde{a}^{R}$, and the other
from estimation of the directional derivatives $\partial _{%
\boldsymbol{\nu }^{L}}u$ and $\partial _{\boldsymbol{\nu }^{R}}u$. First, we
estimate the error from the regression coefficients $\tilde{a}^{L}$ and $%
\tilde{a}^{R}$.\textit{\ }Let $\gamma _{L}(\ell )$ be a geodesic
parameterized with the arc length $\ell$, connecting the points $x^{B}$ and
its "left" nearest neighbor $x^{L}$\ such that $\gamma _{L}(0)=x^{B}$ and $%
\gamma _{L}(\ell )=x^{L}$. Define $\boldsymbol{\nu }^{L}:=\gamma
_{L}^{\prime }(0)\in T_{x^{B}}M$ as a unit tangent vector by noticing that $%
\left\vert \gamma_{L}^{\prime }(t)\right\vert \equiv 1$ for $0\leq t\leq \ell $\
due to the arc length parametrization. Following the proof in Proposition %
\ref{prop:diren}, we have the error estimate
\begin{equation*}
\left\vert \boldsymbol{\nu }^{L}-\boldsymbol{\tilde{\nu}}^{L}\right\vert =%
\mathcal{O}(h).
\end{equation*}%
Similarly, we can define the geodesic $\gamma _{R}(\ell )$ connecting $x^{B}$%
\ and $x^{R}$\ and the unit tangent vector $\boldsymbol{\nu }^{R}:=\gamma
_{R}^{\prime }(0)\in T_{x^{B}}M$. Then, we have the similar error estimate
\begin{equation*}
\left\vert \boldsymbol{\nu }^{R}-\boldsymbol{\tilde{\nu}}^{R}\right\vert =%
\mathcal{O}(h).
\end{equation*}%
Since $-\boldsymbol{\nu ,\nu }^{L},\boldsymbol{\nu }^{R}\in T_{x^{B}}M$ and $%
M$ is a 2D manifold, there exist unique coefficients $a^{L}$ and $a^{R}$ such
that
\begin{equation*}
-\boldsymbol{\nu }=a^{L}\boldsymbol{\nu }^{L}+a^{R}\boldsymbol{\nu }^{R}.
\end{equation*}%
By comparing Eq.~\eqref{Eqn:lc} and noticing that $\left\vert \boldsymbol{\nu }-%
\boldsymbol{\tilde{\nu}}\right\vert =\mathcal{O}(h),$ we have the
estimation for coefficients,%
\begin{equation*}
\left\vert a^{L}-\tilde{a}^{L}\right\vert =\mathcal{O}(h )\text{
and }\left\vert a^{R}-\tilde{a}^{R}\right\vert =\mathcal{O}(h ).
\end{equation*}

Next, we estimate the error between the analytic directional derivative $%
\frac{\partial u}{\partial \boldsymbol{\nu }^{L}}$ and the numerical
estimation $\frac{\Delta u}{\Delta \boldsymbol{\tilde{\nu}}^{L}}:=\frac{%
u\left( x^{L}\right) -u\left( x^{B}\right) }{\left\vert
x^{L}-x^{B}\right\vert }$. Let $\vec{z}=\left( z_{1},\ldots ,z_{d}\right) $
denote the $d$-dimensional ($d=2$) geodesic normal coordinate of $x_L$ defined by an
exponential map $\exp _{x^{B}}:T_{x^{B}}M\rightarrow M$ then $\vec{z}$\
satisfies
\begin{equation*}
\vec{z}=\ell \boldsymbol{\nu }^{L}=\ell \gamma _{L}^{\prime }(0)\text{ and }%
\exp _{x^{B}}\vec{0}=x^{B}\text{, }\exp _{x^{B}}\vec{z}=x^{L},
\end{equation*}%
where $\ell ^{2}=\ell ^{2}\left\vert \gamma _{L}^{\prime }(0)\right\vert
^{2}=\left\vert \vec{z}\right\vert ^{2}=\sum_{i=1}^{d}z_{i}^{2}$. We also
define $\hat{u}(\vec{z}):=u(\exp _{x^{B}}\vec{z})=u(x^{L})$, such that, $%
\hat{u}(\vec{0})=u\left( x^{B}\right) .$ With this definition, we have the
following Taylor's expansion,
\begin{equation*}
\hat{u}(\vec{z})=\hat{u}(\vec{0})+\sum_{i=1}^{d}z_{i}\frac{\partial \hat{u}(%
\vec{0})}{\partial z_{i}}+\frac{1}{2}\sum_{i,j=1}^{d}z_{i}z_{j}\frac{%
\partial ^{2}\hat{u}(\vec{0})}{\partial z_{i}\partial z_{j}}+\mathcal{O}%
(\ell ^{3}),
\end{equation*}%
which is equivalent to,
\begin{equation}
u(x^{L})=u(x^{B})+\frac{\partial u}{\partial \boldsymbol{\nu }^{L}}%
(x^{B})\ell +\frac{1}{2}\left( \boldsymbol{\nu }^{L}\right) ^{\top
}H(u(x^{B}))\boldsymbol{\nu }^{L}\ell ^{2}+\mathcal{O}(\ell ^{3}),\notag%\label{taylorexpansion}
\end{equation}%
by noticing that $\vec{z}=\ell \boldsymbol{\nu }^{L}$ is the normal coordinate.
This is just a Taylor expansion of function $u$ along a geodesic $\gamma
_{L}(\ell )$. Here, $H$ denotes the $d\times d$-dimensional Hessian matrix
whose components are $\nabla _{i}\nabla _{j}u(x^{B})$, where $\nabla _{i}$
denote the covariant derivative in the $i$th direction. Following the proof in Proposition %
\ref{prop:diren}, we have $|x^{L}-x^{B}|^{-1}=\ell ^{-1}\left( 1+\mathcal{O}%
(\ell ^{2})\right) .$ Then, we have the error between the analytic $\frac{%
\partial u}{\partial \boldsymbol{\nu }^{L}}$ and the numerical $\frac{\Delta
u}{\Delta \boldsymbol{\tilde{\nu}}^{L}}$
\begin{eqnarray*}
\frac{u(x^{L})-u(x^{B})}{|x^{L}-x^{B}|} &=&\left( \frac{\partial u}{\partial
\boldsymbol{\nu }^{L}}(x^{B})\ell +\frac{1}{2}\left( \boldsymbol{\nu }%
^{L}\right) ^{\top }H(u(x^{B}))\boldsymbol{\nu }^{L}\ell ^{2}+\mathcal{O}%
(\ell ^{3})\right) \ell ^{-1}\left( 1+\mathcal{O}(\ell ^{2})\right) \\
&=&\frac{\partial u}{\partial \boldsymbol{\nu }^{L}}(x^{B})+\mathcal{O}(\ell
).
\end{eqnarray*}%
One can follow the same steps and deduce for the "right" $x^{R}$,
\begin{equation}
\frac{u(x^{R})-u(x^{B})}{|x^{R}-x^{B}|}=\frac{\partial {u}}{\partial
\boldsymbol{\nu }^{R}}(x^{B})+\mathcal{O}(\ell ), \notag%\label{errorrate1D}
\end{equation}%
where we have introduced an arc-length $\ell $ for the geodesic distance
between $x^{R}$ and $x^{B}$. Since $\ell =\mathcal{O}(h)$, the remainder is
of order-$h$.

Finally, we obtain the result:
\begin{eqnarray*}
\left\vert \frac{\partial u}{\partial \boldsymbol{\nu }}\left( x^{B}\right) -%
\frac{\Delta u}{\Delta \boldsymbol{\tilde{\nu}}}\left( x^{B}\right)
\right\vert &:=&\left\vert \frac{\partial u}{\partial \boldsymbol{\nu }}%
\left( x^{B}\right) +\tilde{a}^{L}\frac{u\left( x^{L}\right) -u\left(
x^{B}\right) }{\left\vert x^{L}-x^{B}\right\vert }+\tilde{a}^{R}\frac{%
u\left( x^{R}\right) -u\left( x^{B}\right) }{\left\vert
x^{R}-x^{B}\right\vert }\right\vert \\
&\leq &\left\vert a^{L}\frac{\partial u}{\partial \boldsymbol{\nu }^{L}}%
(x^{B})-\tilde{a}^{L}\frac{u\left( x^{L}\right) -u\left( x^{B}\right) }{%
\left\vert x^{L}-x^{B}\right\vert }\right\vert +\left\vert a^{R}\frac{%
\partial u}{\partial \boldsymbol{\nu }^{R}}(x^{B})-\tilde{a}^{R}\frac{%
u\left( x^{R}\right) -u\left( x^{B}\right) }{\left\vert
x^{R}-x^{B}\right\vert }\right\vert \\
&\leq &\left\vert a^{L}-\tilde{a}^{L}\right\vert \left\vert \frac{\partial u%
}{\partial \boldsymbol{\nu }^{L}}(x^{B})\right\vert +\left\vert \tilde{a}%
^{L}\right\vert \left\vert \frac{\partial u}{\partial \boldsymbol{\nu }^{L}}%
(x^{B})-\frac{u\left( x^{L}\right) -u\left( x^{B}\right) }{\left\vert
x^{L}-x^{B}\right\vert }\right\vert \\
&&+\left\vert a^{R}-\tilde{a}^{R}\right\vert \left\vert \frac{\partial u}{%
\partial \boldsymbol{\nu }^{R}}(x^{B})\right\vert +\left\vert \tilde{a}%
^{R}\right\vert \left\vert \frac{\partial u}{\partial \boldsymbol{\nu }^{R}}%
(x^{B})-\frac{u\left( x^{R}\right) -u\left( x^{B}\right) }{\left\vert
x^{R}-x^{B}\right\vert }\right\vert \\
&=&\mathcal{O}(h).
\end{eqnarray*}
\end{proof}

%\bibliographystyle{unsrt}
%\bibliography{arxiv_kme}

\end{document}